\documentclass{memo-l}
\usepackage{amssymb,verbatim,graphicx}
\usepackage{graphics,epsfig,psfrag} 
\usepackage{pb-diagram,graphics,epsfig,psfrag}
\usepackage{color}
\usepackage{hyperref}
\definecolor{darkred}{rgb}{0.5,0,0}
\definecolor{darkgreen}{rgb}{0,0.5,0}
\definecolor{darkblue}{rgb}{0,0,0.5}
\hypersetup{ colorlinks,
linkcolor=darkblue,
filecolor=darkgreen,
urlcolor=darkred,
citecolor=darkblue }
\usepackage[displaymath,mathlines,pagewise]{lineno}

\newtheorem{theorem}{Theorem}[chapter]

\newtheorem{corollary}[theorem]{Corollary}

\newtheorem{conjecture}[theorem]{Conjecture}

\newtheorem{proposition}[theorem]{Proposition}
\newtheorem{lemma}[theorem]{Lemma}

\theoremstyle{definition}
\newtheorem{definition}[theorem]{Definition}
\theoremstyle{remark}
\newtheorem{remark}[theorem]{Remark}
\newtheorem{example}[theorem]{Example}

\numberwithin{section}{chapter}
\numberwithin{equation}{chapter}

\newcommand\A{\mathcal{A}}
\newcommand\M{\mathcal{M}}

\renewcommand\S{\mathcal{S}}

\newcommand{\T}{\mathcal{T}}
\newcommand{\J}{\mathcal{J}}

\newcommand{\U}{\mathcal{U}}

\newcommand{\F}{\mathcal{F}}
\newcommand{\N}{\mathbb{N}}
\newcommand{\R}{\mathbb{R}}
\renewcommand{\H}{\mathbb{H}}
\newcommand{\RR}{\mathcal{R}}
\newcommand{\WW}{\mathcal{W}}
\newcommand{\C}{\mathbb{C}}
\newcommand{\CC}{C\kern-1.3ex|}
\newcommand{\cC}{\mathcal{C}}

\newcommand{\bB}{\mathcal{B}}

\newcommand{\Z}{\mathbb{Z}}
\newcommand{\Q}{\mathbb{Q}}

\newcommand{\ddt}{\frac{d}{dt}}

\newcommand{\dds}{\frac{d}{ds}}

\newcommand{\ppt}{\frac{\partial}{\partial t}}

\renewcommand{\P}{\mathbb{P}}
\newcommand{\PP}{\mathcal{P}}

\newcommand\lie[1]{\mathfrak{#1}}

\newcommand{\h}{\lie{h}}
\newcommand{\g}{\lie{g}}

\renewcommand{\t}{\lie{t}}

\newcommand{\so}{\lie{so}}

\newcommand{\on}{\operatorname}
\newcommand{\thin}{\on{thin}}
\newcommand{\thick}{\on{thick}}
\newcommand{\Crit}{\on{Crit}} 
\newcommand{\Critval}{\on{Critval}} 
\newcommand{\univ}{\on{univ}}

\newcommand{\ainfty}{{$A_\infty$\ }}

\newcommand{\dist}{\on{dist}}

\newcommand{\pre}{{\on{pre}}}
\newcommand{\triv}{{\on{triv}}}

\newcommand{\dual}{\vee}

\newcommand{\ab}{\on{ab}}

\newcommand{\Edge}{\on{Edge}}

\newcommand{\Lag}{\on{Lag}}

\newcommand{\loc}{{\on{loc}}}
\newcommand{\Ver}{\on{Vert}}
\newcommand{\Ve}{\on{Vert}}

\newcommand\B{\mathcal{B}}

\newcommand{\End}{\on{End}}

\newcommand{\de}{\delta}

\newcommand{\Aut}{ \on{Aut} }

\newcommand{\Ad}{ \on{Ad} }

\newcommand{\Hom}{ \on{Hom}}
\newcommand{\Ind}{ \on{Ind}}
\renewcommand{\ker}{ \on{ker}}
\newcommand{\coker}{ \on{coker}}

\newcommand{\Spin}{ \on{Spin}}

\newcommand{\Vol}{  \on{Vol}}
\newcommand{\diag}{  \on{diag}}

\newcommand{\codim}{\on{codim}}

\newcommand\dirac{/\kern-1.2ex\partial} 
\newcommand\qu{/\kern-.7ex/} 
\newcommand\lqu{\backslash \kern-.7ex \backslash} 

\newcommand\dr{r_+ \kern-.7ex - \kern-.7ex r_-}
 
\renewcommand{\d}{{\on{d}}}
\newcommand{\ol}{\overline}
\newcommand{\olp}{\ol{\partial}}

\newcommand\Phinv{\Phi^{-1}}
\newcommand\phinv{\phi^{-1}}

\newcommand\eps{\epsilon}

\newcommand{\f}{\frac}
\newcommand{\lan}{\langle}
\newcommand{\ran}{\rangle}
\newcommand{\hh}{{\f{1}{2}}}

\newcommand{\qq}{{\f{1}{4}}}

\newcommand{\ti}{\tilde}

\newcommand\cE{\mathcal{E}}
\newcommand\E{\mathcal{E}}
\newcommand\cL{\mathcal{L}}
\newcommand\cF{\mathcal{F}}

\newcommand\cO{\mathcal{O}}
\newcommand\cG{\mathcal{G}}
\newcommand\cR{\mathcal{R}}
\newcommand\cP{\mathcal{P}}
\newcommand\cT{\mathcal{T}}
\newcommand\TT{\mathcal{T}}

\newcommand\cI{\mathcal{I}}

\renewcommand{\ss}{{\on{ss}}}

\newcommand\mE{\mathcal{E}}

\newcommand\curv{\on{curv}}

\newcommand\Map{\on{Map}}
\newcommand\rank{\on{rank}}

\newcommand\Pic{\on{Pic}}
\newcommand\Vect{\on{Vect}}
\newcommand\ul{\underline}
\newcommand\mO{\mathcal{O}}
\newcommand\G{\mathcal{G}}

\newcommand\Ker{\on{Ker}}
\newcommand\grad{\on{grad}}
\newcommand\reg{{\on{reg}}}

\newcommand\llabel[1]{{}} 
\newcommand\bdefn{\begin{definition}}
\newcommand\edefn{\end{definition}}
\newcommand\bea{\begin{eqnarray*}}
\newcommand\eea{\end{eqnarray*}}
\newcommand\bcv{\left[ \begin{array}{r} }
\newcommand\ecv{\end{array} \right] }

\newcommand\bma{\left[ \begin{array}{l} }
\newcommand\ema{\end{array} \right]}
\newcommand\ben{\begin{enumerate}}
\newcommand\een{\end{enumerate}}
\newcommand\beq{\begin{equation}}
\newcommand\eeq{\end{equation}}
\newcommand\bex{\begin{example}}
\newcommand\bsj{\left\{ \begin{array}{rrr} }
\newcommand\esj{\end{array} \right\}}

\newcommand\Id{\on{Id}}
\newcommand\XX{\mathbb{X}}
\newcommand\DD{\mathbb{D}}

\newcommand\Fuk{\on{Fuk}}

\newcommand\eex{\end{example}}
\newcommand\crit{{\on{crit}}}

\newcommand\sx{*\kern-.5ex_X}

\newcommand\white{\includegraphics[width=.05in]{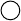}}

\newcommand\black{\includegraphics[width=.05in]{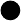}}
\newcommand\whitet{\includegraphics[width=.05in]{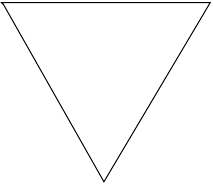}}
\newcommand\greyt{\includegraphics[width=.05in]{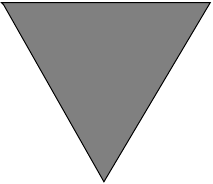}}
\newcommand\blackt{\includegraphics[width=.05in]{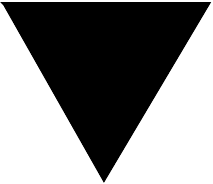}}
\newcommand\lldots{\hbox to 1em{.\hss.\hss.}}

\usepackage{braket}
\usepackage{chngcntr} 
\counterwithin{figure}{chapter} 

\begin{document}

\title[Floer cohomology and flips]{Floer cohomology and flips}

\author{Fran\c{c}ois Charest}

\address{Department of Mathematics, Barnard College - Columbia University,
MC 4433, 2990 Broadway, 
New York, NY 10027}

\email {charest@math.columbia.edu}

\author{Chris T. Woodward}

\address{Mathematics-Hill Center,
Rutgers University, 110 Frelinghuysen Road, Piscataway, NJ 08854-8019,
U.S.A.}  

\email{ctw@math.rutgers.edu}

\thanks{This work was partially supported by NSF grants DMS 1207194
  and 
  DMS 1711070  and a Simons Fellowship. }

\date{July 2, 2022}


\maketitle

\begin{abstract} 
  We show that blow-ups or reverse flips (in the sense of the minimal
  model program) of rational symplectic manifolds with point centers
  create Floer-non-trivial Lagrangian tori. \label{abschange} These
  results are part of a conjectural decomposition of the Fukaya
  category of a compact symplectic manifold with a singularity-free
  running of the minimal model program, analogous to the description
  of Bondal-Orlov \cite{bondal:der} and Kawamata \cite{kaw:der} of the
  bounded derived category of coherent sheaves on a compact complex
  manifold.
\end{abstract}

\tableofcontents

\chapter{Introduction} 

Lagrangian Floer cohomology was introduced in \cite{floer:lag} as a
version of Morse theory for the space of paths from a Lagrangian
submanifold to itself.  Although the theory was introduced decades
ago, it has been far from clear which Lagrangian submanifolds have
well-defined or non-trivial Floer cohomology.  Similarly, one would
like to know whether a compact symplectic manifold contains
Lagrangians with non-trivial Floer cohomology.  The purpose of this
paper is to describe a method for producing Floer-non-trivial
Lagrangians via the surgeries that appear in the minimal model
program, namely {\em flips} and {\em blow-ups}.  In particular, if one
has knowledge of a sufficiently nice minimal model program (mmp) for
the symplectic manifold (in a suitably modified sense) then one can
read off a list of Floer non-trivial Lagrangians that conjecturally
generate the Fukaya category.  At each stage in the mmp running the
Lagrangians that break off have the same slope, in the sense of ratio
of the Maslov class to the disks of minimal area.  In this sense any
such mmp running provide a filtration of the Fukaya category analogous
to the Harder-Narasimhan filtration of a complex vector bundle.  The
results are analogous to those of Bondal-Orlov \cite{bondal:der} and
Kawamata \cite{kaw:der} on the behavior of the bounded derived
category of coherent sheaves on a compact complex manifold under
flips.  The approach taken here is inspired by the work of
Fukaya-Oh-Ohta-Ono \cite{fooo:toric1}, \cite{fooo:toric2} on the toric
case.  Fukaya et al \cite[Theorem 1.4]{fooo:toric1} prove the existence of a Floer
non-trivial torus in any rational toric manifold by shrinking the moment polytope.  Our results
describe a similar mechanism that produces a Floer non-trivial torus
in birationally-Fano smooth projective varieties, by shrinking the
manifold via a topological version of the Kahler-Ricci flow.  The
connection of the location of the Floer non-trivial tori in Fukaya et
al \cite{fooo:toric1} to the minimal model program was noted in work
with Gonz\'alez \cite{gw:surject}.

To describe the results more precisely, recall
that a running of the mmp for a smooth birationally-Fano projective
variety $X$ is a sequence of birational varieties
$X = X_0,X_1,\ldots, X_k$ such that each $X_{i+1}$ is obtained from
$X_i$ by an {\em mmp transition}.  The simplest mmp transition is a
blow-down of a divisor, or more generally a divisorial contraction.
The other operations are {\em flips} which are birational isomorphisms
on the complement of codimension four exceptional loci.  The minimal
model program ends (in the birationally-Fano case) with a {\em Mori
  fibration} which here means a fibration with Fano fiber.  Each
blow-up or flip $X_i \dashrightarrow X_{i+1}$ has a {\em center}
$Z_{i}$ which is a subquotient of both $X_i$ and $X_{i+1}$.  A flip,
in the cases considered here, replaces a weighted-projective bundle
$\P(N_i^+) \to Z_i$ over the center with another weighted-projective
bundle $\P(N_{i+1}^-) \to Z_i$:
\[ \begin{diagram} 
\node{X_i} 
\node{ \P(N_i^+) } \arrow{w}  \arrow{se} 
\node[2]{\P(N_{i+1}^-)}  \arrow{e}  \arrow{sw}
\node{X_{i+1}}
  \\
\node[3]{Z_i}
\end{diagram} .\]
A {\em weighted-projective bundle} means a bundle whose fibers
are the git quotient of a vector space by an action of the
multiplicative group $\C^\times = \C - \{ 0 \}$.  Each
weighted-projective bundle $\P(N_i^\pm)$ is obtained by removing the
zero section $0$ from a vector bundle $N_i^\pm$ over the center $Z_i$
and taking the quotient of $N_i^\pm - 0$ by a fiber-wise linear
$\C^\times$-action with only the zero section $0$ as fixed point set.
In the case of a divisorial contraction, the map
$\P(N_{i+1}^-) \to Z_i$ would be an isomorphism, while in the case of
a Mori fibration, one may think of the center as the result of the
transition: $\P(N_i^+) \cong X_i, Z_i \cong X_{i+1} .$ Sometimes
singularities in the spaces $X_i$ are unavoidable; however, in good
cases (such as toric varieties) one may assume that the $X_i$'s are
smooth orbifolds.  In this case we say that $X$ has a {\em smooth
  running} of the mmp.  Existence of mmp runnings is known for
varieties of low dimension and in many explicit examples.  For toric
varieties, mmp runnings exist by work of Reid \cite{reid:decomp}.

Each of the minimal model transitions has a symplectic analog.  The
local change in the symplectic manifold follows the algebraic model,
and the symplectic class varies in the anticanonical direction.  The
symplectic analog of a blow-up was studied in McDuff \cite{mc:ex},
without the anticanonical variation.  Note that the anticanonical
variation of the symplectic class in our definition of blow-up allows
blow-ups to be infinitely large, that is, there exist reverse
symplectic runnings of the mmp with a single blow-up transition and
infinite duration.  The definitions of symplectic flip and symplectic
Mori fibration are similar, see Definition \ref{sflip}.  We say that a
symplectic manifold $X_+$ is obtained from a symplectic manifold $X_-$
{\em by a symplectic blow-up resp. flip} if there exists a symplectic
running of the mmp with a single transition given by a blow-up
resp. flip with point center.  In this case we say that the {\em
  multiplicity} of the transition is
\[m = \dim(QH(X_+)) -
\dim(QH(X_-)) .\]
A simple argument using Mayer-Vietoris implies that $m$ is always
positive in the cases considered in this paper.

The version of Floer theory that we use requires the notion of {\em
  bounding cochain} of Fukaya-Oh-Ohta-Ono \cite{fooo}.  We begin with
a brief discussion of what we mean by the Fukaya algebra $CF(L)$ of a
Lagrangian $L \subset X$.  There are several foundational systems
which at the moment are not known to be equivalent (or in some cases,
completely written down.)  Working in the Morse cochain model, the
structure coefficients in these Fukaya algebras count rigid elements
$u: C \to X$ in the moduli space of {\em treed holomorphic disks} with
Lagrangian boundary condition given by $L$.  We regularize these
moduli spaces using a {\em stabilizing divisor} as in Cieliebak-Mohnke
\cite{cm:trans}, which means a codimension two symplectic hypersurface
$D \subset X$ which meets any non-constant holomorphic sphere in $X$
in finitely many but at least three points, and making the Lagrangian
$L$ exact in the complement of the hypersurface $D$.  The
intersections $u^{-1}(D)$ of any pseudoholomorphic disk or sphere
$u: C \to X$ stabilize the domain.  Domain-dependent almost complex
structures and Morse functions suffice to make all moduli spaces of
low expected dimension regular for generic perturbations.  While this
foundational scheme is somewhat less general than the other
approaches, it requires no discussion of virtual fundamental classes
and so makes the necessary foundational arguments substantially
shorter.

The Fukaya algebra also depends on the choice of local system with
values in the units in the Novikov ring.  Let $\Lambda$ denote the
universal Novikov field and $\Lambda^\times \subset \Lambda$ the group
of elements with vanishing valuation with respect to the formal
variable.  For a connected Lagrangian $L$ denote by
$\RR(L) = \Hom(\pi_1(L), \Lambda^\times)$ the group of
$\Lambda^\times$-local systems.  For any compact oriented spin
Lagrangian submanifold $L$ equipped with a local system $y \in \RR(L)$
there is a {\em strictly unital Fukaya algebra} $CF(L,{y})$, defined
by counting {\em weighted treed pseudoholomorphic disks}, independent
of all choices up to homotopy equivalence.  The resulting moduli space
of solutions $MC(L,{y})$ to the {\em projective Maurer-Cartan equation} is
independent of all choices and for each solution $b \in MC(L,{y})$ the
{\em Floer cohomology group} $HF(L,{y},b)$ is independent of the
choice of perturbations up to isomorphism.  The main result of this
paper is the existence of Floer non-trivial Lagrangians near mmp
transitions:

\begin{theorem} \label{result} Suppose that $X_+$ is a compact
  rational symplectic manifold obtained from a compact rational
  symplectic manifold $X_-$ by a reverse simple flip or blow-up with
  trivial center with multiplicity $m$, with sufficiently small
  exceptional locus.  In a contractible neighborhood of the
  exceptional locus there exists a Lagrangian torus $L \subset X_+$
  with $m$ distinct local systems and Maurer-Cartan solutions
\[ {y}_k \in \RR(L), \quad 
b_k \in {MC}(L,{y}_k), \quad 
 k =1\ldots, m, \] 
 with non-trivial Floer cohomology
\[  HF(L,b_k,{y}_k) \cong H(L,\Lambda) \neq \{ 0 \} . \] 
\end{theorem} 

Existence of Floer non-trivial Lagrangians in compact symplectic
manifolds is known only in a special cases.  For symplectic manifolds
with anti-symplectic involutions, Floer cohomology of the fixed point
sets is studied in Fukaya-Oh-Ohta-Ono \cite{fooo:anti}.  Cho-Oh
\cite{chooh:toric} and Fukaya-Oh-Ohta-Ono \cite{fooo:toric1} show that
any compact toric manifold contains at least one moment fiber with
non-trivial Floer cohomology.  In semi-Fano hypersurfaces, at least
one Lagrangian has non-trivial Floer cohomology by work of Seidel
\cite{se:ho} and Sheridan \cite{sh:hmsfano}.  Related discussions of
blow-ups can be found in Smith \cite{smith:pq}.

\label{removedthm}
We conjecture that the Lagrangians appearing in Theorem \ref{result},
applied iteratively, split-generate the Fukaya category.  Denote by
$\Fuk(X)$ the Fukaya category of $X$, independent of all choices up to
homotopy equivalence of curved \ainfty algebras.  For any element $w$
in the universal Novikov ring denoted by $\Fuk(X,w)$ the category
whose objects are pairs $(L,b)$ where $L$ is an object of $X$ and $b$
is a weakly bounding cochain satisfying the projective Maurer-Cartan
equation with value $w$ and whose morphism spaces are explained in
\cite{fooo}.  Let $D^\pi \Fuk(X,w)$ be the idempotent completion of
the derived Fukaya category and define
\[ D^\pi \Fuk(X) = \bigsqcup_w D^\pi \Fuk(X,w) .\]

\begin{conjecture} \label{mainconj} Suppose
  $X = X_0,X_1,\ldots, X_k = X'$ is a sequence of compact symplectic
  manifolds such that each $X_{i+1}$ is obtained from $X_i$ by an mmp
  transition with center $Z_i$.  Then the idempotent-closure of the
  derived Fukaya category $D^\pi \Fuk(X_0)$ is isomorphic to the
  disjoint union of categories of centers of the mmp transitions:
\begin{equation} \label{decomp}
D^\pi \Fuk(X) \cong D^\pi \Fuk(X') \sqcup \bigsqcup_{i=1}^k D^\pi
\Fuk(Z_i)^{m_i}
\end{equation}
where
\[m_i = \dim(QH(X_{i})) - \dim(QH(X_{i+1})) \]
is the {\em multiplicity} of the $i$-th mmp transition given as the
difference of the quantum cohomology rings $QH(X_{i+1}), QH(X_i)$.
\end{conjecture}  

\noindent By the disjoint union of categories we mean the category
whose objects are the disjoint union, and whose morphism groups
between elements of different sets in the disjoint union are trivial.
In the case of toric manifolds, the Lagrangians have been announced to
generate the Fukaya category by Abouzaid-Fukaya-Oh-Ohta-Ono, based on
the technique in Abouzaid \cite{abouzaid:gen}.  Since our evidence is
mostly in the birationally-Fano case, it is possible that the
conjecture \ref{mainconj} needs some similar restriction.  A special
case is shown in Venugopalan-Woodward-Xu \cite{vwx}.

The decomposition \eqref{decomp} should be related to the
decomposition by quantum multiplication of the first Chern class in
the following sense: Recall that quantum multiplication by the first
Chern class $c_1(X) \in H^2(X)$ induces an endomorphism
\[c_1(X) \star :QH(X) \to QH(X), \quad \alpha \mapsto c_1(X) \star \alpha  .\]  
The eigenvalues $\lambda_1,\ldots,\lambda_k \in \Lambda$ of
$c_1(X) \star$ that are non-zero have a well-defined $q$-valuation
$\on{val}_q(\lambda_i) \in \R$, given by the exponent of the leading
order term.  We expect that these valuations (when non-zero) are the
transition times in the minimal model program.  The conjecture
\ref{mainconj} provides a method of attack on another conjecture of
Kontsevich, that the Fukaya category $\Fuk(X)$ of a symplectic
manifold $X$ is expected to be a categorification of the quantum
cohomology $QH(X)$, at least in many cases, in the following sense
\cite{kon:hom}: there should be an isomorphism
\begin{equation} \label{konts} H(\Fuk(X)) := \bigoplus_{w}
  H(\Fuk(X,w)) \cong QH(X) \end{equation}
from the Hochschild cohomology $H(\Fuk(X))$ of the Fukaya category to
the quantum cohomology $QH(X)$; here $\Fuk(X,w)$ is the summand of the
Fukaya category with curvature $w$.  If the centers of the mmp
transitions have the property \eqref{konts} then the manifold should
have the same property and hence also the quantum cohomology
decomposes into summands corresponding to transitions.  

The idea that the quantum cohomology of a symplectic manifold should
behave well under minimal model transitions is not new; see for
example Ruan \cite{ruan:surgery}, Lee-Lin-Wang \cite{lee:fmi}, Bayer
\cite{bayer:blowups}, Acosta-Shoemaker \cite{as:qctb}; we also heard
related results in talks of H. Iritani.
\label{minorp} \llabel{minor} See Li \cite{li:disj} for related result
in the case of open symplectic manifolds.  Including  bulk
deformations one should obtain a filtration rather than a splitting of
the Fukaya category, analogous to the situation in complex geometry
where mmp transitions provide a semi-orthogonal decomposition.  

The analogy with results of Bondal-Orlov \cite{bondal:der} and
Kawamata \cite{kaw:der} is somewhat mysterious, since the mirror of
the minimal model program on the algebraic side is not expected to be
the symplectic version of the minimal model program considered here.
Rather, the minimal model program on the symplectic side should
correspond under mirror symmetry to a deformation of the mirror
potential by a change of variables in the potential
\[  W \mapsto \phi_t^* W, \ \phi_t(y) := yq^{-tc_1(X)} . \] 
Because the mirror should be understood as a formal completion at
$q = 0$, such a deformation changes the mirror by eliminating some
critical loci via the formal completion.  This flow
should be related to the renormalization group flow in the physics of
non-linear sigma models as discussed in, for example, Hori-Vafa
\cite{ho:mi}.

The proof of Theorem \ref{result} combines a neck-stretching argument
with local toric computations.  Near the exceptional locus the reverse
flip is toric.  The results in the toric case imply the existence of a
Floer-non-trivial torus in the corresponding toric variety.  This
Lagrangian torus collapses at the singularity of a running of the
minimal model program; in this sense the Lagrangian is a ``vanishing
cycle''.  The exceptional locus of the flip is separated from the rest
of the symplectic manifold by a coisotropic submanifold fibered over a
toric variety.  Stretching the neck, as in symplectic field theory,
produces a homotopy-equivalent {\em broken Fukaya algebra} associated
to the Lagrangian which counts maps to the pieces combined with Morse
trajectories on the toric variety.  Similar arguments are common in
the literature, for example, in the work of Iwao-Lee-Lin-Wang
\cite{lee:fmi}, \cite{lee:flop}.  One computes explicitly, using a
Morse function arising as component of a moment map, that the
resulting broken Fukaya algebra is weakly unobstructed and that the
broken Floer cohomology of the Lagrangian is non-vanishing.  Moduli
spaces of pseudoholomorphic disks in toric varieties with invariant
constraints are never isolated, and this implies the unobstructedness
of Floer cohomology.  The classification of disks of small area
implies the existence of a critical point of the potential.

As applications of Theorem \ref{result} we show in Chapter
\ref{flexamples} that various symplectic manifolds contain Hamiltonian
non-displaceable Lagrangian tori.  For example, in the case of toric
manifolds we reproduce in Chapter \ref{trans} some of the results of
Fukaya-Oh-Ohta-Ono \cite{fooo:toric1}.  We also show existence of
Floer non-trivial tori in symplectic quotients of products of
two-spheres (Lemma \ref{reglab}) and moduli spaces of rank two
parabolic bundles in genus zero (Lemma \ref{reglabm}).

We thank Kai Cieliebak, Octav Cornea, Sheel Ganatra, Fran\c{c}ois
Lalonde, Dusa McDuff, and Klaus Mohnke for helpful discussions.  We
also thank Sushmita Venugopalan and Guangbo Xu for pointing out
mistakes in earlier versions.

\chapter{Symplectic flips}

\label{examplessec} 

The goal of the minimal model program (mmp for short) is to classify
algebraic varieties by finding a {\em minimal model} in each
birational equivalence class, see \cite{km:mmp}.  In the
birationally-Fano case, the hoped-for minimal model is a {\em Mori
  fibration}: a fibration with Fano fiber.  In general, singularities
play an important role in the minimal model program.  Here we assume
that the variety admits a {\em singularity-free} running of the mmp.
While this case is considered somewhat trivial by algebraic geometers,
it includes a number of smooth projective varieties whose symplectic
geometry is poorly understood.  We introduce a symplectic version of
the mmp given as a path (with singularities corresponding to
surgeries) of symplectic manifolds.  By a suggestion of Song-Tian and
others \cite{song:krflow}, an mmp running is conjecturally equivalent
to running K\"ahler-Ricci flow with surgery; the paths in our
symplectic mmp are a ``topological version'' of the K\"ahler-Ricci
flow.

\section{Symplectic mmp runnings} 

Each mmp transition is a special kind of {\em birational
  transformation}.  For the sake of completeness we recall the
terminology.  Let $X_\pm$ be normal projective varieties.  A {\em
  rational map} from $X_+$ to $X_-$ is a Zariski dense subset
$U \subset X_+$ and a morphism $\phi: U \to X_-$; we write
$\phi: X_+ \dashrightarrow X_-$.  A rational map $\phi$ is a {\em
  birational equivalence} if $\phi$ has a rational inverse, that is, a
rational map $(V \subset X_-, \psi: V \to X_+ )$ to $X_+$ such that
the compositions
$\phi \circ \psi |_{ \psi^{-1}(U)}, \psi \circ \phi |_{\phi^{-1}(V)}$
are the identity on the domains of definition.  Birational
equivalences have a natural notion of composition, making birational
equivalence into an equivalence relation.

The minimal model program involves the following types of morphisms,
see for example Hacon-McKernan \cite{hacon:flips}.  The first two
types are birational equivalences.

\begin{definition} \label{trans} A {\em minimal model transition} of
  $X_+$ is one of the following three types:
\begin{enumerate} 
\item {\rm (Divisorial contractions)} A morphism $\tau: X_+ \to X_-$
  that is the contraction of a Cartier divisor (codimension one
  subvariety $Y \subset X_+$); a typical example is a blow-down.
\item {\rm (Flips)} Let $\tau_+: X_+ \to X_0$ be a birational
  morphism.  The morphism $\tau_+$ is {\em small} if and only if
  $\tau_+$ does not contract a divisor.  The morphism $\tau_+$ is a
  {\em flip contraction} if and only if the relative Picard number of
  $\tau_+$ is one.  The {\em flip} of $X_+$ is another small
  birational morphism $\tau_- : X_- \to X_0$ of relative Picard number
  one.
\item {\rm (Mori fibrations)} A Mori fibration is a fibration
  $\tau: X_+ \to X_-$ with Fano fiber, that is, a fibration whose
  relative anticanonical bundle $ K^{-1}_{X_+}/ \tau^* K^{-1}_{X_-}$
  is very ample on the fibers $\tau^{-1}(x) \subset X_+$.
\end{enumerate} 
\end{definition} 

Here the {\em relative Picard number} is the difference in Picard
numbers, that is, the difference in dimensions
\[ \Pic(\tau_+) = \dim(\Pic(X_+)) - \dim( \Pic(X_0) ) \]
in the moduli spaces $\Pic(X_+), \Pic(X_0)$ of line bundles.  The
relative Picard number is one if and only if every two curves
contracted by $\tau_-$ are numerical multiples of each other, that is,
define proportional linear functions on the space of degree two
cohomology classes.

\begin{definition} 
\begin{enumerate} 
\item A {\em running} of the mmp is a sequence
  of smooth projective varieties 
\[ X = X_0,\ldots, X_{k+1} \] 
such that
for $i = 0,\ldots, k$, $X_{i+1}$ is obtained from $X_i$ by a
  divisorial contraction or a flip, and
$X_{k+1}$ is obtained from $X_{k}$ by a fibration with Fano
  fiber (sometimes called Mori fibration).  
The variety $X_{k}$ may be called the {\em minimal model} of $X_0$.
The different minimal models that occur are related by the {\em
  Sarkisov program} \cite{hm:sark}.
\item An {\em extended running} of the mmp is a sequence of smooth
  projective varieties
  $X_0,\ldots, X_{k_1}, X_{k_1 + 1}, \ldots, X_{k_l + 1}$, where
  $k_1,\ldots, k_l$ is an increasing sequence of integers, such that
  $X_{i+1}$ is obtained from $X_i$ by a divisorial contraction or flip
  for $k \neq k_i$ and $X_{k_i + 1}$ is obtained from $X_{k_i}$ by a
  Mori fibration.
\end{enumerate} 
\end{definition}

It is expected that runnings of the mmp exist for all smooth
projective varieties.  This existence has proved up to dimension three
\cite{km:mmp}.  Existence of Mori fibration models is known, in the
birationally Fano case, by \cite[Corollary 1.3.2]{bchk}, but the
existence of runnings in higher dimension seems to be unknown.

\begin{example}
  The simplest example of the mmp occurs for toric surfaces.  Recall
  that any toric surface $X$ corresponds to a two-dimensional fan $\cC
  = \{ C \subset \R^2 \}$ of strictly convex cones.  Elementary
  combinatorics shows that if $\cC$ is complete and has more than four
  vertices then there always exists some invariant prime divisor $D_i$
  with $D_i^2 = -1$; blowing down $D_i$ one eventually reaches a
  Hirzebruch surface or projective plane which is the minimal model,
  see Audin \cite[Theorem VIII.2.9]{au:to}.  Note that toric surfaces
  often admit several runnings of the toric mmp, depending on the
  order in which the $-1$-curves are blown down.
\end{example} 

\begin{example}
\label{dp} 
Del Pezzo surfaces also provide elementary examples.  Let $dP_n$
denote a del Pezzo surface given by the blow up of $\P^2$ at $9-n$
generic points.  For $1 \leq n \leq 9$, $dP_n$ is Fano and admits a
running of the mmp with (in our notation) one transition, since $dP_n$
is itself a Mori fibration over a point.  However, $dP_n$ for $n < 8 $
admits multiple mmp runnings.  For example, for $dP_6$ one can blow
down a the exceptional curves
$E_i \subset dP_6, E_i.E_i = -1, i = 1,2,3$ created by the blow-ups in
any order; on the other hand, one may view $dP_6$ as the
thrice-blow-up of the {\em dual} projective plane via the Cremona
transformation.   Blowing down the exceptional curves
$E'_j, E'_j.E'_j = -1$ in any order gives another six runnings of the
mmp.  Taking the blow-ups to respect the toric structure $dP_6$ is
toric with moment polytope a hexagon shown in Figure \ref{corners}.
The $12$ runnings above correspond to the ways of ``restoring the
corners'' to make a triangle.  The dots in Figure \ref{corners}
correspond to the locations of the Floer-non-trivial Lagrangians
produced
by the main result. 
\end{example}

\begin{figure}[ht]
\includegraphics[height=1.5in]{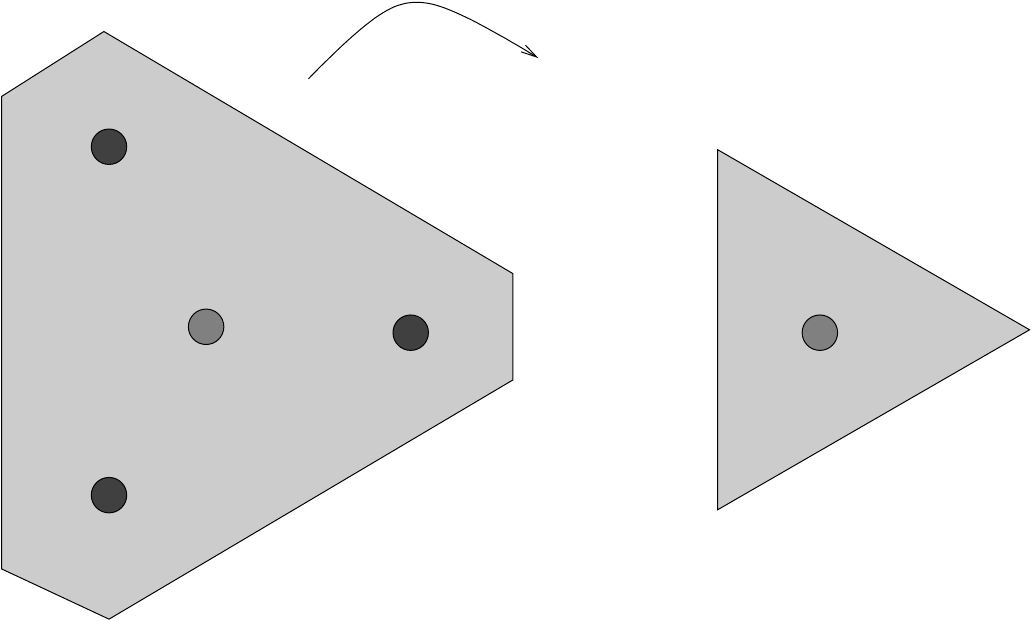} 
\caption{Blowing down $dP_6$ to $\P^2$} 
\label{corners}
\end{figure}

\begin{remark} \label{gitrun} Runnings of the mmp can often be
  obtained from geometric invariant theory as follows, as in Reid
  \cite{reid:flip} and Thaddeus \cite{th:fl}; see also \cite{br:ac}
  and \cite{do:va}.  Let $G$ be a complex reductive group and $X$ a
  smooth projective $G$-variety.  Recall that a {\em linearization} of
  $X$ is an ample $G$-line bundle $\widetilde{X}$.  Given such a
  linearization the geometric invariant theory quotient $X \qu G$ is
  the quotient of the semistable locus, defined as the set of points
  where an invariant section is non-vanishing,
\[X^{\on{ss}} := \{ x \in X \ | \ \exists k > 0 , \ s \in H^0(\widetilde{X}^k)^G, s(x) \neq
0 \}, \]
by the orbit equivalence relation 
\[ x_1 \sim x_2 \iff \ol{G x_1} \cap \ol{G x_2} \cap X^{\on{ss}} \neq
0 .\]
The git quotient $X \qu G$ depends only on the ray
$\{ \widetilde{X}^t, t > 0 \} \subset \Pic^G(X)$ generated by $\widetilde{X}$ in
the equivariant Picard group $\Pic^G(X)$.  That is, tensor powers of
$\widetilde{X}$ define the same git quotient, since if some section of
$\widetilde{X}^k$ is non-vanishing at $x \in X$ then so are all its tensor
powers.  In particular, git quotients are defined for ample elements
of the rational Picard group $\Pic^G_\Q(X) = \Pic^G(X) \otimes_\Z \Q$.

An path of linearizations gives rise to a path of git quotients with a
discrete set of transition times.  Let
$\widetilde{X} = \widetilde{X}_0$ be as above, and let
$\widetilde{X}_1$ be another $G$-equivariant line bundle.  The tensor
powers
\[\widetilde{X}_t = \widetilde{X}_0 \otimes \widetilde{X}_1^{t} \in \Pic_\Q(X), t \in
\Q \]
are ample for $t \in [0,T] \cap \Q$ for $T$ sufficiently small.  We
call $\ti{X}_0$ the {\em base point} of the variation of linearization
and $\ti{X}_1$ the {\em direction}.  Let
\[ X \qu_t G := X_t^{\ss} / \sim \]
be the family of geometric invariant theory quotients obtained from
the linearizations $\widetilde{X}_t$.  As the polarization varies the
git quotients undergo a sequence of divisorial contractions, flips
without the ampleness condition, and fibrations each obtained as
follows: The {\em transition times} are the set of values of $t$ where
the quotient $X \qu_t G$ has stabilizer groups of positive dimension:
\[ \cT := \{ t \in \R, \ \exists x \in X_t^{\ss}, \ \# G_x = \infty \}
.\]
Dolgachev-Hu \cite{do:va} and Thaddeus \cite{th:fl} reduce the study
of the wall-crossings to the case of circle actions as follows.
Define the {\em master space}
\begin{equation} \label{masterspace} X_{t_1,t_2} :=  \P(\widetilde{X}^{t_1} \oplus
\widetilde{X}^{t_2}) \qu G \end{equation} 
for some $t_1,t_2 \in [0,T]$ not in the set of transition times
$\cT$. \llabel{ttime}  \label{ttimep} 
The times $t_1,t_2$ may be taken to be
integers after replacing $\widetilde{X}$ with a tensor power, which does not
affect the computation.  The space $X_{t_1,t_2}$ is equipped with a
natural linearization, obtained by tensoring the relative hyperplane
bundle with the pull-back of $\widetilde{X}_0$.  The $\C^\times$-action on
the fibers induces a $\C^\times$-action on the quotient.
\[ X_{t_1,t_2} \qu \C^\times \cong X \qu_t  G\] 
for $t \in (t_1,t_2)$.

One can now check that variation of quotient produces a flip, with the
ampleness condition satisfied if the direction of the variation
$\widetilde{X}_1$ is the canonical bundle $K$.  In the circle group
case $G = \C^\times$ let $F \subset X_t^{\ss}$ be a component of the
fixed point set that is stable at time $t$.  Let
$\mu_1,\ldots, \mu_k \in \Z$ denote the weights of $\C^\times$ on the
normal bundle $N$ to $F$, and $N_i$ the weight space for weight
$\mu_i, i = 1,\ldots, k$.  Let
\[ N_\pm := \bigoplus_{\pm \mu_i > 0} N_i \subset N \]
denote the positive resp. negative weight subbundles of $N$.  For
$t_- < t < t_+$ with $t_\pm$ close to $t$ the Hilbert-Mumford
criterion and Luna slice theorem imply that the semistable locus
changes by replacing a variety isomorphic to $N_+$ with one isomorphic
to $N_-$, see \cite{th:fl}.  Hence $X \qu_{t_+} G$ is obtained from
$X \qu_{t_-} G$ by replacing the (weighted)-projectivized bundle
$N_+^\times/G$ of with $N_-^\times/G$:
\[ (X \qu_{t_-} G) \backslash (N_-^\times/G) \cong (X \qu_{t_+} G)
  \backslash (N_+^\times/G) .\]
  One checks easily from the local model that these morphisms are
  relatively $K$-ample resp. $-K$-ample over the center.  Thus, in the
  absence of singularities (that is, if $G$ acts freely for generic
  $t$) the spaces $X \qu_t G$ yield a smooth running of the minimal
  model program.\end{remark} 

\begin{remark}
  The symplectic version of these transitions can be described in
  terms of Morse theory of the moment map, as explained in
  Guillemin-Sternberg \cite{gu:bi}.  Let $(X,\omega)$ be a compact
  symplectic manifold equipped with Hamiltonian $U(1)$-action with
  proper moment map $\Phi:X \to \R$.  Let
\[ \Crit(\Phi) = \{ x \in X \ |  \d \Phi(x) = 0 \}, \quad 
\Critval(\Phi) = \Phi(\Crit(\Phi)) \]
denote the set of critical points resp. critical values of $\Phi$.
Given a critical value $c \in \Critval(\Phi)$, we denote by $c_\pm \in
\R$ regular values on either side of $c$, so that $c$ is the unique
critical value in $(c_-,c_+)$.
We suppose for simplicity that $\Phinv(c)$ contains a unique critical
point $x_0 \in X$, so that 
\[ \Critval(\Phi) \cap (c_-,c_+) = \{ c \}, \quad \Crit(\Phi) \cap
\Phinv(c) = \{ x_ 0 \}. \]
By the equivariant Darboux theorem, there exist Darboux coordinates
$z_1,\ldots, z_n$ near $x_0$ and {\em weights}
$\mu_1,\ldots,\mu_n \in \Z$ so that
\[ \Phi(z_1,\ldots, z_n) = c - \sum_{j=1}^n  \mu_j |z_j|^2/2 .\]
In particular, $\Phi$ is Morse and we denote by $W^{\pm}_{x_0}$ the
stable and unstable manifolds of $-\grad(\Phi)$ with respect to some
invariant metric.  The time $t$ gradient flow $\phi_t: X \to X$ of
$- \grad(\Phi)$ induces a diffeomorphism between level sets on the
complement of the stable and unstable manifolds:
\[ \Phinv(c_+) \backslash W^+_{x_0} \to \Phinv(c_-) \backslash
W^-_{x_0}, \quad x \mapsto \phi_t(x), \ t \ \text{such that }
\ \Phi(\phi_t(x)) = c_-
 .\]
Assuming the gradient vector field is defined using an invariant
metric one obtains an identification of symplectic quotients
except on the symplectic quotients of the stable and unstable manifolds:
\[ ( X \qu_{c_+} U(1) ) \backslash (W^+_{x_0}  \qu_{c_+} U(1))
\to ( X \qu_{c_-} U(1) ) \backslash (W^+_{x_0}  \qu_{c_-} U(1)) .\]
By the description from equivariant Darboux, one sees that the
symplectic quotients of the stable and unstable manifolds are weighted
projective spaces with weights given by the positive resp. negative
weights:
\[ (W^+_{x_0} \qu_{c_+} U(1)) \cong \P[ \pm \mu_i, \pm \mu_i > 0 ] .\]
Thus $X \qu_{c_-} U(1)$ is obtained from $X \qu_{c_+} U(1)$ by replacing
one weighted projective space with another.

The change in the symplectic class under variation of symplectic
quotient is described by Duistermaat-Heckman theory \cite{du:on}.  Let
$c_0 \in \R - \Critval(\Phi)$ be a regular value of $\Phi$.  Consider
the product $\Phinv(c_0) \times [c_-,c_+]$ for $c_\pm$ close to $c_0$.
Let $\pi_C, \pi_\R$ be the projections on the factors of
$ \Phinv(c_0) \times [c_-,c_+]$.  Choose a connection one-form and
denote its curvature
\[\alpha \in \Omega^1(\Phinv(c_0))^{U(1)}, \quad \curv(\alpha) \in
\Omega^2(X \qu_{c_0} U(1)).\]
Define a closed two-form on $ \Phinv(c_0) \times [c_-,c_+]$ by
\[ \omega_0 = \pi_C^* \omega_c + \d (\alpha, \pi_\R - c_0)  \in
\Omega^2(\Phinv(c_0) \times [c_-,c_+] ).\]
For $c_-,c_+$ sufficiently close to $c_0$, $\omega_0$ is symplectic
and has moment map given by $\pi_\R$.  By the coisotropic embedding
theorem, for $c_-,c_+$ sufficiently small there exists an equivariant
symplectomorphism 
$ \psi: U \to \Phinv(c_0) \times [c_-,c_+] $
of a neighborhood $U$ of $\Phinv(c_0)$ in $X$ with $ \Phinv(c_0)
\times [c_-,c_+]$.  Hence the symplectic quotients $X \qu_{c} U(1)$
for $c $ close to $c_0$ are diffeomorphic to $X \qu_{c_0} C$, with
symplectic form $\omega_c + (c - c_0) \curv(\alpha)$. 
This ends the Remark.
\end{remark} 

The following describes a symplectic version of an mmp running.  It
requires that the family of symplectic manifolds be given locally by
variation of symplectic quotient, and globally the change in
symplectic class is the canonical class.

\begin{definition} \label{sflip}
\begin{enumerate}
\item {\rm (Simple symplectic flips)} Two symplectic manifolds
  $V_+, V_-$ of dimension $2n$ are related by a {\em simple symplectic
    flip} if there exists a symplectic vector space
  $\ti{V}\cong \C^{n+1}$ with a Hamiltonian action of the circle $S^1$
  with moment map $\Psi: \ti{V} \to \R$ such that
  $V_\pm = \ti{V} \qu_\pm U(1) $ are the symplectic quotients at small
  values $\pm \eps$ for some $\eps > 0$, with the following
  properties:
\begin{enumerate} 
\item the sum of the weights is positive: 
\[ \sum_{i=1}^{n+1} \mu_i > 0; \] 
\item all of the weights are non-zero:
\[ \mu_i \neq 0, \quad  \forall i = 1,\ldots, n+1; \]
\item at least two weights are positive, and at least two weights are
  negative: 
\[ n_\pm := \# \{ \mu_i > 0 \} \ge 2  .\]
\end{enumerate} 
We write $\ti{V}_\pm$ for the sum of the positive resp. negative
weight spaces in $\ti{V}$, so that
\[\ti{V} = \ti{V}_+ \oplus \ti{V}_- .\]
The semi-stable loci are then 
\begin{equation} \label{ssloci} 
 \ti{V}^{\on{ss},-} = (\ti{V}_-  - \{ 0 \} ) \times \ti{V}_+, \quad 
\ti{V}^{\on{ss},+} = \ti{V}_- \times (\ti{V}_+ - \{ 0 \} )
.\end{equation} 
In particular, $V_+$ is obtained from $V_-$ by replacing a
weighted-projective space of dimension $n_- - 1$ with a
weighted-projective space of dimension $n_+ - 1$, as in Remark
\ref{gitrun}, where $n_\pm$ are the number of positive resp. negative
weights.
\item \label{fflip} {\rm (Fibered flips)} More generally, suppose
  that $\ti{V}$ is a vector bundle over a base $Z$ equipped with a
  symplectic structure on the total space and a Hamiltonian
  $U(1)$-action with the same properties as in the simple case.  Then
  the quotients $\ti{V} \qu_+ U(1)$ and $\ti{V} \qu_- U(1)$ are
  related by a symplectic flip {\em with center $Z$}.
\item {\rm (Symplectic flips)} Let $X_\pm$ be non-empty symplectic
  manifolds.  We say $X_+$ is obtained from $X_-$ by a {\em symplectic
    flip} if $X_+$ is obtained {\em locally} by a symplectic flip:
  that is, there exist open covers
\[ X_\pm = U_\pm \cup V_\pm \]
such that the following hold:
\begin{enumerate}
\item $U_+$ is diffeomorphic to $U_-$ and admits a family of symplectic forms
  \[\omega_{U,t} \in \Omega^2(U_\pm), t \in [-\eps,\eps]\]
and embeddings
\[i_\pm :U_\pm \to X_\pm, \quad i_\pm^* \omega_\pm = \omega_{U,\pm
  \eps}.\]
\item The manifolds $V_\pm$ are related by a simple flip as
  in part \ref{sflip} \eqref{fflip}.
\item The difference between symplectic classes on $X_+$ and $X_-$ is
  a positive multiple of the first Chern class in the following sense:
  Under the canonical identification $H^2(X_-) \to H^2(X_+)$ induced
  by the diffeomorphism in codimension at least four the first Chern
  class $c_1(X_-)$ maps to $c_1(X_+)$.  Then for some $\eps > 0$
  \llabel{c1s} \label{c1sp}
  \[ [\omega_+] - [\omega_-] = 2\eps c_1(X_-) .\]
\end{enumerate} 
\item {\rm (Symplectic divisorial contraction)} Symplectic divisorial
  contractions are defined in the same way as symplectic flips, but in
  this case all but one weight is positive.  In this case there exists
  a projection $\pi_+ : X_+ \to X_-$ with exceptional locus
  $E \subset X_+$ a (weighted)-projective bundle.  We require
\[ [\omega_+] = \pi_+^* [\omega_-] + 2\eps  \pi_+^* c_1(X_-) +
  \eps (n-1) [E]^\dual  \] 
  where $[E]^\dual \in H^2(X_+)$ is the Poincar\'e dual of the
  exceptional divisor. Note that this differs from the usual
  definition of symplectic blow-down because of the presence of the
  additional change in symplectic class $c_1(X_-)$.  The term
  $\eps[E]^\dual/(n-1)$ guarantees that the path
  $[\omega_+ + (t - \eps) c_1(X_+)] = [\omega_+ + (t-\eps) \pi_+^*
  c_1(X_-) + (t-\eps)(n-1) [E]^\dual]$
  becomes singular at time $t = 0$.
\item {\rm (Symplectic Mori fibration)} By a {\em symplectic Mori
    fibration} we mean a symplectic fibration $\pi: X_+ \to X_-$ such
  that the fiber $\pi^{-1}(x), x \in X_-$ is a monotone symplectic
  manifold ( the definition of Guillemin-Lerman-Sternberg \cite{gu:sf}
  is more general.)  In all our examples flips $X_-$ will be obtained
  from $X_+$ by a variation of symplectic quotient using a global
  Hamiltonian circle action.  Symplectic Mori fibrations will then be
  Mori fibrations in the usual sense.

\item {\rm (Symplectic mmp running)} A {\em symplectic mmp running} we
  mean a sequence $X = X_0,X_1,\ldots, X_k$ together with, for each $i
  = 0,\ldots, k$, a path of symplectic forms 
  \begin{equation} \label{c1var}
 \omega_{i,t} \in 
  \Omega^2(X_i), t \in [t_i^-,t_i^+], \quad  \ddt [\omega_{i,t} ] =
    c_1(X_i) \end{equation}
and each $ (X_i, \omega_{t_i^-})$ is obtained from $(X_{i-1},
\omega_{t_{i-1}^+})$ by a symplectic mmp transition for $i = 1,\ldots,
k$.
\end{enumerate} 
\end{definition} 

\begin{example}  \label{mmpex} 
\begin{enumerate}   
\item {\rm (Disjoint unions)} If $(X'_t,\omega_t')$ and
  $(X''_t,\omega_t'')$ are symplectic mmp runnings with distinct
  transition times then so is
  $(X'_t \sqcup X_t'', \omega_t' \sqcup \omega_t'')$.  The set of
  centers of the running for $X'_t \sqcup X''_t$ is the union of the
  centers for $X'_t$ and $X''_t$.
\item {\rm (Products)} If $(X'_t,\omega_t')$ and $(X''_t,\omega_t'')$
  are symplectic mmp runnings with distinct transition times then so
  is $(X'_t \times X_t'', \omega_t' \times \omega_t'')$.  The set of
  centers of the running for $X'_t \times X''_t$ is the union of the
  products $Z'_{t_i'} \times X_{t_i'}''$ and
  $X_{t_i''}' \times Z''_{t_i''}$ where $Z'_{t_i'}, Z''_{t_i''}$ are
  the centers for the mmp runnings for $X_t', X_t''$.
\item \label{equiv} {\rm (Equivariant mmp runnings)} Let $G$ be a
  compact Lie group and $(X_t,\omega_t)$ be a symplectic mmp running
  equipped with a Hamiltonian $G$-action with moment maps
  $\Phi_t: X_t \to \g^\dual$.  A family $(X_t,\omega_t)$ is a {\em
    $G$-equivariant symplectic mmp running} if the local models
  $\ti{V}_i$ are $G \times U(1)$-Hamiltonian manifolds, so that the
  mmp transition is given locally by variation of $U(1)$ quotient, the
  variation formula \eqref{c1var} holds in equivariant cohomology, and
  at each transition the identifications $U_\pm, V_\pm \to X_\pm$ are
  identifications of Hamiltonian $G$-manifolds.
\item \label{quotients} {\rm (Quotients)} If $X_t$ is an equivariant
  symplectic mmp running, then the family of symplectic quotients
  $X_t \qu G$ is a symplectic mmp running, with local model given by
  $\ti{V} \qu G$ and centers given by the quotients of the centers of
  $X_t$, assuming that these are free.  Indeed, the family
  $(X_t,\omega_t)$ be represented as a family of $U(1)$ quotients by
  the master space construction \eqref{masterspace} which realizes
  $X_t \qu G$ as a family of $U(1)$-quotients.  Note that $X_t \qu G$
  may have mmp transitions even if $X_t$ does not.  \end{enumerate}
\end{example}

\section{Runnings for toric manifolds} 
\label{trans}  

In this and the following subsections we give a sequence of examples
of mmp runnings.  Each of these examples is in some sense obtained by
variation of quotient as in Remark \ref{gitrun}.  The minimal model
program for toric varieties was established by Reid
\cite{reid:decomp}.  Here we describe mmp runnings given by shrinking
the moment polytope, see Gonz\'alez-Woodward \cite{gw:surject} and
Pasquier \cite{pasq:mmp}.  First we recall the following.

\begin{proposition}\label{tgit}  {\rm (Delzant \cite[Section 3.2]{de:ha},
    Lerman-Tolman \cite{lt}, Coates et al \cite[Section
    3.1]{coates:kt}.)}  Any smooth toric Deligne-Mumford stack with
  projective coarse moduli space has a presentation as a geometric
  invariant theory quotient.
\end{proposition}

We review the construction.  Take $V$ be a Hermitian vector space of
dimension $k$ with an action of a torus $G$.  Denote the weights of
the action
\[\mu_1,\ldots,\mu_k \in \g_\Z^\dual = \Hom(\g_\Z,\Z) .\]  
A linearization of $V$ is determined by an equivariant K\"ahler class
\[ \omega_{V,G} \in H^2_G(V) \cong \g_\Z .\]   
The geometric invariant theory $ V \qu G$ is, if locally free, the
quotient of the semistable locus
\begin{equation} \label{ssl} 
 V^{\ss} = \{ (v_1,\ldots,v_k) \in V, \quad \on{span} \{ \mu_k |
{v_k \neq 0} \} \ni \omega_{V,G} \} \end{equation}
by the action of $G$.  Suppose $G$ is contained in a maximal torus $H$
of the unitary group of $V$; then the residual torus $T = H/G$ acts on
$X = V \qu G$ making $X$ into a toric variety.  The moment polytope
for the action of the residual torus on the quotient can be computed
from the moment polytope for the original action.  The class
$\omega_{V,G} \in H^2_G(V)$ has a lift to $H^2_H(V)$ where
$H = (\C^\times)^k$. In terms of the standard basis vectors
$\eps_i^\dual \in \h^\dual$ the cohomology class of such a lift is
given by 
\[ \omega_{V,H} = \sum c_i \eps_i^\dual \in \h^\dual
\cong H^2_{H}(V) .\]
Choose coordinates $z_1,\ldots,z_k$ on $V$ so that the symplectic form
is 
\[ \omega_V = (1/2i) \sum_{j=1}^k \d z_j \wedge \d \ol{z}_j .\] 
A moment map for the $H$-action is given by the formula
\[ (z_j)_{j = 1}^k \mapsto ( -  c_j - |z_j|^2/2 )_{j= 1}^k .\]
Let $\nu_j$ denote the image of the $j$-standard basis vector
$- \eps_j \in \h$ under the projection $\h \to \g$.  The residual
action of the torus $T = H/G$ on $X = V \qu G$ has moment image
\begin{equation} \label{cjs} P = \{ \lambda \in \t^\dual \ | \ \lan \lambda, \nu_j \ran \ge c_j,
\quad j =1,\ldots, k \} .  \end{equation}  

The description of a toric manifold as a git quotient leads to the
following description of mmp runnings.

\begin{proposition} Suppose that the constants $c_j, j = 1,\ldots, k$
  from \eqref{cjs} are generic.  A running of the mmp for the toric
  variety $X$ is given by the sequence of toric varieties $X_t$
  corresponding to the sequence of polytopes
\[P_t = \{ \mu \in \t^\dual \ | \ \lan \mu, \nu_j \ran \ge c_j + t
\quad j = 1,\ldots, k \} .\]
The transition times are the set of times
\[ \cT := \Set{t \ | \ \exists \mu \in \t^\dual, \ \{ \nu_j \ | \  \lan
 \mu, \nu_j \ran = c_j + t \} \ \text{is linearly dependent} } .\]
Every symplectic manifold $X_t, t \notin \cT $ is an orbifold.  If the
set of normal vectors $\nu_j$ to facets stay the same (that is, if the
inequalities $ \lan \mu,\nu_j \ran = c_j + t$ continue to have
solutions for every $j = 1,\ldots, k$ ) then the transition is a flip;
if a facet disappears then the transition is a divisorial contraction
(where the divisor is the preimage of the disappearing face).
\end{proposition} 

\begin{proof} The statement of the proposition is a special case of
  \ref{mmpex} \eqref{quotients}.  Indeed suppose that $X$ is a
  symplectic quotient of $\C^k$ with respect to the equivariant
  symplectic class determined by the constants
  $(c_1,\ldots, c_N) \in \R^N \cong H_2^G(\C^k)$.  Then
  $(c_1 - t,\ldots, c_N - t)$ represents a canonical variation of
  symplectic class, and so a $G$-equivariant mmp running on $\C^k$.
  By Remark \ref{mmpex} \eqref{equiv} this descends by the quotient
  construct to an mmp running for $X$.  The description of the types
  of the transitions follows from the fact that the boundary divisors
  defined by $z_i = 0$ intersect the stable loci \eqref{ssl} on both
  sides exactly if the unstable locus $V - V^{\ss}$ is complex
  codimension at least two.
\end{proof}

For example, in Figure \ref{p1p1} there are two divisorial
contractions, occurring at the dots shaded in the diagram. 
\begin{figure}[ht]
\includegraphics[height=1in]{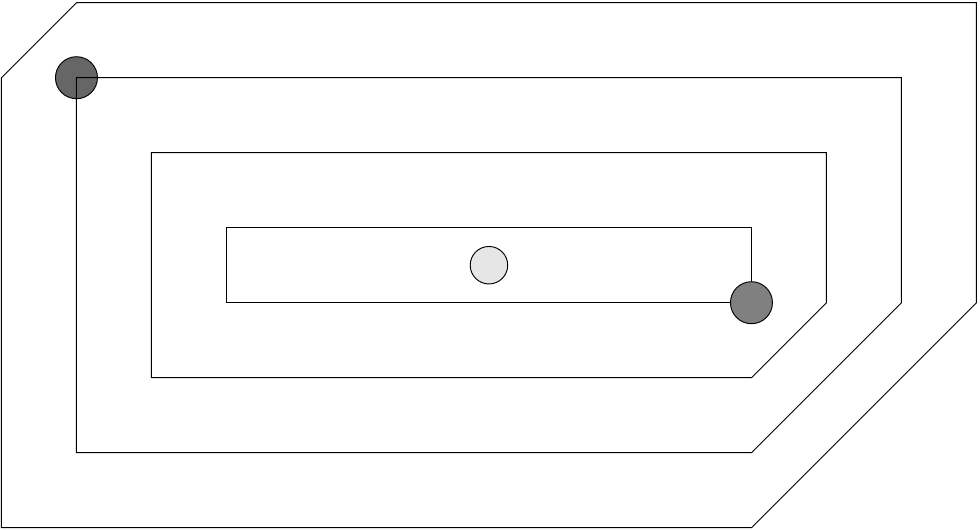}
\caption{Moment images of regular tori for the twice blow-up 
  of $\P^1 \times \P^1$}
\label{p1p1}
\end{figure}
Finally if a generic point $\mu$ in $P_{t_0}$ satisfies
$\lan \mu, \nu_j \ran = c_j + {t_0}$ then $X_t$ undergoes a Mori
fibration at $t_0$ with base given by the toric variety $X_{t_0}$ with
polytope $P_{t_0}$; one can then continue the running with $X_{t_0}$
to obtain an extended running.

\section{Runnings for polygon spaces} 
\label{polygons}

The moduli space of polygons is the quotient of a product of
two-spheres by the diagonal action of the group of Euclidean rotations
in three-space.  This moduli space is often used as one of the primary
examples of geometric invariant theory and symplectic quotients, see
for example Kirwan \cite[Example 2.2]{ki:coh}.  This space is also a special case
of the moduli space of flat bundles on a Riemann surface, which
appears in a number of constructions in mathematical physics.

First recall the Hamiltonian structure of the two-sphere via its
realization as a coadjoint orbit.  Let $S^2 \subset \R^3$ the unit
two-sphere equipped with the $SO(3)$-invariant symplectic form
\[ \omega \in \Omega^2(S^2), \quad \int_{S^2} \omega = 1. \]  
The action of $SO(3)$ on $S^2$ is naturally Hamiltonian.  Viewing
$SO(3)$ as a coadjoint orbit in $\so(3)^\dual \cong \R^3$, the moment
map is the inclusion of $S^2$ in $\R^3$, as a special case of the
Kirillov-Kostant-Souriau construction \cite{so:st}.

Consider the diagonal action on the product of spheres in the
previous paragraph.  Let $n \ge 1$ be an integer, and
$\lambda_1,\ldots, \lambda_n > 0 $ a sequence of positive real
numbers.  The product
\[ \widetilde{X} = (S^2 \times \ldots \times S^2, \lambda_1 \pi_1^* \omega +
\ldots + \lambda_n \pi_n^* \omega) \]
is naturally a symplectic manifold of dimension $2n$.  The group
$SO(3)$ acts diagonally on $\widetilde{X}$ with moment map
\[ \Psi: \widetilde{X} \to \R^3 , \quad (v_1,\ldots, v_n) \mapsto v_1 +
\ldots + v_n .\]
The symplectic quotient $X = \widetilde{X} \qu SO(3)$ is the {\em moduli
  space of $n$-gons}
\[P(\lambda_1,\ldots,\lambda_n) = \Set{ (v_1,\ldots, v_n) \in
(\R^3)^n \ | \ \Vert v_i \Vert = \lambda_i, \forall i, \quad
\sum_{i=1}^n v_i = 0 } .\]

The moduli space may be alternatively realized from the geometric
invariant theory perspective as a quotient by the complexified group.
We view each $v_i$ as a point in the projective line $\P^1$.  A tuple
$(v_1,\ldots,v_n) \in \P^1$ is semistable if and only if for each $w \in \P^1$,
the {\em slope inequality}
\[ \sum_{v_i = w} \lambda_i \leq \sum_{v_i \neq w} \lambda_i \]
holds \cite{mu:ge}.  Then $P(\lambda_1,\ldots,\lambda_n)$ is the
quotient of the semistable locus by the action of $SL(2,\C)$
by a special case of the Kempf-Ness theorem \cite{ke:le}.

A running of the mmp for the moduli space of polygons is given by
varying the lengths in a uniform way.  The first Chern class of the
product of spheres $\widetilde{X}$ is the class  
\[ c_1( \widetilde{X}) = \sum_{j=1}^n \pi_j^* c_1(\P^1) \]
where $\pi_j$ is projection onto the $j$-th factor.  It follows from
Remark \ref{gitrun} that the sequence of moduli spaces
$P(\lambda_1 - t, \ldots, \lambda_n - t)$ is a mmp (after rescaling
the time parameter by $2$) for $P(\lambda_1,\ldots, \lambda_n)$.
Transitions occur whenever there are one-dimensional polygons.  That
is,
\begin{equation} \label{stimes} 
 \cT := \Set{ t \ | \ \exists \eps_1,\ldots,\eps_n \in \{ -1,1 \},
\ \sum_{j=1}^n \eps_j (\lambda_j - t) = 0 } .\end{equation}
We may assume that the number of positive signs is greater at least
the number of minus signs, by symmetry.  For example, if the initial
configuration is $\lambda = (10,10,12,13,14)$ then there are three
transitions, at $t = 5,7,9$, corresponding to the equalities
\[ 5 + 5 + 7 = 8 + 9, \ 3 + 3 + 6 = 5 + 7, \ 1 + 1 + 5 = 3 + 4
.\]
There is a final transition when the smallest edge acquires zero
length.

Each flip or blow-down replaces a projective space of dimension equal
to the number of plus signs in \eqref{stimes}, minus one, with a
projective space of dimension equal to the number of minus signs in
\eqref{stimes}, minus one.  One can explicitly describe the projective
spaces involved in the flips as follows.  For $\pm \in \{ + , - \}$
let
\[ I_\pm = \{ i, \eps_i = \pm 1 \} \] 
denote the set of indices with positive resp. negative signs.  Let
$t_+ > t$ resp. $t_- < t$ and let
\[S_\pm = \Set{ (v_1,\ldots, v_n) \in
P(\lambda_1-t_\pm,\ldots,\lambda_n-t_\pm) \ | \ \R_{> 0} v_i = \R_{>
  0}v_j, \ \forall i,j \in I_\pm } \]
denote the locus where the vectors with indices in $I_\pm$ point in
the same direction.  (Since $\pm$ is a variable, this is a requirement
for $I_+$ or $I_-$ but not both.)  Thus $S_\pm$ is the symplectic
quotient of a submanifold $\ti{S}_\pm$ of the product of two-spheres
with only positive resp. negative weights for the circle action.  So
the space $S_+$ is a projective space and the flip replaces $S_+$ with
$S_-$:

\[ \begin{diagram} 
\node{
P(\lambda_1-t_-,\ldots,\lambda_n-t_-) 
} 
\node[4]{
P(\lambda_1-t_+,\ldots,\lambda_n-t_+) 
}
\\
\node{ S_- } \arrow{n}  \arrow{e}
\node[2]{\on{pt}} 
\node[2]{S_+}  \arrow{w}  \arrow{n}
\end{diagram} .\]
\vskip .1in For example, in the case of lengths $10,10,12,13,14$ for
the transition at $t = 5$, the configurations with edge lengths
$5+\eps,5+\eps,7+\eps,8+\eps,9+\eps$ with $9+\eps,8+\eps$ edges
approximately colinear are replaced with configurations with edges
with lengths $5- \eps,5 - \eps,7 - \eps$ approximately colinear.  The
first set of configurations is a two-sphere, a moduli space of
quadrilaterals, while the latter set of configurations is a point, a
moduli space of triangles.  It follows that the transition is a
blow-down.

There are two situations in which one obtains a Mori fibration.
First, in the case that one of the $\lambda_i$'s becomes zero, say
$\lambda_i - t$ is small in relation to the other weights, there is a
fibration
\[P(\lambda_1 - t,\ldots, \lambda_n - t) \to P(\lambda_1 - t
,\ldots,\lambda_{i-1} - t, \lambda_{i+1} - t, \ldots, \lambda_n - t)
.\]
Symplectically, this is a special case of the results of
Guillemin-Lerman-Sternberg \cite[Section 4]{gu:sf}.  From the
algebraic point of view, in this case the value of $v_i$ does not
affect the semistability condition, and forgetting $v_i$ defines the
fibration.  In the case that the moduli space becomes empty before one
of the $\lambda_i$'s reaches zero, the moduli space is a projective
space at the last stage, by the same discussion as in the case of
flips.  For example, in the case of lengths $10,10,12,13,14$ one
obtains a fibration
$ P(.1,.1,2.1,3.1,4.1)\to P(2,3,4) $
when $t = 10$ over the moduli space $P(2,3,4)$ which is a point.
Therefore, the moduli space at $t = 10 - \eps$ is diffeomorphic to a
product of two-spheres:
$ P(\eps,\eps,2 +\eps ,3 + \eps,4 + \eps) \cong S^2 \times S^2 .$
The conclusion is that $P(10,10,12,13,14)$ is a thrice-blow-up of $S^2
\times S^2$, that is, a del Pezzo surface.

The second way that one may obtain a Mori fibration is that one can
reach a chamber in which the moduli space is empty, because one of the
lengths, say $\lambda_i - t$ is so long compared to the others that
the sum of the remaining lengths is smaller than the first length:
\[ \sum_{j \neq i} \lambda_j - t < \lambda_i - t .\] 
In this case, the last non-empty moduli space admits the structure of a fiber bundle over
the space of reducible polygons corresponding to the transition, which
form a point.  The fiber is diffeomorphic to a projective space, by a repeat of the
arguments above.  See Moon-Yoo \cite{moon:birat} for more details.

\section{Runnings for moduli spaces of flat bundles} 
\label{flatmmp}

The moduli space of flat bundles on a Riemann surface is an example of
an infinite-dimensional symplectic quotient, and studied in for
example Atiyah-Bott \cite{at:mo}:

\begin{proposition} 
  Let $\Sigma$ be a compact Riemann surface and $G$ a simply-connected
  compact Lie group.  The moduli space $\cR(\Sigma)$ of flat
  $G$-bundles on $\Sigma$ has a presentation as an
  infinite-dimensional symplectic quotient of the affine space
  $\A(\Sigma)$ of $G$-connections by the Hamiltonian action of the
  group of gauge transformations $\cG(\Sigma)$.
\end{proposition} 

\begin{proof}[Recall of construction] The trivial $G$-bundle
  $P = \Sigma \times G$ has space of connections $\A(P)$ canonically
  identified with the space of $\g$-valued one-forms which are
  $G$-invariant and induce the identity on the vertical directions:
  \[ \A(P) := \{ \alpha \in \Omega^1( \Sigma \times G)^G, \quad
  \alpha(\xi_P) = \xi, \forall \xi \in \g \} \]
  where $\xi_P \in \Vect(P)$ is the vector field generated by $\xi$.
  The space $\A(P)$ is an affine space modelled on
  $\Omega^1(\Sigma,\g)$, A symplectic structure is given by wedge
  product and integration:
\[ \Omega^1(\Sigma,\g)^2 \to \R, \quad (a_1,a_2) \mapsto \int_\Sigma
(a_1 \wedge a_2) .\]
Here $(a_1 \wedge a_2) \in \Omega^2(\Sigma)$ is the result of
composition \label{resup} \llabel{resu}
\[\Omega^1(\Sigma,\g)^{\otimes 2} \to
\Omega^2(\Sigma, g^{\otimes 2}) \to 
 \Omega^2(\Sigma,\R)\] 
where the latter is induced by an invariant inner product $\g \times
\g \to \R$.  The action of the group $\G(P)$ of gauge transformations
on $\A(P)$ by pullback is Hamiltonian with moment map given by the
curvature:
\[ \A(P) \to \Omega^2(\Sigma,P(\g)), \quad \alpha \mapsto F_\alpha .\] 
The symplectic form on $\A(P)$ descends to a closed two-form on the
symplectic quotient
\[ \cR(\Sigma) := \A(P) \qu \G(P) = \{ \alpha \in \A(P) \ | \ F_\alpha
= 0 \} / \G(P) \]
the moduli space of flat connections on the trivial bundle. The
tangent space to $\cR(\Sigma)$ at the isomorphism class of a
connection $\alpha$ has a natural identification
\[ T_{[\alpha]} \cR(\Sigma) \cong H^1(\d_\alpha) \]
with the cohomology $H^1(\d_\alpha)$ of the associated covariant derivative
$\d_\alpha$ in the adjoint representation.  The Hodge star furnishes a
K\"ahler structure on the moduli space.
\end{proof}

\begin{definition} {\rm (Moduli spaces of bundles on surfaces with
    boundary)}  
  Extensions of the construction to the case that the surface has
  boundary are given in Mehta-Seshadri \cite{ms:pb}.  Suppose $\Sigma$
  is a compact oriented surface of genus $g$ with $n$ boundary
  components.  That is, $\Sigma$ is obtained from a closed compact
  oriented surface of genus $g$ by removing $n$ disjoint disks as in
  Figure \ref{surf}.  Let $Z_k \subset \Sigma$ be the $k$-th boundary
  circle, and $[Z_k] \in \pi_1(\Sigma)$ the class defined by a small
  loop around the $k$-th boundary component for $k = 1,\ldots, n$.
  Let $G = SU(2)$ denote group of special unitary $2 \times 2$
  matrices.  The space of conjugacy classes $G/\Ad(G)$ is naturally
  parametrized by an interval:
\[ [0,1/2] \to G/\Ad(G) , \quad \lambda \mapsto \diag(\exp( \pm 2\pi i  \lambda)) .\]
Let $\lambda_1,\ldots, \lambda_n \in [0,1/2]$ be {\em labels} attached
to the boundary components.
\begin{figure}[ht]
\begin{picture}(0,0)%
\includegraphics{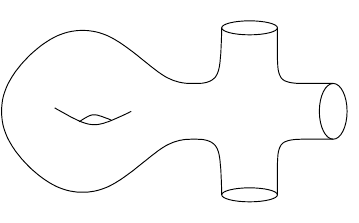}%
\end{picture}%
\setlength{\unitlength}{3947sp}%
\begingroup\makeatletter\ifx\SetFigFont\undefined%
\gdef\SetFigFont#1#2#3#4#5{%
  \reset@font\fontsize{#1}{#2pt}%
  \fontfamily{#3}\fontseries{#4}\fontshape{#5}%
  \selectfont}%
\fi\endgroup%
\begin{picture}(2856,1774)(1810,-3975)
\put(3637,-3940){\makebox(0,0)[lb]{\smash{{\SetFigFont{10}{6.0}{\rmdefault}{\mddefault}{\updefault}{\color[rgb]{0,0,0}$\lambda_1$}%
}}}}
\put(4651,-3116){\makebox(0,0)[lb]{\smash{{\SetFigFont{10}{6.0}{\rmdefault}{\mddefault}{\updefault}{\color[rgb]{0,0,0}$\lambda_2$}%
}}}}
\put(3656,-2270){\makebox(0,0)[lb]{\smash{{\SetFigFont{10}{6.0}{\rmdefault}{\mddefault}{\updefault}{\color[rgb]{0,0,0}$\lambda_3$}%
}}}}
\end{picture}%
\caption{Compact surface with labelled boundary}
\label{surf}
\end{figure}
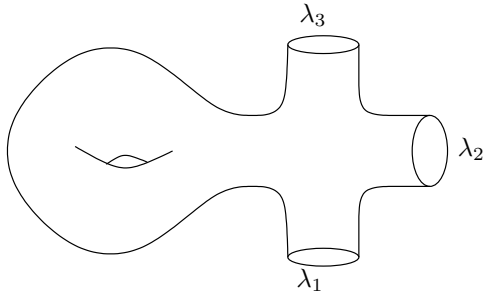 
Choose a base point in $\Sigma$ and let $\pi_1(\Sigma)$ denote the
fundamental group of homotopy classes of based loops with respect to
some base point.  Each loop $Z_k$ defines an element
$[Z_k] \in \pi_1(\Sigma)$, by connecting $Z_k$ to the base point, and
is well-defined up to conjugacy.  For numbers $\mu_1,\mu_2$ we denote
by $\diag(\mu_1,\mu_2)$ the diagonal $2\times2$ matrix with diagonal
entries $\mu_1$ and $\mu_2$. Let $\cR(\lambda_1,\ldots,\lambda_n)$
denote the moduli space of isomorphism classes of flat bundles with
holonomy around the boundary circles given by
$\exp(2 \pi i \diag(\lambda_k,-\lambda_k)), k = 1,\ldots, n$.  Since
any flat bundle is described up to isomorphism by the associated
holonomy representation of the fundamental group, we have the explicit
description
\begin{equation} \label{explicit}
\cR(\lambda_1,\ldots,\lambda_n)
 = \Set{ \begin{array}{l} \varphi
  \in \Hom(\pi_1(\Sigma),SU(2)) \\ \varphi([Z_k]) \in SU(2) \exp(2 \pi
  i \diag(\lambda_k,-\lambda_k)) \\ k = 1,\ldots, n \end{array}
} / SU(2) .\end{equation}
By Mehta-Seshadri \cite{ms:pb}, the moduli space of flat bundles may
be identified with the moduli space of {\em parabolic bundles} with
weights $\lambda_1,\ldots,\lambda_n$.  Here a parabolic bundle means a
holomorphic bundle $ E \to \Sigma$ on closed Riemann surface $\Sigma$
with markings $z_1,\ldots,z_n \in \Sigma$ with the additional datum of
one-dimensional subspace $L_i \subset E_{z_i}$ in the fiber $E_{z_i}$
at each marking, together with the weights
$\lambda_1,\ldots, \lambda_n$.\end{definition}

\begin{remark} {\rm (Moduli of spherical polygons)} In the case of
  rank two bundles there is a simple interpretation of these moduli
  spaces in terms of {\em spherical polygons}.  Namely $\pi_1(\Sigma)$
  is generated by homotopy classes of paths
  \[ \gamma_1,\ldots, \gamma_n \in \pi_1(\Sigma), \quad \gamma_1 \ldots
  \gamma_n = 1 .\]
Thus a representation of the fundamental group corresponds to a tuple
\[g_1,\ldots,g_n \in SU(2), \quad g_1 \ldots g_n = e \]  
where $e \in SU(2)$ is the identity.  Consider the polygon in $SU(2)$
with vertices
\[ ( e,\ g_1 ,\ g_1g_2 ,\ldots, \ g_1 \ldots g_{n-1}) \in SU(2)^n. \]
Choose a metric on $SU(2) \cong S^3$ invariant under the left and
right actions.  Because the metric on $SU(2)$ is invariant under the
right action, the distance between the $j-1$-th and $j$-th vertices is
the distance between $e$ and $g_j$.  Using invariance again it
suffices to assume that
$g_j = \diag( 2 \pi i (\lambda_j, -\lambda_j))$.  The distance is then
$\lambda_j$, once the metric is normalized so the maximal torus has
volume one.  Via the identification of $S^3$ with $SU(2)$, any
representation gives rise to a polygon in $S^3$ with edge lengths
$\lambda_1,\ldots, \lambda_n$.  Conversely, any closed polygon
determines a representation assigning the edge elements to the
generators of $\pi_1(\Sigma)$, and this correspondence is bijective up
to isometries of the three-sphere.
\end{remark} 

For generic weights the moduli space of flat bundles has a smooth
running of the mmp given by varying the labels in a uniform way.
First, a result of Boden-Hu \cite{boden:var} and Thaddeus
\cite[Section 7]{th:fl} shows that varying the labels leads to a
generalized flips in the sense that all conditions are satisfied
except the condition that the morphisms to the singular quotient are
relatively ample.  For this the variation of K\"ahler class should be
in the canonical direction as we now explain.  In the case without
boundary, the anticanonical class was computed by Drezet-Narasimhan
\cite{dr:pi} and in the case of parabolic bundles by
Biswas-Raghavendra \cite{br:par}; see Meinrenken-Woodward
\cite{me:can} for a symplectic perspective.  The anticanonical class
is expressed in terms of the symplectic class and the line bundles
$L_j$ associated to the eigenspaces of the holonomy around the
boundary components by
\[ c_1( \cR(\lambda_1,\ldots,\lambda_n) ) = 4
[\omega_{\cR(\lambda_1,\ldots,\lambda_n)}] - \sum_{j=1}^n ( 4 \lambda_j
- 1) 2c_1(L_j) .\]
In particular, if all weights $\lambda_i = \qq$ then the moduli space
is Fano.

\begin{proposition}  The moduli space $\cR(\lambda,\ldots, \lambda_n)$ has a smooth running of the mmp given by
the sequence of moduli spaces
\begin{equation} \label{sequence} 
\RR \left( \frac{\lambda_1 - t}{1 - 4t}
, \ldots, \frac{\lambda_n - t}{1-
  4t} \right) .\end{equation}
\end{proposition} 

\begin{proof} 
This family can be produced as a variation of symplectic quotient
using the construction of \cite{me:can} as follows: Let $LG$ denote
the loop group of $G = SU(2)$, $\cR$ denote the moduli space of flat
$G$-connections with framings on the boundary equipped with its
natural Hamiltonian action of $LG^n$, and
$\cO_{\lambda_1},\ldots,\cO_{\lambda_n}$ the $LG$-coadjoint orbits
through $\lambda_1,\ldots,\lambda_n$.  Then
$\cR(\lambda_1,\ldots, \lambda_n)$ has a realization as a symplectic
quotient
\[ \cR(\lambda_1,\ldots, \lambda_n) = 
( \cR \times \cO_{\lambda_1}
\times \ldots \times \cO_{\lambda_n})
 \qu LG^n .\]
Consider the product of anticanonical bundles
\begin{equation} \label{anticanon}
K_{\cR}^\dual \boxtimes K^\dual_{\cO_{\lambda_1}} \boxtimes \ldots
\boxtimes K^\dual_{\cO_{\lambda_n}} \to \cR \times \cO_{\lambda_1}
\times \ldots \times \cO_{\lambda_n} \end{equation}
in the sense of \cite{me:can}.  Its total space minus the zero section
has closed two form given by
\[ \pi^* \omega_{\cR \times \cO_{\lambda_1} \times \ldots \times
  \cO_{\lambda_n}} + \d (\alpha, \phi) \in \Omega^2 ((K^\dual_{\cR}
\boxtimes K^\dual_{\cO_{\lambda_1}} \boxtimes \ldots \boxtimes
K^\dual_{\cO_{\lambda_n}}) - \{ 0 \})\]
where $\alpha$ is a connection one-form and $\phi$ is the logarithm of
the norm on the fiber.  This two-form is non-degenerate on the region
defined by $(\lambda_i - \phi) / (1 - 4 \phi) \in (0,1/2)$ for each
$i$.  The action of $S^1$ by scalar multiplication in the fibers is
Hamiltonian with moment map $\phi$.  The quotient
\[\ti{\cR}(\lambda_1,\ldots,\lambda_n) := (K^\dual_{\cR} \boxtimes
K^\dual_{\cO_{\lambda_1}} \boxtimes \ldots \boxtimes
K^\dual_{\cO_{\lambda_n}} - \{ 0 \}) \qu LG^n\]
has a residual action of $S^1$ whose quotients are the family given
above:
\[ \ti{\cR}(\lambda_1,\ldots,\lambda_n)
 \qu_t S^1 = \cR \left(
\frac{\lambda_1 - t}{1 - 4t} , \ldots, \frac{\lambda_n - t}{1- 4t}
\right) .\]
At any fixed point the action of $S^1$ on the anticanonical bundle has
positive weight, by definition.  Note that the case $\lambda_1 =
\ldots = \lambda_n = \qq$ has a trivial mmp.  One should think of the
markings as moving away from the ``center'' $\qq$ of the Weyl alcove
$[0,1/2]$ under the mmp.\end{proof}

The flips or blow-downs occur at transition times at which there are
reducible (abelian) bundles.  More precisely, the set of transition
times is
\label{z2p} 
\begin{equation} \label{z2}  \cT = \Set{ t \ | \ \exists \eps_1,\ldots,\eps_n \in \{ -1,1 \},
\quad \sum_{i=1}^n \eps_i \frac{\lambda_i - t}{1 - 4t} \in \frac{1}{2}
\Z
} .\end{equation} 

The projective bundles involved in the flip can be explicitly
described as follows using the description of the moduli space as loop
space quotient in \cite{me:lo}; we focus on the genus zero case and
omit the proofs.  Let 
\[ I_\pm = \{ i, \eps_i = \pm 1 \} .\]  
Fix a decomposition of the curve $\Sigma$ into Riemann surfaces with
boundary $\Sigma_+, \Sigma_-$ such that $\Sigma_\pm$ contains the
markings in $I_\pm$.  Let $S_\pm$ denote the moduli space of bundles
that are abelian on $\Sigma_\pm$:
\[ S_\pm = \Set{ [A] \in \cR \left(\frac{\lambda_1 - t_\pm}{1-4t},\ldots,
    \frac{\lambda_n - t_\pm}{1-4t} \right) |  \dim( \Aut(A) ) = 1
} .\]
Then the flip replaces $S_+$ with $S_-$:
\[ {\small \begin{diagram} \node{ \cR
 \left( \left(
          \frac{\lambda_j-t_-}{1-4t} \right)_{j=1}^n \right) 
}
    \node[4]{ \cR \left( \left(\frac{\lambda_j-t_+}{1-4t}
        \right)_{j=1}^n \right) } \\ \node{ S_- } \arrow{n} \arrow{e}
    \node[2] {
\cR^{\ab} 
 \left( \left(
          \frac{\lambda_j-t_-}{1-4t} \right)_{j=1}^n \right) 
}
 \node[2]{S_+} \arrow{w}
    \arrow{n}
\end{diagram}} \]
where 
\[ 
\cR^{\ab} 
 \left( \left(
          \frac{\lambda_j-t_-}{1-4t} \right)_{j=1}^n \right) 
= \Set{ \rho \in \cR \left(\frac{\lambda_1 -
      t}{1-4t},\lldots, \frac{\lambda_n- t}{1-4t} \right) \ | \
  \rho(\pi_1(\Sigma)) \in T } \]
is the moduli space of representations in the maximal torus
$T = \{ \diag(e^{i\theta}, e^{-i \theta}) \} \subset SU(2)$.  Thus a
projective bundle over the Jacobian gets replaced with another
projective bundle.

As in the case of polygon spaces, there are two ways of obtaining Mori
fibrations: First, fibrations with $\P^1$-fiber occur whenever one of
the markings $\lambda_i - t$ reaches $0$ or $1/2$, with base the
moduli space of flat bundles with labels
\[\frac{\lambda_1 - t}{1-4t},\lldots, \frac{\lambda_{i-1} -t}{1-4t},
\frac{\lambda_{i+1} -t}{1-4t} ,\lldots, \frac{\lambda_n -t}{1-4t} \]
resp.
\[\frac{\lambda_1 - t}{1-4t},\lldots,\frac{ \lambda_{i-1} -t}{1-4t} , 1/2,
\frac{\lambda_{i+1} -t}{1-4t} ,\lldots,  \frac{ \lambda_n -
  t}{1-4t} \]
if the marking $\lambda_i$ reaches $0$ resp. $1/2$.  Second, in genus
zero the moduli space can become empty before any of the markings
reach $0$ or $1/2$.  By a special case of Agnihotri-Woodward
\cite{ag:ei}, proved earlier by Treloar \cite{tr:sph} we have
\begin{equation} \label{treloar} \cR(\lambda_1,\lldots,\lambda_n) = \emptyset \iff \exists I = \{ i_1
\neq \lldots \neq i_{2k+1} \} , \quad \sum_{i \in I} \lambda_i > k +
\sum_{i \notin I} \lambda_i .\end{equation}
Thus in the last stage there is either a fibration over a moduli space
with one less parabolic weight, with two-sphere fiber, or in genus
zero one can also have a projective space at the last stage if the
moduli space becomes empty.  One can then continue with the base to
obtain an extended running.  The mmp of this moduli
space is discussed in greater detail in Moon-Yoo \cite{moon:birat} as
well as Boden-Hu \cite{boden:var} and Thaddeus \cite[Section
  7]{th:fl}.

\begin{remark} {\rm (Runnings for flag varieties)} 
  Flag varieties admit mmp runnings given by fibrations over partial
  flag varieties.  Let $X$ be the variety of complete flags in a
  vector space of dimension $n$ with symplectic class
  $[\omega] \in H^2(X) \cong \R^n$ corresponding to an n-tuple
  $\lambda = (\lambda_1 \ge \ldots \ge \lambda_n) \in \Z^n$.  The
  space $X$ has a natural transitive action of the unitary group which
  induces a diffeomorphism $X \cong U(n)/U(1)^n$.  The unique
  polarization of $X$ is $\ti{X} = (U(n) \times \C) / U(1)^n $ where
  $U(1)^n$ acts on $U(n)$ from the right and on the left on $\C$ with
  weight $\lambda$.  We identify the Lie algebra with $\R^n$, the
  weight lattice with $\Z^n$ and let
  $\eps_1,\lldots, \eps_n \in \Z^n $ denote the standard basis of
  weights for $U(1)^n$.  We identify the set of positive weights with
  the intersection of $\Z^n$ with
\[ \t_+^\dual = \{ (\mu_1 \ge \ldots \ge \mu_n) \} \subset \R^n .\]
The tangent bundle of $X$ is the
associated fiber bundle
\[  TX \cong U(n) \times_{U(1)^n} \bigoplus_{1 \leq i < j \leq n} \C_{\eps_i - \eps_j} \]
where $\C_{\eps_i - \eps_j}$ is the space on which $U(1)^n$ acts with
weight $\eps_i - \eps_j$.  Hence the canonical bundle of $X$ is 
\[ \Lambda^{\on{top}} TX \cong U(n) \times_{U(1)^n} \C_{2\rho} \]
where $\C_{2\rho}$ denotes the one-dimensional representation
of $U(1)^n$ with weight 
\[ 2\rho := ((n-1)\eps_1 + (n-3) \eps_2 + \lldots + (1-n) \eps_n) .\]
An extended running of the mmp is the sequence of
partial flag varieties obtained as follows.  Consider the piecewise
linear path $\lambda_t$ starting with $\lambda_t = \lambda - \rho t$
and continuing as follows:  Whenever $\lambda_t$ hits a wall
$\sigma \subset \t_+^\dual$ of the positive chamber the mmp running
continues with the path
\[\lambda_t = \lambda_{t_i} - (t - t_i) \pi_\sigma \rho \]
where $\pi_\sigma$ is the projection onto $\sigma$.  Each transition
is a Mori fibration with Grassmann fiber and base the partial flag
variety corresponding to the element $\lambda_{t_i}$.

A simple example is the variety of complex flags in a
three-dimensional complex vector space which admits the structure of a
Mori fibration in two ways.  For example, let 
\[X = \on{Fl}(\C^3) := \{ V_1 \subset V_2 \subset \C^3 \ | \ \dim(V_k)
= k, k =1,2 \}\] 
be the variety of full flags in $\C^3$.  A running of the mmp is given
by $X,\P^2,{\rm pt}$.  There are no flips or divisorial contractions
in this case, so no regular Lagrangians.
\end{remark} 

\chapter{Lagrangians associated to flips} 
\label{regl}

We introduce a class of Lagrangians associated to minimal model
transitions which we call {\em regular}; based on the local models
studied in Fukaya et al. \cite{fooo:toric1}.  Regularity refers to the
fact that the Maslov index two disks in a toric neighborhood have
equal area.  These will be shown in Theorem \ref{unobs} below to be
Floer non-trivial.

\section{Regular Lagrangians} 

Each regular Lagrangian is given locally in a toric model and bounds a
collection of holomorphic disks of equal area.  We introduce the
following terminology.  
\begin{definition} \label{tori} 
Let $X$ be obtained by a reverse flip or 
blow-up with center $Z$. 
 \label{toricpiece}
 The reverse flip or blow-up replaces a projective bundle
 $\P(N_-) \to Z$ with a projective bundle $\P(N_+) \to Z$.  By the
 constant rank embedding theorem \cite{ma:so}, a neighborhood $U$ of
 $\P(N_+)$ in $X$ is symplectomorphic to a neighborhood $V$ of the
 zero section in a symplectic vector bundle $E_+$ over $\P(N_+)$.  Let
 $\phi: U \to \R_{\ge 0}$ be the canonical moment map for the
 Hamiltonian $S^1$-action on the fibers of $E_+$, so that
 $\P(N_+) = \phinv(0)$.  Let $J$ denote an almost complex structure
 equal to the toric complex structure on the fibers of $E_+$.
\begin{enumerate}
\item A Lagrangian $L \subset \P(N_+)$ is {\em toric} if the
  restriction of $\pi: \P(N_+) \to Z$ to $L$ fibers over a submanifold
  $L_Z$ of $Z$ with fiber a standard Lagrangian (in the fiber) torus
  orbit $L_F \subset \pi^{-1}(z)$ for some $z \in Z$.  A Lagrangian
  $L$ in $E_+$ is {\em toric} if the restriction of
  $\pi: E_+ \to \P(N_+)$ fibers over a toric Lagrangian in $\P(N_+)$
  with fiber a standard torus $L_F$ in the fiber of $E_+$.
\item A Lagrangian $L$ in $X$ is {\em regular} if there exists a
  constant $c > 0$ such that
  \label{critlag}
\begin{enumerate}
\item $L$ is a toric Lagrangian in $\phinv([0,c))$;
\item the holomorphic disks $u: (D,\partial D) \to (X,L)$ of Maslov
  index two contained in $\phinv([0,c))$ all have equal area $A_0$;
\item any non-constant holomorphic disk $u: (D,\partial D) \to (X,L)$
  meeting the complement of $U$ has area greater than $A_0$.
\end{enumerate}
\end{enumerate} 
\end{definition}

\begin{figure}[ht]
\includegraphics[height=1.5in]{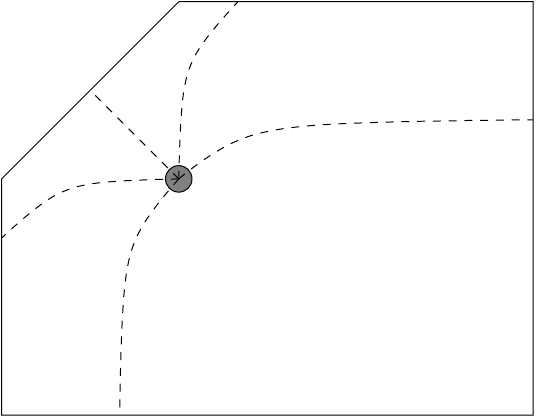}
\caption{Representation of Maslov index two disks} 
\label{threeprim}
\end{figure}

\begin{example}  
 {\rm (Blow-up at a point)} Let $X = \on{Bl}_0(\C^n)$ be the
  symplectic blow-up of $\C^n$ at a point.  The symplectic manifold
  $X$ admits a Hamiltonian torus action with all weights $-1$ and
  moment map $\Phi: X \to \t^\dual \cong \R^n$ with image
\[\Phi(X) = \{ (x_1,\ldots, x_n) \in \R_{\geq 0}^n \ | \ x_1 + \ldots
  + x_n \geq 1 . \} .\]  
  The inverse image 
\[L = \Phi^{-1}(1,\ldots, 1)\] 
is the unique regular torus.  Indeed, $X$ can be realized as the
quotient $\C^{n+1} \qu \C^\times$ of $\C^{n+1} $ by the
$\C^\times$-action with weights $-1,\ldots, -1,1$.  The only  disks of 
Maslov index two are given by 
\[u_i: B \to \C^{n+1}, \quad z \mapsto (1,\ldots, 1, z, 1,\ldots,
1) .\] 
All disks $u_i , i = 1,\ldots, n+1$ have equal area. A similar local
model applies to toric blow-ups in general; see Chapter \ref{choohsec}
for more details.  A regular Lagrangian in a blow-up of a product of
projective lines is shown in Figure \ref{threeprim}.
\end{example} 

\section{Regular Lagrangians for toric manifolds}
\label{choohsec}

The follow classification of disks in the toric case is used in the
definition.  We continue with the notation from Chapter \ref{trans}.

\begin{proposition} \label{chooh} {\rm (Cho-Oh \cite[Section
    4]{chooh:toric})} Let $X$ be a compact symplectic toric manifold
  equipped with the action of a torus $T$, whose moment polytope
  $\Phi(X)$ has $k$ facets.  Holomorphic disks with boundary in a
  Lagrangian torus orbit are classified by $k$-tuples of non-negative
  integers $\ul{d} \in \Z^k_{\ge 0}$.  The area and index of a disk
  corresponding to $\ul{d}$ are
\[ A(\ul{d}) = \sum_{j=1}^k  \lan \lambda, \nu_j \ran - c_j, \quad 
I(\ul{d}) = 2 \sum_{j=1}^k d_j . \]
\end{proposition}

  \begin{proof} We recall the classification.  Suppose $X$ is realized
    as a symplectic quotient of a vector space $V \cong \C^k$ by the
    action of a torus $G$, and let $L \subset V$ be a Lagrangian orbit
    of $T$.  Let $\ti{L} \cong (S^1)^k \subset V$ denote the lift of
    $L$ to $V$, given as the orbit of a point
    $(\ti{\mu}_1,\ldots, \ti{\mu}_k)$ under the standard torus action:
\begin{equation} \label{tiL}
 \ti{L} = \{ (e^{i\theta_1} \ti{\mu}_1,\ldots, e^{i \theta_k} \ti{\mu}_k) \ |
\ \theta_1,\ldots, \theta_k \in \R \} .\end{equation}
Let 
\[ C = \{ z \in \C \ | \ |z| \leq 1 \} \] 
denote the unit disk.  A {\em Blaschke product} of degree
$(d_1,\ldots,d_k)$ is a map from $C$ to $\C^k$ with boundary in a
toric Lagrangian prescribed by coefficients 
\[ a_{i,j} \in \C, \quad |a_{i,j}| < 1, \quad i \leq n, \quad j \leq
d_i :\]
\begin{equation} \label{blaschke}  u: C \to \C^k, \quad z \mapsto 
\left( \prod_{j=1}^{d_i} \frac{ z - a_{i,j}}{1 - z \ol{a_{i,j}}}
\right)_{i = 1,\dots,n} .\end{equation}
As in Cho-Oh \cite{chooh:toric}, the products \eqref{blaschke} are a
complete description of holomorphic disks with boundary in $L$.  Since
the image of $\ti{u}(z)$ is disjoint from the semistable locus, the
Blaschke products descend to disks $u: (C,\partial C) \to (X,L)$.  We
compute their Maslov index using the splitting
$ \ti{u}^* TV \cong u^* TX \oplus \g .$
Since the Maslov index of bundle pairs is additive, and the second
factor has Maslov index zero, the Maslov index of the disk is given by
\[ I(u) = \sum_{i=1}^k 2 d_j = 2 \# u(D). ( \sum_{i=1}^i [D_i]) \]
twice the sum of the intersection number with the anticanonical
divisor, which is the disjoint union of prime invariant divisors.  In
particular the disks of index two are those maps $u_i$ with lifts of
the form
\[ \ti{u}_i(z)= ( \ti{\mu}_1,\ldots, \ti{\mu}_{i-1}, \ti{\mu}_i z,
\ti{\mu}_{i+1}, \ldots, \ti{\mu}_k) .\]
The area of this disk is 
\[ A(u_i) = A(\ti{u}_i) = \lan \lambda, \nu_i \ran - c_ i \]
by a standard computation in Darboux coordinates.
\label{primitive}
The homology class of higher index Maslov disks
$u: C \to X, I(u) > 2 $ is a sum of these homology classes, and so the
area $A(u)$ of such a disk $u$ is the sum of the areas of disks of
index two.
\end{proof} 

The moment fibers that are regular Lagrangians in Definition
\ref{tori} are described as follows.  Let $\mu \in P$ and
\[t(\mu) = \min_j \lan \mu, \nu_j \ran - c_j .\]
The real number $t(\mu)$ is the time at which $\mu$ ``disappears''
under the mmp.  Suppose the set of normal vectors of facets meeting
the singular point
\[ N(\mu) := \{ \nu_j \ | \ \lan \mu, \nu_j \ran - c_j = t \} \]
is linearly dependent.  Then $L_\mu := \Phi^{-1}(\mu)$ satisfies the
first two parts of the definition of regularity in Definition
\ref{tori} \eqref{critlag}.  To see this we first compute the areas of
disks with boundary on the Lagrangian.  Suppose that $X$ is realized
as the git quotient of a vector space $V \cong \C^k$ by a torus $G$.
We may assume that $\dim(X) > 1$ so that the real codimension of the
unstable locus is at least four.  Let $\ti{L}_\mu$ denote the preimage
of $L_\mu$ in $C^k$.  The Lagrangian $\ti{L}_\mu$ is a Lagrangian
torus orbit of the group $U(1)^k$ acting on $\C^k$, that is,
$ \ti{L}_\mu = U(1)^k (\ti{\mu}_1,\ldots, \ti{\mu}_k) $ for some
constants $(\ti{\mu}_1,\ldots, \ti{\mu}_k)$.  Each map of a
holomorphic disk to $\C^k$ with boundary in $\ti{L}_\mu$ corresponds
to a collection of maps from disks to $\C$ with boundary on
$U(1) \ti{\mu}_j$.  If $L_\mu$ is a regular Lagrangian then the
minimal areas $A(u)$ of the holomorphic disks $u: C \to X$ with
Lagrangian boundary condition $L_\mu$ are $t$.  These correspond to
disks in the components corresponding to facets distance $t$ from
$\mu$, and each of these disks $u$ has Maslov index $I(u)$ two.  The
last assumption in Definition \ref{tori} \eqref{critlag} holds if the
other facets are sufficiently far away.  For example, in Figure
\ref{p1p1} we have two regular Lagrangians, given as the inverse
images of the darkly shaded dots in the Figure under the moment map.

\section{Regular Lagrangians for polygon spaces} 

A natural family of Lagrangian tori is generated by the {\em bending
  flows} studied in Klyachko \cite{kl:poly} and Kapovich-Millson
\cite{km:poly}.  Fix a subset $I \subset \{ 1,\ldots, n \}$ of the
edges of the polygon and define a {\em diagonal length} function
\[ \ti{\Psi}_I: S^2 \times \ldots \times S^2 \to \R_{\ge 0}, \quad
(v_1,\ldots, v_n) \mapsto \Vert v_I \Vert, \quad 
v_I := \sum_{i \in I} v_i.  \]
The diagonal length is smooth on the locus where it is positive. 
The Hamiltonian flow of the diagonal length function is given by
rotating part of the polygon around the diagonal.  
The following is,
for example, explained in \cite[Section 3]{km:poly}:

\begin{lemma} \label{bflowlem}
The function $\ti{\Psi}_I$ generates on
  $\ti{\Psi}_I^{-1}( \R_{>0 })$ a Hamiltonian circle action given by
  rotating the vectors $v_i, i \in I$ around the axis spanned by $v_I
  := \sum_{i \in I} v_i$:
  \begin{equation} \label{bflows} v_i \mapsto R_{\theta,v_I} v_i,
    \quad i = 1,\ldots, n \end{equation}
where $R_{\theta, v_I}$ is rotation by angle $\theta$ around the span
of $v_I$.  Furthermore if $\cT$ is a collection of subsets such that
for all $I ,J \in \cT$ either $I \subseteq J$ or $J \subseteq I$ then
the associated flows $R_{\theta,v_I}, R_{\theta',v_J}$ commute.
\end{lemma}

\begin{proof} We provide a proof for completeness.  By the symplectic
  cross-section theorem \cite[Theorem 26.7]{gu:sy}, the inverse of the
  interior of the positive Weyl chamber $\R_{> 0}$ under the moment
  map $\Psi$ is a symplectic submanifold $\Psi^{-1}(\R_{> 0}\times \{
  0 \} \times \{ 0 \}) \subset \widetilde{X}$.  This inverse image is
  the locus
\[ \widetilde{X}_I = \Set{ (v_1,\ldots, v_n) \in \widetilde{X} \ | \ v_I \in \R_{> 0} \times \{ 0 \} \times \{ 0 \} \subset \R^3 } .\]
Since the stabilizer of any point in $\widetilde{X}_I$ is contained in
$SO(2)$, the flow-out of $\widetilde{X}_I$ is
$SO(3) \widetilde{X}_I = SO(3) \times_{SO(2)} \widetilde{X}_I .$ As a
result, any $SO(2)$-equivariant Hamiltonian diffeomorphism of
$\widetilde{X}_I$ extends uniquely to a Hamiltonian diffeomorphism of
$SO(3) \widetilde{X}_I$ which is equivariant with respect to the
$SO(3)$-action.  Now on $\widetilde{X}_I$, the function $\ti{\Psi}_I$
is the first component of the moment map and so the flow of
$\ti{\Psi}_I$ is rotation around the first axis.  It follows that the
flow of $\ti{\Psi}_I$ is rotation around the line spanned by
$\sum_{i \in I} v_i$, as long as this vector is non-zero.  In
particular the flow of $\ti{\Psi}_I$ is $SO(3)$-equivariant and so
descends to a function $\Psi_I$ generating a circle action on a dense
subset of $P(\lambda_1,\ldots, \lambda_n ) = \widetilde{X} \qu SO(3)$.
If $I \subset J$ then we may assume that the ordering of vectors
$v_1,\ldots, v_n$ is such that $I,J$ consist of adjacent indices.  The
vectors $v_I$ and $v_J$ break the polygon into three pieces, and the
flows of $\Psi_I$ and $\Psi_J$ are rotation of the first and third
pieces around the diagonals $v_I, v_J$ respectively.  In particular,
these flows commute.
\end{proof}
\begin{figure}[ht]
\begin{picture}(0,0)%
\includegraphics{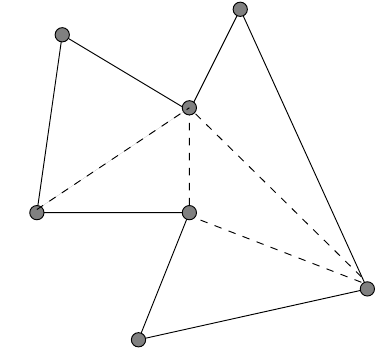}%
\end{picture}%
\setlength{\unitlength}{3947sp}%
\begingroup\makeatletter\ifx\SetFigFont\undefined%
\gdef\SetFigFont#1#2#3#4#5{%
  \reset@font\fontsize{#1}{#2pt}%
  \fontfamily{#3}\fontseries{#4}\fontshape{#5}%
  \selectfont}%
\fi\endgroup%
\begin{picture}(3004,2774)(2161,-2921)
\put(2176,-1153){\makebox(0,0)[lb]{\smash{{\SetFigFont{10}{6.0}{\rmdefault}{\mddefault}{\updefault}{\color[rgb]{0,0,0}$v_1$}%
}}}}
\put(3066,-594){\makebox(0,0)[lb]{\smash{{\SetFigFont{10}{6.0}{\rmdefault}{\mddefault}{\updefault}{\color[rgb]{0,0,0}$v_2$}%
}}}}
\put(4286,-2830){\makebox(0,0)[lb]{\smash{{\SetFigFont{10}{6.0}{\rmdefault}{\mddefault}{\updefault}{\color[rgb]{0,0,0}$v_5$}%
}}}}
\put(3142,-2449){\makebox(0,0)[lb]{\smash{{\SetFigFont{10}{6.0}{\rmdefault}{\mddefault}{\updefault}{\color[rgb]{0,0,0}$v_6$}%
}}}}
\put(2939,-1966){\makebox(0,0)[lb]{\smash{{\SetFigFont{10}{6.0}{\rmdefault}{\mddefault}{\updefault}{\color[rgb]{0,0,0}$v_7$}%
}}}}
\put(2913,-1381){\makebox(0,0)[lb]{\smash{{\SetFigFont{10}{6.0}{\rmdefault}{\mddefault}{\updefault}{\color[rgb]{0,0,0}$v_{1,2}$}%
}}}}
\put(4210,-1432){\makebox(0,0)[lb]{\smash{{\SetFigFont{10}{6.0}{\rmdefault}{\mddefault}{\updefault}{\color[rgb]{0,0,0}$v_{3,4}$}%
}}}}
\put(3727,-1610){\makebox(0,0)[lb]{\smash{{\SetFigFont{10}{6.0}{\rmdefault}{\mddefault}{\updefault}{\color[rgb]{0,0,0}$v_{1,2,7}$}%
}}}}
\put(3727,-594){\makebox(0,0)[lb]{\smash{{\SetFigFont{10}{6.0}{\rmdefault}{\mddefault}{\updefault}{\color[rgb]{0,0,0}$v_3$}%
}}}}
\put(4540,-1076){\makebox(0,0)[lb]{\smash{{\SetFigFont{10}{6.0}{\rmdefault}{\mddefault}{\updefault}{\color[rgb]{0,0,0}$v_4$}%
}}}}
\put(4108,-2220){\makebox(0,0)[lb]{\smash{{\SetFigFont{10}{6.0}{\rmdefault}{\mddefault}{\updefault}{\color[rgb]{0,0,0}$v_{5,6}$}%
}}}}
\end{picture}%
\caption{Triangulated polygon} 
\label{triangp}
\end{figure} 
Given a triangulation of a polygon we associate a moment map for a
densely-defined torus action as follows.  A triangulation
$\cE = \{ T \subset \{ 1, \ldots, n \} \} $ of the abstract $n$-gon
\llabel{triang}  \label{triangp} 
with edges $v_1,\ldots, v_n$ is a collection of
subsets, called triangles.  Each triangle $T \in \cE$ is a set of size
three $I_1,I_2,I_3$ whose elements indicate which sums
$v_{I_j} := \sum_{ i \in I_j} v_i$ form the edge.  If $I_j$ has size
more than one, the corresponding vector $v_{I_j}$ is a {\em diagonal}.
See Figure \ref{triangp}.  Any triangulation $\cE$ triangulation gives
rise to a map
\[ \Psi_{\cE}: P(\lambda_1,\ldots,\lambda_n) \to \R_{\ge 0}^{n-3},
\quad [v_1,\ldots, v_n] \to ( \Vert v_{I_j} \Vert )_{j=1}^{n-3} \]
given by taking the edge lengths of the diagonals.  By Lemma
\ref{bflowlem}, the map $\Psi_{\cE}$ is, where smooth, a moment map
for the action of an $n-3$-dimensional {\em bending torus}
$G \cong U(1)^{n-3}$ which acts as follows: Let
$ (\exp(i\theta_1), \ldots, \exp(i \theta_{n-3})) \in G$ and
$[v_1,\ldots,v_n]$ be an equivalence class of polygons.  For each
diagonal $v_I$, divide the polygon into two pieces along $v_I$, and
rotate one of those pieces, say $(v_i)_{i \in I}$ around the diagonal
by the given angle $\theta_I$.  The resulting polygon is independent
of the choice of which piece is rotated, since polygons related by an
overall rotation define the same point in the moduli space.  See
\cite[Section 3]{km:poly} for more details.

The regular Lagrangian tori are described as follows as fibers of the
Goldman map for which all triangles have the same ``looseness''.
Suppose that $P(\lambda_1 - t, \ldots, \lambda_n - t)$ is an mmp
running for $P(\lambda_1,\ldots, \lambda_n)$.  As noted in Example
\ref{trans}, mmp transitions correspond to partitions 
\[\{ 1, \ldots, n \} = I_+ \cup I_-, \quad
\sum_{i \in I_+} \lambda_i - t = \sum_{i \in I_-} \lambda_i -
  t .\]  
For each triangle $T$ in the triangulation with labels $\mu_1,\mu_2,
\mu_3$ we denote by $l(T)$ the {\em looseness} of the triangle
\[ l(T): = \min_{i,j,k \ \text{distinct}} ( \mu_i + \mu_j - \mu_k )/2.\]
In other words, the looseness measures the failure of the triangle to
be degenerate.  A labelling $\mu \in \R_{\ge 0}^{n-3}$ is called {\em
  regular} if $l(T)$ is independent of $T$ and less than
$\min_i \lambda_i$.  For example, if the edge lengths are $2,3,4,7$,
and the triangulation separates the first two edges from the last two,
then a regular triangulation is obtained by assigning $4$ to the
middle edge.  The looseness of each triangle is
$1 = 2 + 3 - 4 = 4 + 4 - 7 $.  See Figure \ref{loosequad}.

\begin{figure}[ht]
\includegraphics[height=1.5in]{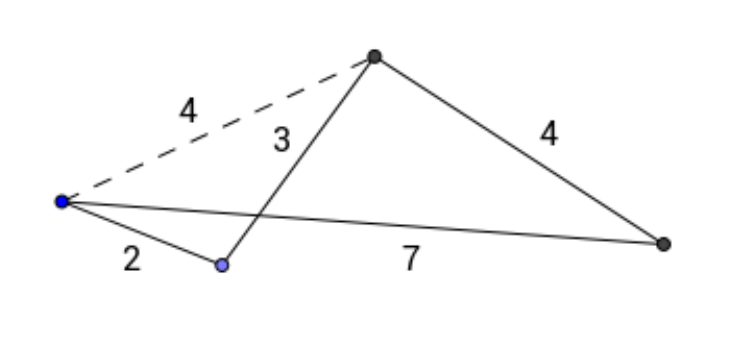}
\caption{A regular triangulation of a quadrilateral}
\label{loosequad}
\end{figure}

\begin{proposition} \label{reglab} Let $\mu \in\R_{\ge 0}^{n-3}$ be a
  regular labelling, and $\mu(t)$ the family of labellings obtained by
  replacing each $\mu_i$ with $\mu_i - t$ which becomes singular at
  first time $t = t_i$.  Then for $\eps > 0$ sufficiently small and
  $t \in (t_i - \eps, t_i)$, any labelling $\mu(t)$ has the property
  that $\Psi^{-1}(\mu)$ is regular.
\end{proposition}

\begin{proof} 
  A local toric structure on $P(\lambda_1,\ldots, \lambda_n)$ is given
  by choosing a triangulation compatible with the partition into
  positive and negative edges, and the action of the bending torus $T$
  above.  Let $v \in P(\lambda_1 - t,\ldots,\lambda_n - t)$ denote the
  one-dimensional polygon corresponding to the transition time.  Each
  triangle in the triangulation is degenerate for $v$ and so for each
  $T$, $\mu_i + \mu_j = \mu_k$ for some edges $i,j,k$ of $T$.  The
  inequalities defining $\Psi_t(T)$ near $v$ are of the form
  $ \mu_i + \mu_j \ge \mu_k $ as $i,j,k$ range over all possible
  indices.  It follows that the polytope defining the image of the map
  $\Psi$ is given locally by the triangle inequalities, $l(T) \ge 0$
  for each of the $n-2$ triangles in the triangulation:
\begin{multline} \Psi_{\cE}(P(\lambda_1,\ldots,\lambda_n)) \\ = \Set{ (\mu_1,\ldots, \mu_{n-3} )
\in \R_{\ge 0}^{n-3} \ | \   \ \forall T \in \cE, (T = \{ v_i , v_j ,v_k
\}) \implies \mu_i + \mu_j \ge \mu_k  } .\end{multline}
The Maslov index two disks $u: C \to X$ with respect to this structure
whose areas go to zero at the transition time have areas $A(u)$ given
by the differences $\mu_i + \mu_j - \mu_k$, and by assumption the
non-constant disks with lowest area all have equal area.  On the other
hand, for times close to the transition time, any other Maslov index
two disk $u: C\to X$ has larger area, since otherwise it would be
contained in the toric piece $X_\subset$ by the diameter estimate in
Sikorav \cite[4.4.1]{sikorav}.
\end{proof}

\section{Regular Lagrangians for moduli spaces of flat bundles} 

The analog of the bending flow was introduced by Goldman \cite{go:in}.
First one constructs a densely defined circle action on the moduli
space of bundles associated to any circle on the surface.  Given any
circle $C \subset \Sigma$ disjoint from the boundary, the holonomy
$\varphi(C)$ of the flat bundle $P$ around $C$ is given by an element
$\exp( \diag( \pm 2 \pi i \mu))$ up to conjugacy.  After gauge
transformation, the holonomy is  $\exp( \diag ( \pm 2 \pi i
\mu))$.  Given an element $\exp( 2\pi i \tau) \in U(1)$, one may
construct a bundle $P_\tau$ by cutting $\Sigma$ along $C$ into pieces
$\Sigma_+, \Sigma_-$ and gluing back the restrictions $P|\Sigma_+,
P|\Sigma_-$ together using the transition map $e(\tau):= \diag ( \exp
( \pm 2 \pi i \tau)) $:
\begin{equation} \label{gflow} P_\tau := (P|\Sigma_+) \bigcup_{e(\tau)} (P|\Sigma_-) .\end{equation}
See Figure \ref{twist}.

\begin{figure}[ht]
\begin{picture}(0,0)%
\includegraphics{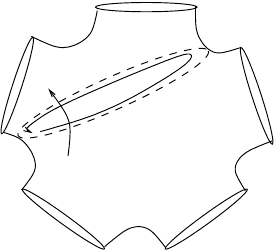}%
\end{picture}%
\setlength{\unitlength}{3947sp}%
\begingroup\makeatletter\ifx\SetFigFont\undefined%
\gdef\SetFigFont#1#2#3#4#5{%
  \reset@font\fontsize{#1}{#2pt}%
  \fontfamily{#3}\fontseries{#4}\fontshape{#5}%
  \selectfont}%
\fi\endgroup%
\begin{picture}(2191,1987)(251,-2491)
\put(1304,-1393){\makebox(0,0)[lb]{\smash{{\SetFigFont{10}{7.2}{\rmdefault}{\mddefault}{\updefault}{\color[rgb]{0,0,0}$e(\mu)$}%
}}}}
\put(969,-1064){\makebox(0,0)[lb]{\smash{{\SetFigFont{10}{7.2}{\rmdefault}{\mddefault}{\updefault}{\color[rgb]{0,0,0}$C$}%
}}}}
\put(619,-1902){\makebox(0,0)[lb]{\smash{{\SetFigFont{10}{7.2}{\rmdefault}{\mddefault}{\updefault}{\color[rgb]{0,0,0}$e(\tau)$}%
}}}}
\end{picture}%
\caption{Twisting a bundle along a circle}
\label{twist}\end{figure} 
\noindent The automorphism given by $\diag( \exp( \pm  2 \pi i \tau))$ commutes with
the holonomy so the resulting bundle has a canonical flat structure,
whose holonomies around loops $\Sigma_+,\Sigma_-$ are equal, but
parallel transport from $\Sigma_+$ to $\Sigma_-$ is twisted by $ \diag
( \exp ( \pm 2 \pi i \tau)) $.  Let $\cR(\lambda_1,\lldots,\lambda_n)^C$
denote the locus where $\mu \notin \{ 0, 1/2 \}$, for which the
construction $[P] \mapsto [P_\tau]$ is well-defined and independent of
all choices.  The map
\[ \cR(\lambda_1,\lldots, \lambda_n)^C 
\to \cR(\lambda_1,\lldots,
       \lambda_n)^C, \quad [P] \mapsto [P_\tau] \]
       defines a circle action.  Furthermore, if $C_1,C_2$ are
       disjoint circles then the circle actions defined above commute
       on the common locus
\[  \cR(\lambda_1,\lldots, \lambda_n)^{C_1} 
\cap \cR(\lambda_1,\lldots, \lambda_n)^{C_2}  .\]
Recall that a {\em pants decomposition} of a surface is a
decomposition into three-holed spheres.  Any compact oriented Riemann
surface with boundary admits a finite pants decomposition, by choosing
sufficiently many separating surfaces so that each piece has Euler
characteristic one.  Choose a pants decomposition $\cP$ that refines
the decomposition into pieces $\Sigma_+,\Sigma_-$.  Given a pants
decomposition, one repeats the construction for each interior circle
in the pants decomposition to obtain a moment map
\begin{equation} \label{goldman}
 \Psi_{\cP}: \cR(\lambda_1,\lldots,\lambda_n) \to [0,1/2]^{n - 3} \end{equation}
for a densely-defined torus action, see \cite[Section 4]{go:in} or,
for a summary, Jeffrey-Weitsman \cite{jw}.  In the genus zero case,
the generic fibers are Lagrangian tori.  For each pairs of pants $P$
in the pants decomposition with labels $\mu_1,\mu_2,\mu_3$, define the
{\em looseness} of $P$ by
\[ l(P) :=  \min 
\left( \min_{i \neq j \neq k} 
( \mu_i + \mu_j - \mu_k) , 1 -
\mu_1 - \mu_2 - \mu_3 \right) .\]
A labelling $\mu \in [0,1/2]^{n-3}$ is {\em regular} if the looseness
$l(P)$ is the same for each pair of pants $P \in \cP$,
\[ \# \{ l(P) | P \in \cP \} = 1  \]
and if the first fibration in the running occurs at a time greater
than $l(P)$.  See Figure \ref{pantses} for two examples in the case $n
= 5$.

The regular Lagrangians are described as follows.  
Consider an mmp running as in \eqref{sequence}
 $\RR \left( \frac{\lambda_1 - t}{1 - 4t}, \lldots, \frac{\lambda_n -
    t}{1- 4t} \right) $.
 The transition times $\cT$
are the times $t$ for which there is an abelian representation.  Given
such a representation with holonomies $\diag(\exp( \pm \eps_j \mu_j))$
define a partition of the surface $\Sigma$ into pieces
$\Sigma_+, \Sigma_-$ containing the markings $\mu_j$ for which
$\eps_j$ is positive resp. negative.  We claim that if $\mu$ is
regular and $l(\mu)$ is sufficiently small then the Goldman fiber
$L_\mu := \Psi_{\cP}^{-1}(\mu) $ is regular.  The Goldman bending flow
\eqref{gflow} induces a toric structure on
$\cR(\lambda_1,\lldots,\lambda_n)$ in a neighborhood of the
exceptional locus.  The image of the Goldman map \eqref{goldman} is
given by
\[ \Psi_{\cP}(\cR(\lambda_1,\lldots,\lambda_n)) = \Set{ \mu \in
[0,1/2]^{n-3} \ | \  \forall P \in \cP, l(P) \ge 0 };\]
that is, for each pair of pants in the pants decomposition the
looseness is non-negative \cite{jw}.  It follows that if $l $ is a
regular labelled pants decomposition then $\Psi^{-1}(l)$ is toric and
there are $n-2$ homotopy classes of disks with Maslov index two and
boundary in $\Psi^{-1}(l)$, of equal area $l(\mu)$ while (for
sufficiently small looseness) the remaining disks of Maslov index two
have area $A(u) > l(\mu)$.  Figure \ref{pantses} gives an example of a
labelled pants decomposition giving rise to a regular Lagrangian,
corresponding to an mmp transition at time $t = .06$.

\chapter{Fukaya algebras}

The Fukaya algebra is a homotopy-associative algebra whose higher
composition maps are counts of configurations involving perturbed
pseudoholomorphic disks with boundary in the Lagrangian
\cite{fuk:garc}.  Because the moduli spaces of disks involved in the
construction are usually singular, there are technical issues involved
in its construction similar to those involved in the construction of
virtual fundamental classes for moduli spaces of pseudoholomorphic
curves.
\label{foooref}
Fukaya-Oh-Ohta-Ono \cite{fooo} introduced a method of solving these
issues using {\em Kuranishi structures} in which one first constructs
local thickenings of the moduli spaces and then introduces
perturbations constructed locally.  In this section we construct
Fukaya algebras of Lagrangians in a compact {\em rational} symplectic
manifold using a perturbation scheme that we find particularly
convenient for various computations: the {\em stabilizing divisors}
scheme introduced by Cieliebak-Mohnke \cite{cm:trans}.  We also
incorporate Morse gradient trees introduced by Fukaya \cite{fuk:garc}
and Cornea-Lalonde \cite{cl:clusters}, see also Seidel
\cite{seidel:genustwo} and Charest \cite{charest:clust}.  Stasheff's
homotopy-associativity equation follows from studying the boundary
strata in the moduli space of treed disks as in Figure \ref{MW}.
\begin{figure}[ht]
\includegraphics[width=4.5in]{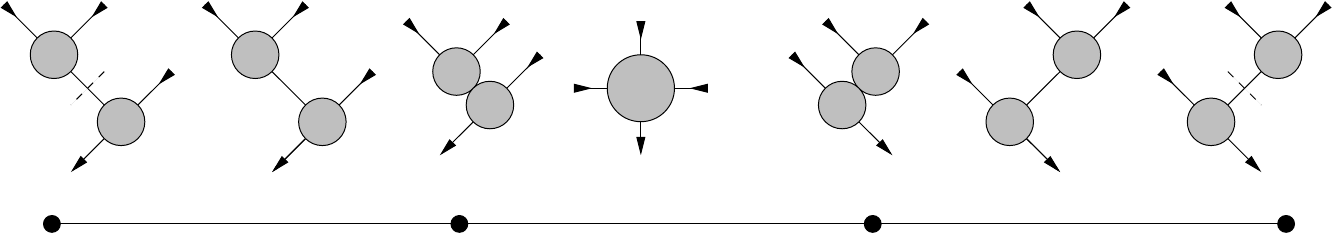}
\caption{Moduli space of stable treed disks}
\label{MW}
\end{figure} 
This construction allows us to take our Floer cochain spaces to be
finite-dimensional.  The structure constants for the Fukaya algebras
in the stabilizing divisors approach count pseudoholomorphic disks
with Lagrangian boundary conditions and Morse gradient trajectories on
the Lagrangians with domain-dependent almost complex structures and
Morse functions depending on the position of additional markings
mapping to a stabilizing divisor.  Because the additional marked
points must be ordered in order to obtain a domain without
automorphisms, this scheme gives a multi-valued perturbation.  The
resulting structure maps
\[ \mu^n : CF(L)^{\otimes n} \to CF(L), \quad n \ge 0 \]
for the Fukaya algebra are defined only using rational coefficients.
We also equip Fukaya algebras with strict units so that disk
potentials are defined.  To achieve this we incorporate a slight
enhancement, similar to that of homotopy units \label{hunits} in
Fukaya-Oh-Ohta-Ono \cite[(3.3.5.2)]{fooo}, in which perturbation
systems compatible with breakings are homotoped to perturbation
systems that admit forgetful maps.

Given the construction of strictly unital Fukaya algebras described
above, the Floer cohomology is defined over a space of projective
Maurer-Cartan solutions.  Let $e_L \in CF(L)$ denote the resulting
strict unit and $CF(L)^{\on{odd},+} \subset CF(L)$ the subset of odd
elements whose coefficients all have positive $q$-valuation.  The
Maurer-Cartan map
\begin{equation} \label{mcmap} \mu: CF(L)^{\on{odd},+} \to CF(L), \quad b \mapsto
  \sum_{n \ge 0} \mu^n \left(\underbrace{b, \ldots, b}_n
  \right) \end{equation}
has solution space
\[ {MC}(L) := \mu^{-1}(\Lambda e_L) \subset CF(L) \]
also denoted $MC(L,y)$ if we wish to emphasize the dependence on the
local system $y$.  The {\em Floer cohomology} is the fiber-wise
cohomology of operator $\mu^1_b$ defined below in Chapter
\ref{maurercartan}
\[HF(L) = \cup_{b \in {MC}(L)} HF(L,b), \quad HF(L,b) :=
\ker(\mu^1_b)/ \on{im}(\mu^1_b), \quad b \in {MC}(L) .\]
The Floer cohomology $HF(L)$ is said to be non-vanishing if the fiber
$HF(L,b) $ is non-vanishing for some $b \in {MC}(L)$.  We prove the
following:

\begin{theorem} \label{fukalg} Let $(X,\omega)$ be a compact
  symplectic manifold with rational symplectic class
  $[\omega] \in H^2(X,\Q)$ and $L \subset X$ a compact rational
  embedded Lagrangian submanifold equipped \label{equip} with a
  relative spin structure and grading.  For a comeager subset of
  perturbation data, counting weighted treed pseudoholomorphic disks
  defines a convergent \ainfty structure with strict unit independent
  of all choices up to convergent strictly-unital \ainfty homotopy.
  Furthermore, for any $b \in MC(L)$ the Floer cohomology $HF(L,b)$ is
  independent of all choices up to gauge equivalence (to be explained
  below).
\end{theorem} 

Theorem \ref{fukalg} is a combination of Theorem \ref{yields} and
Corollary \ref{diffdegree} below.

\section{\ainfty algebras}
\label{orient}

Homotopy-associative algebras were introduced by Stasheff \cite{st:ho}
in order to capture algebraic structures on the space of cochains on
loop spaces.  We follow the sign convention in Seidel
\cite{seidel:sub}.  Let $g > 0$ be an even integer.  A {\em
  $\Z_g$-graded \ainfty algebra} consists of a $\Z_g$-graded vector
space $A$ together with for each $d \ge 0$ a multilinear degree zero
{\em composition map}
\[ \mu^d: \ A^{\otimes d} \to A[2-d] \]
satisfying the {\em \ainfty-associativity equations} \cite[(2.1)]{seidel:sub}
\begin{multline} \label{ainftyassoc}  
0 = \sum_{\substack{n,m \ge 0 \\ n+m \leq d}} (-1)^{ n + \sum_{i=1}^n |a_i|}
\mu^{d-m+1}(a_1,\ldots,a_n, \\ \mu^m(a_{n+1},\ldots,a_{n+m}),
a_{n+m+1},\ldots,a_d) 
\end{multline}
for any tuple of homogeneous elements $a_1,\ldots,a_d$ with degrees
$|a_1|, \ldots, |a_d| \in \Z_g$.  The signs are the {\em shifted
  Koszul signs}, that is, the Koszul signs for the shifted grading in
which the structure maps have degree one as in Kontsevich-Soibelman
\cite{ks:ainfty}.  The notation $[2-d]$ denotes a degree shift by
$2-d$, so that without the shifting \label{shiftingl} \llabel{shifting} $\mu^1$ has
degree $1$, $\mu^2$ has degree $0$ etc.  The element $\mu^0(1) \in A$
(where $1 \in \Lambda$ is the unit) is called the {\em curvature} of
the algebra.  The \ainfty algebra $A$ is {\em flat} if the curvature
vanishes.  A {\em strict unit} for $A$ is an element $e_A \in A$ such
that
\begin{equation} \label{strictunit} 
 \mu^2(e_A,a) = a = (-1)^{|a|} \mu^2(a,e_A), \quad \mu^n(\ldots, e_A,
 \ldots) = 0, \forall n \neq 2 .\end{equation}
A {\em strictly unital} \ainfty algebra is an \ainfty algebra equipped
with a strict unit.  The {\em cohomology} of a flat \ainfty algebra
$A$ is defined by
\[ H(\mu^1) = \frac{ \on{ker}(\mu^1)}{\on{im}(\mu^1)} .\]
The algebra structure on $H(\mu^1)$ is given by
\begin{equation} \label{hcomp}
[a_1 a_2] = (-1)^{|a_1|} [\mu^2( a_1,a_2)] .\end{equation}
An element $e_A \in A$ is a {\em cohomological unit} if $[e_A]$ is a
unit for $H(\mu^1)$.  A result of Seidel \cite[Corollary 2.14]{se:bo}
implies that any flat \ainfty algebra with a cohomological unit is
equivalent to an \ainfty algebra with strict unit.  However, we will
construct strict units by a {\em homotopy unit} construction.

\section{Associahedra} 

The combinatorics of the \ainfty associativity axiom
\eqref{ainftyassoc} is closely related to a sequence of cell complexes
introduced by Stasheff \cite{st:ho} under the name {\em associahedra}.
One realization of these spaces is as {\em moduli spaces of metric
  ribbon trees}.  An {\em oriented tree} is a connected, cycle-free
graph given by a pair
\[{\Gamma} = (\Edge({\Gamma}),\Ver({\Gamma})) \]
where $\Ver({\Gamma})$ is the set of vertices and $\Edge({\Gamma})$ is the set of
edges equipped with {\em head} and {\em tail} maps 
\[ h, t : \Edge({\Gamma}) \to \Ver({\Gamma}) \cup \{ \infty \} .\]
The {\em valence} $|v|$ of any vertex $v \in \Ver({\Gamma})$ is the
number of edges $e \in h^{-1}(v) \cup t^{-1}(v)$ meeting the vertex
$v$.  An edge $e \in \Edge({\Gamma})$ is {\em combinatorially finite}
if $ \infty \notin \{ h^{-1}(e), t^{-1}(e) \}$ {\em semi-infinite} if
$\{ h^{-1}(e), t^{-1}(e) \} = \{v , \infty \}$ for some
$v \in \Ver({\Gamma})$ and {\em infinite} if \label{verv}
$\Ver(\Gamma) = \emptyset$ and $\Edge({\Gamma})$ has a single element
$e$.  If $\Ver({\Gamma})$ is non-empty, then we denote by
$\Edge_{-}({\Gamma})$ resp. $\Edge_{\rightarrow}({\Gamma})$ the set of
combinatorially finite resp. semi-infinite edges.  In the special case
$\Ver({\Gamma})$ is empty we denote by $\Edge_{\rightarrow}({\Gamma})$
two copies $e_+,e_-$ of the single edge $e \in \Edge(\Gamma)$ (so that
there is a single incoming $e_-$ and single outgoing $e_+$
semi-infinite edge).  A {\em ribbon tree} is a tree $\Gamma$ equipped
with a {\em ribbon structure}: a cyclic ordering
$o_v: \{ e \in \Edge({\Gamma}), e \ni v \} \to \{ 1, \ldots, |v| \}$
of the edges \llabel{ribbon}\label{ribbonp} incident to each vertex
$v \in \Ver({\Gamma})$; a cyclic ordering is an equivalence class
$[o_v]$ of orderings where two orderings $o_v, o_v'$ are equivalent if
they are related by a cyclic permutation
$o_v'(\tau(e)) = o_v(\tau(e)) + k \ \text{mod} \ |v| $.  A single edge
$e_0 \in \Edge_{\rightarrow}({\Gamma})$ is outgoing (with head at
$\infty$) and called the {\em root} of the tree.  All the edges of
$\Gamma$ are oriented towards the root $e_0$. \label{rootorient}
Incoming semi-infinite edges
$e \in \Edge_{\rightarrow}({\Gamma}), e \neq e_0$ are called {\em
  leaves}. \llabel{leaf}

A moduli space of metric ribbon trees is obtained by allowing the
finite edges to acquire lengths.  A {\em metric ribbon tree} is a pair
$T=({\Gamma},\ell)$ consisting of a ribbon tree ${\Gamma}$ equipped with a {\em
  metric $\ell$}.  By definition a metric is a labelling
\[ \ell: \ \Edge_{-}({\Gamma}) \to [0,\infty]\]
of its combinatorially finite edges $e \in \Edge_-({\Gamma})$ by
elements $\ell(e)$ of $[0,\infty]$ called {\em lengths}.  We think of
$T$ as the topological space obtained by joining together intervals
$T_e$ of length $\ell(e)$ at the vertices.  From this point of view
there is a natural equivalence relation on ribbon metric trees defined
by collapsing edges of length zero: Given a tree $T$ with an edge $e$
of length $\ell(e)$ zero, removing the interior $\on{int}(e)$ of the
edge $e$ from $T$ and identifying its head $h(e)$ and tail $t(e)$
gives an equivalent metric tree
\begin{equation} \label{equivT} T' = (T - \on{int}(e))/ (h(e) \sim
  t(e)) .\end{equation}

To obtain a compactification of the moduli space we also allow the
lengths of the edges to go to infinity in which case the edge becomes
a {\em broken edge}: The interior of the edge $e$ with length
$\ell(e) = \infty$ is equipped with a finite number of points
$b_1,\ldots, b_k \subset \on{int}(e)$ called {\em breakings}.  A {\em
  broken metric tree} is obtained from a finite collection of metric
trees by gluing roots to leaves as follows: given two metric trees
$T_1,T_2$ and semi-infinite root edge
$e_2 \in \Edge_{\rightarrow}(\Gamma_2)$ and leaf edge
$e_1 \in \Edge_{\rightarrow}(\Gamma_2)$, let $\ol{T}_1$
resp. $\ol{T}_2$ denote the space obtained by adding a point
$\infty_2$ resp. $\infty_1$ at the open end of $e_2$ resp. $e_1$.  The
space
\begin{equation} \label{glueT}
 T := \ol{T}_1 \cup_{\infty_1 \sim \infty_2} \ol{T}_2 \end{equation}
is a broken metric tree, the point $\infty_1 \sim \infty_2$ being
called a {\em breaking}.    See Figure \ref{cuttree}.
\begin{figure}[ht]
\includegraphics[height=.85in]{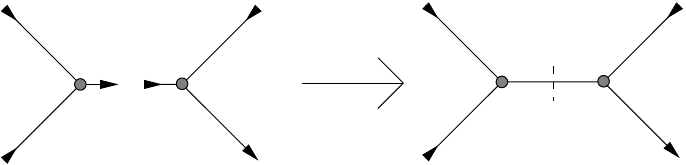} 
\caption{Creating a broken tree}
\label{cuttree}
\end{figure}
\noindent In general, broken metric trees $T$ are obtained by from
broken metric trees $T_1,T_2$ as in \eqref{glueT} in such a way that
the resulting space $T$ is connected and has no non-contractible
cycles, that is, $\pi_0(T)$ is a point and $\pi_1(T)$ is the trivial
group.  We think of the gluing points as breakings rather than
vertices, so that there are no new vertices in the glued treed $T$.
\label{clarify}
If a combinatorially finite edge $e \in \Edge_-(T)$ has infinite
length $\ell(e) = \infty$ then one attaches an additional positive
integer $b(e)$ to that edge indicating its number of breakings, see
\cite{charest:clust}.

In order to obtain a compact moduli space of broken trees a stability
condition is imposed.  A broken metric tree $T = (\Gamma,\ell,b)$ is
{\em stable} if and only if \label{stabletree} \llabel{stabletreel}
\label{stabletreelp}
each combinatorially semi-infinite edge $e \in \Edge_{\rightarrow}(T)$
is unbroken, that is, $b(e) = 0$; each combinatorially finite edge
$e' \in \Edge_{-}(T)$ is broken at most once, that is, $b(e') \leq 1$;
and the valence $|v|$ of each vertex $v \in \Ver(T)$ is at least $3$.
The moduli space of stable metric ribbon trees with a fixed number of
semiinfinite edges is a finite cell complex studied in, for example,
Boardman-Vogt \cite{boardman:hom}.  This moduli space is the first
realization of Stasheff's associahedron as a moduli space of geometric
objects.  However, the natural cell structure on this moduli space is
a {\em refinement} of the canonical cell structure on the
associahedra.

A second realization of the associahedron that reproduces the canonical cell
structure involves nodal disks with boundary markings.

\begin{definition} \label{ndisk} 
\begin{enumerate} 
\item A {\em holomorphic disk} is a complex surface with boundary
  diffeomorphic to the complex unit disk   $B^2 = \Set{ z \in \C  |  \Vert z \Vert \leq 1 }$.  A {\em nodal disk} with a
  single boundary node is a topological space $S$ obtained from a
  disjoint union of holomorphic disks $S_1,S_2$ by identifying pairs
  of boundary points $w_{12} \in S_1, w_{21} \in S_2$ on the boundary
  of each component so that
\begin{equation} \label{glueS} 
S = S_1 \cup_{w_{12} \sim w_{21}} S_2 .\end{equation} 
See Figure \ref{cutfig2}.
\begin{figure}[ht]
\includegraphics[height=.85in]{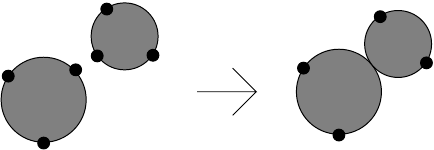} 
\caption{Creating a nodal disk}
\label{cutfig2}
\end{figure}
The image of $w_{12}, w_{21}$ in the space $S$ is the {\em nodal
  point}.  A nodal disk $S$ with multiple nodes
$w_{ij}, i,j \in \{ 1,\ldots, k\}, i \neq j$ is obtained by repeating
this construction \eqref{glueS} with $S_1,S_2$ nodal disks with fewer
nodes, and $w_{12}, w_{21}$ distinct from the previous nodes.  For an
integer $n \ge 0$ a {\em nodal disk with $n+1$ boundary markings} is a
nodal disk $S$ equipped with a finite ordered collection of points
$\ul{x} = (x_0,\ldots,x_n)$ on the boundary $\partial S$, disjoint
from the nodes, in counterclockwise cyclic order around the boundary
$\partial S$.
\item {\rm (Stable nodal disks)} An $(n+1)$-marked nodal disk
  $(S,\ul{z})$ is {\em stable} if each component $S_v$ has at least
  three special (nodal or marked) points.  The moduli space of
  $(n+1)$-marked stable disks forms a compact cell complex, isomorphic
  as a cell complex to the associahedron from Stasheff \cite{st:ho}.
\item {\rm (Sphere components and interior markings)} A {\em
    holomorphic sphere} is a complex surface biholomorphic to the
  projective line $S^2 \cong \P^1$.  We allow sphere components
  $\P^1 \cong S_v \subset S$ and interior markings
  $z_1,\ldots, z_n \in \on{int}(S)$ in the definition of marked nodal
  disks $S$.  A nodal disk $S$ with a single interior node $w \in S$
  is defined similarly to that of a boundary node by using the
  construction \eqref{glueS}, except in this case $S$ is obtained by
  gluing together a nodal disk $S_1$ with a holomorphic marked sphere
  $S_2$ with $w_{12}, w_{21}$ points in the interior 
  $\on{int}(S)$. 
 \llabel{glueI}   \label{glueIp} 
\end{enumerate}
\end{definition} 

A combination of the above constructions involves both trees and
disks, as in Oh \cite{oh:fl1}, Cornea-Lalonde \cite{cl:clusters},
Biran-Cornea \cite{bc:ql}, and Seidel \cite{seidel:genustwo}.  A {\em
  treed disk} $C$ is obtained from a nodal disk $S$ by replacing each
node $w$ with a (possibly broken) edge $e$ of some length $\ell(e)$
and $b(e) \ge 0$ breakings; that is, by replacing $w$ with two copies
$w_{12}, w_{21}$ and gluing to the endpoints of $T_e$.  We also allow
a (possibly broken) edge $e$ with no disks in which case $C \cong \R$
has two semi-infinite edges $e_+,e_-$ and an arbitrary number of
breakings $b(e)$.  Let $\Gamma(C)$ be the {\em combinatorial type} of
$C$, equal to the combinatorial type of the nodal disk $S$ but
equipped with the additional data of a number of breakings
$b: \Edge(C) \to \Z_{\ge 0}$.  Thus a treed disk $C$ consists of a
surface part
\[S = (S_v,\ul{x}_v,\ul{z}_v)_{v \in \Ver(\Gamma)}\]
(where $\ul{x}_v$ resp. $\ul{z}_v$ denotes the ordered set of boundary
resp. interior markings) a tree part
\[T = (T_e, \ell(e), b(e))_{e \in \Edge(\Gamma)} \]
and an ordering
\[o: \Edge_{\black,\rightarrow}(\Gamma) \to \{ 1, \ldots, n \}\]
of the set of interior leaves.  Denote by
\[ z_e = T_e \cap S, e \in \Edge_{\black,\rightarrow}(\Gamma) \]  
the attaching points of the interior leaves and call them {\em
  interior markings}.  A treed disk $C = S \cup T$ is {\em stable}
\llabel{stabletd} \label{stabletdp} if and only if the nodal disk $S$
is stable, each combinatorially-finite edge is broken at most once,
and each semi-infinite edge is unbroken.  The set of vertices
$\Ver(\Gamma)$ is equipped with a partition
\[ \Ver({\Gamma})= \Ver_{\white}({\Gamma}) \sqcup \Ver_{\black}({\Gamma})\]
into vertices corresponding to disks and vertices correspond to
spheres. Similarly, the set of edges is partitioned into sets of edges
\[\Edge({\Gamma}) = \Edge_{\white}({\Gamma}) \sqcup \Edge_{\black}({\Gamma})\] 
representing boundary nodes resp. interior nodes.  For each
$v \in \Ver_{\white}({\Gamma})$, the edges $e \in \Edge_{\white}(E)$
incident to $v$ are equipped with a cyclic ordering $o_v$ induced by
the orientation on the boundary of the disk $S_v$.  We call the
resulting tree ${\Gamma}$ equipped with a length function
$\ell: \Edge_-(\Gamma) \to [0,\infty]$ also a {\em metric ribbon
  tree}, although only some of the incident edges
$e \in (h \times t)^{-1}(v)$ are equipped with a cyclic ordering.
\begin{figure}[ht]
\includegraphics[height=1.5in]{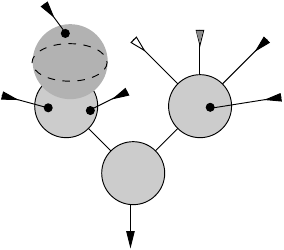}
\caption{A treed disk with three disk components and one sphere component}
\label{treeddisk}
\end{figure} 
\noindent See Figure \ref{treeddisk}.  An {\em isomorphism} of treed
disks $\phi: C \to C'$ with positive edge lengths is a collection of
isomorphisms $\phi_v: S_v \to S_v'$ of nodal disks preserving the
markings and a collection of isomorphisms $T_e \to T_e'$ of broken
edges (that is, preserving the number of breakings and lengths).
These data combine to a homeomorphism $\phi:C \to C'$ that is an
isometry $\phi |_T: T \to T'$ on the tree part and a biholomorphism
$\phi|_S : S \to S'$ on the surface part.  A treed disk $C$ with an
edge length $\ell(e)$ zero is declared equivalent to the treed disk
$C'$ where the edge $e$ is replaced by a node $w \in C'$.

In order to obtain Fukaya algebras with strict units, we attach
additional parameters to certain of the semi-infinite edges called
{\em weightings} as in Ganatra \cite{ganatra}.  When the weighting of
an edge is infinite, we will assume that the perturbation data is
pulled back under the forgetful map forgetting that edge and
stabilizing.  For this reason, the edges where the weightings are
forced to be infinite are called {\em forgettable}.
\llabel{forgettable}
\label{forgettablep}

\begin{definition} \label{wdef} {\rm (Weightings)} A {\em weighting} of
  a treed disk $C= S \cup T$ of type $\Gamma$ is 
\begin{enumerate}
\item {\rm (Weighted, forgettable, and unforgettable edges)} a
  partition of the boundary semi-infinite edges
\[\Edge^{\greyt}(\Gamma) \sqcup \Edge^{\whitet}(\Gamma) \sqcup \Edge^{\blackt}(\Gamma) =
  \Edge_{\white,\rightarrow}(\Gamma) \]
into {\em weighted} resp. {\em forgettable} resp.  {\em unforgettable}
edges, and
\item {\rm (Weighting) } a map 
\[\rho: \Edge_{\white,\rightarrow}(\Gamma) \to [0,\infty]\]
satisfying the property: each of the semi-infinite $e$ edges is
assigned a {\em weight} $\rho(e)$ such that
 \[ 
	  \rho(e) \in 
\begin{cases}  \{ 0 \}  &  e \in
  \Edge^{\blackt}(\Gamma) \\  [0,\infty] & e \in
  \Edge^{\greyt}(\Gamma) \\
  \{ \infty \} & e \in \Edge^{\whitet}(\Gamma) 
\end{cases}.
\] 
\end{enumerate} 
If the outgoing edge $e_0 \in \Edge_{\rightarrow}(\Gamma)$ is
unweighted (forgettable or unforgettable) then an isomorphism
$\psi: (C,\rho) \to (C',\rho')$ of weighted treed disks is an
isomorphism of treed disks $C \to C'$ that preserves the types of
semi-infinite edges
$e \in \Edge_{\rightarrow}(\Gamma) \cong \Edge_{\rightarrow}(\Gamma')$
and weightings: $\rho(e) = \rho'(e')$ for all corresponding edges
$e \in \Edge_{\white,\rightarrow}(\Gamma), \ e' \in
\Edge_{\white,\rightarrow}(\Gamma')$.  This ends the definition.
\end{definition} 

The case that the outgoing edge $e_0\in \Edge_{\rightarrow}(\Gamma)$
is weighted $\rho(e_0) > 0$ is rare in our examples and should be
considered an exceptional case.  There is an additional notion of
equivalence in this case: If the outgoing edge $e_0$ is weighted then
an isomorphism of weighted treed disks $C \to C'$ is an isomorphism of
treed disks preserving the types of semi-infinite edges
$e \in \Edge_{\white,\rightarrow}(\Gamma)$ and the weights
$\rho(e), e \in \Edge_{\white,\rightarrow}(\Gamma)$ up to scalar
multiples:
\llabel{scalarl} \label{scalarlp} 
\begin{equation} \label{scalar} 
\exists \lambda\in (0,\infty), \ \forall e \in \Edge_{\white,\rightarrow}(\Gamma), e' \in
 \Edge_{\white,\rightarrow}(\Gamma'), \ \rho(e) = \lambda \rho'(e').\end{equation}
In particular, any weighted tree $T$ such that 
$\Ver(\Gamma) = \emptyset$ and a single edge
$e \in \Edge_{\white,\rightarrow}(\Gamma)$ that is weighted
$\rho(e) \in (0,\infty)$ is isomorphic to any other such configuration
$T'$ with a different weight \llabel{ends} \label{endsp}
$\rho(e') \in (0,\infty), e \in \Edge(\Gamma')$.

The {\em combinatorial type} of any weighted treed disk is the tree
associated to the underlying nodal disk with additional data recording
which lengths resp. weights are zero or infinite.  Namely if
$C = S \cup T$ is a weighted treed disk then its combinatorial type is
the tree $\Gamma = \Gamma(C)$ \llabel{graphtotree}
\label{graphtotreep}
 obtained by gluing 
together the combinatorial types $\Gamma(S_v)$ of the disks $S_v$
along the edges corresponding to the edges of $T$; and equipped with 
the additional data of 
\begin{enumerate} 
\item the subsets 
\[\Edge^{\greyt}(\Gamma) \ \text{resp.} \ \Edge^{\whitet}(\Gamma) \ \text{resp.}
  \ \Edge^{\blackt}(\Gamma) \subset \Edge_{\white,\rightarrow}(\Gamma) \]
of weighted, resp. forgettable, resp. unforgettable semi-infinite 
edges;
\item the subsets 
\[\Edge_{-}^\infty(\Gamma) \ \text{resp.} \Edge_{-}^0(\Gamma) \ \text{resp.}
  \Edge_{-}^{(0,\infty)}(\Gamma) \subset \Edge_{-}(\Gamma) \]
of combinatorially finite edges of infinite resp. zero length resp.
non-zero finite length;
\item the subset 
\[\Edge^{\greyt,\infty}(\Gamma) \ \text{resp.}
  \ \Edge^{\greyt,0}(\Gamma) \subset \Edge^{\greyt}(\Gamma)
\]
of weighted edges with infinite resp. zero weighting.
\end{enumerate} 

A well-behaved moduli space of weighted treed disks is obtained after
imposing a stability condition.

\begin{definition} \label{wstable} A weighted treed disk
  $C = S \cup T$ of type $\Gamma$ is {\em stable} if either
\begin{enumerate}
\item {\rm (At least one disk)} there is at least one disk component
  $S_v, v \in \Ver_{\white}(\Gamma)$, and the following conditions hold:
\begin{enumerate} 
\item each disk component $S_v, v \in \Ver_{\white}(\Gamma)$ has at least
  three edges $e \in \Edge(\Gamma)$ attached to the boundary $\partial S_v$
  or at least one edge attached to the boundary $\partial S_v$ and one
  edge to the interior $\on{int}(S_v)$;
\item each sphere component $S_v, v \in \Ver_{\black}(\Gamma)$ has at least
  three edges $e \in \Edge(\Gamma)$ attached; 
\item each combinatorially-finite edge $e \in \Edge_-(\Gamma)$ is
  broken at most once, and each semi-infinite edge
  $e \in \Edge_{\rightarrow}(\Gamma)$ is unbroken;
\item if the outgoing edge is weighted
  $e_0 \in \Edge^{\greyt}(\Gamma)$ then at least one leaf
  $e_i \in \Edge_{\white,\rightarrow}(\Gamma), i > 0$ is also
  weighted, that is, $e_i \in \Edge^{\greyt}(\Gamma)$. 
\end{enumerate} 
\item {\rm (No disks)} if there are no disks, so that
  $\Ver(\Gamma) = \emptyset$, there is a single weighted leaf
  $e_1 \in \Edge^{\greyt}(\Gamma)$ and an unweighted (forgettable or
  unforgettable) root 
  $e_0 \in \Edge^{\whitet}(\Gamma) \cup \Edge^{\blackt}(\Gamma)$.
\end{enumerate} 
\end{definition} 

These conditions guarantee that the moduli space $\M_\Gamma$ of stable
weighted treed disks of each combinatorial type $\Gamma$ is expected
dimension, see Remark \ref{examples} below.  Because a configuration
$C \cong \R$ with no disks is allowed, the stability condition for
weighted treed disks is not equivalent to the absence of non-trivial
automorphisms, that is, the triviality $\Aut(C) \cong \{ 1 \}$ of the
group $\Aut(C)$ of automorphisms of $C$.

The moduli spaces of stable weighted treed disks are naturally cell
complexes with multiple cells of top dimension.  For integers
$n,m \ge 0$ denote by $\ol{{\M}}_{n,m}$ the moduli space of
isomorphism classes of stable weighted treed disks $C$ with $n$
boundary leaves and $m$ interior leaves. For each combinatorial type
$\Gamma$ denote by ${\M}_{\Gamma} \subset \ol{{\M}}_{n,m}$ the set of
isomorphism classes of weighted stable treed disks of type $\Gamma$.
The dimension of $\M_\Gamma$ is equal to $n+ 2m -2$, in the absence of
weighted edges.  In particular, if $\Gamma$ has no vertices
$v \in \Ver(\Gamma)$ then the dimension $\dim(\M_\Gamma)$ of
$\M_\Gamma$ is zero.  The moduli spaces decomposes into strata of
fixed type
\[ \ol{\M}_{n,m} = \bigcup_\Gamma \M_\Gamma .\]
Let $\ol{\M}_\Gamma$ denote the closure of $\M_\Gamma$ in
$\ol{\M}_{n,m}$.  In Figure \ref{interior} a subset of the moduli
space $\ol{\M}_{2,1}$ with one interior marking is shown, where the
interior marking is constrained to lie on the line half-way between
the special points on the boundary.  We denote by
$\M_{n,m} \subset \ol{\M}_{n,m}$ the interior, by which we mean the
(disconnected) union of top-dimensional strata.

\begin{remark} \label{examples} The moduli spaces of weighted treed
  disks are related to unweighted moduli spaces by taking products
  with intervals: If $\Gamma$ has at least one vertex
  $v \in \Ver(\Gamma)$ and $\Gamma'$ denotes the combinatorial type of
  $\Gamma$ obtained by setting the weights $\rho(e)$ to zero and the
  outgoing edge $e_0$ of $\Gamma$ is unweighted
  $e_0 \in \Edge_{\rightarrow}^{\blackt}(\Gamma)$ then
  \[ \M_\Gamma \cong \M_{\Gamma'} \times
  (0,\infty)^{|\Edge^{\greyt,(0,\infty)}(\Gamma)|} .\]
  If the outgoing edge $e_0$ is weighted
  $e_0 \in \Edge_{\rightarrow}^{\greyt}(\Gamma)$ and at least one
  leaf $e \in \Edge_{\rightarrow}^{\greyt}(\Gamma)$ is
  weighted then
\[ \M_\Gamma \cong \M_{\Gamma'} \times
(0,\infty)^{|\Edge^{\greyt,(0,\infty)}(\Gamma)| - 2} \]
since only the ratios of the weightings of leaves must be preserved by
the isomorphisms, see \eqref{scalar}.  In particular, if $\Gamma$ is a
type with a single weighted leaf
$e \in \Edge_{\rightarrow}^{\greyt}(\Gamma)$ and no vertices,
$\Ver(\Gamma) = \emptyset$, the outgoing edge may be unforgettable or
forgettable, the weightings on the leaf are irrelevant and $\M_\Gamma$
is a point.
\end{remark}

In general moduli spaces of stable curves only admit universal curves
in an orbifold sense. In the setting here orbifold singularities are
absent and the moduli spaces of stable treed disks admit honest
universal curves.  For any stable combinatorial type $\Gamma$ let
$\ol{\U}_\Gamma$ denote {\em universal treed disk} consisting of
isomorphism classes of pairs $(C,z)$ where $C$ is a treed disk of type
$\Gamma$ and $z$ is a point in $C$, possibly on a disk component
$S_v \cong \{ |z| \leq 1 \}$, a sphere component $S_v \cong \P^1$, or
one of the edges $e$ of the tree part $T \subset C$.  The map
\[ \ol{{\U}}_\Gamma \to \ol{{\M}}_\Gamma, \quad [C,z] \to [C] \]
is the universal projection.  Because of the stability condition,
there is a natural bijection
\[ \ol{\U}_{\Gamma}= \bigcup_{[C] \in \ol{\M}_\Gamma} C .\]
Denote by
\[ \ol{{\S}}_\Gamma = \{ [C = S \cup T, z] \in \ol{{\U}}_\Gamma \ |
\ z \in S \} \]
the locus where $z$ lies on a disk or sphere of $C$.  Denote by 
\[\ol{{\T}}_\Gamma = \{ [C = S \cup T, z] \in \ol{{\U}}_\Gamma \ |
\ z \in T \} \]
the locus where $z$ lies on an edge of $C$.  Hence
\[ \ol{{\U}}_\Gamma = \ol{{\S}}_\Gamma \cup \ol{{\T}}_\Gamma \] 
and $\ol{\S}_\Gamma \cap \ol{\T}_\Gamma$ is the set of points on the
boundary of the disks meeting the edges of the tree.  In case $\Gamma$
has no vertices we define $\ol{\U}_\Gamma$ to be the real line,
considered as a fiber bundle over the point $\ol{\M}_\Gamma$.  The
tree part splits into {\em interior} and {\em boundary} tree parts
depending on whether the edge is attached to an interior point or a
boundary point of a disk or sphere:
\begin{equation} \label{boundint}
 \ol{\T}_\Gamma = \ol{\T}_{\white,\Gamma} \cup \ol{\T}_{\black,\Gamma}
 .\end{equation}

\begin{figure}[ht]
\includegraphics[width=5in]{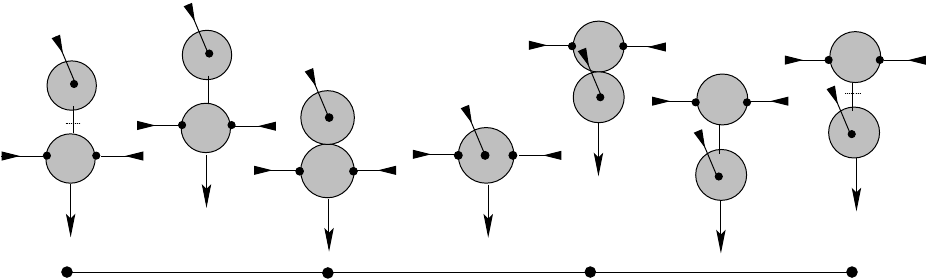}
\caption{Treed disks with interior leaves}
\label{interior}
\end{figure}

Later we will need local trivializations of the universal treed disk
and the associated families of complex structures and metrics on the
domains.  For a 
stable combinatorial type $\Gamma$ let
\[\tau^i : {{\U}}_{\Gamma}^i \to {{\M}}_{\Gamma}^i \times C,  \quad  i =
1,\ldots, l \]
be a collection of local trivializations of the universal treed disk.
The trivialization $\tau^i$ identifies each fiber $C'$ with the fixed
treed disk $C$.  The complex structures on the fibers of induce a family
\begin{equation} \label{localtriv} {\M}_{\Gamma}^i \to \J(S), \quad m \mapsto j(m) \end{equation}
of complex structures on the two-dimensional locus $S \subset C$.

The following operations on treed disks will be referred to in the
coherence conditions on perturbation data. 

\begin{definition} \label{bmgraphs} {\rm (Morphisms of graphs)} 
A {\em morphism} of graphs $\Upsilon: \Gamma \to \Gamma'$ is a
surjective morphism of the set of vertices $\Ve(\Gamma) \to
\Ve(\Gamma')$ obtained by combining the following {\em elementary
  morphisms}:
\begin{enumerate}
\item {\rm (Cutting edges) } \llabel{cutting} \label{cuttingp}
 $\Upsilon$ {\em cuts
    an edge with infinite length} with a single breaking if there
  exists
\[e \in \Edge_{-}(\Gamma'), \quad \ell(e) = \infty \]
so that the map $\Ver(\Gamma) \to \Ver(\Gamma')$ on vertices is a
bijection, and
\[\Edge_{\white}(\Gamma) \cong \Edge_{\white}(\Gamma') - \{ e \} + \{
e_+,e_- \} \]
where $e_\pm \in \Edge_{\white,\rightarrow}(\Gamma)$ are attached to
the vertices contained in $e$.  Since our graphs are trees, $\Gamma$
is disconnected with pieces $\Gamma_-,\Gamma_+$ that are types of
stable treed disks.  The ordering on
$\Edge_{\black,\rightarrow}(\Gamma')$ is required to agree with the
ordering on $\Edge_{\black,\rightarrow}(\Gamma_\pm)$ by viewing the
latter as a subset of the former.

The weighting and type of the cut edges are defined as follows.
Suppose that $\Gamma_-$ is the component of $\Gamma$ not containing
the root edge.  If $\Gamma_-$ has any interior leaves, set
$\rho(e_\pm) =0$ and $e_\pm \in \Edge^{\blackt}(\Gamma)$.  Otherwise
(and these are relatively rare exceptional cases in our examples and
used only for the construction of strict units) if there are no
interior leaves let $e_1,\ldots,e_k$ denote the leaves of
$\Gamma_-$.
\begin{enumerate} 
\item If any of $e_1,\ldots, e_k$ are unforgettable then $e_\pm \in
  \Edge^{\blackt}(\Gamma)$ are also unforgettable.
\item If none of $e_1,\ldots, e_k$ are unforgettable and at least one
  of $e_1,\ldots, e_k$ is weighted then $e_\pm \in
  \Edge^{\greyt}(\Gamma)$ are also weighted.
\item If $e_1,\ldots, e_k$ are forgettable then $e_\pm \in
  \Edge^{\whitet}(\Gamma)$ are also forgettable.
\end{enumerate} 
\begin{figure}[ht]
\includegraphics[height=1in]{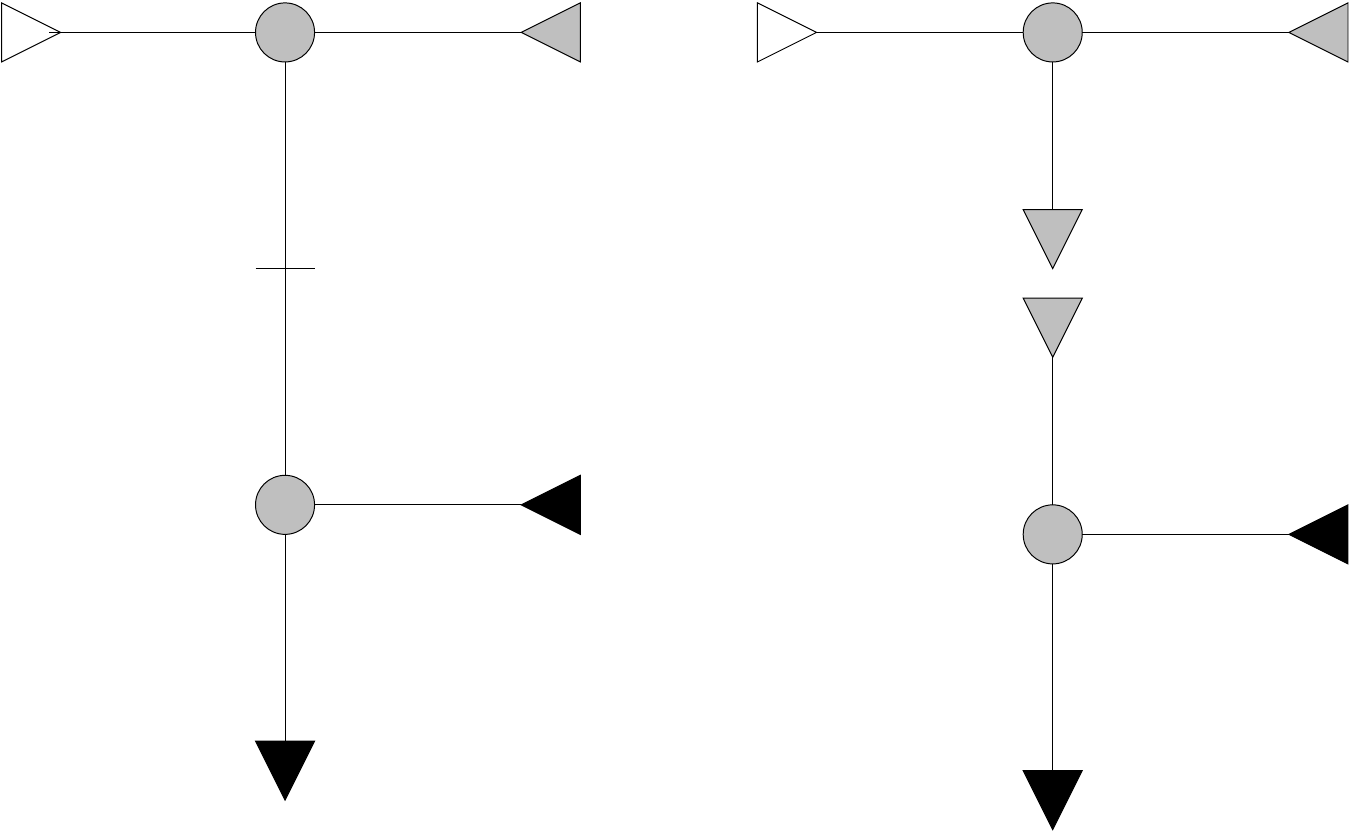}
\caption{Cutting an edge} 
\label{break}
\end{figure} 
 See Figure \ref{break} for an example.  Define the weighting on the
cut edges
$\rho(e_\pm) = \min(\rho(e_1),\ldots, \rho(e_k)) .$ 
In particular if $\Gamma_-$ has all zero weights $\rho(e_l) = 0, l =
1,\ldots, k $ then $\rho(e_\pm) = 0$.

\item {\rm (Collapsing edges)} $\Upsilon$ {\em collapses an edge} if
  the map on vertices is a bijection except for a single vertex $v'
  \in \Ve(\Gamma')$
\[\Ve(\Upsilon): \Ve(\Gamma) \to \Ve(\Gamma'), \quad \Ve(\Upsilon)^{-1}(v')
= \{ v_-, v_+ \} \]
that has two pre-images $v_\pm \in \Ve(\Gamma)$.  The vertices
$v_-,v_+$ are connected by an edge $e \in \Edge(\Gamma)$ so that
$\Edge(\Gamma') \cong \Edge(\Gamma) - \{ e \} .$  
See Figure \ref{collapse}.

\llabel{collapse} \label{collapsep}
 \begin{figure}[ht]
\includegraphics[height=0.85in]{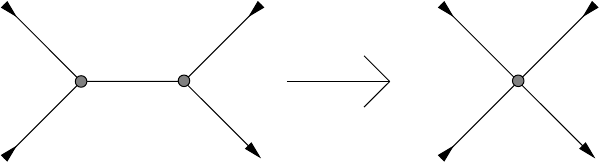} 
\caption{Collapsing an edge}
\label{collapse}
\end{figure}

\item {\rm (Making an edge length finite or non-zero)} $\Upsilon$ {\em
    makes an edge finite resp. non-zero} if $\Gamma'$ is the same
  graph as $\Gamma$ with the same orderings and the lengths of the
  edges of $\Gamma'$ and $\Gamma$ are the same except for a single
  edge $e$:
\[\ell |_{ \Edge_{-}(\Gamma) - \{ e \} } = \ell' |_{
      \Edge_{-}(\Gamma') - \{ e \} } .\]
For the edge $e$ we require 
\[\ell(e) = \infty \ \text{resp.} \ 0, \quad \ell(e') \in (0,\infty)
.\]
\item {\rm (Forgetting tails)} $\Upsilon: \Gamma \to \Gamma'$ {\em
    forgets a tail} (semi-infinite edge) and collapses edges to make
  the resulting combinatorial type stable.  The ordering on
  $\Edge_{\black,\rightarrow}(\Gamma)$ is required to agree with the
  one on $\Edge_{\black,\rightarrow}(\Gamma')$ viewing the latter as a
  subset.  See Figure \ref{forgettail}.
\item {\rm (Making an edge weight finite or non-zero)} $\Upsilon$ {\em
    makes a weight finite or non-zero} if $\Gamma'$ is the same graph
  as $\Gamma$ and the weights of the edges
  $\rho(e), e \in \Edge^{\greyt}(\Gamma)$ are the same with orderings
  except for a single edge $e$,
\[\rho |_{ \Edge_{-}(\Gamma) - \{ e \} } = \rho' |_{
  \Edge_{-}(\Gamma') - \{ e \} } .\]
For the edge $e$ we have 
$\rho(e) = \infty \ \text{resp.} \ 0$ and $ \rho'(e) \in (0,\infty).$

\begin{figure}[ht]
\includegraphics[height=.85in]{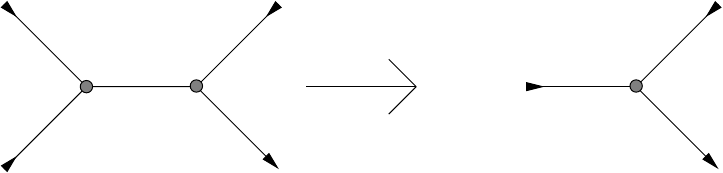} 
\caption{Forgetting a tail}
\label{forgettail} 
\end{figure}
\end{enumerate} 

\end{definition}

The operations of cutting edges commute.  For example, if $\Gamma'$ is
obtained from $\Gamma$ by cutting two edges, say
$e',e'' \in \Edge_{-}(\Gamma)$, then the induced weighting on
$\Gamma'$ is independent of the order of the cutting.  This follows
from the identity
\[\min(\rho(e_1),\ldots, \rho(e_j),
\min( \rho(e_{j+1}),\ldots, \rho(e_{j+k})), \ldots, \rho(e_i))=
\min(\rho(e_1),\ldots, \rho(e_i)) .\]
Each of the above operations on graphs corresponds to a map of moduli
spaces of stable marked treed disks.

\begin{definition} {\rm (Morphisms of moduli spaces)} 
\begin{enumerate} 
\item {\rm (Cutting edges)} Suppose that $\Gamma'$ is obtained from
  $\Gamma$ by cutting an edge $e$.  There are diffeomorphisms
\[\ol{\M}_{\Gamma} \to \ol{\M}_{\Gamma'}, \quad [C] \to [C'] \]
obtained as follows.  Given a treed disk $C$ of type $\Gamma'$, let
$z_+,z_-$ denote the endpoints at infinity of the edge corresponding
to $e$.  Form a treed disk $C'$ by identifying $z_+ \sim z_-$ and
choosing the labelling of the interior leaves to be that of $\Gamma'$.

\item {\rm (Collapsing edges)} Suppose that $\Gamma$ is obtained from
  $\Gamma'$ by collapsing an edge. There is an embedding
  $\iota_{\Gamma}^{\Gamma'}: \ol{\M}_{\Gamma} \to \ol{\M}_{\Gamma'}$.
  In the case of an edge of $\Edge_{-}^0(\Gamma)$, the image of
  $\iota_{\Gamma}^{\Gamma'}( \ol{\M}_{\Gamma})$ is a 1-codimensional
  corner of $\ol{\M}_{\Gamma}$.  In the case of an edge of
  $\Edge_{\black,\rightarrow}(\Gamma)$ the image
  $\iota_{\Gamma}^{\Gamma'}( \ol{\M}_{\Gamma})$ is a 2-codimensional
  submanifold of $\ol{\M}_{\Gamma'}$.
\item {\rm (Making an edge or weight finite resp. non-zero )} If
  $\Gamma$ is obtained from $\Gamma'$ by making an edge finite
  resp. non-zero then $\ol{\M}_{\Gamma}$ also embeds in
  $\ol{\M}_{\Gamma'}$ as the 1-codimensional corner.  The image is the
  set of configurations where the edge $e$ reaches infinite resp. zero
  length $\ell(e)$ or weight $\rho(e)$.
\item {\rm (Forgetting tails)} Suppose that $\Gamma'$ is obtained from
  $\Gamma$ by forgetting a tail (either in
  $\Edge_{\black,\rightarrow}(\Gamma)$ or
  $\Edge_{\white,\rightarrow}(\Gamma')$).  Collapsing the unstable
  vertices $v \in \Ver(\Gamma')$ and combining the lengths $\ell(e)$
  (if any) into the new metric
  $\ell(e') = \sum_{e \mapsto e'} \ell(e)$ defines a map
  $\ol{\M}_{\Gamma} \to \ol{\M}_{\Gamma'}$.  Each weighted
  semi-infinite edge $e$ for $\Gamma$ defines a weighted semi-infinite
  edge $e'$ for $\Gamma'$ with the same weight $\rho(e) = \rho(e')$.
\end{enumerate} 
\end{definition} 

 Each of the maps involved in the operations {\rm (Collapsing
   edges/Making edges or weights finite or non-zero), (Forgetting tails),
   (Cutting edges)} extends to a smooth map of universal treed disks.
 In the case that the type is disconnected we have 
\[\Gamma = \Gamma_1 \sqcup \Gamma_2 \ \ \implies \ \ \ol{\M}_\Gamma
 \cong \ol{\M}_{\Gamma_1} \times \ol{\M}_{\Gamma_2} .\]
In this case the universal disk $\ol{\U}_{\Gamma}$ is the disjoint
union of the pullbacks of the universal disks $\ol{\U}_{\Gamma_1}$ and
$\ol{\U}_{\Gamma_2}$: If $\pi_1, \pi_2$ are the projections on the
factors above then
\[ \ol{\U}_{\Gamma} = \pi_1^* \ol{\U}_{\Gamma_1} \sqcup \pi_2^*
\ol{\U}_{\Gamma_2} .\]
In the case of forgetting a tail, we denote by
$f^{\Gamma'}_{\Gamma}: \ol{\U}_{\Gamma} \to \ol{\U}_{\Gamma'}$ the map
of universal curves obtained by collapsing components that becomes
unstable after forgetting the tail, and by
$\ol{\U}_{\Gamma}^{\on{st}} \subset \ol{\U}_{\Gamma}$ the union of
components {\em not} collapsed by the forgetful morphism.  The
restriction of $f^{\Gamma'}_\Gamma$ to $\ol{\U}_{\Gamma'}^{\on{st}}$
has finite fibers, and injective except at the points $w \in T$ where
edges $T_{e'}, T_{e''}$ are glued together after a disk component
$S_v, v \in \Ver(\Gamma)$ collapses.

Orientations on the main strata (i.e. of maximal dimension) of the
moduli space of (non-weighted) treed disks may be constructed as
follows.

\begin{definition} \label{orientdef} {\rm (Orientations on moduli of treed disks)} 
\begin{enumerate} 
\item {\rm (A single disk)} Consider a stratum $\M_\Gamma$ of treed
  disks having a single disk corresponding to a single vertex
  $ \Ver(\Gamma) = \{ v \}$.  One can identify a smooth disk $S$ with
  $n+1$ boundary attaching points $x_0,\ldots,x_n$ and $m\geq 1$
  attaching points of interior leaves $z_1,\ldots,z_m$ with the
  positive half-space $\H \subset \C$ by a map
  $ \phi: S - \{ x_0 \} \to \H, \quad z_1 \mapsto i $ so that the
  boundary attaching points $x_i$, $i\geq 1$ map to an ordered tuple
  in $ \R \subset \C$.  For $m=0$ interior leaves, there are
  $n+1\geq 3$ boundary leaves.  We identify
  $ \phi: S - \{ x_0 \} \to \H$ so that
  $ x_1 \mapsto 0, \quad x_2 \mapsto 1 .$ The remaining boundary
  points $x_i$, $i\geq 3$ map to an ordered tuple of
  $]1,\infty[ \subset \R \subset \C$.  The moduli space $\M_\Gamma$ of
  disks of this type then inherits an orientation
  $O_\Gamma: \M_\Gamma \to \det( TM_\Gamma)$ from the canonical
  orientation on $\R^{n-2}\times \C^{m}$.
\item {\rm (Multiple disks)} Consider a stratum $\M_\Gamma$ treed
  disks having more than a single disk.  The closure $\ol{\M}_\Gamma$
  contains strata $\ol{\M}_{\Gamma'}$ with fewer edges with finite
  non-zero lengths
\[ \# \{ e \in \Edge (\Gamma') \ | \ \ell(e) \in (0,\infty) \} < \# \{
  e \in \Edge (\Gamma) \ | \ \ell(e) \in (0,\infty) \}  .\] 
  The inclusion identifies treed disks $C \to C'$ by identifying a
  boundary node $w \in C$ with an edge $e \subset C'$ of length
  $\ell(e)$ zero.  The addition of an edge $e$ of finite non-zero
  length $\ell(e) > 0$ corresponds to identifying the closures of two
  main strata on a 1-codimensional corner strata $\M_{\Gamma'}$.
  Choose an orientation $O_\Gamma$ so that the induced orientations on
  the boundary strata $\M_{\Gamma'}$ corresponding to a zero length
  are opposite.  The orientations $O_\Gamma$ then glue together to an
  orientation $O_{n,m}$ on $\ol{\M}_{n,m}$.
\end{enumerate} 
\end{definition} 

\section{Treed pseudoholomorphic disks}

The composition maps in the Fukaya algebra will be obtained
by counting treed pseudoholomorphic disks, which we now define.

\begin{definition} 
\begin{enumerate} 
\item {\rm (Gradient flow lines)} Let $L$ be a 
compact connected smooth manifold.  Denote by 
\[\cG(L) \subset \Map(TL^{\otimes 2},\R) \]
the space of smooth Riemannian metrics on $L$.  Fix a metric
$G \in \cG(L)$ and a Morse function $F: L \to \R$ having a unique
maximum $x_M \in L$.  Let $I \subset \R$ be a connected subset
containing at least two elements, that is, an open or closed interval.
The {\em gradient vector field} of $F$ is defined by
\[ \grad_F: L \to TL, \quad G( \grad_F, \cdot ) = \d F \in \Omega^1(L)
.\]
A {\em gradient flow line} for $-F$ is a map
\[u: I \to L, \quad \dds u = - \grad_F(u) \]
\label{gradsign}  \label{gradsignp} 
where $s$ is a unit velocity coordinate on $I$.
Given a time $s \in \R$ let
\[\phi_s: L \to L, \quad \dds \phi_s(x) = - \grad_F(\phi_s(x)), \quad \forall
x \in L \]
\llabel{gradsign}  denote the time $s$ gradient flow of $-F$.  
\item {\rm (Stable and unstable manifolds)} Denote by
\[\cI^{\on{geom}}(L) := \crit(F) \subset L \]
the space of critical points of $F$.  Taking the limit of the gradient
flow determines a discontinuous map
\[ L \to \crit(F), \quad y \mapsto \lim_{s \to \pm \infty} \phi_t(y)
.\]
By the stable manifold theorem each $x \in \cI^{\on{geom}}(L)$ determines stable
and unstable manifolds
\[W_x^\pm := \Set{ y \in L | \lim_{s \to \pm \infty} \phi_s(y) = x }
\subset L\]
consisting of points whose downward resp. upwards gradient flow
converges to $x$.  Denote by 
\begin{equation} \label{imap} i: \cI^{\on{geom}}(L) \to \Z_{\ge 0},
  \quad x \mapsto \dim(W_x^-) 
\end{equation} 
the {\em index map}.  The pair $(F,G)$ is {\em Morse-Smale} if the
intersections
\[W_{x_-}^+ \cap W_{x_+}^- \subset L\] 
are transverse for $x_+,x_- \in \cI^{\on{geom}}(L)$, and so a smooth manifold of
dimension $i(x_+) - i(x_-)$.   The additive group $\R$ acts on the
intersection $W_{x_-}^+ \cap W_{x_+}^-$ by the flow of $-\grad(F)$,
and the quotient 
\[ \M(x_+,x_-) = (W_{x_-}^+ \cap W_{x_+}^-) / \R \]
is canonically identified with the space of Morse trajectories from
$x_+$ to $x_-$.
\item {\rm (Almost complex structures)} Let $(X,\omega)$ be a
  symplectic manifold.  An almost complex structures on $(X,\omega)$
  given by
\[J : TX \to TX, \quad J^2 = -I\]
is {\em tamed} if and only if $\omega( \cdot, J \cdot)$ is positive
definite and {\em compatible} if in addition $\omega(\cdot, J \cdot)$
is symmetric, hence a Riemannian metric on $X$.  Denote by
$\J_\tau(X)$ the space of smooth tamed almost complex structures and
$\J_\tau(X)_l$ the space of tamed almost complex structures of class
$C^l$.  The space $\J_\tau(X)_l$ has a natural Banach manifold
structure modelled on the space of $C^l$ sections $\delta J$ of the
endomorphism bundle $\End(TX) \to TX$ satisfying the linearized
condition
\[ \delta J: X \to \End(TX), \quad J (\delta J) = - (\delta J) J .\]
\end{enumerate} 
\end{definition} 




In order to obtain the necessary transversality our Morse functions
and almost complex structures must be allowed to depend on a point in
the domain.  Fix a compact {\em thick part} of the universal tree
$ \ol{{\T}}^{\thick}_{\Gamma}
\subset 
\ol{{\T}}_{\Gamma} $
with the following property: Its interior
$\on{int}( \ol{{\T}}^{\thick}_{\Gamma} )$ contains at least one point
on each edge:
\[  
\on{int}(\ol{{\T}}^{\thick}_{\Gamma}) \cap \on{int}(e) \neq
\emptyset, \quad \forall \ \text{edges} \ e .\]
Also fix
a compact subset
\[ \ol{\S}_\Gamma^{\thick}
 \subset \ol{\S}_\Gamma - \{ w_e \in
\ol{\S}_\Gamma, \ e \in \Edge_{-}(\Gamma) \} \]
disjoint from the fiber-wise boundary $\partial \ol{\S}_\Gamma$ and
spherical nodes, with the property: $\ol{\S}_\Gamma^{\thick}$ contains
in its interior $\on{int}( \ol{\S}_\Gamma)$ at least one point on each
sphere and disk component $S_v \subset S, v \in \Ver(\Gamma)$ in each
fiber $S \subset \ol{\S}_\Gamma$.  Thus the complement
$
{\T}_\Gamma^{\thin} =  \ol{{\T}}_{\Gamma} 
- \ol{{\T}}^{\thick}_{\Gamma} 
\subset  \ol{{\T}}_{\Gamma}  $
is a neighborhood of infinity on each edge.  Furthermore, the
complement
$ {\S}_\Gamma^{\thin} = \ol{{\S}}_{\Gamma} 
- \ol{{\S}}^{\thick}_{\Gamma} 
\subset  \ol{{\S}}_{\Gamma}  $
is a neighborhood of the boundary and nodes.

\begin{definition} 
\begin{enumerate} 
\item {\rm (Domain-dependent Morse functions)} Let $\Gamma$
  be a type of stable treed disk.  Let $\ol{\T}_\Gamma \subset
  \ol{\U}_\Gamma$ be the tree part of the universal treed disk, and
  $\ol{\T}_{\white,\Gamma}$ its boundary part as in \eqref{boundint}.  Let
  $(F,G)$ be a Morse-Smale pair.  For an integer $l \ge 0$ a {\em
    domain-dependent perturbation} of $F$ of class $C^l$ is a $C^l$
  map
\begin{equation} \label{FGam}  F_{\Gamma}: \ol{{\T}}_{\white,\Gamma}
 \times L \to \R \end{equation}
equal to the given function $F$ away from the compact part:
\[ F_\Gamma | (\ol{{\T}}_{\white,\Gamma} - \ol{{\T}}^{\thick}_{\white,\Gamma}) =
\pi_2^* F \]
where $\pi_2$ is the projection on the second factor in \eqref{FGam}. 
\item {\rm (Domain-dependent almost complex structure)} Let
  $J \in \J_\tau(X)$ be a tamed almost complex structure.  Let
  $l \ge 0$ be an integer.  A {\em domain-dependent almost complex
    structure} of class $C^l$ for treed disks of type $\Gamma$ and
  base $J$ is a map
  \[ J_{\Gamma} : \ \ol{{\S}}_{\Gamma} \times X \to \End(TX) .\]
  We require that $J_\Gamma$ is equal to the given $J$ away from the
  compact part:
\[ J_\Gamma | 
(\ol{{\S}}_\Gamma - 
\ol{{\S}}^{\thick}_\Gamma )
= \pi_2^* J \]
where $\pi_2$ is the projection on the second factor in \eqref{FGam}.
\item {\rm (Perturbation data)} A perturbation datum for a type
  $\Gamma$ of holomorphic treed disk is a pair
  $P_\Gamma = (F_\Gamma,J_\Gamma)$ consisting of a domain-dependent
  Morse function $F_\Gamma$ and a domain-dependent almost complex
  structure $J_\Gamma$.
\end{enumerate}
\end{definition} 

The following are three operations on perturbation data related to the
morphisms of weighted graphs in Definition \ref{bmgraphs}.

\begin{definition} \label{pertops}
\begin{enumerate}
\item {\rm (Collapsing edges/making an edge or weight finite or
  non-zero)} Suppose that $\Gamma'$ is obtained from $\Gamma$ by
  collapsing an edge or making an edge length/weight finite/non-zero.  Any
  perturbation datum $P_{\Gamma'}$ for $\Gamma'$ induces a datum for
  $\Gamma$ by pullback of $P_{\Gamma'}$ under $\iota_\Gamma^{\Gamma'}:
  \ol{\U}_{\Gamma} \to \ol{\U}_{\Gamma'}$.
\item {\rm (Cutting edges)} Suppose that $\Gamma'$ is a combinatorial
  type obtained by cutting an edge of $\Gamma$.  A perturbation datum
  for $\Gamma'$ gives rise to a perturbation datum for $\Gamma$ by
  pushing forward $P_{\Gamma'}$ under the map
  $\pi_\Gamma^{\Gamma'}: \ol{\U}_{\Gamma} \to \ol{\U}_{\Gamma'}$.
  That is, define
\[J_{\Gamma}(z,x) = J_{\Gamma'}(z',x), \quad \forall z \in 
  (\pi_\Gamma^{\Gamma'})^{-1}(z') .\] 
  The definition is independent of the choice of lift $z$ by the
  (Constant near the nodes and markings) axiom.  The definition for
  $F_{\Gamma}$ is similar.

\item {\rm (Forgetting tails)} Suppose that $\Gamma'$ is a
  combinatorial type of stable treed disk is obtained from $\Gamma$ by
  forgetting a semi-infinite edge.  Consider the map of universal
  disks $f_{\Gamma'}^\Gamma: \ol{\U}_{\Gamma} \to \ol{\U}_{\Gamma'}$
  given by forgetting the edge and stabilizing.  Any perturbation
  datum $P_{\Gamma'}$ induces a datum $P_{\Gamma}$ by pullback of
  $P_{\Gamma'}$.
\end{enumerate}
\end{definition} 

We are now ready to define coherent collections of perturbation data.
These are data that behave well 
with each type of operation in 
Definition \ref{pertops}.  Given a type $\Gamma$, denote by 
$\Gamma_\circ$ the combinatorial type obtained by collapsing all 
vertices $v \in \Ver_{\black}(\Gamma)$ corresponding to spherical 
vertices to markings, and let $\Gamma(v) \subset \Gamma$ denote the 
subgraph with vertex $v$ and adjacent edges.

\begin{definition} \label{coherent} {\rm (Coherent families of
    perturbation data)} A collection of perturbation data
  $ \ul{P} = (P_\Gamma ) $ is {\em coherent} if it is compatible with
  the morphisms of moduli spaces $\M_\Gamma \to \M_{\Gamma'}$ 
  induced by morphisms of weighted treed disks $\Gamma \to \Gamma'$ in
  the sense that
\begin{enumerate} 
\item {\rm (Collapsing edges/making an edge or weight finite or non-zero)}
  if $\Gamma'$ is obtained from $\Gamma$ by collapsing an edge or
  making an edge/weight finite/non-zero, then $P_{\Gamma}$ is the
  pullback of $P_{\Gamma'}$;

\item {\rm (Cutting edges) } if $\Gamma'$ is obtained from $\Gamma$ by
  cutting an edge of infinite length, then $P_{\Gamma'}$ is the
  pullback of $P_{\Gamma}$.  If $\Gamma'$ is the union of types
  $\Gamma_1,\Gamma_2$ obtained by cutting an edge of $\Gamma$, then
  $P_{\Gamma'}$ is obtained from $P_{\Gamma_1}$ and $P_{\Gamma_2}$ as
  follows: Let
\[\pi_k: \ol{\M}_{\Gamma'} \cong \ol{\M}_{\Gamma_1} \times
  \ol{\M}_{\Gamma_2} \to \ol{\M}_{\Gamma_k}\] 
  denote the projection on the $k$th factor.  The universal curve
  $\ol{\U}_{\Gamma'}$ is the union of $\pi_1^* \ol{\U}_{\Gamma_1}$ and
  $\pi_2^* \ol{\U}_{\Gamma_2}$.  We require that $P_{\Gamma'}$ is
  equal to the pullback of $P_{\Gamma_k}$ on
  $\pi_k^* \ol{\U}_{\Gamma_k}$ for each $k \in \{ 1 ,2 \}$:
  \begin{equation} \label{require} P_{\Gamma'} | \ol{\U}_{\Gamma_k} =
    \pi_k^* P_{\Gamma_k} .\end{equation}
  In particular suppose that $\Gamma_1$ corresponds to a configuration
  $u: C \to X$ with a single unmarked disk $S_v \subset C$ and two
  leaves $T_{e'}, T_{e''} \subset C$, one of which, say $e'$ is
  weighted resp. forgettable as in the bottom row in Figure
  \ref{triv}.  Then by our conventions for (Cutting Edges) the
  corresponding leaf $e$ of $\Gamma_2$ is weighted resp. forgettable,
  with the same weight $\rho(e) = \rho(e')$ of the leaf $e'$ of
  $\Gamma_1$, and we require \eqref{require}.

\item {\rm (Locality axiom)} For any spherical vertex $v$ in a type
  $\Gamma$, the perturbation datum $P_\Gamma$ restricts to the
  pull-back of a perturbation $P_{\Gamma,v}$ on the image of
  $\pi^* \U_{\Gamma(v)}$ in $\U_\Gamma$.  (Note that the perturbation
  $P_{\Gamma,v}$ is allowed to depend on $\Gamma$, not just
  $\Gamma(v)$).  
Furthermore, the restriction of $P_\Gamma$ to the 
  disk components and boundary edges in $\U_\Gamma$ is the pull-back 
  of a perturbation datam $P_{\Gamma_\circ}$ from $\U_{\Gamma_\circ}$.

\item \label{infweight} {\rm (Forgettable edges)} If some weight
  parameter $\rho(e), e \in \Edge_{\rightarrow}(\Gamma')$ of a type
  $\Gamma'$ is equal to infinity and $\Gamma$ is the type obtained by
  (Forgetting tails) then
  $P_{\Gamma'} = (f^{\Gamma'}_\Gamma)^* P_{\Gamma}$ is pulled back
  under the forgetful map $f^{\Gamma'}_\Gamma$ forgetting the first
  forgettable leaf $e_i \in \Edge_{\rightarrow}^{\whitet}(\Gamma)$ and
  stabilizing from the perturbation datum $P_{\Gamma}$ given by
  (Forgetting tails).  The last sentence of the previous item
  guarantees that this condition is compatible with the product axiom,
  in the case that forgetting an edge $e$ with infinite weight
  $\rho(e) = \infty$ leads to a collapse of a disk component.
\end{enumerate} 
\end{definition}  

\begin{remark} The (Locality Axiom) implies that on any disk component
  $S_v, v \in \Ver_\circ(\Gamma)$ that the perturbation
  $J_\Gamma |S_v$ depends only on special points attached to disk
  components and the lengths $\ell(e)$ of the edges
  $e \in \Edge(\Gamma)$.
\end{remark}

Domain-dependent perturbations are obtained by pull-back from
inclusions into the universal curve.  Let $C$ be a stable treed disk
of type $\Gamma$.  Since $C$ is stable, $C$ admits a unique
identification with a fiber of a universal treed disk $\U_{\Gamma}$.
Given a perturbation datum $P_\Gamma$ for type $\Gamma$, we obtain a
domain-dependent almost complex structure and Morse function for $C$,
still denoted $J_\Gamma,F_\Gamma$, by pull-back under the map
$C \to \U_\Gamma$.  If $C$ is an arbitrary treed disk of type $\Gamma$
with at least one vertex, let $C \to C^{\on{st}}$ denote the
stabilization obtained by collapsing unstable sphere and disk
component and collapsing trajectories of infinite length (e.g. in the
case of combinatorially finite edges with more than one breaking or
sphere or disk components with too few special points) to points.  Let
$\Gamma^{\on{st}}$ denote the combinatorial type of $C^{\on{st}}$.  By
pull-back, one obtains domain-dependent perturbations for $C$ from
those for $C^{\on{st}}$.  A perturbation datum for type $\Gamma$ is a
perturbation for type $\Gamma^{\on{st}}$.  In the case that $C$ has
only one leaf, so that $C$ is a string of disks
$S_v \cong B^2 , v \in \Ver(\Gamma)$ and segments
$T_e, e \in \Edge(\Gamma(C))$, we take $J_\Gamma = J$ and
$F_\Gamma = F$ to be domain-independent.

\begin{definition} {\rm (Perturbed pseudoholomorphic treed disks)}  
  Given a perturbation datum $P_\Gamma$ for weighted treed disks of
  type $\Gamma$, a pseudoholomorphic treed disk in $X$ with boundary in $L$
  of type $\Gamma$ consists of a treed disk $C = S \cup T$ of type
  $\Gamma$ and a continuous map $u = (u_S,u_T): C \to X$ such that the
  following holds: Let $T = T_{\white} \cup T_{\black}$ be the
  splitting into boundary and interior parts as in \eqref{boundint}.
\begin{enumerate} 
\item {\rm (Lagrangian boundary condition)} On the boundary $u ( \partial S \cup T_{\white}) \subset L$.
\item {\rm (Surface equation)} On the surface part $S$ of $C$ the map
  $u$ is $J_\Gamma$-holomorphic for the given domain-dependent almost
  complex structure: if $j$ denotes the complex structure on $S$ then
  \[ J_{\Gamma,u(z),z} \ \d u_S = \d u_S \ j. \]
\item {\rm (Boundary tree equation)} On the boundary tree part $T_{\white}
  \subset C$ the map $u$ is a collection of gradient trajectories:
  \[ \dds u_T = -\grad_{ F_{\Gamma, (s,u(s))} } u_T \]
  where $s$ is a coordinate on the segment so that the segment has the
  given length.  Thus for each edge $e \in \Edge_{-}(\Gamma)$ the
  length of the trajectory is given by the length $u |_{e \subset T}$
  is equal to $\ell(e)$.
\item {\rm (Interior tree condition)} On the interior tree  $T_{\black}
  \subset T$ the map $u$ is constant.
\end{enumerate} 
\end{definition} 

The (Interior tree condition)  means that the interior parts $T_{\black}$ of the
tree $T$ are essentially irrelevant from our point of view.  However,
from a conceptual viewpoint if one is going to replace boundary
markings with edges then one should also replace interior markings
with leaves; this conceptual point will become important in the proof
of homotopy invariance in Chapter \ref{hinv}. 

A compact Hausdorff moduli space of treed disks is obtained by
requiring an energy bound and the following stability condition.  The
moduli space that we define will, however, not be smooth; we will
achieve smoothness for a perturbation of these moduli spaces later
using Cieliebak-Mohnke perturbations.  By a {\em node} of a treed disk
$C = S \cup T$ we mean an intersection point $S_{v'} \cap S_{v''}$ of
two components $S_{v'}, S_{v''}$ of the surface part $S$ or the
intersection $e \cap S_v$ of an edge $e \subset T$ with a surface
component $S_v \subset S$.

\begin{definition} \label{stable1} {\rm (Stable pseudoholomorphic
    treed disks)} A pseudoholomorphic weighted treed disk $u: C \to X$
  with interior nodes $z_1,\ldots,z_k$ and boundary nodes
  $w_1,\ldots, w_m$ is {\em stable} if
\begin{enumerate} 
\item each disk component on which $u$ is constant has at least three
  boundary nodes or one boundary node and one interior node:
\[ \d u(S_v) = 0 \ \ S_v \ \text{disk} \ \implies \ 2 \# \{ z_k,w_k
\in \on{int}(S_v) \} + \# \{ w_k \in \partial S_v \} \ge 3 ;
  \]
\item each sphere component $S_v \subset C$ component on which $u$ is constant
  has at least three special points:
\[ \d u(S_v) = 0, \ \ S_v \ \text{sphere} \ \ \ \implies \ \ \ \# \{
  z_k, w_k \in S_v \} \ge 3
  \]
and
\item \label{infline} each infinite line on which $u$ is constant has a
weighted leaf and an unforgettable or forgettable root:
\[ \d u (C) = 0 , \ \ C \ \text{line} \ \ \ \implies e_0 \in
  \Edge^{\greyt}(\Gamma) \ \text{and} \ \ e_1 \in
  \Edge^{\blackt}(\Gamma) \cup \Edge^{\whitet}(\Gamma) .\]
\end{enumerate} 
Note that we allow a configuration $u: C \to X$ with no disks or
spheres, so that $S = \emptyset$, and a single edge $T_e \cong \R$
equipped with a non-constant Morse trajectory $u: T_e \to L$.
\end{definition} 

The stability condition in Definition \ref{stable1} is not quite the
same as a map $u: C \to X$ having no non-trivial automorphisms
$\phi: C \to C, \phi^* u = u$.  Indeed given a configuration
$u: C \cong \R \to X$ on which $u$ is constant as in \eqref{infline},
we have $\Aut(u) \cong \R $, corresponding to translations of the line
$C \cong T \cong \R$.  However, the moduli space of such maps
$u: C \to X$ is still of expected dimension $0$ because the weighting
$\rho(e) \in (0,\infty)$ raises the expected dimension of the moduli
space $\M_\Gamma$.

There is a further notion of equivalence of pseudoholomorphic weighted
treed disks related to attaching constant trajectories.  Given a
non-constant pseudoholomorphic treed disk $u: C \to X$ with leaf $e_i$
for which the weighting $\rho(e_i) = \infty$ resp. $0 $, we declare
$u$ to be {\em equivalent} to the pseudoholomorphic treed disk
$u' : C \to X$ obtained by attaching to $e_i$ a constant trajectory
$u'': \R \to L$ with weighted incoming $e_i^-$ and forgettable
resp. unforgettable outgoing edge $e_i^+$.  See Figure \ref{triv0}.
\begin{figure}[ht]
\begin{picture}(0,0)%
\includegraphics{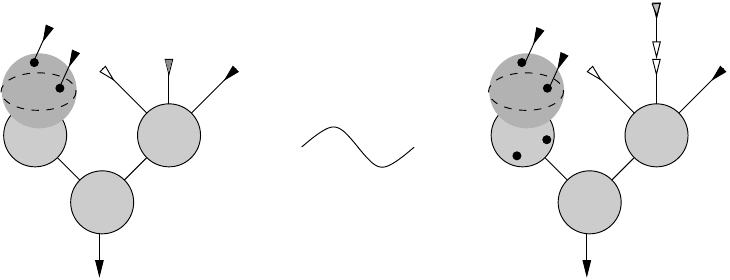}%
\end{picture}%
\setlength{\unitlength}{3947sp}%
\begingroup\makeatletter\ifx\SetFigFont\undefined%
\gdef\SetFigFont#1#2#3#4#5{%
  \reset@font\fontsize{#1}{#2pt}%
  \fontfamily{#3}\fontseries{#4}\fontshape{#5}%
  \selectfont}%
\fi\endgroup%
\begin{picture}(5818,2224)(-912,983)
\put(4393,2951){\makebox(0,0)[lb]{\smash{{{$\rho=\infty$}%
}}}}
\put(301,2789){\makebox(0,0)[lb]{\smash{{{$\rho = \infty$}%
}}}}
\end{picture}%
\caption{Equivalent weighted treed disks}
\label{triv0}
\end{figure}
Conversely, removing a constant segment $u'': \R \to X$ and
relabelling gives equivalent holomorphic treed disks.  Also, any two
configurations $u: C \to X, u':C' \to X$ with an outgoing weighted
edge $e_0$ with the same underlying tree $\Gamma$ are considered
equivalent.  See Figure \ref{equiv}.

\begin{figure}[ht]
\begin{picture}(0,0)%
\includegraphics{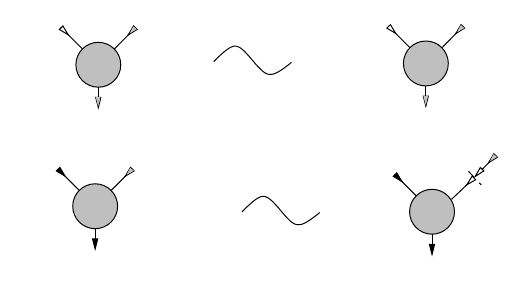}%
\end{picture}%
\setlength{\unitlength}{4144sp}%
\begingroup\makeatletter\ifx\SetFigFont\undefined%
\gdef\SetFigFont#1#2#3#4#5{%
  \reset@font\fontsize{#1}{#2pt}%
  \fontfamily{#3}\fontseries{#4}\fontshape{#5}%
  \selectfont}%
\fi\endgroup%
\begin{picture}(3846,2129)(3777,-4185)
\put(4434,-3014){\makebox(0,0)[lb]{\smash{{\SetFigFont{8}{9.6}{\rmdefault}{\mddefault}{\updefault}{\color[rgb]{0,0,0}$x^{\greyt}$}%
}}}}
\put(4709,-2193){\makebox(0,0)[lb]{\smash{{\SetFigFont{8}{9.6}{\rmdefault}{\mddefault}{\updefault}{\color[rgb]{0,0,0}$x^{\greyt}$}%
}}}}
\put(4115,-2173){\makebox(0,0)[lb]{\smash{{\SetFigFont{8}{9.6}{\rmdefault}{\mddefault}{\updefault}{\color[rgb]{0,0,0}$x^{\whitet}$}%
}}}}
\put(4790,-2415){\makebox(0,0)[lb]{\smash{{\SetFigFont{8}{9.6}{\rmdefault}{\mddefault}{\updefault}{\color[rgb]{0,0,0}$\rho_1$}%
}}}}
\put(4646,-2804){\makebox(0,0)[lb]{\smash{{\SetFigFont{8}{9.6}{\rmdefault}{\mddefault}{\updefault}{\color[rgb]{0,0,0}$\rho_1$}%
}}}}
\put(6929,-3005){\makebox(0,0)[lb]{\smash{{\SetFigFont{8}{9.6}{\rmdefault}{\mddefault}{\updefault}{\color[rgb]{0,0,0}$x^{\greyt}$}%
}}}}
\put(7205,-2183){\makebox(0,0)[lb]{\smash{{\SetFigFont{8}{9.6}{\rmdefault}{\mddefault}{\updefault}{\color[rgb]{0,0,0}$x^{\greyt}$}%
}}}}
\put(6611,-2163){\makebox(0,0)[lb]{\smash{{\SetFigFont{8}{9.6}{\rmdefault}{\mddefault}{\updefault}{\color[rgb]{0,0,0}$x^{\whitet}$}%
}}}}
\put(7608,-3140){\makebox(0,0)[lb]{\smash{{\SetFigFont{8}{9.6}{\rmdefault}{\mddefault}{\updefault}{\color[rgb]{0,0,0}$x^{\greyt}$}%
}}}}
\put(7286,-2406){\makebox(0,0)[lb]{\smash{{\SetFigFont{8}{9.6}{\rmdefault}{\mddefault}{\updefault}{\color[rgb]{0,0,0}$\rho_2$}%
}}}}
\put(7142,-2794){\makebox(0,0)[lb]{\smash{{\SetFigFont{8}{9.6}{\rmdefault}{\mddefault}{\updefault}{\color[rgb]{0,0,0}$\rho_2$}%
}}}}
\put(7509,-3459){\makebox(0,0)[lb]{\smash{{\SetFigFont{8}{9.6}{\rmdefault}{\mddefault}{\updefault}{\color[rgb]{0,0,0}$\rho=\infty$}%
}}}}
\put(6977,-4134){\makebox(0,0)[lb]{\smash{{\SetFigFont{8}{9.6}{\rmdefault}{\mddefault}{\updefault}{\color[rgb]{0,0,0}$x^{\blackt}$}%
}}}}
\put(6659,-3292){\makebox(0,0)[lb]{\smash{{\SetFigFont{8}{9.6}{\rmdefault}{\mddefault}{\updefault}{\color[rgb]{0,0,0}$x^{\blackt}$}%
}}}}
\put(4686,-3271){\makebox(0,0)[lb]{\smash{{\SetFigFont{8}{9.6}{\rmdefault}{\mddefault}{\updefault}{\color[rgb]{0,0,0}$x^{\greyt}$}%
}}}}
\put(4726,-3553){\makebox(0,0)[lb]{\smash{{\SetFigFont{8}{9.6}{\rmdefault}{\mddefault}{\updefault}{\color[rgb]{0,0,0}$\rho=\infty$}%
}}}}
\put(4410,-4092){\makebox(0,0)[lb]{\smash{{\SetFigFont{8}{9.6}{\rmdefault}{\mddefault}{\updefault}{\color[rgb]{0,0,0}$x^{\blackt}$}%
}}}}
\put(4092,-3251){\makebox(0,0)[lb]{\smash{{\SetFigFont{8}{9.6}{\rmdefault}{\mddefault}{\updefault}{\color[rgb]{0,0,0}$x^{\blackt}$}%
}}}}
\end{picture}%
\caption{Equivalent weighted treed disks, ctd.}
\label{equiv}
\end{figure}

We introduce notation for various moduli spaces of equivalence classes
of stable weighted treed disks.  For non-negative integers $n,m$
denote by $\ol{\M}_{n,m}(L)$ the (possibly empty) moduli space of
equivalence classes of stable treed pseudoholomorphic disks with $n$
boundary leaves and $m$ interior leaves.  A natural extension of the
Gromov topology on $\ol{\M}_{n,m}(L)$ is Hausdorff as in \cite[Section
5.6]{ms:jh}.  For any connected combinatorial type $\Gamma$ of treed
pseudoholomorphic disk, denote by $\M_\Gamma(L)$ the subset of type
$\Gamma$ so that
\[ \ol{\M}_{n,m}(L) = \bigcup_{\Gamma} \M_\Gamma(L) .\]
Denote by $\ol{\M}_\Gamma(L)$ the union of strata $\M_{\Gamma'}(L)$
such that there exists a morphism $\Gamma' \to \Gamma$.  After the
regularization below, $\ol{\M}_\Gamma(L)$ is the closure of
$\M_\Gamma(L)$, at least for the strata we consider.  The moduli space
of treed disks decomposes further into components depending on the
limits along the semi-infinite edges.  Define
\[\cI(L) = \left( \cI^{\on{geom}}(L) - \{ x_M \} \right) \cup \{x^{\whitet}, x^{\greyt},
x^{\blackt} \} . \]
Thus $\cI(L)$ is the set of  critical points of $F$, with the
maximum $x_M$ replaced by three copies $x^{\whitet}, x^{\greyt},
x^{\blackt}$.  We extend the index map on $\cI^{\on{geom}}(L)$ to $\cI(L)$
by
\[i(x^{\whitet}) = i(x^{\blackt}) = 0, \quad i(x^{\greyt})=-1 .\]
Define the set of generators of degree $d$ 
\begin{equation} \label{degreepart} 
 \cI_d(L) = \{ x \in \cI(L) \ | \ i(x) = d \}
 .\end{equation}
An {\em admissible labelling} of a (non-broken) weighted treed disk
$C$ with leaves $e_1,\ldots, e_n$ and outgoing edge $e_0$ is a
sequence $\ul{x} = (x_0,\ldots, x_n) \in \cI(L)$ satisfying:
\begin{enumerate}
\item {\rm (Label axiom)} If $x_i = x^{\greyt}$ resp. $x^{\whitet}$
  resp. $x^{\blackt}$ then the corresponding semi-infinite edge $e_i$
  is required to be weighted resp. forgettable resp. unforgettable,
  that is,
\begin{multline}  x_i = x^{\greyt} \ \text{resp.} \ x^{\whitet} \
  \text{resp.} \ x^{\blackt} \\ \implies
  \ e_i \in \Edge^{\greyt}(\Gamma)\ \text{resp.} \ e_i \in
  \Edge^{\whitet}(\Gamma) \ \text{resp.} \ e_i \in
  \Edge^{\blackt}(\Gamma) .\end{multline} 
Furthermore, in this case the limit along the $i$-th semi-infinite
edge is required to be $x_M$:
\[\lim_{s \to \infty} u(\varphi_e(s)) = x_M .\]
In every other case the semi-infinite edge is required to be
unforgettable:
\[ x_i \notin \{ x^{\greyt}, x^{\whitet}, x^{\blackt} \} \ \implies \ e_i \in
\Edge^{\blackt}(\Gamma) .\]
\vskip .1in
\item {\rm (Outgoing edge axiom)}  
\begin{enumerate} 
\item The outgoing edge $e_0$ is weighted,
  $e_0 \in \Edge^{\greyt}(\Gamma)$ only if there are two leaves
  $e_1,e_2 \in \Edge_{\white,\rightarrow}(\Gamma)$, exactly one of
  which, say $e_1$ is weighted with the same weight
  $\rho(e_1) = \rho(e_0)$, and the other, say $e_2$ is forgettable
  with weight $\rho(e_2) = \infty$, and there is a single disk $S$
  with no interior leaves, that is, $T \cap \on{int}(S) = \emptyset$.
\item The outgoing edge $e_0$ can only be forgettable, that is, $e_0
  \in \Edge^{\whitet}(\Gamma)$ if either
\begin{itemize}
\item there are two forgettable  leaves $e_1,e_2 \in
  \Edge^{\whitet}(\Gamma)$, or
\item there is a single leaf $e_1 \in \Edge^{\greyt}(\Gamma)$ that is
  weighted and the configuration $u: C \to X$ has no interior leaves,
  that is, $\Edge_{\black,\rightarrow}(\Gamma) = \emptyset$.
\end{itemize}
See Figure \ref{triv}.
\end{enumerate} 
\end{enumerate} 

The (Outgoing edge axiom) treats the case of constant 
treed disks.  As is typical in Floer theory, constant configurations
must be treated with great care.  Denote by
\[ \M_\Gamma(L,\ul{x}) \subset \M_\Gamma(L) \] 
the moduli space of isomorphism classes of pseudoholomorphic treed disks of
type $\Gamma$ with boundary in $L$ and admissible labelling
$\ul{x} = (x_0, \ldots, x_n)$.

\begin{remark} \label{const} {\rm (Constant trajectories)} If
  $x_1 = x^{\greyt}$ and $x_0 = x^{\blackt}$ resp.
  $x_0= x^{\whitet}$ then the moduli space $\M(L,x_0,x_1)$ contains
  a configuration with no disks and single edge on which $u$ is
  constant, corresponding to a weighted leaf
  $e \in \Edge^{\greyt}(\Gamma)$ and a root edge
  $e_0 \in \Edge(\Gamma)$ that is unforgettable resp.  forgettable.
  These trajectories are pictured in Figure \ref{triv}.
\end{remark}  

\begin{figure}[ht]
\begin{picture}(0,0)%
\includegraphics{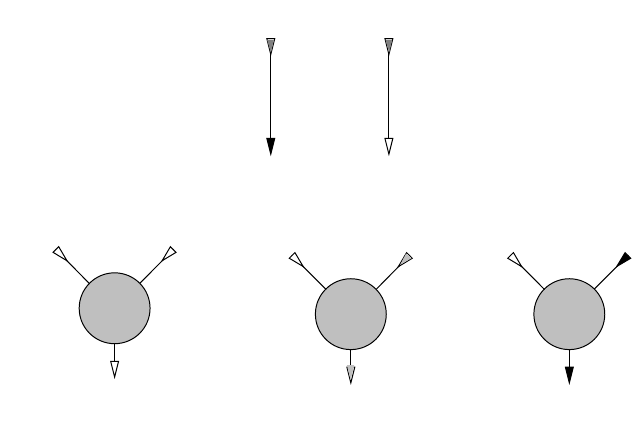}%
\end{picture}%
\setlength{\unitlength}{4144sp}%
\begingroup\makeatletter\ifx\SetFigFont\undefined%
\gdef\SetFigFont#1#2#3#4#5{%
  \reset@font\fontsize{#1}{#2pt}%
  \fontfamily{#3}\fontseries{#4}\fontshape{#5}%
  \selectfont}%
\fi\endgroup%
\begin{picture}(4819,3220)(2277,-3634)
\put(4804,-3561){\makebox(0,0)[lb]{\smash{{{$x^{\greyt}$}%
}}}}
\put(3701,-1173){\makebox(0,0)[lb]{\smash{{{$\rho = 0$}%
}}}}
\put(4111,-573){\makebox(0,0)[lb]{\smash{{{$x^{\greyt}$}%
}}}}
\put(5111,-593){\makebox(0,0)[lb]{\smash{{{$x^{\greyt}$}%
}}}}
\put(4141,-1783){\makebox(0,0)[lb]{\smash{{{$x^{\blackt}$}%
}}}}
\put(5581,-1783){\makebox(0,0)[lb]{\smash{{{$x^{\whitet}$}%
}}}}
\put(5491,-1163){\makebox(0,0)[lb]{\smash{{{$\rho = \infty$}%
}}}}
\put(2592,-2194){\makebox(0,0)[lb]{\smash{{{$x^{\whitet}$}%
}}}}
\put(3424,-2207){\makebox(0,0)[lb]{\smash{{{$x^{\whitet}$}%
}}}}
\put(2944,-3512){\makebox(0,0)[lb]{\smash{{{$x^{\whitet}$}%
}}}}
\put(4276,-2260){\makebox(0,0)[lb]{\smash{{{$x^{\whitet}$}%
}}}}
\put(5239,-2265){\makebox(0,0)[lb]{\smash{{{$x^{\greyt}$}%
}}}}
\put(6040,-2251){\makebox(0,0)[lb]{\smash{{{$x^{\whitet}$}%
}}}}
\put(6945,-2256){\makebox(0,0)[lb]{\smash{{{$x^{\blackt}$}%
}}}}
\put(6501,-3565){\makebox(0,0)[lb]{\smash{{{$x^{\blackt}$}%
}}}}
\end{picture}%
\caption{Unmarked treed disks}
\label{triv}
\end{figure}

The moduli spaces of stable pseudoholomorphic treed disks of any fixed
combinatorial type are cut out locally by Fredholm maps in suitable
Sobolev completions.  Let $p \ge 2$ and $k \ge 1$ be integers and let
$C = S \cup T$ be a treed disk.  Denote by $\Map^{k,p}(C,X,L)$ the
space of continuous maps $u$ from $C$ to $X$ of Sobolev class
$W^{k,p}$ on each disk, sphere and edge component such that the
boundaries $\partial C := \partial S \cup T$ of the disks and edges
map to $L$.  In each local chart for each component of $C$ and $X$ the
map $u$ is given by a collection of continuous functions with $k$
partial derivatives of class $L^p$.  Let $\exp: TX \to X$ denote
geodesic exponentiation with respect to a metric on $X$ is chosen for
which $L$ is totally geodesic, that is, $TL$ maps to $L$.  The space
$\Map^{k,p}(C,X,L)$ has the structure of a smooth Banach manifold,
with a local chart at $u \in \Map^{k,p}(C,X,L)$ 
\[ W^{k,p}(C, u^* TX, (u |_{\partial C})^* TL) \to \Map^{k,p}(C,X,L),
\quad \xi \mapsto \exp_u(\xi) .\]
Denote by
\begin{equation} \label{mapkp} \Map^{k,p}_{\Gamma}(C,X,L) \subset
  \Map^{k,p}(C,X,L) \end{equation}
the subset of maps $u: C \to X$ such that $u$ has the prescribed
homology class $[u | S_v]$ on each component
$S_v \subset S, v \in \Ver(\Gamma)$.  For each local trivialization of
the universal tree disk $\U_\Gamma$ as in \eqref{localtriv} we
consider the ambient moduli space defined as follows.  Let
$\Map^{k,p}_\Gamma(C,X,L)$ denote the space of maps of Sobolev class
$k \ge 1,p \ge 2, kp >2$ mapping the boundary of $C$ into $L$, defined
using a Sobolev norm induced by a choice of connections on $C$ and
$X$, independent of such choice.  The base of the universal bundle in
the local trivialization is
\[ \B^i_{k,p,\Gamma} := {\M}^i_{\Gamma} \times
\Map^{k,p}_\Gamma(C,X,L) .\]
Consider the map given by the local trivialization
\[ {\M}^{i}_{\Gamma} \to \J(S), \ m \mapsto j(m).\]
Consider the smooth Banach vector bundle $\E^i = \E^i_{k,p,\Gamma}$
over $\B^i_{k,p,\Gamma}$ whose fibers are given by
\[ (\E^i_{k,p,\Gamma})_{m, u, J } \subset
\Omega^{0,1}(C, u^* TX)_{k-1,p} := 
\Omega^{0,1}_{j,J,\Gamma}(S, u_S^* TX)_{k-1,p} \oplus \Omega^1(T,
u_T^* TL)_{k-1,p} \]
the space of $0,1$-forms with respect to $j(m),J$ that vanish to order
$m_\pm(e)-1$ at the nodes corresponding to each edge $e$.  Local
trivializations of $\E_{k,p,\Gamma}^i$ are defined by geodesic
exponentiation from $u$ and parallel transport using the Hermitian
connection defined by the almost complex structure
\[ \Phi_\xi: \Omega^{0,1}(S, u_S^* TX)_{k-1,p} \to 
\Omega^{0,1}(S, \exp_{u_S}(\xi)^*TX)_{k-1,p} \]
see for example \cite[p. 48]{ms:jh}.  The Cauchy-Riemann and shifted
gradient operators may be applied to the restrictions $u_S$
resp. $u_T$ of $u$ to the two resp. one dimensional parts of
$C = S \cup T$.  These define a section \llabel{gradf}\llabel{gradfp}
\begin{equation} \label{olp0} \olp_\Gamma: \B_{k,p,\Gamma}^i \to
  \cE_{k,p,\Gamma}^i, \quad (m,u) \mapsto
\left( \olp_{j(m),J} u_S , \left( \dds + \grad_F \right)u_T \right) \end{equation}
where 
\begin{equation} \label{olp2} \olp_{j(m),J} u := \hh (\d u_S + J \d
  u_S j(m)),
\end{equation} 
and $s$ is a local coordinate on $T$ with unit speed with respect to
the given metric.  Let $\Aut(C)$ be the group of automorphisms of $C$.
The {\em local moduli space} is
\[{\M}^{\univ,i}_{\Gamma}(L) = \olp^{-1} \B^i_{k,p,\Gamma}/\Aut(C) \]
where $\B^i_{k,p,\Gamma}$ is embedded as the zero section.  Define a
non-linear map on Banach spaces using the local trivializations
\begin{multline} \label{linop} 
\cF_{u}: \M_\Gamma^i \times \Omega^0(C,
  u^* (TX,T L))_{k,p} 
\to \Omega^{0,1}(S, u_S^*
  TX)_{k-1,p} \oplus \Omega^{1}(, u_T^* TL)_{k-1,p}, \\ \quad (\xi,C) 
  \mapsto \Phi_\xi^{-1} \olp_\Gamma \exp_{u}(\xi) .
\end{multline}
denote the linearization of $\cF_u$ (c.f. Floer-Hofer-Salamon  
\cite[Section 5]{fhs:tr}) 
\begin{equation} \label{Du} \ti{D}_u = D_0 \cF_u = \ddt |_{t = 0}
  \cF_u ( t \xi, t m ) .\end{equation}
Standard arguments show that the operator $\ti{D}_u$ is Fredholm.

\section{Transversality} 

In this section we regularize the moduli space of stable
pseudoholomorphic treed disks with boundary in a Lagrangian
submanifold using domain-dependent almost complex structures and
metrics.  The domains are stabilized using Donaldson hypersurfaces as
in Charest-Woodward \cite{cw:traj}; this construction extends that of
Cieliebak-Mohnke \cite{cm:trans}.  For a submanifold $L \subset X$,
let
\[h_2: \pi_2(X,L) \rightarrow H_2(X,L), \quad \text{resp.} \quad
[\omega]^\dual: H_2(X,L) \to \R\]
be the degree two relative Hurewicz morphism, resp.  the map induced
by pairing with $[\omega] \in H^2(X,L)$.  A symplectic manifold $X$
with two-form $\omega \in \Omega^2(X)$ will be called {\em rational}
if the class $[\omega] \in H^2(X,\R)$ is rational, that is, in the
image of $H^2(X,\Q) \to H^2(X,\R).$ Equivalently, $(X,\omega)$ is
rational if there exists a {\em linearization} of $X$: a line bundle
$\widetilde{X} \to X$ with a connection whose curvature is
$(2\pi k/i) \omega$ for some integer $k > 0$.
A Lagrangian $L \subset X$ of a rational symplectic manifold $X$ with
linearization $\widetilde{X}$ will be called {\em strongly rational}
if $\widetilde{X} | L$ has a covariant constant section
$L \to \widetilde{X} | L$.  A Lagrangian submanifold $L \subset X$ is
{\em rational} if and only if the set of areas of disks is discrete:
\[ \exists e > 0, \quad [\omega]^\dual \circ h_2(\pi_2(X,L)) = \Z \cdot e \subset \R.\] 

These rationality assumptions guarantee that we can find {\em
  stabilizing divisors} in the following sense. A {\em divisor} in $X$
is a closed codimension two symplectic submanifold $D \subset X$.  An
almost complex structure $J: TX \to TX$ is {\em adapted} to a divisor
$D$ if $D$ is an almost complex submanifold of $(X,J)$.  A divisor
$D \subset X$ is {\em strongly stabilizing} for a Lagrangian
submanifold $L \subset X$ if and only if $D$ is disjoint from $L$ and
any disk $u: (C,\partial C) \rightarrow (X,L)$ with non-zero area
$\omega([u])>0$ intersects $D$ in at least one point.  A divisor
$D \subset X$ is {\em weakly stabilizing} for a Lagrangian submanifold
$L \subset X$ if and only if $D$ is disjoint from $L$ and there exists
an almost-complex structure $J_D \in \mathcal{J}(X,\omega)$ adapted to
$D$ such that any non-constant $J_D$-holomorphic disk
$u: (C,\partial C) \rightarrow (X,L)$ intersects $D$ in at least one
point.

The existence of stabilizing divisors is an application of the theory
of Donaldson-Auroux-Gayet-Mohsen stabilizing divisors \cite{don:symp},
\cite{auroux:complement}.  Roughly speaking one chooses an
approximately holomorphic section $\sigma$ of $\ti{X}$ concentrated on
the Lagrangian $L$; then a generic perturbation defines the desired
divisor by $D = \sigma^{-1}(0)$.

\begin{theorem} \cite[Section 4]{cw:traj} There exists a divisor
  $D \subset X - L$ that is \label{weaklyp} \llabel{weakly}
 weakly stabilizing for
  $L$ representing $k[\omega]$ for some sufficiently large integer
  $k$. Moreover, if $L$ is rational resp. strongly rational then there
  exists a divisor $D \subset X - L$ that is weakly stabilizing for
  $L$ resp. strongly stabilizing for $L$ representing $k [\omega]$ for
  some large $k$ and such that $L$ is exact in $(X-D,\omega |_{X-D})$.
\end{theorem}

In the case that $X$ is a smooth projective algebraic variety,
stabilizing divisors may be obtained using a result of
Borthwick-Paul-Uribe \cite{bo:le}.  There is also a time-dependent
version of this result which will be used later to prove independence
of the homotopy type of the Fukaya algebra from the choice of base
almost complex structure: If $J^t, t \in [0,1]$ is a smooth path of
compatible almost complex structures on $X$ then (see \cite[Lemma
4.20]{cw:traj} ) there exists a path of $J^t$-stabilizing divisors
$D^t, t \in [0,1]$ connecting $D^0,D^1$.

\begin{definition}
  {\rm (Adapted treed disks)} Let $L$ be a compact Lagrangian and $D$
  a codimension two submanifold disjoint from $L$.  A stable
  pseudoholomorphic treed disk $u: C \to X$ with boundary in $L$ is {\em
    adapted} to $D$ if and only if
\begin{enumerate} 
\item {\rm (Stable surface axiom)} The domain $C= S \cup T$ has stable
  or empty surface part $S$; and
\item {\rm (Leaf axiom)} each interior semi-infinite leaf
  $e \subset C$ maps to $D$ and each connected component of
  $u^{-1}(D)$ contains an interior semi-infinite leaf $e$.
\end{enumerate} 
\end{definition}
\noindent The (Stable surface axiom) implies that the domain is stable
except that arbitrary breakings of each edge are allowed; in
particular, the case that $S$ is empty corresponds to a broken edge
with an arbitrary number of breakings.

The moduli space of adapted treed pseudoholomorphic disks is
stratified by combinatorial type as follows.  Denote by $\Pi(X)$
resp. $\Pi(X,L)$ the set of homotopy classes of maps from the two
sphere resp. disk with boundary in $L$.  The multiplicities of the
intersection with the divisor $D$ at $z_e \in \on{int}(S) \cap T$ are
the winding numbers $m_\pm(e)$ of small loop around $z_e$ in the
complement of the zero section in a tubular neighborhood of $D$ in
$X$. The combinatorial type of a pseudoholomorphic treed disk
$u: C \to X$ adapted to $D$ consists of
\begin{enumerate} 
\item the combinatorial type $\Gamma = (\Ver(\Gamma), \Edge(\Gamma))$
  (orderings etc. omitted from the notation) of its domain $C$
  together with
\item the labelling
\[d: \Ver(\Gamma) \to \Pi(X)
\cup \Pi(X,L)\] 
of each vertex $v$ of $\Gamma$ corresponding to a disk or sphere
component with the corresponding homotopy class and
\item the labelling
  \begin{equation} \label{me} m= (m_+,m_-): \Edge_{\black}(\Gamma) \to
    \Z_{\ge 0}^2 \end{equation}
  recording the intersection multiplicity of the map $u$ to the
  divisor $D$ at each of the nodes $z_e^\pm$ at the end of each edge
  $ e \in \Edge_{\black,\rightarrow}(\Gamma)$; we assign
  $m_\pm(e) = 0$ if the point $z_e^\pm$ does not map to $D$ or the
  entire component maps to $D$.
\end{enumerate}

We introduce the following notation.  Let $\ol{\M}(L,D)$ the moduli
space of adapted treed marked disks in $X$ with boundary in $L$ and
$\M_\Gamma(L,D)$ the locus of combinatorial type $\Gamma$.  For
$\ul{x} \in \cI(L)^n$ let
\[ {\M}_\Gamma(L,D,\ul{x}) \subset \M_\Gamma(L,D)\]
denote the adapted subset made of pseudoholomorphic treed disks of type
$\Gamma$ adapted to $D$ with limits $\ul{x} = (x_0,\ldots, x_n) \in
\cI(L)$ along the root and leaves.

The expected dimension of the moduli space is given by a formula
involving the Maslov indices of the disks and the indices and number
of semi-infinite edges.  Denote the components $u_v= u | S_v$ of $u$,
denote by the Maslov index $I(u_v)$ resp. twice the Chern number, if
$S_v$ is a sphere.  The expected dimension of
${\M}_\Gamma(L,D,\ul{x})$ at $[u: C \to X]$ is given by
\begin{multline} \label{expdim} i(\Gamma,\ul{x}) := i(x_0)- \sum_{i
    =1}^n i(x_i) + \sum_{v \in \Ver(\Gamma)} I(u_v) + n - 2 - |
  \Edge_{-}^0(\Gamma) | - |\Edge_-^\infty(\Gamma)| \\- 2 |
  \Edge_{\black,-}(\Gamma) | - \sum_{e \in
    \Edge_{\black,\rightarrow}(\Gamma)} 2m(e) - \sum_{e \in
    \Edge_{\black,-}(\Gamma)} 2m(e)\end{multline}
where $m(e) = m_+(e) + m_-(e)$ is the intersection multiplicity of the
map  with the stabilizing divisor at each end of the edge, see
\eqref{me}.   Denote by 
\[ {\M}_\Gamma(L,D,\ul{x})_d \subset \M_\Gamma(L,D,\ul{x}) \]
the locus of expected dimension $d$.    We denote by 
\[ \M(L,D,\ul{x}) = \bigcup_\Gamma \M_\Gamma(L,D,\ul{x}) \]
the union over strata  of top dimension, that is, for which the edge
lengths in $\Gamma$ connecting disk and sphere components are finite
and non-zero.  Finally, denote by
\[ \M(L,D,\ul{x})_d \subset \M(L,D,\ul{x}) \]
the locus of expected dimension $d$.  We call elements of
$\M(L,\ul{x})_0$ {\em rigid}.  The elements of
$\M_\Gamma(L,D,\ul{x})_0$ for types $\Gamma$ corresponding to strata
that are not of maximal dimension in $\ol{\M}(L,D,\ul{x})$ are
stratum-wise rigid, but at (formally) may be deformed in the direction
normal to the stratum.

The energy and the number of interior leaves are related as follows.
Let $\Gamma$ be a type of stable treed disk.  Disconnecting the
components that are connected by boundary nodes with positive length
one obtains types $\Gamma_1,\ldots, \Gamma_l$, and a decomposition of
the universal curve $\U_\Gamma$ into components
$\ol{\U}_{\Gamma_1},\ldots, \ol{\U}_{\Gamma_l}$.  Let $n(\Gamma_i)$
denote the number of interior leaves on $\ol{\U}_{\Gamma_i}$.  We
assume that $D$ has Poincar\'e dual given by $k[\omega]$.  Exactness
of $L$ in the complement of $D$ implies the following:

\begin{proposition}  \label{control}
  Any stable treed disk $u: C \to X$ with domain of type $\Gamma$ and
  only transverse intersections with the divisor has energy equal to
\[ E(u | C_i) \leq n(\Gamma_i,k) := \frac{n(\Gamma_i)}{k}\]
on the component $C_i \subset C$ contained in $\ol{\U}_{\Gamma_i}$.
\end{proposition} 

In order to obtain transversality we begin by fixing an open subset of
the universal curve on which the perturbations will vanish.  Let
$\ol{\U}_{\Gamma}^{\thick} \subset \ol{\U}_{\Gamma} $ be a compact
subset disjoint from the nodes and attaching points of the edges such
that the interior of $\ol{\U}_{\Gamma}^{\thick} $ in each two and
one-dimensional component is open.  Suppose that perturbation data
$P_{\Gamma'}$ for all boundary types
$\U_{\Gamma'} \subset \ol{\U}_\Gamma$ have been chosen.  Let
\[\cP_\Gamma^l(X,D) = \{ P_\Gamma = (F_\Gamma, J_\Gamma) \}\]
denote the space of perturbation data $P_\Gamma = (F_\Gamma,
J_\Gamma)$ of class $C^l$ that are
\begin{itemize} 
\item equal to the given pair $(F,J)$ on $ \ol{\U}_\Gamma -
  \ol{\U}_\Gamma^{\thick}$, and such that
\item the restriction of $P_{\Gamma}$ to $\ol{\U}_{\Gamma'}$ is equal
  to $P_{\Gamma'}$, for each boundary type $\Gamma$', that is, type of
  lower-dimensional stratum $\ol{\M}_{\Gamma'} \subset
  \ol{\M}_\Gamma$.
\end{itemize}
The second condition will guarantee that the resulting collection
satisfies the (Collapsing edges/Making edges or weights finite or
non-zero) axiom of the coherence condition Definition \ref{coherent}.
Let $\cP_\Gamma(X,D)$ denote the intersection of the spaces
$\cP_\Gamma^l(X,D)$ for $l \ge 0$.

One cannot expect, using stabilizing divisors, to obtain
transversality for all combinatorial types.  The reason is an analog
of the multiple cover problem: once one has a ghost bubble
$S_v \subset S$ mapping to the divisor $D$ then one has configurations
with arbitrary number of interior leaves $T_e \subset C$ meeting the
component $S_v$.  The expected dimension $\dim \M_\Gamma(L,D)$ of the
component $\M_\Gamma(L,D)$ of the moduli space containing these
configurations has limit minus infinity but each $\M_\Gamma(L,D)$ is
non-empty.  A type $\Gamma$ will be called {\em uncrowded} if each
maximal ghost component contains at most one endpoint of an interior
leaf.  Write $\Gamma' \prec \Gamma$ if and only if $\Gamma$ is
obtained from $\Gamma'$ by (Collapsing edges/making edge lengths or
weights finite/non-zero) or $\Gamma$ is obtained from $\Gamma'$ by
(Forgetting a forgettable tail).

\begin{theorem} \label{main} {\rm (Transversality)} Suppose that
  $\Gamma$ is an uncrowded type of adapted pseudoholomorphic treed
  disk of expected dimension $i(\Gamma,\ul{x}) \leq 1$, see
  \eqref{expdim}.  Suppose admissible perturbation data for types of
  adapted treed marked disk 
  $\Gamma' \prec \Gamma$ are given.  Then there exists a comeager
  subset
\[ {\cP}^{\reg}_\Gamma(X,D) \subset {\cP}_\Gamma(X,D) \]
of {\em regular perturbation data} for type $\Gamma$ coherent with the
previously chosen perturbation data such that if $P_\Gamma \in
\cP_\Gamma^{\reg}(X,D)$ then
\begin{enumerate} 
\item \label{transs} {\rm (Smoothness of each stratum)} the stratum
  ${\M}_{\Gamma}(L,D)$ is a smooth manifold of expected
  dimension;
\item \label{tubs} {\rm (Tubular neighborhoods)} if $\Gamma$ is obtained from
  $\Gamma'$ by collapsing an edge of $\Edge_{\white,-}(\Gamma')$ or
  making an edge or weight finite/non-zero  or by gluing
  $\Gamma'$ at a breaking then the stratum ${\M}_{\Gamma'}(L,D)$ has a
  tubular neighborhood in $\ol{\M}_{\Gamma}(L,D)$; and
\item \label{oriens} {\rm (Orientations)} there exist orientations on
  $\M_\Gamma(L,D)$ compatible with the morphisms {\rm (Cutting an
    edge)} and {\rm (Collapsing an edge/Making an edge/weight
    finite/non-zero)} in the following sense:
\begin{enumerate} 
\item If $\Gamma$ is obtained from $\Gamma'$ by {\rm (Cutting an
  edge)} then the isomorphism $\M_{\Gamma'}(L,D) \to \M_{\Gamma}(L,D)$
  is orientation preserving.
\item If $\Gamma$ is obtained from $\Gamma'$ by {\rm (Collapsing an
  edge)} or {\rm (Making an edge/weight finite/non-zero)} then the
  inclusion $\M_{\Gamma'}(L,D) \to \M_{\Gamma}(L,D)$ has orientation
  (using the decomposition
\[ T\M_{\Gamma}(L,D) | \M_{\Gamma'}(L,D) \cong \R \oplus
  T\M_{\Gamma'}(L,D)\]
and the outward normal orientation on the first factor) given by a
universal sign depending only on $\Gamma,\Gamma'$. 
\end{enumerate} 

\end{enumerate} 
\end{theorem} 

\begin{proof} 
  We discuss only \eqref{transs}; the item \eqref{tubs} is a
  combination of standard gluing theorems, c.f.  Schm\"aschke
  \cite[Section 7]{schmaschke}, while orientations \eqref{oriens} are
  discussed further in Remark \ref{oriensrem} below.  \eqref{transs}
  is an application of the Sard-Smale theorem to a universal moduli
  space.  In the first part of the proof we assume that $\Gamma$ has
  no forgettable leaves and construct a perturbation datum $P_\Gamma$
  by extending the given perturbation data on the boundary of the
  universal moduli space.  Let $p \ge 2$ and $k \ge 1$ be integers
  with $kp > 2$.  As in the discussion before \eqref{mapkp} denote by
  $\Map^{k,p}(C,X,L)$ the space of continuous maps $u$ from $C$ to $X$
  of Sobolev class $W^{k,p}$ on each disk, sphere and edge component
  such that the boundaries $\partial C := \partial S \cup T$ of the
  disks and edges mapping to $L$.  That is, the restriction of $u$ to
  each component is in class $W^{k,p}_{\loc}$ (which is independent of
  the choice of Sobolev norm) and has finite $W^{k,p}$ norm with
  respect to some choice of connections on the tangent bundles of the
  domain and target (which depends strongly on the choice of norm;
  however in the end the moduli spaces will be independent of such
  choices.)  In each local chart for each component of $C$ and $X$ the
  map $u$ is given by a collection of continuous functions with $k$
  partial derivatives of class $L^p$.  The space $\Map^{k,p}(C,X,L)$
  has the structure of a Banach manifold, with a local chart at
  $u \in \Map^{k,p}(C,X,L)$ given by the geodesic exponential map
\[ W^{k,p}(C, u^* TX, (u |_{\partial C})^* TL) \to \Map^{k,p}(C,X,L),
\quad \xi \mapsto \exp_u(\xi) \]
where we assume that the metric on $X$ is chosen so that $L$ is
totally geodesic, that is, preserved by geodesic flow.  For any
combinatorial type $\Gamma$ with vertices labelled by homology classes
denote by
$\Map^{k,p}_{\Gamma}(C,X,L) \subset \Map^{k,p}(C,X,L) $
the subset of maps such that $u$ has the prescribed homology class on
each component.

For each local trivialization of the universal tree disk as in
\eqref{localtriv} we define a Banach manifold that combines variations
in the domain with variations in the map, depending on a choice of
Sobolev norm.  Let $\Map^{k,p}_\Gamma(C,X,L,D)$ denote the subspace of
maps mapping the boundary of $C$ into $L$, the interior markings into
$D$, and constant on each disk with no interior marking, and with the
prescribed order of vanishing at the intersection points with $D$.
Let $l \gg k$ be an integer and
\[ \B^i_{k,p,l,\Gamma} := {\M}^i_{\Gamma} \times 
\Map^{k,p}_\Gamma(C,X,L,D) \times {\cP}^l_\Gamma(X,D) .\]
Consider the map given by the local trivialization
\[ {\M}^{\univ,i}_{\Gamma} \to \J(S), \ m \mapsto j(m).\]
Consider the Banach bundle $\E^i_{k,p,l,\Gamma}$ over
$\B^i_{k,p,l,\Gamma}$ given by
\[ (\E^i_{k,p,l,\Gamma})_{m, u, J } \subset
\Omega^{0,1}_{j,J,\Gamma}(S, (u_S)^* TX)_{k-1,p} \oplus \Omega^1(T,
(u_T)^* TL)_{k-1,p} \]
the space of $0,1$-forms with respect to $j(m),J$ that vanish to order
$m_\pm(e)-1$ at the nodes corresponding to the edge $e$, if $m_\pm(e)$
is positive.  The space $\E^i_{k,p,l,\Gamma}$ is then a $C^q$-Banach
vector bundle for $q < l - k$.  The Cauchy-Riemann and shifted
gradient operators may be applied to the restrictions $u_S$
resp. $u_T$ of $u$ to the two resp. one dimensional parts of
$C = S \cup T$.  These define a $C^q$ section
\begin{equation} \label{olp} 
\olp_\Gamma: \B_{k,p,l,\Gamma}^i \to \E_{k,p,l,\Gamma}^i, \quad (m,u,
J,F) \mapsto 
\left( \olp_{j(m),J} u_S , (\dds + \grad_F)u_T \right) \end{equation}
where 
\begin{equation} 
  \olp_{j(m),J} u :=  \hh ( \d u_S + J \d u_S j(m)), 
\end{equation} 
and $s$ 
is a local coordinate with unit speed.  The {\em local universal
  moduli space} is
\begin{equation} \label{luniversal}
{\M}^{\univ,i}_{\Gamma}(L,D) = \olp^{-1}
\B^i_{k,p,l,\Gamma} \end{equation} 
where $\B^i_{k,p,l,\Gamma}$ is embedded as the zero section. 

To show the local universal moduli space is cut out transversally, we
show that an element of the cokernel of the linearized operator
vanishes everywhere.  Suppose that
\[\eta = (\eta_2,\eta_1) \in \Omega^{0,1}_{j,J,\Gamma}(S, (u_S)^*
TX)_{k-1,p} \oplus \Omega^1(T, (u_T)^* TL)_{k-1,p} \]
is in the cokernel of derivative $\ti{D}_u$ of \eqref{olp}, introduced
in \eqref{Du} but now with an additional term arising from the
variation of almost complex structure and Morse function, with
$2$-dimensional part $\eta_2$ and one-dimensional part $\eta_1$.
Variation of \eqref{olp} with respect to the section $\xi_1$ on the
one-dimensional part $T$ gives
\[ 0 = \int_T (D_{u_1} \xi_1, \eta_1) \d s = \int_T (\xi_1, D_{u_1}^*
\eta_1) \d s, \quad \forall \xi \in \Omega^0_c(u_1^* TL) \]
where $D_{u_1}$ is the operator of \eqref{Du}.  Hence
\begin{equation} \label{covconst} \nabla_\xi * \eta_1 = 0 .\end{equation} 
On the other hand, the linearization of \eqref{olp} with respect to
the Morse function is pointwise surjective:
\[ \{ - \grad_{F_\Gamma} u_T(s) \ | \ F_{\Gamma} \in C^\ell_c(
\ol{{\T}}_{\white,\Gamma} \times L) \} = T_{u_T(s)} L .\]
This surjectivity implies that
\begin{equation} \label{pointvan} 
\eta_1 (u_T(s)) = 0, \quad \forall s \in \ol{\U}_\Gamma -
\ol{\U}_\Gamma^{\thin} .\end{equation}
Combining \eqref{covconst} and \eqref{pointvan} implies that
$\eta_1 = 0$.  Similarly the two-dimensional part $\eta_2$ satisfies
\begin{equation}  (D_u \xi_2, \eta_2) = 0, \quad \int_S ( Y \circ \d u \circ j ) \wedge \eta_2
= 0 \end{equation}
for every $\xi_2 \in \Omega^0(S, u^* TX)$ with given orders of
vanishing at the intersection points with $D$ and variation of almost
complex structure $Y \in \Omega^0(S, u^* \End(TX))$ as in
\cite[Chapter 3]{ms:jh}.  Hence in particular $D_u^* \eta_2 =0$ away
from the intersection points with $D$.  So if $\d u(z) \neq 0$ for
some $z \in S$ not equal to an intersection point then $\eta_2$
vanishes in an open neighborhood of $z$.  By unique continuation
$\eta_2$ vanishes everywhere on the component of $S$ containing $z$
except possibly at the intersection points with $D$.  At these points
$\eta_2$ could in theory be a sum of derivatives of delta functions.
That $\eta_2$ also vanishes at these intersection points follows from
\cite[Lemma 6.5, Proposition 6.10]{cm:trans}.

It remains to consider components of the two-dimensional part on which
the map is constant.  On such components variation of the almost
complex structure $J_\Gamma$ produces no variation of the one-form
$\olp_\Gamma u$ obtained by applying the Cauchy-Riemann operator.  Let
$S' \subset S$ be a union of disk and sphere components of
$S \subset C$.  If $u: C \to X$ is a map that is constant on
$S' \subset C$ the linearized operator is also constant is surjective
by standard arguments, c.f. Oh \cite{oh:rh}.  However, we also must
check that the matching conditions at the nodes are cut out
transversally.  Let $S''$ denote the normalization of $S'$, obtained
by replacing each nodal point in $S'$ with a pair of points in $S''$.
Since the combinatorial type of the component $S'$ is a connected
subgraph $\Gamma'$ of a tree, the combinatorial type $\Gamma'$ must
itself be a tree.  Denote by $T_u L$ the tangent space at the constant
value of $u$ on $S'$.  Taking the differences of the maps at the nodes
defines a map
\begin{equation} \label{diffs}  \delta: \ker(D_u | S'') \cong T_u L^k \to T_u^m, \quad \xi \mapsto (
\xi(w_i^+) - \xi(w_i^-) )_{i=1}^m .\end{equation}
An explicit inverse to $\delta$ is given by defining recursively as
follows.  Consider the orientation on the combinatorial type $\Gamma'
\subset \Gamma$ induced by the choice of outgoing semi-infinite edge
of $\Gamma$.  For $\eta \in T_u^m$ define an element $\xi \in T_u L^k$
by 
\[\xi(t(e)) = \xi(h(e)) + \eta(e) \] 
whenever $t(e),h(e)$ are the head and tail of an edge $e$
corresponding to a node.  This element may be defined recursively
starting with the edge $e_0''$ of $\Gamma'$ closest to the outgoing
edge of $\Gamma$ and taking $\xi(v) = 0$ for $v$ the vertex
corresponding to the disk component closest to outgoing edge.  The
matching conditions at the nodes connecting $S'$ with the complement
$S - S'$ are also cut out transversally.  Indeed on the adjacent
components the linearized operator restricted to sections vanishing at
the node is already surjective as in \cite[Lemma 6.5]{cm:trans}.  This
implies that the conditions at the boundary nodes cut out the moduli
space transversally.  A similar discussion for collections of sphere
components on which the map is constant implies that the matching
conditions at the spherical nodes are also cut out transversally.
This completes the proof that the parametrized linearized operator is
surjective.  The surjectivity of the parametrized linearized and the
implicit function theorem implies that the local universal moduli
space is a Banach manifold.  More precisely,
${\M}^{\univ,i}_{\Gamma}(L,D)$ is a Banach manifold of class $C^q$,
and the forgetful morphism
$\varphi_i: {\M}^{\univ,i}_{\Gamma}(L,D) \to \cP_{\Gamma}(L,D)_l $
is a $C^q$ Fredholm map.  Let
${\M}^{\univ,i}_{\Gamma}(L,D)_d \subset {\M}^{\univ,i}_{\Gamma}(L,D) $
denote the component on which $\varphi_i$ has Fredholm index $d$. 
\end{proof} 

In order to obtain compactness later it will also be necessary to
obtain a certain kind of transversality for {\em crowded}
combinatorial types.  

\begin{definition} {\rm (Regularity for crowded types)} For each such
  crowded combinatorial type $\Gamma$ let $\Gamma'$ denote the
  combinatorial type obtained by forgetting all but one marking on the
  ghost components.  The same proof shows that there exists a comeager
  subset $\PP^{\reg}_{\Gamma}(X,D)$ such that the moduli space
  $\M_{\Gamma'}(L,D)$ is a smooth manifold whose dimension is the sum
  of the expected dimension, necessarily empty if the expected
  dimension of the crowded type was at most one. We call a
  perturbation $P_\Gamma$ {\em regular} if the moduli space
  $\M_{\Gamma'}(L,D)$ is cut out transversally.
\end{definition}

\begin{definition} {\rm (Regular perturbations for types of disks)}
  For any combinatorial type $\Gamma$ of adapted treed disk there
  exist ( up the multiple breakings of edges) finitely many
  combinatorial types $\Gamma_X$ of {\em pseudoholomorphic} treed
  disks whose underlying type of treed disk is $\Gamma$.  Indeed, if a
  component $u|S_v: S_v \to X$ is non-constant, then its energy
  $E(u|S_v)$ is controlled by the number of interior leaves as in
  Proposition \ref{control}.  On the other hand, if $u|S_v$ is
  constant then its energy is zero.  Thus the energy of the map $u(C)$
  is controlled by the number $\# \Edge_{\black,\rightarrow}$ of
  interior leaves.  This implies that there are a finite number of
  combinatorial types $\Gamma_X$ that correspond to any particular
  $\Gamma$.  Given a type $\Gamma$ of domain, define
\[ \PP^{\reg}_\Gamma(X,D) = \cap_{\Gamma_X} \PP^{\reg}_{\Gamma_X}(X,D) \]
where $\Gamma_X$ ranges over types of pseudoholomorphic treed disks
with type of domain $\Gamma$.  We call perturbation data
$P_\Gamma \in \PP^{\reg}_\Gamma(X,D)$ {\em regular}.  Since countable
intersections of comeager sets are comeager, Theorem \ref{main}
implies the regular perturbation data form a comeager subset.  This
ends the definition. \end{definition}

\begin{remark} The strata $\M_{\Gamma'}(L,D,\ul{x})$ of expected
  dimension one, with $\M_{\Gamma'}$ of maximal dimension, have
  boundary strata $\M_{\Gamma}(L,D,\ul{x})$ that are of two possible
  types.
\begin{enumerate} 
\item The first possibility is that $\Gamma$ has a broken edge, as in Figure \ref{brokenfig}. 
\begin{figure}[ht]
\includegraphics[height=2in]{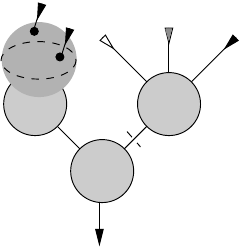}
\caption{Treed disk with a broken edge}
\label{brokenfig}
\end{figure} 
Note that the broken edge $e$ may be a leaf and that leaf may have
label $x^{\greyt}$ and weight $0$ or $\infty$ as in Remark
\ref{const}.  In this case any configuration $u$ of type $\Gamma$ is
constant on the first segment in the leaf and the normal bundle is
again isomorphic to $\R_{\ge 0}$. The normal direction corresponds to
deformations that replace that segment with an unbroken segment
labelled $x^{\greyt}$ and deform the weight $\rho(e)$ away from $0$
resp. $\infty$.
\item The second possibility is that $\Gamma$ corresponds to a stratum
  with a boundary node: either $T$ has an edge of length zero or
equivalently $S$ has a disk with a boundary node. See Figure \ref{zero}. 
\end{enumerate} 

\begin{figure}[ht]
\includegraphics[height=2in]{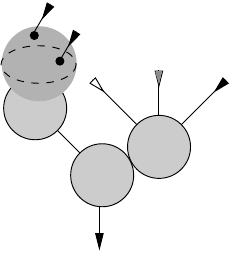}
\caption{Treed disk with a zero-length edge}
\label{zero}
\end{figure}

We call $\M_\Gamma(L,D)$ in the first resp. second case a {\em true}
resp.  {\em fake boundary stratum.}  A fake boundary stratum
$\M_\Gamma(L,D)$ is not a part of the boundary $\partial \ol{\M}(L,D)$
of the moduli space of adapted treed disks in the sense that
$\ol{\M}(L,D)$ is homeomorphic to an open ball, rather than a
half-ball near a $\M_\Gamma(L,D)$ if regular.  Indeed any such treed
disk $u:C \to X$ may be deformed either by deforming the disk node
$z \in C$, or deforming the length $\ell(e) $ of the edge corresponding
to the node to a positive real number $\ell(e) > 0$.
\end{remark}

\begin{remark} \label{oriensrem} 
{\rm (Orientations)} Orientations on moduli spaces of
  treed pseudoholomorphic disks are defined in the presence of
  relative spin structures on the Lagrangian as in Fukaya-Oh-Ohta-Ono
  \cite[Chapter 10]{fooo}, \cite{orient}.  A {\em relative spin
    structure} for an $n$-dimensional oriented Lagrangian
  $L \subset X$ is a lift of the \v{C}ech class of its tangent bundle
  $TL$ to relative non-abelian cohomology with values in $\Spin(n)$
  relative to the map $i: L \to X$; equivalently, a collection of
  lifts 
\[ \ti{\psi}_{\alpha \beta}: U_{\alpha} \cap U_\beta \to
\Spin(n) \]  
for $TL$ of the transition maps
$\psi_{\alpha \beta}: U_{\alpha} \cap U_\beta \to SO(n)$ for $TL$ with
respect to some open cover $\{ U_\alpha, \alpha \in I \}$ of $L$ to
$\Spin(n)$ satisfying the cocycle condition
\[ \ti{\psi}_{\alpha \beta} \ti{\psi}_{\alpha \gamma}^{-1} \ti{\psi}_{\beta \gamma} =
i^* \phi_{\alpha \beta \gamma} \]
where
$(\phi_{\alpha \beta \gamma} \in \{ \pm 1 \})_{\alpha \beta \gamma}$
is a $2$-cycle on $X$.  Given a relative spin structure, the
determinant line on the treed pseudoholomorphic disk is oriented as
follows: Fix coherent orientations on the stable and unstable
manifolds of the Morse function on $L$.  Let
$(u : C \to X) \in \M_\Gamma(L,D)$ be an adapted pseudoholomorphic
treed disk of combinatorial type $\Gamma$ of expected dimension $0$.
Assuming that regularity has been achieved, the linearization
$D \olp_\Gamma|_{P_\Gamma}$ of the section $\olp_\Gamma$ of
\eqref{olp} restricted to the perturbation $P_\Gamma$ is an
isomorphism.  Choose orientations
\[ O_{x_i}: W^\pm(x_i) \to \det (T W^{\pm(x_i)}) \] 
on the stable and unstable manifolds of the Morse function
$W^\pm(x_i)$ for $x_i \in \cI(L)$ so that the map
\[ T_{x_i} W^-(x_i) \oplus T_{x_i} W^+(x_i) \to T_{x_i} L \]
induces an orientation-preserving isomorphism of determinant lines
\begin{equation} \label{coho} \det (T_{x_i} W^-(x_i) \oplus T_{x_i} W^+(x_i)) \to \det(T_{x_i} L)
.\end{equation}
In case $x_i = x^{\greyt}$ we define 
$W^\pm(x^{\greyt}) = W^\pm(x_M) \times \R$
and choose orientations similarly.  One naturally obtains an
orientation of the determinant line $\det(\ti{D}_u)$ of the linearized
operator $\ti{D}_u$ of \eqref{Du} from the isomorphism
\begin{equation} \label{orienteq} \det(\ti{D}_u) \to
  \det(T\M_{\Gamma}) \otimes \det(TL) \otimes \det( TW^+(x_0)) \otimes
  \bigotimes_{i=1}^n \det(TW^-(x_i)), \end{equation}
the orientations on the stable and unstable manifolds $W^\pm(x_i)$ and
the orientation on the underlying moduli space $\M_\Gamma$ of treed
disks.  See \cite{charest:clust}, \cite{orient} for similar
discussions.   In particular, denote  by 
\begin{equation} \label{signs} \eps: \M(L,D)_0 \to \{ \pm 1
  \} \end{equation}
the map assigning to any rigid treed disk $u$ the associated sign.
The case of the type $\Gamma$ of a trivial trajectory
$u: C \cong \R \to X$ with no disks is treated separately: In the case
of a trajectory connecting $x^{\greyt}$ with $x^{\whitet}$
resp. $x^{\blackt}$ the moduli space is a point and we define the
orientation to agree resp. disagree with the standard orientation.
This ends the Remark.
\end{remark} 

\section{Compactness}  

In this section we show that the subset of the moduli space satisfying
an energy bound is compact for suitable perturbation data. 

\begin{definition} \label{Estab} For $E > 0$, an almost complex structure
  $J_D \in \J_\tau(X)$ is {\em $E$-stabilized} by a divisor $D$ if and only if
\begin{enumerate}
\item {\rm (Non-constant spheres)} $D$ contains no
  non-constant $J_D$-holomorphic spheres of energy less than $E$;
  and
\item {\rm (Sufficient intersections)} each non-constant
  $J_D$-holomorphic sphere $u: C \to X$ resp. $J_D$-holomorphic disk
  $u: (C,\partial C) \to (X,L)$ with energy less than $E$ has at least
  three resp. one intersection points resp. point with the divisor
  $D$:
\[ E(u) < E \ \ \implies \ \ \# u^{-1}(D) \ge 1 + 2 (\chi(C) - 1) \]
where $\chi(C)$ is the Euler characteristic.
\end{enumerate} 
Denote by 
\[H_2(X,\Z)_{\black} \subset H_2(X,\Z), \quad \text{resp.} \quad
H_2(X,L,\Z)_{\white} \subset H_2(X,L,\Z) \] 
the set of classes representing non-constant $J_D$-holomorphic
spheres, resp.  non-constant $J_D$-holomorphic disks with boundary in
$L$.  A divisor $D$ with Poincar\'e dual $[D]^\dual = k [\omega]$ for
some $k \in \N$ has {\em sufficiently large degree} for an almost
complex structure $J_D$ if and only if
\begin{equation} 
\label{large} 
\begin{array}{lll} 
([D]^\dual,\alpha) & \geq 2(c_1(X),\alpha) + \dim(X) + 1 &  \forall \alpha \in H_2(X,\Z)_{\black}\\ 
([D]^\dual,\beta) & \geq 1 & \forall \beta \in H_2(X,L,\Z)_{\white} .
\end{array}
\end{equation} 
This ends the definition.
\end{definition} 

We introduce the following notation for spaces of almost complex
structures sufficiently close to the given one.  Given $J \in
\J_\tau(X,\omega)$ denote by 
\[\J_\tau(X,D,J,\theta) = \{ J_D \in \J_\tau(X,\omega) \ | \Vert J_D -
J \Vert < \theta, \quad J_D(TD) = TD \} \]
the space of tamed almost complex structures close to $J$ in the sense
of \cite[p. 335]{cm:trans} and preserving $TD$.  The following lemma
on existence of stabilizing almost complex structures is a special
case of the results of \cite{cm:trans}.

\begin{lemma} \label{largelem} {\rm (Cieliebak-Mohnke \cite[Proposition
  8.14, Corollary 8.20]{cm:trans}. )} For $\theta$ sufficiently small,
  suppose that $D$ has sufficiently large degree for an almost complex
  structure $\theta$-close to $J$.  For each energy $E > 0$, there
  exists an open and dense subset
\[\J^*(X,D,J,\theta,E) \subset \J_\tau(X,D,J,\theta)\] 
such that if $J_D \in \J^*(X,D,J,\theta,E)$, then $J_D$ is
$E$-stabilized by $D$.  Similarly, if $D = (D^t)$ is a family of
divisors for $J^t$, then for each energy $E > 0$, there exists a dense
and open subset of time-dependent almost complex structures
\[\J^*(X,D^t,J^t,\theta,E) \subset \J_\tau(X,D^t,J^t,\theta)\] 
such that if $J_{D^t} \in \J^*(X,D^t,J^t,\theta,E)$, then $J_{D^t}$ is
$E$-stabilized for all $t$.
\end{lemma} 

We restrict to perturbation data taking values in
$\J^*(X,D,J,\theta,E)$ for a (weakly or strictly) stabilizing divisor
$D$ having sufficiently large degree for an almost-complex structure
$\theta$-close to $J$.  Let $J_D \in \J_\tau(X,D,J,\theta)$ be an
almost complex structure that is stabilized for all energies, for
example, in the intersection of $\J^*(X,D,J,\theta,E)$ for all $E$.
For each energy $E$, there is a contractible open neighborhood
$\J^{**}(X,D,J_D,\theta,E)$ of $J_D$ in $\J^*(X,D,J,\theta,E)$ that is
$E$-stabilized.

\begin{definition} \label{stabilized} A perturbation datum
  $P_\Gamma = (F_\Gamma,J_\Gamma)$ for a type of stable treed disk
  $\Gamma$ is {\em stabilizing} with respect to $D$ if $J_\Gamma$
  takes values in $\J^*(X,D,J,\theta, n(\Gamma_i,k))$ on
  $\ol{\U}_{\Gamma_i}$ (in particular, if $J_\Gamma$ takes values in
  $\J^{**}(X,D,J,\theta,n(\Gamma_i,k))$.)
\end{definition} 

To save space we call a perturbation system $\ul{P} = (P_\gamma)$ {\em
  admissible} if it is coherent, regular, and stabilizing.

\begin{theorem} \label{compthm} {\rm (Compactness for fixed type)} For
  any collection $\ul{P} = (P_\Gamma)$ of admissible perturbation data
  and any uncrowded type $\Gamma$ of expected dimension at most one,
  the moduli space $ \ol{\M}_\Gamma(L,D)$ of adapted treed marked
  disks of type $\Gamma$ is compact and the closure of
  $\M_\Gamma(L,D)$ contains only configurations with disk bubbling.
\end{theorem} 

\begin{proof} (c.f.  \cite[Proof of Proposition 4.10]{cw:traj}.)
  Because of the existence of local distance functions, similar to
  \cite[Section 5.6]{ms:jh}, it suffices to check sequential
  compactness.  Let $u_\nu: C_\nu \to X$ be a sequence of stable treed
  disks of type $\Gamma$, necessarily of fixed energy $E(\Gamma)$.
  After passing to a subsequence we may assume that the sequence of
  stable treed disks $[C_\nu]$ converges to a limiting stable disk
  $[C]$ in $\ol{\M}_\Gamma$.  Considering each disk or sphere
  component or Morse trajectory of $u_\nu$ separately one sees the
  sequence of maps $u_\nu: C_\nu \to X$ has subsequence that admits a
  stable Gromov limit $u: \hat{C} \to X$, where
  $\hat{C} = \hat{S} \cup \hat{T}$ is a possibly unstable sphere or
  disk with stabilization $C$ of type $\Gamma_\infty$.

  We show that the limit constructed in the previous paragraph is
  adapted, that is, interior leaves correspond to intersections with
  the Donaldson hypersurface.  By definition we have
  $J_\Gamma(x,z) = J_D (x) , \ \forall x \in D.$ Now
  $J_D \in \J^*(X,D,J,\theta,n(\Gamma_i,k))$ was chosen so that $D$
  contains no non-constant $J_D$-holomorphic spheres.  To show the
  (Leaf axiom) property, note that each connected component $C_i$ of
  $u^{-1}(D)$ has surface part $S_i = S \cap C_i$ either a point $z_i$ or a
  union of sphere and disk components.  In the first case, the
  intersection multiplicity $m(z_i)$ with the divisor at the point
  $z_i$ is positive while in the second, the intersection multiplicity
  $m(w)$ at each node $w \in S_i$ connecting the component $S_i$ with
  the rest of the domain $S - S_i$ is positive.  In the first case, by
  topological invariance of the intersection multiplicity there exists
  a sequence $z_\nu$ of isolated points in $u_\nu^{-1}(D)$ with
  positive intersection multiplicity converging to $z_i$.  By the
  (Leaf axiom) property of $u_\nu$, the points $z_\nu$ must be
  markings.  Hence $z_i$ is a marking as well.  In the second case,
  the positivity of the intersection multiplicity implies that there
  exists a sequence of points
  $z_{\nu,i} \in u_\nu^{-1}(D) \subset C_\nu$ converging to $C_i$.
  Either $z_{\nu,i}$ are isolated in $u_\nu^{-1}(D)$, or lie on some
  sequence of connected components $C_{\nu,i}$ of
  $u_\nu^{-1}(D) \subset C_\nu$ on which $u_\nu$ is constant in which
  case any limit point of $C_{\nu,i}$ lies in $C_i$.  Since each
  $C_{\nu,i}$ contains a marking by the (Leaf axiom) property, so does
  $C_i$.  Note that if $u_\nu(z_{i,\nu}) \in D$ then $u(z_i) \in D$,
  by convergence on compact subsets of complements of the nodes.  This
  shows the (Leaf axiom) property.

  We next show the (Stable surface axiom).  Note that since $D$ is
  stabilizing for $L$, any non-constant disk component $\hat{S}_i$ of
  $\hat{S}$ must have at least one interior intersection point, call
  it $z$, with $u^{-1}(D)$.  Since $z$ lies in the interior of $S_i$,
  and the boundary $\partial S_i$ has at least one special point the
  component $\hat{S}_i$ is stable.  Next consider a spherical
  component $\hat{S}_i$ of $\hat{S}$ attached to a component of $S$.
  First suppose that $u$ is non-constant on $\hat{S}_i$.  Suppose that
  $\hat{S}_i$ is attached at a point of $S$ contained in
  $\ol{\U}_{\Gamma_i}$ for some type $\Gamma_i$, so that the energy of
  $u |_{\hat{S}_i}$ is at most $n(\Gamma_i,k)$.  Since $J_\Gamma$ is
  constant and equal to an element of
  $\J^*(X,D,J,\theta, n(\Gamma_i,k))$ on $\hat{S}_i$, the restriction
  of $u$ to $\hat{S}_i$ has at least three intersection points
  resp. one intersection point with $D$.  By definition these points
  must be markings, which contradicts the instability of $\hat{S}_i$.
  Hence the stable map $u$ must be constant on $\hat{S}_i$, and thus
  $\hat{S}_i$ must be stable.  This shows that $\hat{S}$ is equal to
  $S$ and the (Stable surface axiom) follows.

  It remains to rule out sphere components in the limit.  Suppose the
  limiting domain $C$ has a spherical component $S_v \subset S$.  By
  the (Locality Axiom), forgetting all but one marking on maximal
  ghost components gives a configuration $u_{\on{red}}: C_{\on{red}} \to X$ in
  an uncrowded moduli space $\M_{\Gamma_{\on{red}}}(L,D)$ for some type
  $\Gamma_{\on{red}}$ with respect to some perturbation datum
  $P_{\Gamma_{\on{red}}}^{\on{red}}$ given by the perturbation datum
  $P_{\Gamma_\circ}$ on disk components and p perturbations
  $P_{\Gamma_v}^{\on{red}}$ for sphere components $S_v^{\on{red}}$ of $C_{\on{red}}$.
  For generic choices of $P_\Gamma$, the moduli space
  $\M_{\Gamma_{\on{red}}}(L,D)$ with respect to this induced perturbation
  datum is of expected dimension.  Since the configuration $u_{\on{red}}$
  has a spherical node, this expected dimension is at least two less
  than the expected dimension of $\M_\Gamma(L,D)$, which is at most
  one.  Hence the dimension $\dim \M_{\Gamma_{\on{red}}}(L,D)$ is negative,
  a contradiction.
\end{proof}

\section{Composition maps} 
\label{novsec}

In this section we use pseudoholomorphic treed disks to define the structure
coefficients of the Fukaya algebra.  Let $q$ be a formal variable and
$\Lambda$ the {\em universal Novikov field} of formal sums with
complex coefficients
\begin{equation} \label{novikov} \Lambda = \Set{ \sum_i c_i
    q^{\alpha_i} \ | \  c_i \in \C, \ \alpha_i \in \R, \quad \alpha_i
    \to \infty . } \end{equation}
Denote by $\Lambda_{\ge 0}$ resp. $\Lambda_{ >0}$ the subalgebra with
only non-negative resp.  positive exponents.  Denote by
\[\Lambda^\times = \Set{ c_0 + \sum_{i > 0} c_i q^{\alpha_i} \in 
  \Lambda_{\ge 0} \ | \ c_0 \neq 0 } \]
the subgroup of formal power series with invertible leading
coefficient.  Although our Fukaya algebras will be defined over the
rationals, allowing complex number provides additional solutions to
the Maurer-Cartan equation.

Lagrangians will be equipped with additional data called {\em brane
  structures}.  Let $\Lag(X)$ denote the fiber bundle over $X$ whose
fiber $\Lag(X)_x$ at $x$ consists of Lagrangian subspaces of $T_x X$.
Let $g$ be an even integer.  A {\em Maslov cover} is an $g$-fold cover
$\Lag^g(X) \to \Lag(X)$
such that the induced two-fold cover 
\[\Lag^2(X) := \Lag^g(X)/\Z_{g/2} \to \Lag(X)\]
is the oriented double cover.  A Lagrangian submanifold $L$ is {\em
  admissible} if $L$ is compact and oriented; we also assume for
simplicity that $L$ is connected.  A {\em grading} on $L$ is a lift of
the canonical map
\[L \to \Lag(X), \quad l \mapsto T_l L \]
to $\Lag^g(X)$, see Seidel \cite{se:gr}.  A {\em brane structure} on a
compact oriented Lagrangian is a relative spin structure together with
a grading and a $\Lambda_{\times}$ local system.  (The grading will be
irrelevant until we discuss Fukaya bimodules.)  An {\em admissible
  Lagrangian brane} is an admissible Lagrangian submanifold equipped
with a brane structure.  For $L$ an admissible Lagrangian brane define
the space of {\rm Floer cochains}
\[ CF(L) = \bigoplus_{d \in \Z_g} CF_d(L), \quad
CF_d(L) = \bigoplus_{x
  \in \cI_d(L)
 } \Lambda x \]
 where the degree $d$ generators $ \cI_d(L)$
\label{CFdef} 
 are as in \eqref{degreepart}.  Let $CF^{\on{geom}}(L)
 \subset CF(L)$ the subspace generated by $x \in \cI^{\on{geom}}(L)$.

 The composition maps are defined in the following.  Given a
 Lagrangian brane $L$ and a tree disk $u: C \to X$ with boundary in
 $L$, we denote by $y(u) \in \Lambda^\times$ the evaluation of the
 local system $y \in \RR(L)$ on the homotopy class of loops
 $[\partial u] \in \pi_1(L)$ defined by going around the boundary of
 each disk component $S_v \subset C$ in the treed disk $C$ once.
 Denote by $\sigma(u) \in \Z_{\ge 0}$ the number of interior leaves of
 $u \in \ol{\M}(L,D,\ul{x})$.  Let $E(u)$ denote the energy, or
 equivalently area, of surface part of $u$, c.f. \eqref{energyarea}
 below.

 \begin{definition} {\rm (Composition maps)} For admissible
   perturbation data $(P_\Gamma)$ define
\[ \mu^{n}: CF(L)^{\otimes n} \to CF(L) \]
on generators by 
\begin{equation} \label{ongenerators} \mu^{n}(x_1,\ldots,x_n) =
  \sum_{x_0,u \in \M(L,D,\ul{x})_0} (-1)^{\heartsuit}
  (\sigma(u)!)^{-1} y(u) q^{E(u)} \eps(u) x_0 \end{equation} 
where $ \heartsuit = {\sum_{i=1}^n i|x_i|}$ and $\eps(u)$ is the sign of
\eqref{signs}.
\end{definition} 

\begin{remark} {\rm (Zero-th composition map is a quantum correction)}   
  Any configuration $u: C \to X$ with no boundary leaves
  $T_e \subset T$ must have at least one non-constant disk component
  $S_v \subset S$.  Indeed, since each vertex has valence at least
  three, any configuration must have leaves but in this case all
  leaves must be interior.  Hence $\mu^0(1)$ has
  $\on{val}_q(\mu^0(1)) > 0$.
\end{remark} 

\begin{remark} {\rm (Leading order term in the first composition map)}  
  The constant trajectories at the maximum $x_M$ with weighted leaf
  and outgoing unweighted root are elements of
  ${\M}_\Gamma(L,D,x_0,x^{\greyt})_0$ for some type $\Gamma$, as in
  Remark \ref{examples}.  The orientations on these trajectories are
  determined by the orientation on $\M_\Gamma$. By the discussion
  after \eqref{orienteq} the orientation is negative resp. positive
  for $x_0 = x^{\blackt}$ resp.  $x_0 = x^{\whitet}$.  Hence
  \begin{multline} \label{diff} \mu^{1}(x^{\greyt}) = x^{\whitet} -
    x^{\blackt}  + \sum_{x_0,[u] \in \M(L,D,x^{\greyt},x_0)_0, E(u)
      > 0} (-1)^{\heartsuit} (\sigma(u)!)^{-1} q^{E(u)} \eps(u)y(u)
    x_0 .\end{multline}
This formula is similar to that in Fukaya-Oh-Ohta-Ono 
\cite[(3.3.5.2)]{fooo}.  Presumably the discussion here is a version
of their treatment of homotopy units.
\end{remark}

\begin{theorem} \label{yields} {\rm (\ainfty algebra for a Lagrangian)}   
For any admissible perturbation system $\ul{P} =
(P_\Gamma)$ the maps $(\mu^n)_{n \ge 0}$ satisfy the axioms of a 
(possibly curved) \ainfty algebra $CF(L)$ with strict 
unit.  The subspace $CF^{\on{geom}}(L)$ is a subalgebra without unit. 
\end{theorem} 

\begin{proof}  
  By Theorems \ref{main} and \ref{compthm}, the one-dimensional
  component of the moduli space with energy bound $\ol{\M}^{< E}(L)_1$
  is a finite union of compact oriented one manifolds
  $\M^{< E}_\Gamma(L)$ for some combinatorial types $\Gamma$ with a
  single broken edge.  Hence the boundary points of the moduli space
  come in pairs.  This produces the identity mod $2$
  \begin{equation} \label{ident} 0 = \sum_{\Gamma \in \TT_{n,m}}
    \sum_{ [u] \in
    \partial \ol{\M}_{\Gamma}(L,D,\ul{x})_1} \eps(u) (\sigma(u)!)^{-1}
  q^{E(u)}  y(u) \end{equation}
where $\TT_{n,m}$ denotes the set of types of treed disks with $n$
boundary and $m$ interior edges.  In case the moduli spaces do not
involve weightings then each combinatorial type $\Gamma \in \TT_{n,m}$
with a single interior edge $e \in \Edge_{-}(\Gamma)$ of infinite
length $\ell(e) = \infty$ is obtained by gluing together graphs
$\Gamma_1,\Gamma_2$ with $n-n_2+1$ and $n_2$ leaves along
semi-infinite edges $e_-,e_+$, say with $m_1$ resp.  $m_2$ interior
markings.  By the (Cutting edges axiom) we have an isomorphism
\begin{multline} \ol{\M}_{\Gamma}(L,D,\ul{x}) = \bigcup_{w \in \cI(L)} 
	\ol{\M}_{\Gamma_1}(L,D,x_0,x_1,\ldots,x_{i-1},w,x_{i+n_2}, \dots, x_n)
	\\	\times \ol{\M}_{\Gamma_2}(L,D,w,x_{i},\ldots,x_{i+n_2-1}) 
  .\end{multline}
Since there are $m$ choose $m_1,m_2$ ways of distributing the interior
leaves to the two component graphs, mod $2$ we have
\begin{multline}
 0 = \sum_{ \substack{i,m_1 + m_2 = m \\ [u_1] \in
    {\M}_{\Gamma_1}(L,D,x_0,x_1,\ldots,x_{i-1},y,x_{i+n_2}, \dots,
    x_n)_0 \\ [u_2] \in
    {\M}_{\Gamma_2}(L,D,y,x_{i},\ldots,x_{i+n_2-1})_0 } } (m!)^{-1}
\left( \begin{array}{l} m \\ m_1 \end{array} \right) q^{E([u_1]) +
  E([u_2])}\\ \eps([u_1]) (\eps[u_2]) {y}(u_1) {y}(u_2).
 \end{multline}
 This gives the \ainfty axiom \eqref{ainftyassoc} up to signs.  We
 refer to \cite{ainfty} for the sign computation.  In the case of a
 weighted leaf $e$ one has additional boundary terms corresponding to
 types where the weighting $\rho(e)$ becomes zero or infinity.  Those
 configurations correspond to splitting off a constant Morse
 trajectory with weighted leaf
 $e \in \Edge_{\rightarrow}^{\greyt}(\Gamma)$ and outgoing forgettable
 or unforgettable root
 $e_0 \in \Edge_{\rightarrow}^{\blackt}(\Gamma) \cup
 \Edge_{\rightarrow}^{\whitet}(\Gamma)$.
 These correspond to the terms $x^{\whitet}$ and $x^{\blackt}$ in
 $\mu^{1}(x^{\greyt})$.

 The existence of strict units follows from the (Forgettable edges
 axiom).  We claim that a strict unit is given by the element
 $e_L = x^{\whitet} \in CF(L)$.  By the (Forgettable edges) axiom, the
 perturbation data $P_\Gamma$ used to define
 $\ol{\M}_\Gamma(L,D,\ldots, x_{i-1}, x_i = x^{\whitet}, x_{i+1},
 \ldots, )$
 is pulled back, on the stable part of the universal curve, from data
 $P_{\Gamma'}$ for
 $\ol{\M}_{\Gamma'}(L,D,\ldots,x_{i-1}, x_{i+1}, \ldots, )$ for the
 type $\Gamma'$ obtained from $\Gamma$ by forgetting the corresponding
 leaf $e_i$.  As long as $\Gamma'$ has at least one vertex, the moduli
 space $\ol{\M}_{\Gamma'}(L,D,\ldots,x_{i-1}, x_{i+1}, \ldots, )$ is
 one dimension lower, and so empty.  Thus the compositions
 $\mu^n(\ldots, x^{\whitet}, \ldots)$ vanish except for the case that
 the resulting type has no stabilization: In the case $n = 2$, the
 underlying configuration $u: C' \to X$ has no non-constant disks
 $S_i \subset C$, as in Figure \ref{triv}.  One obtains from a
 configuration with type $\Gamma'$ and constant values the identity
\[\mu^2(x,x^{\whitet}) = (-1)^{\deg(x)} \mu^2(x^{\whitet},x) =
x , \quad \forall x \in \cI(L) . \qedhere \]
\end{proof} 

\section{Divisor equation} 

The divisor equation \cite[2.2.4]{km} expresses the fact that the
number of pseudoholomorphic curves mapping a point on the domain to a
codimension two submanifold (divisor) is the pairing of the homology
class of the map with the class of the divisor.  In the case of
pseudoholomorphic curves with boundary, the resulting equation
\cite[Proposition 6.3]{cho} gives an equality once one sums over all
possible places of the insertion (since the theory is not invariant
under permutations of the insertions.)  In any particular perturbation
scheme, the existence of such an equation relies on the existence of
forgetful maps.  The Morse moduli spaces above do not admit forgetful
morphisms in general, hence the divisor axiom is not satisfied.
However, there is a weak version for the case of a single leaf.  Let
$\Gamma$ be a combinatorial type with a single boundary leaf and at
least one interior leaf and $f(\Gamma)$ the type obtained by
forgetting the leaf.  We consider perturbations of the form
\begin{equation} \label{fp}  P_\Gamma = f^*
  P_{f(\Gamma)} \end{equation} 
for type $\Gamma$ pulled back under the map of universal curves
$f: \U_\Gamma \to \U_{f(\Gamma)}$.

\begin{lemma} \label{fexist} Let $\Gamma$ be a combinatorial type of adapted
  pseudoholomorphic treed disk with a single boundary leaf of expected
  dimension at most one and suppose that $J_D \in \J_\tau(X)$ is such
  that all non-constant $J_D$-holomorphic disks $u$ in $X$ with
  boundary in $L$ have positive Maslov index $I(u) > 0$.  For a
  comeager set of perturbations $\PP_{f(\Gamma)}^{\reg,\on{div}}$, if
  $P_{f(\Gamma)} \in \PP_{f(\Gamma)}^{\reg,\on{div}}$ then the moduli
  space $\M_\Gamma(X,L)$ is smooth of expected dimension.
\end{lemma}

\begin{proof} Let $C$ be a tree disk of type $\Gamma$ with single
  boundary leaf $e$ and $f(C)$ the treed disk obtained by forgetting
  $e$ and stabilizing.  In the construction of the local universal
  moduli space \eqref{luniversal} perturbations of the form
  $f^* P_{f(\Gamma)}$ suffice to achieve transversality.  Indeed, at
  most one disk component of $C$ is collapsed under the forgetful map
  (depending on whether the disk containing $e \cap S$ is constant or
  not).  If the disk $S_v$ containing $e \cap S$ is not collapsed,
  then variations $\xi \in \Omega^0(S_v , (u| S_v)^* TX)$ that vanish
  at the point $S \cap e$ are enough to force an element $\eta$ in the
  cokernel of the linearized operator to satisfy $D_u^* \eta = 0$ on
  $S_v$.  The same argument as in that construction produces a
  comeager subset of perturbations.  If $S_v$ is collapsed, then it
  maps to a point on an edge of $f(C)$ which is necessarily
  combinatorially semi-infinite in at most one direction, since $C$
  must contain at least one non-constant disk.  Let $z = f(S \cap e)$
  be the collapsed disk in $f(C)$.  If $z$ lies on an edge $T_e$ not
  diffeomorphic to $\R$ (that is, a semi-infinite rather than infinite
  interval) then domain-dependent perturbations $F_\Gamma$ of the
  Morse function suffice to make the moduli space of maps
  $v: f(C) \to X$ with $v(z) \in W_{l_1}^-$ transverse, and this
  moduli space is isomorphic to $\M_\Gamma(X,L)$.  The final case is
  that $z$ maps to a broken segment $e \cong \R$ of a semi-infinite
  edge.  The latter corresponds to a broken configuration $u: C \to X$
  that splits into pieces $u_1: C_1 \to X$ with no leaves, a
  configuration $u_2: C_2 \to X$ with two leaves but only non-constant
  disks, and a Morse trajectory $u_3: \R \to L$.  Since the
  configuration $u_1$ must be non-negative expected dimension, the
  breaking must be at a critical point $l_2 \in \cI(L)$ of index zero.
  There are no such configurations, since $u_3$ must flow to a
  critical point of positive index. \end{proof}

Note that in general one may not choose perturbations pulled back
under forgetful maps: Configurations $u \in \M_\Gamma$ with constant
disks correspond to intersections of the unstable manifolds
$W_x^-, x \in \crit(m)$ of the unperturbed Morse function
$m: L \to \R$.  The stable manifolds $W_x^-,W_y^-$ are not
transversally intersecting for obvious reasons (e.g. they may be equal
if $x = y$.) \label{complicatedp} \llabel{complicated}

Perturbations satisfying the forgetful property \eqref{fp} give the
following weak form of the divisor equation.  Let
$c \in \ker(\mu^{1,0})$ be a cocycle for the Morse operator
$\mu^{1,0}$ and $[c] \in H(L)$ the Morse homology class, represented
by the sum of unstable manifolds $W^-(x)$ of the critical points $x$
appearing in $c$.  That is, if $c = \sum c_x x$ then
\[ [c ] = \left[ \sum c_x W^-(x) \right] .\]
  As in Fukaya-Oh-Ohta-Ono \cite[Lemma 
  13.1]{fooo:toric1} we write 
 \begin{equation} \label{decompose}  \mu^d = \sum_\beta q^\beta
   \mu^{d,\beta} \end{equation} 
  for some complex numbers $\mu^{d,\beta}$.  

  \begin{proposition} {\rm (Weak divisor equation)} Suppose that an
    admissible perturbation system $\ul{P} = (P_\Gamma)$ is such that
    $P_\Gamma = f^* P_{f(\Gamma)}$ for any type $\Gamma$ with one
    leaf.  For any Morse cycle $c \in \ker(\mu^{1,0}) \cap CF^1(L)$,
    we have
\begin{eqnarray}\label{diveq} 
 \mu^1(c) &=& \sum_{u \in \M_1(l_0)} q^{A(u)} \eps(u) \sigma(u) ( [
              \partial u], [c]) y(u) \\ &=&
                                               \sum (\partial \beta,[c])  q^{\beta} \mu^{0,\beta} (c) 
                                               \end{eqnarray}
where $\partial \beta$ is the image of $\beta \in H_2(X,L)$ in
$H_1(L)$.  
\end{proposition}

\begin{proof} By assumption there is a forgetful morphism
\begin{equation} \label{forgetful} 
f: \M_\Gamma(X,L,D,x_1,x_0) \to 
\M_{f(\Gamma)}(X,L,D,x_0)\end{equation}
obtained by forgetting the map on the leaf.  The fiber over 
an element 
\[ [u: C \to X] \in \M_{\Gamma'}(x_0) \] 
is the set of points in the boundary $\partial C$ mapping to
$W^-(x_1)$,
\[ f^{-1}(u) \cong (u | \partial C)^{-1}(W^-(x_1)) \]
and is cut out transversally.  The areas, numbers of interior leaves,
and holonomies of $u$ and $f(u)$ are equal.  Furthermore, by
construction the sign $\eps(u)$ is equal to $\eps(f(u))$ times the
sign of the corresponding intersection of
$(u | \partial C)^{-1}(W^-(x_1))$.  Indeed in the case that
$\{ z \} = S \cap e$ is a point on a constant disk with one other
leaf $e' \subset T$, attaching the edge $e$ at the disk with
the leaves reversed gives a contribution with opposite sign;
thus we may ignore such contributions.  In the case that
$\{ z \} = S \cap e$ attaches to a non-constant disk $S_v \subset S$,
the construction of orientations in \eqref{orienteq} defines the
orientation on these moduli spaces as that on
$ \det(T f^{-1}(C)) \otimes \det(TW^-(x_1))$ which is identified with
the trivial determinant line $\R$ via the determinant line of
$\det(D_z u) \ \on{mod} TW^-(x_1)$. This identification is positive
 if the intersection of $\partial u$ with $W^-(x_1)$ is
positive.  The claimed equality follows.
\end{proof}

The following argument similar to that Fukaya-Oh-Ohta-Ono
\cite[Section 13]{fooo:toric1} in the toric case implies that critical
points of the disk potential determine Floer-non-trivial brane
structures: Let
\[ \cR(L) := \Hom(\pi_1(L),\Lambda^\times) \]
denote the space of isomorphism classes of local systems on $L$.  Any
representation descends to the abelianization $H_1(L)$ of $\pi_1(L)$.
Coordinates on $\cR(L)$ are given by exponentiation:
\[ \exp:  H^1(L,\Lambda_{\ge 0}) 
\to \cR(L), \quad f \mapsto \exp(f) .\]
In particular we have an isomorphism of tangent spaces 
\[ T_{y} \cR(L) \cong H^1(L,\Lambda_{\ge 0}) \]
at any ${y} \in \cR(L)$.  Any element of $\cR(L)$ is given by an
element $[f] \in H^1(L,\Lambda_{\ge 0})$ in the sense that the map is
given by
\[ \pi_1(L) \to \Lambda^\times, \quad [\gamma] \mapsto 
{y}_{[f]}([\gamma]) = \exp( [\gamma], [f]) .\]
The derivative of evaluation at a class $[\gamma] \in \pi_1(L)$ is the
evaluation of the cohomology class on the homology class of the loop:
\[ \partial_{[c]} {y}_{[f]}( [\gamma]) = ([c], [\gamma])  .\]

The argument for the Floer non-triviality involves variation of the
local system. Suppose that $0 \in {MC}(L,y)$ for every choice of local
system $y \in \cR(L)$.  Set
\begin{equation} \label{WW} \WW: \cR(L) \to \Lambda_{\ge 0}, \quad
  \WW(y) e_L = \mu^0(1) .\end{equation}
Continuing the computation from \eqref{diveq} we identify a
first cohomology class with a tangent vector to the space of local
systems. Then \llabel{wehave} \label{wehavep}
\begin{eqnarray*}  \label{derivw} \mu^1(c) &=& 
\sum_{\beta > 0} q^\beta 
(\partial \beta, [c]) 
\mu^{0,\beta} \\ 
&=& 
\sum_{x_0,u \in {\M}_1(L,D,x_0)_0} (-1)^{\heartsuit}
([\partial u],[c]) (\sigma(u)!)^{-1}  y(u)  q^{E(u)} \eps(u) x_0 \\
&=&  \partial_{[c]} \mu^0(1) =   \partial_{[c]} \WW(y) e_L .
\end{eqnarray*} 

In the following proposition, similar to Fukaya-Oh-Ohta-Ono
\cite[Lemma 13.1]{fooo:toric1}, 
we have in mind the case that the
Lagrangian is a torus and the Morse function is the standard one
obtained from taking the product of the standard height function on
the component circles, so that the Morse coboundary vanishes.  A
related result in the monotone case is in Biran-Cornea \cite[Theorem
1.2.2]{bc:rigid}. \llabel{induct} \label{inductp}

\begin{proposition} \label{critW} Suppose that the Morse operator
  $\mu^{1,0} = \mu^1 |_{q = 0}$ vanishes, $0 \in {MC}(L)$ and for some
  ${y} \in \cR(L)$, for all $b \in H^1(L,\Lambda_{\ge 0})$, we have
  $\partial_b \WW(y) = 0 $.  Then the Floer cohomology is isomorphic
  to the ordinary cohomology:
\[ H(\mu^1_{y},\Lambda_{\ge 0}) = H(L,\Lambda_{\ge 0}) .\]
\end{proposition} 

\begin{proof}  
  For completeness we include the proof is a double induction on
  energy and classical degree.  Let $\hbar > 0$ be the energy
  quantization constant so that $ E(\beta) > \hbar$ for all classes
  $\beta \in H_2(\XX,L)$ represented by non-constant treed holomorphic
  disks.  Suppose that $m_1^{\beta}= 0$ for all $\beta$ with
  $E(\beta) \leq E $.  We claim that $m_1^{\beta} = 0$ for all
  $\beta$ with $E(\beta) \leq E + \hbar$.  The base step is furnished
  by the assumption that the Morse function is standard, so the Morse
  differential vanishes.  By induction on degree we may assume that
  $m_1^{\beta}(c) = 0$ for degree $\deg(c) \leq d$.  The inductive
  step is furnished by the assumption that $\mu^1(c) = 0$ for
  $\deg(c) = 1$ via the computation \eqref{derivw}.  We claim that
  $m_1^{\beta}(c_1 \cup c_2) = 0$ for all classes $c_1,c_2$ with
  $1 \leq \deg(c_1), \deg(c_2) \leq d$.  Using the decomposition
  \eqref{decompose} and comparing terms with coefficient $q^\beta$ in
  the \ainfty associativity relation we obtain
  \begin{multline} m_1^{\beta}( m_2^{0}(c_1,c_2)) = \sum_{\beta_1 +
      \beta_2 = \beta} \pm m_2^{\beta_1}(m_1^{\beta_2}(c_1),c_2) \\ +
    \sum_{\beta_1 + \beta_2 = \beta} \pm
    m_2^{\beta_1}(c_1,m_1^{\beta_1}(c_2)) + \sum_{\beta_1 + \beta_2 =
      \beta, \beta_2 \neq 0} \pm m_1^{\beta_1}(m_2^{\beta_2}(c_1,c_2))
    .\end{multline}
  The first two terms on the right vanish by the inductive hypothesis
  for degree.  The last term vanishes by the inductive hypothesis for
  energy since $E(\beta_1) < E(\beta) - \hbar \leq E$. Since any class
  $c$ of degree $d + 1$ is a linear combination of classes
  $c_1 \cup c_2$ with $1 \leq \deg(c_1),\deg(c_2) \leq d$ this shows
  that $\mu^1(c) = 0$ for classes of degree $d + 1$. 
\end{proof}

\begin{example} {\rm (Potential for the projective line)}  
  Suppose that $X = S^2$ with area $A$ and $L \cong S^1$ separates $X$
  into pieces of areas $A_1,A_2$.  Thus the Maslov-index-two disks
  have areas $A_1,A_2$ with opposite boundary homotopy classes.  The
  disk potential is
\[ \WW(y) = q^{A_1} y + q^{A_2}/y .\]
The Floer differential vanishes if and only if 
\[ 0 = y \partial \WW/\partial y = q^{A_1} y - q^{A_2}/y \]
for some $y$.  The equation $y^2 = q^{A_2 - A_1}$ has a solution in
$\cR(L) \cong \Lambda^\times$ if and only if $A_1 = A_2$.  Thus only in this case
(the case that the Lagrangian is Hamiltonian isotopic to the equator)
one has non-trivial Floer cohomology.
\end{example} 

Brane structures corresponding to non-degenerate critical points are
particularly well-behaved.  A Lagrangian torus $L$ equipped with brane
structure is {\em non-degenerate} if the leading order part $\WW_0$ of
the potential $\WW$ (that is the sum of terms with lowest $q$-degree)
has a non-degenerate critical point.  Existence of a non-degenerate
critical point is invariant under perturbation by Fukaya-Oh-Ohta-Ono
\cite[Theorem 10.4]{fooo:toric1}.  We reproduce the proof here for the
sake of completeness.  Since $\RR(L) = \Hom(H_1(L), \Lambda^\times)$,
$\RR(L)$ is a smooth manifold at any $y_0 \in \RR(L)$ in the sense
that there exists a coordinate chart
$\exp_{y_0}: \Lambda_{\ge 0}^{\otimes N} \supset U \to \RR(L)$, where
$U$ is an open neighborhood of $0$.  The notion of critical points of
$\WW$ and the Hessian at those critical points is therefore
well-defined.

\begin{theorem} \label{ift} Suppose that
  $\WW: \RR(L) \to \Lambda_{\ge 0} $ is a function of the form
  $\WW = \WW_0 + \WW_1$ where $\WW_0$ consists of the terms of lowest
  $q$-valuation.  Suppose that $y_0 \in \Crit(\WW_0)$ and $\WW_0$ has
  non-degenerate Hessian $D^2_{y_0} \WW_0$ at $y_0$.  Then there
  exists
\[y_1 \in \Hom(\pi_1(L), \Lambda_{> 0}), \quad y = y_0\exp( y_1) \in
\Crit(\WW)\] 
with the terms in $y_1$ having higher $q$-valuation than those in
$y_0$.
\end{theorem} 

\begin{proof} The proof is an application of a formal implicit
  function theorem adapted to the setting of Novikov rings.  Given a
  critical point $y_0$ of $\WW_0$ we solve for a critical point $y$ of
  $\WW$ order by order, using non-degeneracy of the Hessian of $\WW$
  at $y_0$.  Suppose that
\[\WW = q^{\eps_0} \WW_0 + q^{\eps_1} \WW_1\] 
where $\WW_0$ has vanishing $q$-valuation and $\eps_1 > \eps_0$.  Suppose
that $y_0$ is a critical point of $\WW$ mod terms divisible by
$q^{\delta}$ for some $\delta$ with $\delta > \eps_0$.  Taking a
Taylor expansion of $D_y \WW$ at $y = y_0$ we solve
\begin{eqnarray} \label{wesolve} 0 &=& D_{y_0 \exp( y_1)} \WW =
  D_{y_0} \WW + D^2_{y_0} \WW (y_1) +
                                       F_{y_0}( y_1) \\
                                   &=& q^{\delta} G(y_0) + D^2_{y_0}
                                       \WW(y_1) + F_{y_0}( y_1)
\end{eqnarray}
where $F_{y_0}(y_1)$ is a linear map satisfying a quadratic
bound for small $y_1$.  Here the Hessian $D^2_{y_0} \WW$ is viewed as
a map
\[ D^2_{y_0} \WW:  T_{y_0} \RR(L) \to \Hom(T_{y_0} \RR(L), \Lambda_{
  \ge 0 } ) .\]
We solve
\[ D^2_{y_0} q^{-\eps_0} \WW(y_1) = -q^{\delta - \eps_0} G(y_0) .\]
Then \eqref{wesolve} fails to hold only because of the term
$F_{y_0}( y_1)$.  Since $\WW$ is divisible by $q^{\eps_0}$ and $y_1$
is divisible by $q^{\delta - \eps_0}$, $F_{y_0}(y_1)$ is divisible by
$q^{2(\delta - \eps_0)+\eps_0} $.  Replacing $y_0$ with
$y_0 \exp( y_1)$ and continuing by induction one obtains a solution to
all orders.
\end{proof}

Since any critical point of $\WW$ is, to leading order, a critical
point of $\WW_0$, we have the following corollary:

\begin{corollary} 
  If every critical point of $\WW_0$ is non-degenerate then there
  exists a bijection $\Crit(\WW_0)\cong \Crit(\WW) $ between the set
  $\Crit(\WW_0)$ of critical points of $\WW_0$ and critical points
  $\Crit(\WW)$ of $\WW$.
\end{corollary}

\section{Maurer-Cartan moduli space} 
\label{maurercartan} 
 
This section contains a version of the results of Fukaya-Oh-Ohta-Ono
\llabel{refhere}  \label{refherep} 
\cite[Section 3.6]{fooo} on solutions to the projective Maurer-Cartan
equation, adapted to the Morse setting.  In particular we define a
{\em cohomology complex} associated to an \ainfty algebra satisfying
suitable convergence conditions.  

The Floer complex is a complex of vector bundles over a space of
solutions to a {\em projective Maurer-Cartan equation}.  Let $A$ be an
\ainfty algebra free and finitely generated over the Novikov ring
$\Lambda_{\ge 0}$.  Such $A$ is {\em convergent} if
$\on{val}_q(\mu^0(1)) > 0$.  Let
\[A^+ = \{ a \in A \ | \ \on{val}_q(a) > 0 \}\] 
denote the subalgebra of terms with positive q-valuation.

\begin{lemma}  \label{curvlem}
For $b \in A^+$ the sum
\begin{equation} \label{curv} 
 \mu^0_b(1) := \mu^0(1) + \mu^1(b) + \mu^2(b,b) + \ldots
 \end{equation}
is well-defined.
\end{lemma} 

\begin{proof} Since $b$ has positive $q$-valuation, we may write
  $b = q^E b_0$ for some $b_0 \in A$ and $E> 0$.  Then the coefficient
  of each generator of $A$ in the composition $\mu^n(b,\ldots, b)$ has
  $q$-valuation $\on{val}_q(\mu^n(b,\ldots,b))> nE$, if non-zero.
  This inequality implies that the sum \eqref{curv} converges.
\end{proof}

More generally the same argument implies convergence of the {\em
  deformed composition map}
\[ \mu^n_b(a_1,\ldots,a_n) = \sum_{i_1,\ldots,i_{n+1}} \mu^{n + i_1 +
  \ldots + i_{n+1}}(\underbrace{b,\ldots, b}_{i_1}, a_1,
\underbrace{b,\ldots, b}_{i_2}, a_2,b, \ldots, b, a_n,
\underbrace{b,\ldots, b}_{i_{n+1}}) \]
over all possible combinations of insertions of the element
$b \in A^+$ between (and before and after) the elements
$a_1,\ldots, a_n$.  For $b$ odd the maps $\mu^n_b$ define an \ainfty
structure on $A$.  In particular
\[ (\mu^1_b)^2(a_1) = \mu^2_b(\mu^0_b(1), a_1) -
\mu^2_b(a_1,\mu^0_b(1)) .\]
The {\em projective Maurer-Cartan equation} for $b \in A^+$ is
\begin{equation} \label{wmc} 
 \mu^0_b(1) = \mu^0(1) + \mu^1(b) + \mu^2(b,b) + \ldots \in \Lambda
 e_A .\end{equation}
Denote by 
\begin{equation} \label{mcspace} {MC}(A) = \{ b \in A^{+,\on{odd}} \ |
  \ \mu^0_b(1) \in \Lambda e_A \} \end{equation}
the space of odd solutions to the projective Maurer-Cartan equation \eqref{wmc}.
Any solution to the projective Maurer-Cartan equation defines an \ainfty
algebra such that $(\mu^1_b)^2 = 0$ and so has a well-defined
cohomology
\[H(\mu^1_b) = \frac{\on{ker}(\mu^1_b)}{\on{im}(\mu^1_b)} .\]
An \ainfty algebra is {\em weakly unobstructed} if there exists a
solution to the projective Maurer-Cartan equation.

We introduce a notion of gauge equivalence for solutions to the projective
Maurer-Cartan equation, so that cohomology is invariant under gauge
equivalence.  For $b_0,\ldots,b_n \in A^+$ of odd degree define
\begin{multline} \label{lotsbs}
 \mu^n_{b_0,b_1,\ldots,b_n}(a_1, \ldots, a_n) =
 \\ \sum_{i_1,\ldots,i_{n+1}} \mu^{n + i_1 + \ldots +
   i_{n+1}}(\underbrace{b_0,\ldots, b_0}_{i_1}, a_1,
 \underbrace{b_1,\ldots, b_1}_{i_2}, a_2,b_2, \ldots, b_2, \ldots,
 a_n, \underbrace{b_n,\ldots, b_n}_{i_{n+1}}). \end{multline}
Two projective Maurer-Cartan solutions $b_0,b_1$ are {\em gauge equivalent}
if and only if
\[ \exists h \in A^+, \ b_1 - b_0 = \mu^1_{b_0,b_1}(h), \quad \deg(h)
\ \text{even} .\]
We then write $b_0 \sim_h b_1$.

Gauge equivalence is an equivalence relation, by a discussion parallel
to that in Seidel \cite[Section 1h]{se:bo}.  To show transitivity,
\label{gtrans} suppose that $b_0 \sim_{h_{01}} b_1$ and
$b_1 \sim_{h_{12}} b_2$.  Then $b_0 \sim_{h_{02}} b_2$ where
\[h_{02} = h_{01} + h_{12} - \mu^2_{b_0,b_1,b_2}(h_{01},h_{12}) .\]
Indeed, the first composition maps involving $b_0,b_1,b_2$ are related
by 
\begin{equation} \label{switch} \mu^1_{b_0,b_1} = \mu^1_{b_0,b_2} - \mu^2_{b_0,b_1,b_2}(\cdot, b_2 - b_1), 
\quad 
\mu^1_{b_1,b_2} = \mu^1_{b_0,b_2} + \mu^2_{b_0,b_1,b_1}(  b_1- b_0, \cdot) .\end{equation}
It follows that 
\begin{eqnarray*}
 \mu^1_{b_0,b_2}(h_{02}) 
&=& 
 \mu^1_{b_0,b_2}(h_{01} + h_{12} + \mu^2_{b_0,b_1,b_2}(h_{01},h_{12}))  \\
&=& 
\mu^1_{b_0,b_1}(h_{01}) 
+ \mu^2_{b_0,b_1,b_2}(h_{01},b_2 - b_1) 
+ \mu^1_{b_1,b_2}(h_{12}) \\
&& - \mu^2_{b_0,b_1,b_2}(b_1 - b_0,h_{12}) 
+ \mu^1_{b_0,b_2}( \mu^2_{b_0,b_1,b_2}(h_{01},h_{12})) ) \\
&=&  (b_1 - b_0)
+ \mu^2_{b_0,b_1,b_2}(h_{01},\mu^1_{b_1,b_2}(h_{12}))
+ (b_2 - b_1) \\ 
&&- \mu^2_{b_0,b_1,b_2}(\mu^1_{b_0,b_1}(h_{01}),h_{12})
+ \mu^1_{b_0,b_2}(\mu^2_{b_0,b_1,b_2}(h_{01},h_{12})) \\
&=& b_2 - b_0 
\end{eqnarray*} 
\label{transproof} where the cancellation of the terms involving $\mu^2_{b_0,b_1,b_2}$
\[ \mu^2_{b_0,b_1,b_2}(h_{01},\mu^1_{b_1,b_2}(h_{12})) -
\mu^2_{b_0,b_1,b_2}(\mu^1_{b_0,b_1}(h_{01}),h_{12}) -
\mu^1_{b_0,b_2}(\mu^2_{b_0,b_1,b_2}(h_{01},h_{12})) \]
follows from the \ainfty axiom for the deformed maps \eqref{lotsbs}.
To prove symmetry \llabel{define} \label{definep}
suppose that $b_0 \sim_{h_{01}} b_1$.  Define
\[ 
 \phi(h_{10}) = h_{10}  - 
\mu^2_{b_1,b_0,b_1}(h_{10}, h_{01}) \quad
 \psi(h_{11}) = h_{11}  +
 \mu^2_{b_1,b_1,b_0}( h_{01}, h_{11}).
\] 
The identities \eqref{switch} imply that $\phi,\psi$ are chain maps:
Indeed, 
\begin{eqnarray*} 
 \phi (\mu^1_{b_1,b_0}(h_{10})) 
&=& \mu^1_{b_1,b_0}(h_{10}) -
\mu^2_{b_1,b_0,b_1}(\mu^1_{b_1,b_0}(h_{10}), h_{01})) \\ &=&
 \mu^1_{b_1,b_1}(h_{10}) + \mu^2_{b_1,b_0,b_1}(h_{10}, \mu^1_{b_0,b_1}(h_{01})) -
 \mu^2_{b_1,b_0,b_1}(\mu^1_{b_1,b_0}(h_{10}), h_{01})) \\ &=&
 \mu^1_{b_1,b_1}(h_{10})  - 
  \mu^1_{b_1,b_1}(
 \mu^2_{b_1,b_0,b_1}(h_{10}, h_{01})) \\ &=& 
\mu^1_{b_1,b_1} (\phi(h_{10}))
\end{eqnarray*} 
and similarly for $\psi$.  Since the $q =0$ part of
$\mu^2_{b_1,b_0,b_1}(\cdot, h_{01})$ resp. $\mu^2_{b_1,b_1,b_0}(
h_{01}, \cdot)$ has negative degree, the maps $\phi,\psi$ are
invertible.  Furthermore,
\begin{eqnarray} \label{phipsi} 
\phi( b_0 - b_1) &=& \mu^1_{b_1,b_1}(h_{01}) \\
\psi( \mu^1_{b_1,b_1}(h_{01})) &=& 
\mu^1_{b_0,b_1} 
\mu^2_{b_0,b_1,b_1}(h_{01},h_{01}) + b_0 - b_1 .
\end{eqnarray} 
Hence if we define
\[ h_{10} :=
 ( \psi \circ \phi)^{-1} ( - h_{01} - 
\mu^2_{b_0,b_1,b_1}(
h_{01}, h_{01}) ) \]
then $b_1 \sim_{h_{10}} b_0$:
\begin{eqnarray*} 
 \mu^1_{b_1,b_0} (h_{10} )  &=& 
\mu^1_{b_1,b_0}  \phi^{-1}   \psi^{-1}  
 (- h_{01} - 
 \mu^2_{b_0,b_1,b_1}(
h_{01}, h_{01}) ) \\
&=&
 \phi^{-1}   \psi^{-1}  
\mu^1_{b_0,b_1} 
 (- h_{01} - 
\mu^2_{b_0,b_1,b_1}(
h_{01}, h_{01}) ) \\
&=& 
 \phi^{-1}   \psi^{-1}  
( 
 \mu^1_{b_1,b_1} 
\mu^2_{b_0,b_1,b_1}(h_{01},h_{01}) + b_0 - b_1)  \\
&=& b_0 - b_1 .\end{eqnarray*} 
Also $b \sim_0 b$ for any $b$, hence $\sim$ is reflexive.

We define the potential of the algebra as a function on the moduli
space of solutions to the projective Maurer-Cartan equation.  Denote by
$\ol{MC}(A)$ the quotient of ${MC}(A)$ by the gauge equivalence
relation,
\[ \ol{MC}(A) = {MC}(A)/ \sim .\]
The quotient $\ol{MC}(A)$ is the {\em moduli space} of solutions to the
projective Maurer-Cartan equation.  Define a {\em potential}
\[ {W}: {MC}(A)
\to \Lambda\] 
on the space of solutions to the projective Maurer-Cartan equation by
\[ {W}(b) 1_A = \mu^0_b(1) .\]
We remark that ${W}$ would be related to the potential $\WW$ defined
in \eqref{WW} via the divisor equation, but the divisor equation in
the perturbation system given here only holds in weak form
\eqref{diveq}. 

\begin{lemma}  
The potential $W$ is gauge-invariant and so descends to a
potential $\ol{W}: \ol{MC}(A) \to \Lambda$.
\end{lemma} 

\begin{proof} 
Suppose $b_0 \sim_h b_1$ so $b_1 - b_0 = \mu^1_{b_0,b_1}(h)$.  We have
\begin{eqnarray*} 
\label{muzero} 
\mu^0_{b_1}(1) - \mu^0_{b_0}(1) &=& 
\sum_{i,j} \mu^{i + j + 1}( \underbrace{b_0,\ldots,b_0}_i, b_1 - b_0 , \underbrace{b_1,\ldots,b_1}_j) \\
&=& 
\sum_{i,j} \mu^{i + j + 1}( \underbrace{b_0,\ldots,b_0}_i,\mu^1_{b_0,b_1}(h), \underbrace{b_1,\ldots,b_1}_j) \\ \\
&=& 
\sum_{i,j,k} \mu^{i + j + k + 2}( \underbrace{b_0,\ldots,b_0}_i,
\mu^0_{b_0}(1), 
 \underbrace{b_0,\ldots,b_0}_j,
h, 
\underbrace{b_1,\ldots,b_1}_k) \\ 
&-& 
\sum_{i,j,k} \mu^{i + j + k + 1}( \underbrace{b_0,\ldots,b_0}_i,
h, 
 \underbrace{b_1,\ldots,b_1}_j,
\mu^0_{b_1}(1), 
\underbrace{b_1,\ldots,b_1}_k) \\ 
&=& \mu^{2}(W(b_0) e_A,h) 
- \mu^{2}(h, W(b_1) e_A) \\ 
&=&   ( W(b_0) - W(b_1)) h .
\end{eqnarray*} 
It follows that $ (e_A + h ) W(b_1) = (e_A + h) W(b_0) .$ Since
$h \in A^+$ the sum $e_A + h$ is non-zero.  Hence $W(b_1) = W(b_0)$ as
claimed.
\end{proof} 

\begin{corollary}   If $b_0 \sim_h b_1$, then $\mu^1_{b_0,b_1}$ is a differential. 
\end{corollary} 

\begin{proof}   Using the \ainfty relations and strict unitality we have  
\begin{eqnarray*} (\mu^1_{b_0,b_1})^2(a)  &=& 
\mu^2( \mu^0_{b_0}(1), a )
-
(-1)^{|a|} \mu^2(a, \mu^0_{b_1}(1)) \\ 
&=& (W(b_1) - W(b_0)) a = 0 . \qedhere \end{eqnarray*}
\end{proof} 

The cohomology of a curve \ainfty algebra is a collection of vector
spaces over the space of solutions to the projective Maurer-Cartan
equation For any $b \in {MC}(A)$ define
\[ H(b) := H(\mu^1_b) = \frac{ \ker(\mu^1_b)}{\on{im}(\mu^1_b)} .\]
The {\em cohomology complex} is the resulting complex of sheaves over
${MC}(A)$:
\begin{equation} \label{cohcom}  
A \times {MC}(A) \to A \times {MC}(A), \quad (a,b)
\mapsto (\mu_1^b(a), b) \end{equation}
The ``stalks'' of this complex fit together to the cohomology
\begin{equation} \label{cohomology}
H(A) := \bigcup_{b \in {MC}(A)} H(b) .\end{equation}

\begin{lemma} \label{gaugeinv} The cohomology $H(A)$ is
  gauge-invariant in the sense that if $b_0 \sim_{h_{10}} b_1$, then
  $H(b_0) \cong H(b_1)$.
\end{lemma}

\begin{proof}
One can verify that
    \[
        \mu^2_{b_1,b_0,b_0}(h_{10},a) = \sum_{n_1,n_2,n_3} 
        \mu^{2  + n_1 + n_2 +n_3}(\underbrace{b_1,\ldots, b_1}_{n_1}, h_{10}, 
        \underbrace{b_0,\ldots, b_0}_{n_2}, a, \underbrace{b_0,\ldots, b_0}_{n_3}) 
    \]
 satisfies 
 \[
     \mu^2_{b_1,b_0,b_0}(h_{10},\mu^1_{b_0}(a)) - 
     \mu^1_{b_1,b_0}(\mu^2_{b_1,b_0,b_0}(h_{10},a) ) = 
     \mu^1_{b_0}(a) -
     \mu^1_{b_1,b_0}(a).
 \]
Hence the operator
\[\mu_{b_1,b_0,b_0}^2(h_{10},\_)-\on{Id}(\_): (A,\mu^1_{b_0})\to (A,\mu^1_{b_1,b_0}) ,\]
is a chain morphism, see \eqref{phipsi}.  For the same reasons,
\[\mu^2_{b_0,b_1,b_0}(h_{10},\_)-\on{Id}(\_): (A,\mu^1_{b_1,b_0}) \to (A,\mu^1_{b_0}) .\]
is a chain morphism.  Consider the map 
\[A \to A, \quad a \mapsto
H^{b_0,b_1,b_0,b_0}_1(h_{01},h_{10},a)\] 
where
\begin{multline} 
H^{b_0,b_1,b_0,b_0}_1(h_{01},h_{10},a) := \sum_{n_1,n_2,n_3,n_4}
\mu^{3 + n_1 + n_2 +n_3 +n_4}(\underbrace{b_0,\ldots, b_0}_{n_1},
h_{01}, \\ 
\underbrace{b_1,\ldots, b_1}_{n_2}, h_{10},
\underbrace{b_0,\ldots, b_0}_{n_3}, a, \underbrace{b_0,\ldots,
  b_0}_{n_4}) .
    \end{multline}
This map is a chain homotopy between
\[
(\mu^2_{b_0,b_1,b_0}(h_{10},\_)-\on{Id}(\_)) 
\circ(\mu^2_{b_1,b_0,b_0}(h_{10},\_)-\on{Id}(\_)) 
\] 
and the chain map 
\[\Phi(a)=a + \mu^2_{b_0}(h_{01} + h_{10} +
\mu^2_{b_0,b_1,b_0}(h_{01},h_{10}),a) .\]  
The latter can be seen to induce an automorphism of $H(b_0)$. In the
same way, the reverse composition of these maps is homotopic to a map
inducing an automorphism of
\[H(b_1,b_0) :=H(A,\mu^1_{b_1,b_0}) .\]
    Similarly, $\mu^{b_1,b_1,b_0}_2(\_,h_{10})-\on{Id}(\_)$ and
    $\mu^{b_1,b_0,b_1}_2(\_,h_{01})-\on{Id}(\_)$ define a chain morphism from
    $(A,\mu^1_{b_1})$ to $(A,\mu^1_{b_1,b_0})$ and
    $(A,\mu^1_{b_1,b_0})$ to $(A,\mu^1_{b_1})$, respectively.  This
    shows that $H(b_1,b_0)\cong H(b_1)$.  Hence
$ H(b_0) \cong H(b_1,b_0) \cong H(b_0) $
as claimed.
\end{proof}

We consider the special case of the projective Maurer-Cartan solutions for a
Fukaya algebra.  The following Lemma will be used in the proof of
Theorem \ref{unobs} below to show that Lagrangians are weakly
unobstructed.

\begin{lemma}\label{Wx} Suppose that $\mu^0(1) \in \Lambda x^{\blackt}$ and 
  every non-constant disk has positive Maslov index.  Then ${MC}(L)$
  is non-empty. 
\end{lemma} 

\begin{proof} Suppose $\mu^0(1) = W x^{\blackt}$ and the condition 
  in the Lemma holds.  Equation \eqref{diff} becomes 
  $ \mu^{1}(x^{\greyt}) = x^{\whitet} - x^{\blackt} .$ Hence 
  with notation as in \eqref{mcmap} we have 
\[ \mu(W x^{\greyt})= \mu^0(1) + W \mu^1(x^{\greyt}) = W x^{\blackt}
+ W (x^{\whitet} - x^{\blackt}) = W x^{\whitet} \in \Lambda 
x^{\whitet} .\]
Thus $W x^{\greyt} \in {MC}(L)$. 
\end{proof} 

We also mention that the addition of homotopy units does not affect
the Floer cohomology:

\begin{lemma} \label{nochange} Suppose that $b \in MC(L)$.  Then
  $HF(L,b)$ is isomorphic to the cohomology $HF^{\on{geom}}(L,b)$ of
  $\mu^1_b$ acting on $CF^{\on{geom}}(L)$.
\end{lemma}

\begin{proof} Choose $\hbar$ so that the first term in $\mu^1_b$ with
  non-vanishing $q$-valuation has valuation at least $\hbar$.
  Consider the first page in the spectral sequence for $CF(L)$ induced
  by the filtration $q^{n \hbar} CF(L),\ n \in \Z_{\ge 0}$.  The
  differential $\mu^1_{b,1}$ on the first page $E^1$ of this spectral
  sequence comes from trajectories with no disks, and in particular
  \eqref{diff} becomes
  $\mu^1_{b,1}(x^{\greyt}) = x^{\whitet} - x^{\blackt}$.  It follows
  that $H(\mu^1_{b,1})$ is the cohomology of the Morse differential
  $\mu^1_b |_{q =0 }$ on $CF^{\on{geom}}(L) $.  The claim follows.
\end{proof}

\chapter{Homotopy invariance}
\label{hinv} 

In this and following sections we show that the Fukaya algebra
constructed above is independent, up to \ainfty homotopy invariance,
of the choice of perturbation system.  The argument uses moduli spaces
of {\em quilted treed disks}, introduced without trees in
Ma'u-Woodward \cite{mau:mult}.

\section{\ainfty morphisms} 

Recall the definition of \ainfty morphisms and homotopies. 

\begin{definition} \label{morphisms}
\begin{enumerate} 
\item {\rm (\ainfty morphisms)}
 \label{ainftyfunctor}
Let $A_0,A_1$ be \ainfty algebras.  An {\em \ainfty morphism}
$\F$ from $A_0$ to $A_1$ consists of a sequence of linear maps
\[\F^d: \ A_0^{\otimes d} \to A_1[1-d], \quad d \ge 0\]
such that the following holds:
\begin{multline} \label{faxiom}
 \sum_{i + j \leq d} (-1)^{i + \sum_{j=1}^i |a_j|} \F^{d - j +
     1}(a_1,\ldots,a_i, \mu_{A_0}^j(a_{i+1
   },\ldots,a_{i+j}),a_{i+j+1},\ldots,a_d) = \\ \sum_{i_1 + \ldots + i_m = d}
   \mu_{A_1}^m(\F^{i_1}(a_1,\ldots, a_{i_1}), \ldots,
   \F^{i_m}(a_{i_1 + \ldots + i_{m-1} + 1},\ldots,a_d))
\end{multline}
where the first sum is over integers $i,j$ with $i+j \leq d$, the
second is over partitions $d = i_1 + \ldots + i_m$.  An \ainfty
morphism $\F$ is {\em unital} if and only if
\[\F^1(e_0) = e_1, \quad \F^k(a_1, \ldots, a_i, e_0,
a_{i+2}, \ldots, a_k) = 0\] 
for every $k\geq 2$ and every $0\leq i \leq k-1$, where $e_0$
resp. $e_1$ is the strict unit in $A_0$ resp. $A_1$. 
\item {\rm (Composition of \ainfty morphisms)} The {\em composition} of
  \ainfty morphisms $\F_0,\F_1$ is defined by 
\begin{multline} \label{composefunc}
 (\F_0 \circ \F_1)^d(a_1,\ldots,a_d)
 = \sum_{i_1 + \ldots + i_m =d}
  \F_0^{m}( \F_1^{i_1}(a_1,\ldots,a_{i_1}),  \\ \ldots, \F_1^{i_m}(a_{d -
    i_m + 1},\ldots,a_d)). \end{multline}
\item {\rm (\ainfty natural transformations)}  
Let $\F_0,\F_1: A_0 \to A_1$ be \ainfty morphisms.  A {\em pre-natural
  transformation} $\TT$ from $\F_0$ to $\F_1$ consists of for each $d
\geq 0$ a multilinear map
\[ \TT^d: \ A_0^d \to A_1 .\]
Let $\Hom(\F_0,\F_1)$ denote the space of pre-natural transformations
from $\F_0$ to $\F_1$.  Define a differential on $\Hom(\F_0,\F_1)$ by
\begin{multline} \label{mu1}
 (\mu^1_{\Hom(\F_0,\F_1)} \TT)^d (a_1,\ldots,a_d) = \sum_{k,m}
  \sum_{i_1,\ldots,i_m} (-1)^\dagger \mu^m_{A_2}(
  \F_0^{i_1}(a_1,\ldots,a_{i_1}), \F_0^{i_2}(a_{i_1 + 1},\ldots, a_{i_1+i_2}), 
  \ldots, \\ \TT^{i_k}(a_{i_1 + \ldots + i_{k-1} + 1},\ldots, a_{i_1 +
    \ldots + i_k}), \F_1^{i_{k+1}}( a_{i_1 + \ldots + i_k + 1},\ldots,
  ) ,\ldots, \F_1^{i_m}(a_{d - i_m},\ldots, a_d)) \\ - \sum_{i,e}
  (-1)^{i + \sum_{j=1}^i |a_j| + |\TT| - 1} \TT^{d - e +
    1}(a_1,\ldots,a_i, \mu^e_{A_1}(a_{i+1},\ldots,
  a_{i+e}),a_{i+e+1},\ldots,a_d) \end{multline}
where 
\[ \dagger = (|\TT|-1)( |a_1| + \ldots + |a_{i_1 + \ldots +
  i_{k-1}}| - i_1 - \ldots - i_{k-1}) .\] 
A {\em natural transformation} is a closed pre-natural transformation.
An \ainfty (pre)natural transformation $\TT$ from a
unital morphism $\F_0$ to a unital morphism $\F_1$ is {\em unital} if
and only if 
\[\TT^k(a_1, \ldots, a_i, e_0, a_{i+2}, \ldots, a_k) = 0 \quad
\forall k\geq 1, \ 0\leq i \leq k-1 \]
where $e_0$ is the unit in $A_0$.
\item {\rm (Composition of natural transformations)} Given two
  pre-natural transformations 
\[\TT_1:\F_0 \to \F_1, \quad \TT_2:\F_1 \to
  \F_2 \] 
  define $ \mu^2(\TT_1,\TT_2)$ by
\begin{multline}  \label{T2T1}
 (\mu^2(\TT_1,\TT_2))^d(a_1,\ldots,a_d) =
\sum_{m,k,l} \sum_{i_1,\ldots,i_m} 
(-1)^\ddag \mu^m_{A_2}( \F_0^{i_1}(a_1,\ldots,a_{i_1}),
\ldots, \F_0^{i_{k-1}}(\ldots),
 \\
\TT_1^{i_k}(a_{i_1 + \ldots + i_{k-1} + 1},\ldots, a_{i_1 + \ldots + i_k}),
\F_1^{i_{k+1}}(\ldots),\ldots, \F_1^{i_{l-1}}(\ldots), 
\\
\TT_2^{i_l} (a_{i_1 + \ldots + i_{l-1} + 1},\ldots, a_{i_1 + \ldots + i_{l}}),
\F_2^{i_{l+1}}(\ldots),\ldots, \F_2^{i_m}(a_{d - i_m},\ldots,a_d)) 
\end{multline}
where 
\[ \ddag = \sum_{i = 1}^{i_1 + \ldots + i_{k-1}} ( |\TT_1| - 1) ( |a_i| - 1) + 
\sum_{i = 1}^{i_1 + \ldots + i_{l-1}} ( |\TT_2| - 1) ( |a_i| - 1)  .\]
Let $\Hom(A_0,A_1)$ denote the space of \ainfty morphisms from $A_0$
to $A_1$, with morphisms given by pre-natural transformations.
Compositions give $\Hom(A_0,A_1)$ the structure of an \ainfty category
\cite[10.17]{fu:fl1}, \cite[8.1]{le:ai}, \cite[Section 1d]{se:bo}.
\item {\rm (\ainfty homotopies)} Suppose that $\F_1,\F_2: A_0 \to A_1$
  are morphisms.  A {\em homotopy} from $\F_1$ to $\F_2$ is a
  pre-natural transformation $\TT \in \Hom(\F_1,\F_2)$ such that
 \begin{equation} \label{homotopy} 
	 \F_1 - \F_2 = \mu_{\Hom(\F_1,\F_2)}^1 (\TT).
 \end{equation}
where $\mu_{\Hom(\F_1,\F_2)}^1 (\TT)$ is defined in \eqref{mu1}.  If a
homotopy exists we say that $\F_1$ is homotopic to $\F_2$ and write
$\F_1 \equiv \F_2$.  As shown in \cite[Section 1h]{se:bo}, homotopy of
\ainfty morphisms is an equivalence relation.
\item {\rm (Composition of homotopies)} Given homotopies $\TT_1$ from
  $\F_0$ to $\F_1$, and $\TT_2$ from $\F_1$ to $\F_2$, the sum
\begin{equation} \label{composehom}
\TT_2 \circ \TT_1: = \TT_1 + \TT_2 + \mu^2(\TT_1,\TT_2) \in
  \Hom(\F_0,\F_2) \end{equation} 
is a homotopy from $\F_0$ to $\F_2$.
\item \label{compfun} {\rm (Composition of morphisms)} Let
  $A_0,A_1,A_2$ be \ainfty algebras.  Given a morphism $\F_{12}: A_1
  \to A_2$ resp.  $\F_{01}: A_0 \to A_1$, right composition with
  $\F_{12}$ resp. left composition with $\F_{01}$ define \ainfty
  morphisms
\[ \begin{array}{lll}
 \cR_{\F_{12}} &:& \Hom(A_0,A_1) \to \Hom(A_0,A_2)  \\
  \ \cL_{\F_{01}} &:& \Hom(A_0,A_2) \to \Hom(A_1,A_2) .\end{array}\]
The action on pre-natural transformations is given as follows
\cite[Section 1e]{se:bo}: Let $\F_{01}', \F_{01}'':A_0 \to A_1$ be
\ainfty morphisms and $T_{01}$ a pre-natural transformation from
$\F_{01}'$ to $\F_{01}''$.  Define
\begin{multline} 
 (\cR_{\F_{12}}(T_{01}))_d(a_1,\ldots,a_d) \\ = \sum_{r,j} \sum_{i_1 +
    \ldots + i_r = d} (-1)^{\ddag} \F_{12}( \F_{01}'(
  a_1,\ldots,a_{i_1}), \ldots, \F_{01}'( \ldots ), \\ T_{01}(a_{i_1 +
    \ldots, + i_{j-1} + 1}, \ldots, a_{i_1 + \ldots + i_j}), 
  \F_{01}'' (\ldots ), \\ \ldots, \F_{01}''(a_{i_1 + \ldots + i_{r-1}
    + 1},\ldots, a_d )) .\end{multline}
\item {\rm (Homotopy equivalence of \ainfty algebras)} Two  \ainfty algebras $A_0,A_1$ are {\em homotopy equivalent} if there
  exist morphisms $\F_{01} : A_0 \to A_1$ and $\F_{10} : A_1 \to A_0$
  such that $\F_{01} \circ \F_{10}$ and $\F_{10} \circ \F_{01}$ are
  homotopic to the respective identities.  Homotopy equivalence of
  \ainfty algebras is an equivalence relation: Symmetry and
  reflexivity are immediate.  For transitivity note that by the
  previous item, if $\F_{01}'$ and $\F_{01}''$ are homotopy equivalent
  then so are $\F_{12} \circ \F_{01}'$ and $\F_{12} \circ \F_{01}''$;
  repeating the argument for left composition, if $\F_{01}'$ and
  $\F_{01}''$ are homotopic and also $\F_{12}'$ and $\F_{12}''$ then
  $\F_{12}' \circ \F_{01}'$ is homotopic to $\F_{12}'' \circ
  \F_{01}''$.  Hence if 
  \[\F_{01}: A_0 \to A_1, \F_{10}: A_1 \to A_0, \quad \F_{12} : A_1
  \to A_2, \F_{21} : A_2 \to A_1\]
are homotopy
  equivalences then
\begin{eqnarray*} 
 (F_{10} \circ \F_{21}) \circ (\F_{12} \circ \F_{01}) 
&=& \F_{10} \circ (\F_{21} \circ \F_{12}) \circ \F_{01} \\
&\cong& \F_{10} \circ \F_{01} 
\cong \on{Id} \end{eqnarray*}
and similarly for $ (\F_{12} \circ \F_{01}) \circ (F_{10} \circ
\F_{21})$.
\end{enumerate} 
\end{definition} 

The following is a version of the material on Maurer-Cartan moduli
spaces in \cite[Chapter 4]{fooo}.  Let $A_0,A_1$ be convergent \ainfty
algebras.  Let $\F: A_0 \to A_1$ be an \ainfty morphism or pre-natural
transformation.  Such $\F$ is {\em convergent} if and only if
the zero-th term $\F^0(1) $ lies in $A^+$.

\begin{lemma} \label{mcmapthm} \llabel{mcmapl} {\rm (Map between
    Maurer-Cartan moduli spaces)} Suppose that $A_0,A_1$ are
  convergent strictly-unital \ainfty algebras and $\F: A_0 \to A_1$ is
  a convergent unital \ainfty morphism.  Then
\begin{equation} \label{phiF}
\phi_{\F} ({b})  = \F^0(1) + \F^1({b}) + \F^2({b},{b}) + \F^3({b},{b},{b}) 
+ \ldots \end{equation}
defines a map from $A_0^+$ to $A_1^+$ and restricts to a map
${MC}(A_0) \to {MC}(A_1)$ of the moduli spaces of
solutions to the projective Maurer-Cartan equation and descends to a map
$\ol{MC}(A_0)$ to $\ol{MC}(A_1)$.  That is: For every ${b} \in
{MC}(A_0)$, $\phi_{\F}({b}) \in {MC}(A_1)$; and
\[ ({b} \sim {b}') \  \implies \ ( \phi_{\F}({b}) \sim \phi_{\F}({b}')),
\quad \forall {b}, {b}' \in {MC}(A_0) .\]

Moreover, if $\F_0,\F_1: A_0 \to A_1$ are unital \ainfty morphisms
that are homotopic by an \ainfty homotopy $\TT: A_0 \to A_1$
then $\F_0,\F_1$ induce the same map on Maurer-Cartan moduli spaces,
that is, $\phi_{\F_0} = \phi_{\F_1}$. In particular, if $\F: A \to A$
is convergent and convergent-homotopic to the identity then $\F$
induces the identity on $\ol{MC}(A)$.  Hence the set $\ol{MC}(A)$ is an
invariant of the homotopy type of $A$.
\end{lemma}

\begin{proof} The proof that the sum \eqref{phiF} converges is
  essentially the same as that for Lemma \ref{curvlem} and left to the
  reader.  Regarding gauge invariance, first notice that for every
  $b\sim b' \in {MC}(A_0)$ with $h \in A_0$ such that
  $b- b'= \mu^1_{b,b'}(h)$,
	\begin{eqnarray*}
		\phi_{\F}(b) - \phi_{\F}(b')
        &=& \sum_{n_0,n_1} \F^{n_0 + n_1 + 1}(\underbrace{b, \ldots, b}_{n_0}, b-b', \underbrace{b', \ldots, b'}_{n_1})\\
	&=& \sum_{n_0,n_1,n_2,n_3}  \F^{n_0 + n_1 + 1}(\underbrace{b, \ldots, b}_{n_0}, \mu^{n_2+n_3+1}(\underbrace{b, \ldots, b}_{n_2}, h,
\\ && \ \quad \quad \underbrace{b', \ldots,b'}_{n_3}), 
\underbrace{b', \ldots, b'}_{n_1}) \\
%
    &=& \mu^1_{\phi_{\F}(b),\phi_{\F}(b')} \left( \sum_{n_4,n_5} \F^{n_4+n_5+1}(\underbrace{b, \ldots,b}_{n_4}, h,\underbrace{b', \ldots,b'}_{n_5} ) \right)
	\end{eqnarray*}
so that $\phi_{\F}(b) \sim \phi_{\F}(b')$. Note that the third
equality above uses the fact that $b, b' \in {MC}(A_0)$ and
the unitality of $\F$.

We now prove the second part of the claim.  Suppose that
$\F_0 - \F_1 = \mu_{\Hom(\F_0,\F_1)}^1 (\TT)$ for some unital
pre-natural transformation $\TT$.  Then
	\begin{eqnarray*}
		\phi_{\F_0}(b) - \phi_{\F_1}(b) &=& \sum_{k\geq 0}
                (\mu_{\Hom(\F_0,\F_1)}^1 \TT)^k (b, \ldots,b) \\ &=&
                \sum_{k\geq 0} \sum_{m} \sum_{i_1+\ldots+i_m=k}
                (-1)^\dagger \mu^m_{A_2}( \F_0^{i_1}(b,\ldots,b),
                \F_0^{i_2}(b,\ldots,b), \ldots, \\ &&
                \TT^{i_k}(b,\ldots, b ), \F_1^{i_{k+1}}( b,\ldots, b)
                ,\ldots, \F_1^{i_m}(b,\ldots, b)) \\ && - \sum_{i,e}
                (-1)^{i + \sum_{j=1}^i |b| + |\TT| - 1} \TT^{k - e +
                  1}(\underbrace{b,\ldots,b}_{i},
                \mu^e_{A_1}(b,\ldots, b),b,\ldots,b) \\ & =:& \clubsuit.
 \end{eqnarray*} 
Using that $b$ is a projective Maurer-Cartan solution we continue
\begin{eqnarray*}
\clubsuit 	&=& \sum_{k\geq 0} \sum_{m}
  \sum_{i_1+\ldots+i_m=k} (-1)^\dagger \mu^m_{A_2}(
  \F_0^{i_1}(b,\ldots,b), \F_0^{i_2}(b,\ldots,b),
  \ldots, \\ 
  && \TT^{i_k}(b,\ldots, b
    ), \F_1^{i_{k+1}}( b,\ldots,
  b) ,\ldots, \F_1^{i_m}(b,\ldots, b)) \\
 && - \sum_{i,e}
  (-1)^{i + \sum_{j=1}^i |b| + |\TT| - 1} \TT^{k - e +
    1}(b,\ldots,b, \lambda e_{A_1}
  ,b,\ldots,b) \\
	&=& \sum_{k\geq 0} \sum_{m}
  \sum_{i_1+\ldots+i_m=k} (-1)^\dagger \mu^m_{A_2}(
  \F_0^{i_1}(b,\ldots,b), \F_0^{i_2}(b,\ldots,b),
  \ldots, \\ 
  && \TT^{i_k}(b,\ldots, b
    ), \F_1^{i_{k+1}}( b,\ldots,
  b) ,\ldots, \F_1^{i_m}(b,\ldots, b)) \\
  &=& \mu^1_{\phi_{\F_0} (b),\phi_{\F_1} (b)} \left(
  \sum_{k} \TT^{k}(b,\ldots, b) \right).
\end{eqnarray*}
Since $b \in A^+$ the sum
\[h_{01} := \sum_{k} \TT^{k}(b,\ldots, b) \]
exists in $A$.  Furthermore since $\TT^0(1) \in A^+ $ we have
$h_{01} \in A^+$.  Hence $\phi_{\F_0}(b) \sim_{h_{01}} \phi_{\F_1}(b)$
as claimed.
\end{proof}

Similarly one has a homotopy invariance property of the cohomology vector
bundle introduced in \eqref{cohomology}:   

\begin{lemma} \label{uptogauge} {\rm (Maps of cohomology bundles)} Any
  convergent \ainfty morphism $\cF:A_0 \to A_1$ induces a morphism
  $H(\cF): H(A_0) \to H(A_1)$, that is, a morphism $H(b) \to
  H(\cF(b))$ for each $b \in {MC}(A_0)$.  If $\cF_0, \cF_1:
  A_0 \to A_1$ are convergent morphisms related by a convergent
  homotopy then $H(\cF_0)$ is equal to $H(\cF_1)$ up to gauge
  transformation.  In particular, if there exist convergent \ainfty
  maps $\F_{01}: A_0 \to A_1$ and $\F_{10}:A_1 \to A_0$ such that
  $\F_{01} \circ \F_{10}$ and $\F_{10} \circ \F_{01}$ are homotopic to
  the identities via convergent homotopies then $H(A_0)$ is isomorphic
  to $H(A_1)$ up to gauge equivalence in the sense that
\[H(b_0) \cong H(\phi_{\F_{01}}(b_0)), \quad H(b_1) \cong H(\phi_{\F_{10}}(b_1)) \]
for any $b_0 \in {MC}(A_0), b_1 \in {MC}(A_1)$.
\end{lemma} 

The proof is similar to that of Lemma \ref{gaugeinv} and omitted.
Thus having a non-trivial cohomology is an invariant of the homotopy
type of a convergent, strictly-unital \ainfty algebra.

\section{Multiplihedra} 
\label{mult} 

The terms in the \ainfty morphism axiom correspond to codimension one
cells in a cell complex called the multiplihedron introduced by
Stasheff \cite{st:ho}.  Stasheff's definition identifies the
$n$-multiplihedron as the cell complex whose vertices correspond to
expressions involving bracketings of formal variables
$x_1,\ldots,x_n$, together with expressions $f( \cdot )$ so that every
$x_j$ is contained in an argument of some $f$.  For example, the
second multiplihedron is an interval with vertices $f(x_1)f(x_2)$ and
$f(x_1x_2)$.  A geometric realization of this polytope was given by
Boardman-Vogt \cite{boardman:hom} in terms of what we will call {\em
  quilted (metric ribbon) trees}.  \llabel{ribbontrees} A quilted
metric ribbon tree is a ribbon metric tree
\[T = (\Gamma, e_0 \in \Edge_{\white,\rightarrow}(\Gamma), \ell: \Edge_{-}(\Gamma) \to
[0,\infty])\]
(ribbon structure omitted from the notation) together with a subset
\[\Ver^1(\Gamma) \subset \Ver(\Gamma)\] of {\em
  colored vertices}.  This set is required to satisfy the condition
that every simple path of edges $e_0, e_1, \ldots, e_k$ from the root
$e_0$ to a leaf $e_k$ meets precisely one colored vertex
$v \in \Ver^1(\Gamma)$, and the metric $\ell$ is required to satisfy the
following:
\begin{itemize}
    \item[] {\rm (Balanced lengths condition)} For any two {\em
            colored vertices} $v_1,v_2 \in \Ver^1(\Gamma)$, 
\begin{equation} \label{balanced}
\sum_{e \in P_+(v_1,v_2)} \ell(e) = \sum_{e \in P_-(v_1,v_2)} \ell(e)
\end{equation}
where $P(v_1,v_2)$ is the (finite length) oriented non-self-crossing
path from $v_1$ to $v_2$ and $P_+(v_1,v_2)$ resp. $P_-(v_1,v_2)$ is
the part of the path towards resp. away from the root edge, see
Ma'u-Woodward \cite{mau:mult}.
\end{itemize} 
The set of combinatorially finite resp. semi-infinite edges is denoted
$\Edge_{-}(\Gamma)$ resp. $\Edge_{\rightarrow}(\Gamma)$; the latter
are equipped with a labelling by integers $0,\ldots, n$.  A quilted
tree is {\em stable} if each colored vertex $v \in \Ver^1(\Gamma)$ has
valence $|v|$ at least two and any non-colored vertex
$v \in \Ver(\Gamma) - \Ver^1(\Gamma)$ has valence $|v|$ at least
three.  Broken quilted trees are defined as in the non-quilted case,
but requiring that any simple path from the root of the broken tree to
a leaf still meets only one colored vertex.  There is a natural notion
of convergence of quilted trees $T_i \to T_\infty$, in which edges
$e_i$ of $\Gamma_i$ whose length $\ell(e_i)$ approaches zero are
contracted and edges $e_i$ whose lengths $\ell(e_i)$ go to infinity
are replaced by broken edges.

A different realization of the multiplihedron is the moduli space of
stable {\em quilted} disks in Ma'u-Woodward \cite{mau:mult}.  In this
realization, one obtains Stasheff's cell structure on the
multiplihedron naturally.  

\begin{definition} \label{quiltings} 
\begin{enumerate}
\item
A {\em quilted disk} is a datum
  $(S,Q,x_0,\ldots, x_n \in \partial S)$ consisting of a marked
  complex disk $(S,x_0,\ldots, x_n \in \partial S)$ (the points are
  required to be in counterclockwise cyclic order) together with a
  circle {\em seam} $Q \subset S$ (here we take $S$ to be a ball in
  the complex plane, so the notion of circle makes sense) tangent to
  the $0$-th marking $x_0$.  An {\em isomorphism} of quilted disks
  from $(S,Q,x_0,\ldots,x_n)$ to $(S',Q',x_0',\ldots,x_n')$ is an
  isomorphism of pseudoholomorphic disks $S \to S'$ mapping $Q$ to
  $Q'$ and $x_0,\ldots, x_n$ to $x_0',\ldots, x_n'$.
\item 
An {\em affine structure} on a disk
  $S$ with boundary point $z_0 \in \partial S$ is an isomorphism with
  a half-space $\phi: S - \{ z_0 \} \to \H$.  Two affine structures
  $\phi,\phi': S - \{ z_0 \} \to \H$ are considered equivalent if
  $\phi(z) = \phi'(z) + \zeta$ for some $\zeta \in \R$.  A quilting is
  equivalent to an affine structure, by taking the quilting to be
  $Q = \{ \on{Im}(z) = 1 \}$.
\item
In this context the notion of quilted
  disk admits a natural generalization to the notion of a {\em quilted
    sphere}: a marked sphere $(C, (z_0,\ldots,z_n))$ equipped with an
  isomorphism $\phi$ from $C - \{z_0 \} \to \C$ to the affine line
  $\C$.  Again, two such isomorphisms $\phi,\phi'$ are considered
  equivalent if they differ by a translation:
  $\phi(z) = \phi'(z) + \zeta$ for some $\zeta \in \C$.
\item
The {\em combinatorial type} $\Gamma$ of a
  quilted nodal marked disk $(S,Q,\ul{x})$ is defined similar to the
  combinatorial type of a nodal marked disk disk.  The set of vertices
  $\Ver(\Gamma)$ has a distinguished subset $\Ver^1(\Gamma)$ of {\em
    colored vertices} corresponding to the quilted components.  Thus
  the unique non-self-crossing path $\gamma_e$ from the root edge
  $e_0$ of the tree to any leaf $e$ is required to pass through
  exactly one colored vertex $v \in \Ver^1(\Gamma)$.
\item
 A {\em nodal quilted disk} of type
  $\Gamma$ is a union of disks
  $S_v , v \in \Ver_{\white}(\Gamma) - \Ver_{\white}^1(\Gamma)$,
  spheres
  $S_v , v \in \Ver_{\black}(\Gamma) - \Ver_{\black}^1(\Gamma)$,
  quilted disks $S_v, v \in \Ver_{\white}^1(\Gamma)$, and quilted
  spheres $S_v, v \in \Ver_{\black}^1(\Gamma)$, with the property that
  along a non-self-crossing path of components
  $S_v, v \in \Ver(\Gamma)$ from the root edge $e_0$ to any other leaf
  $e_i$, a single colored vertex $v \in \Ver^1(\Gamma)$ occurs.  A
  nodal quilted disk $S$ is {\em stable} if there are no non-trivial
  automorphisms $\Aut(S) - \{ 1 \}$, or equivalently, each component
  $S_v, v \in \Ver(\Gamma)$ has no automorphisms.  In the case of
  sphere components $S_v, v \in \Ver_{\black}(\Gamma)$, this means
  that $S_v$ has at least three special points if it is unquilted, or
  at least two special points if it is quilted.
\end{enumerate}
\end{definition} 

The moduli space of stable quilted disks with interior and boundary
markings is a compact cell complex.  As the interior and boundary
markings go to infinity, they bubble off onto either quilted disks or
quilted spheres.  The case of combined boundary and interior markings
is a straight-forward generalization of the boundary and interior
cases treated separately in \cite{mau:mult}.

There is a combined moduli space which includes both quilted disk,
spheres, and possibly broken segments.  A quilted treed disk $C$ is
obtained from a quilted nodal disk $S$ by replacing each node
$w(e) \in S, e \in \Edge(\Gamma(S))$ with a (possibly broken) segment
$T_e$ by attaching the endpoints of $T_e$ to two copies of $w(e)$.
The {\em combinatorial type} $\Gamma(C)$ of a quilted treed disk $C$
is the combinatorial type $\Gamma(S)$ of the surface with the
additional labellings of the number of breakings $b(e)$ of each edge
$T_e$.  A quilted treed disk is {\em stable} if the underlying quilted
disk is stable.  Let $\ol{\M}_{n,m,1}$ denote the moduli space of
stable marked quilted treed disks $u: C \to X$ with $n$ boundary
leaves and $m$ interior leaves.  See Figure \ref{MWc} for a picture of
$\ol{\M}_{2,0,1}$.  A picture of $\ol{\M}_{1,1,1}$ is shown in Figure
\ref{qspheres}.  In the picture the quilted disks
$S_v \subset S, v \in \Ver^1(\Gamma)$ are those with two shadings;
while the ordinary disks $S_v, v \notin \Ver^1(\Gamma)$ have either
light or dark shading depending on whether they can be connected to
the zero-th edge without passing a colored vertex.  The hashes on the
line segments $T_e$ indicate breakings.
\begin{figure}[ht]
\includegraphics[width=5in]{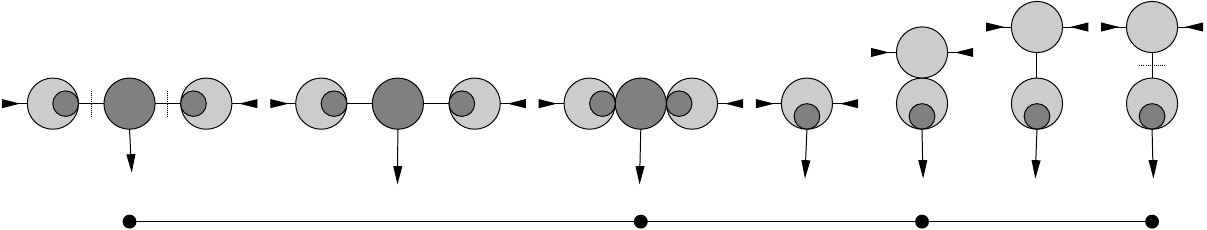}
\caption{Moduli space of stable quilted treed disks}
\label{MWc} 
\end{figure} 
Any non-self-crossing path from the root edge $e_0$ to a leaf $e$ must
pass through exactly one colored vertex $v \in \Ver^1(\Gamma)$
correspond to either a quilted disk or quilted sphere $S_v \subset C$.
See Figure \ref{qspheres} for the combinatorics of the top-dimensional
cells in the case of one boundary leaf
$e \in \Edge_{\white,\rightarrow}(\Gamma)$ and one interior leaf
$e \in \Edge_{\black,\rightarrow}(\Gamma)$; the $s$ indicates a
quilted sphere component
$S_v \cong \P^1, v \in \Ver_{\black}(\Gamma)$.  Orientations of the
stratum $\M_\Gamma$ of the moduli space of quilted treed disks can be
defined from the moduli space $\M_{\Gamma'}$ of unquilted disks with
the same number of markings via the morphism
$\M_\Gamma \to \R \times \M_{\Gamma'}$ where the extra factor $\R$
prescribes the position of the seam in Definition \ref{quiltings};
this orientation extends over the other top-dimensional cells as in
the unquilted case.

\begin{figure}[ht]
\includegraphics[height=2.5in]{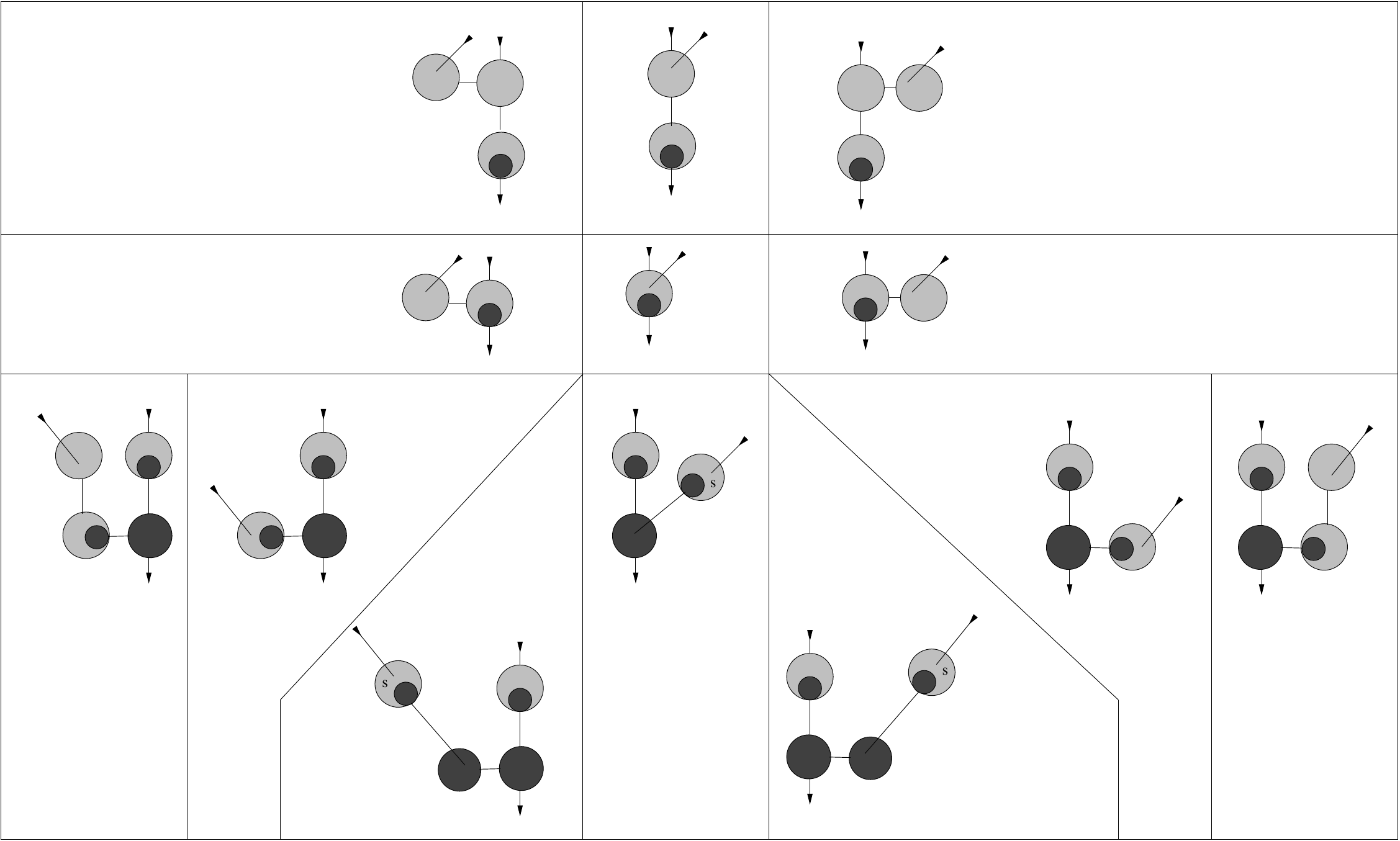}
\caption{Moduli space of stable quilted treed disks with a boundary
  leaf and an interior leaf}
\label{qspheres} 
\end{figure} 

The boundary of the moduli space of quilted treed disks is a union of
moduli spaces involving quilted and unquilted disks with fewer
markings.  For any integers $n,m \ge 1$ we have
\begin{equation} \label{unionseq} \partial \ol{\M}_{n,m,1} \cong 
\bigcup_{\substack{i,j \\ m_1+m_2 =m}} \left( \ol{\M}_{n-i + 1,m_1,1}
\times \ol{\M}_{i,m_2}\right) \cup \bigcup_{\substack{i_1,\ldots,
    i_r\\ m_0 + \sum m_j=m}} \left( \ol{\M}_{r,m_0} \times
\prod_{j=1}^r \ol{\M}_{i_j,m_j,1} \right) .\end{equation}
In the first sum $j$ is the index of the attaching leaf in the quilted
tree and so ranges from $1$ to $n-i + 1$, while $i_1,\ldots, i_r$ in
the second union range over partitions $i_1 + \ldots + i_r = n$.  By
construction, for the facet of the first type, the sign of the
inclusions of boundary strata are the same as that for the
corresponding inclusion of boundary facets of $\ol{\M}_{n,m,1}$, that
is, $(-1)^{i(n-i-j) + j} $.  For facets of the second type, the gluing
map
$ (0,\infty) \times \M_{r,m_0} \times \Pi_{j=1}^r \M_{|I_j|,m_j,1} \to
\M_{n,m,1} $ is for boundary markings
\begin{multline} \label{gluemap2} 
    (\delta,x_1,\ldots, x_r,
(x_{1,j} = 0, x_{2,j},\ldots, x_{|I_j|,j} )_{j=1}^r) \mapsto \\ 
(x_1, x_1 + \delta^{-1}x_{2,1}, \ldots,  x_1 + \delta^{-1}x_{|I_1|,1},
\ldots,
x_r, x_r + \delta^{-1}x_{2,r} ,\ldots, x_r + \delta^{-1}x_{|I_r|,r}).
\end{multline}
This map changes orientations by $\sum_{j=1}^r (r-j) (|I_j| - 1);$ in
case of non-trivial weightings, $|I_j|$ should be replaced by the
number of incoming markings or non-trivial weightings on the $j$-th
component.

The combinatorial type of a quilted disk is the tree obtained as in
the unquilted case.  The resulting tree $\Gamma$ has vertices
$\Ver(\Gamma)$ equipped with a distinguished subset $\Ver^1(\Gamma)$
of colored vertices $v$ corresponding to quilted components
$S_v \subset C$.  Morphisms of quilted trees (Cutting infinite length
edges, collapsing edges, making lengths/weights finite/non-zero, and
forgetting tails) induce morphisms of moduli spaces of stable quilted
treed disks as in the unquilted case. In the special case of cutting
an edge $e$ of infinite length $\ell(e) = \infty$, one of the pieces
$\Gamma_1$ will be a (possibly disconnected) quilted type and the
other $\Gamma_2$ an unquilted type.  The uncolored vertices
$v \notin \Ver^1(\Gamma)$ admit a $\{ 0, 1 \}$-labelling
$\Ver(\Gamma) - \Ver^1(\Gamma) \to \{ 0 , 1 \}$ with value $0$
resp. $1$ if component is further away from the root resp. closer to
the root edge $e \subset \Edge_{\rightarrow}(\Gamma)$ than the quilted
components with respect to any non-self-crossing path of components.
For any combinatorial type $\Gamma$ of quilted disk there is a {\em
  universal quilted treed disk} $\ol{\U}_\Gamma \to \ol{\M}_\Gamma$
which is a cell complex whose fiber over $C$ is isomorphic to $C$, and
splits into surface and tree parts
$ \ol{\U}_\Gamma = \ol{\S}_\Gamma \cup \ol{\T}_{\white,\Gamma} \cup
\ol{\T}_{\black,\Gamma}$,
where the last two sets are the boundary and interior parts of the
tree respectively.  Labels and weights on the semi-infinite ends
$e \in \Edge_{\rightarrow}(\Gamma)$ of types $\Gamma$ of quilted tree
disks are defined as in the case of treed disks.  We suppose there is
a partition of the boundary semi-infinite edges
\[\Edge^{\greyt}(\Gamma) \sqcup \Edge^{\whitet}(\Gamma) \sqcup \Edge^{\blackt}(\Gamma) =
\Edge_{\white,\rightarrow}(\Gamma) \]
into {\em weighted} resp. {\em forgettable} resp.  {\em unforgettable}
edges as in the unquilted case, except that now the root of the
quilted tree with one weighted leaf and no marking is weighted with
the same weight as the leaf, see Figure \ref{trivq}.  The moduli space
with a single weighted leaf
$e \in \Edge_{\rightarrow}(\Gamma), \rho(e) \in (0, \infty)$ and no
markings is then a point.

\begin{figure} 
\begin{picture}(0,0)%
\includegraphics{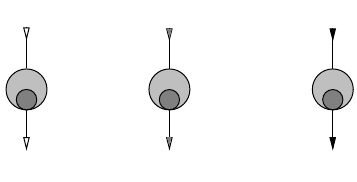}%
\end{picture}%
\setlength{\unitlength}{4144sp}%
\begingroup\makeatletter\ifx\SetFigFont\undefined%
\gdef\SetFigFont#1#2#3#4#5{%
  \reset@font\fontsize{#1}{#2pt}%
  \fontfamily{#3}\fontseries{#4}\fontshape{#5}%
  \selectfont}%
\fi\endgroup%
\begin{picture}(2700,1371)(4216,-2041)
\put(4231,-1986){\makebox(0,0)[lb]{\smash{{\SetFigFont{8}{9.6}{\rmdefault}{\mddefault}{\updefault}{\color[rgb]{0,0,0}$x^{\whitet}$}%
}}}}
\put(5352,-1986){\makebox(0,0)[lb]{\smash{{\SetFigFont{8}{9.6}{\rmdefault}{\mddefault}{\updefault}{\color[rgb]{0,0,0}$x^{\greyt}$}%
}}}}
\put(6597,-1986){\makebox(0,0)[lb]{\smash{{\SetFigFont{8}{9.6}{\rmdefault}{\mddefault}{\updefault}{\color[rgb]{0,0,0}$x^{\blackt}$}%
}}}}
\put(4259,-785){\makebox(0,0)[lb]{\smash{{\SetFigFont{8}{9.6}{\rmdefault}{\mddefault}{\updefault}{\color[rgb]{0,0,0}$x^{\whitet}$}%
}}}}
\put(5417,-799){\makebox(0,0)[lb]{\smash{{\SetFigFont{8}{9.6}{\rmdefault}{\mddefault}{\updefault}{\color[rgb]{0,0,0}$x^{\greyt}$}%
}}}}
\put(6662,-799){\makebox(0,0)[lb]{\smash{{\SetFigFont{8}{9.6}{\rmdefault}{\mddefault}{\updefault}{\color[rgb]{0,0,0}$x^{\blackt}$}%
}}}}
\put(5515,-1082){\makebox(0,0)[lb]{\smash{{\SetFigFont{8}{9.6}{\rmdefault}{\mddefault}{\updefault}{\color[rgb]{0,0,0}$\rho$}%
}}}}
\put(5523,-1653){\makebox(0,0)[lb]{\smash{{\SetFigFont{8}{9.6}{\rmdefault}{\mddefault}{\updefault}{\color[rgb]{0,0,0}$\rho$}%
}}}}
\end{picture}%
\caption{Unmarked stable treed quilted disks}
\label{trivq}
\end{figure} 

\section{Quilted pseudoholomorphic disks}  
\label{qdisks}

In the next few sections we prove invariance of the homotopy type of
the Fukaya algebra under a change of divisors of the same degree,
built from homotopic sections of the same line bundle; the general
case is treated in Chapter \ref{stabilize} below.  We introduce the
following notation:

\begin{enumerate} 
\item {\rm (Paths of divisors)} Suppose that $D^0,D^1 \subset X$ are
  stabilizing divisors for $L$ built from homotopic unitary sections
  of $\ti{X}$ of the same degree that are asymptotically holomorphic
  with respect to compatible almost complex structures
  $J^0, J^1 \in \J(X)$.  Choose a path $J^t \in \J(X)$ from $J^0$ to
  $J^1$.  By Lemma \ref{largelem} above, there exists a path of
  $J^t$-stabilizing divisors $D^t \subset X, t \in [0,1]$ connecting
  $D^0,D^1$, and a path $J_{D^t} \in \J(X,D^t)$ of compatible almost
  complex structures such that $D^t$ contains no $J_{D^t}$-holomorphic
  spheres.
\item {\rm (Distance-to-seam function)} In order to specify which
  divisor of the above family to use at a given point of a quilted
  domain, define a function as follows.  For every point $z \in C$ of
  a quilted treed disk $C$, let
  \[d(z) := \pm \sum_{e \in P(z)} \ell(e) \in [-\infty,\infty]\]
  be the distance of $z$ to the quilted components of $C$, that is,
  the sum of the lengths of the edges $e$ in the path $P(z)$ between
  $z$ and the quilted components times $1$ resp. $-1$ if $z$ is above
  resp. below the quilted components (that is, further from
  resp. closer to the root than the quilted components).
\item {\rm (Perturbation morphisms)} Given perturbation data
  $\ul{P}^0$ and $\ul{P}^1$ with respect to metrics
  $G^0,G^1 \in \G(L)$ over unquilted treed disks for $D^0$
  resp. $D^1$, a {\em perturbation morphism} $\ul{P}^{01}$ from
  $\ul{P}^0$ to $\ul{P}^1$ for the quilted combinatorial type $\Gamma$
  consists of
\begin{enumerate} 
\item a smooth function $\delta^{01}_\Gamma: [-\infty,\infty] \to
  [0,1]$ (to be composed with $d$)
\item a smooth domain-dependent choice of metric
\[ G_\Gamma^{01}: \ol{\T}_{\white,\Gamma} \to \R \]
constant to $G^0$ resp. $G^1$ on the neighborhood of the points at
infinity of the semi-infinite edges 
\[ \ol{{\T}}_{\white,\Gamma} - \ol{{\T}}^{\thick}_{\white,\Gamma} \]
for which $d=-\infty$ resp. $d=\infty$,
\item a domain-dependent Morse function
\[ F_\Gamma^{01}: \ol{\T}_{\white,\Gamma} \to \R \]
constant to $F^0$ resp. $F^1$ on the neighborhood
$\ol{{\T}}_{\white,\Gamma} - \ol{{\T}}^{\thick}_{\white,\Gamma}$ of
the endpoints for which $d=-\infty$ resp. $d=\infty$ and equal to
$F_{\Gamma_0}^0$ resp. $F_{\Gamma_1}^1$ on the (unquilted) treed disks
components of type $\Gamma_0$, $\Gamma_1$ for which $d=-\infty$
resp. $d=\infty$, and

\item a domain-dependent almost complex structure
\[ J_\Gamma^{01}: \ol{\S}_\Gamma \to \J_\tau(X) \]
given as the product of maps pulled back from 
\[ J_{\Gamma(v)}^{01} : \ \ol{{\S}}_{\Gamma(v)} \times X \to \End(TX) \]
with the property that for any surface component $C_i$ of $C$,
$J_\Gamma^{01}$ is
\begin{enumerate}
\item  equal to the given $J$ away from the compact part:
\[ J_\Gamma | \ol{{\S}}_\Gamma - \ol{{\S}}^{\thick}_\Gamma = \pi_2^*
J \]
where $\pi_2$ is the projection on the second factor in \eqref{FGam}
and 
\item equal to the complex structures $J_{\Gamma_0}^0$
resp. $J_{\Gamma_1}^1$ on the (unquilted) treed disks components of
type $\Gamma_0$, $\Gamma_1$ for which $d=-\infty$ resp. $d=\infty$:
Let $\iota_k: \ol{\S}_{\Gamma_k} \to \ol{\S}_\Gamma$ denote the
inclusion of the unquilted components. Then we require
\[ J_\Gamma | \iota_k: \ol{{\S}}_{\Gamma_k} = J_{\Gamma_k}^k , \quad k
\in \{ 0 ,1 \} .\]
\end{enumerate}
\end{enumerate} 
\item \label{qindep} {\rm (Quilt-independence)} If the perturbations
  $\ul{P}^ = (P^k_\Gamma)$ are equal one can also require the
  following independence property: A perturbation system for quilts
  $\ul{P}^{01} = (P^{01}_\Gamma = (F_\Gamma^{01}, J^{01}_\Gamma,
  G^{01}_\Gamma))$
  is {\em quilt-independent} if \label{inva_axio} $G_\Gamma^{01}$,
  $F_\Gamma^{01}$, and $J_\Gamma^{01}$ are pull-backs under the
  forgetful morphism forgetting the quilting $Q_v \subset S_v$ on the
  quilted disk components $S_v, v \in \Ver^1(\Gamma)$.
\end{enumerate}

Pseudoholomorphic quilted disks with respect to domain-dependent
structures are defined as follows.  Let $C$ be a stable quilted treed
disk of type $\Gamma$.  By the stability condition $C$ may be
identified with a fiber $\pi^{-1}[C]$ of the universal curve
$\pi: \U_\Gamma \to \M_\Gamma$.  By pullback we obtain a perturbation
on $C$, still denoted
$(\delta^{01}_\Gamma, J^{01}_\Gamma, F^{01}_\Gamma,G^{01}_\Gamma)$.
If $C = S \cup T$ is obtained from a stable quilted treed disk $C'$ by
attaching an infinite segment $e \subset T$ at any of the
semi-infinite ends of $C'$; choose perturbation data on $C$ such that
the Morse function $F^{01}_\Gamma$ on $C$ is constant on the infinite
segment $e$.  A {\em pseudoholomorphic quilted treed disk}
$u: C \to X$ of combinatorial type $\Gamma$ is a continuous map $u$
from a quilted treed disk $C$ that is smooth on each component
$S_v ,T_e \subset C$, $J_\Gamma^{01}$-holomorphic on the surface
components $S_v, v \in \Ver(\Gamma)$, a $-\grad(F_\Gamma^{01})$-Morse
trajectory with respect to the metric $G_\Gamma^{01}$ on each boundary
tree segment $T_e \subset C, e \in \Edge_{\white}(\Gamma)$, and
constant on the tree segments of sphere type
$T_e \subset C, e \in \Edge_{\black}(\Gamma)$.  The interior leaves of
the tree are irrelevant for our purposes; but they do affect the
combinatorics of the boundary of the moduli spaces because of the
balanced condition, see Figure \ref{qspheres}.  A stable holomorphic
quilted tree disk is {\em adapted} if and only if (Stable surface
axiom) the surface part is stable and (Leaf axiom) each interior leaf
$T_e$ maps to $D^{\delta^{01}_\Gamma \circ d(z_i)}$ and for each
$t \in [0,1]$, each component of
$u^{-1}(D^t) \cap (\delta_\Gamma^{01})^{-1}(t)$ contains an interior
leaf $e$.  Two adapted weighted disks
$u_0 : C_0 \to X, u_1: C_1 \to X$ are {\em isomorphic} if there exists
an isomorphism of weighted disks $\phi: C_0 \to C_1$ so that
$u_0 \circ \phi = u_1 .$ Note that if $C_0,C_1$ have a single unmarked
quilted disk component
$S_{v,k}, v \in \Ver^1(\Gamma), k \in \{ 0, 1 \}$ and a single leaf
$e_{1,k} \in \Edge_{\rightarrow}(\Gamma)$, the weightings
$\rho(e_{1,0}), \rho(e_{1,1})$ are not required to be equal.  Given a
non-constant pseudoholomorphic quilted treed disk $u: C \to X$ with
leaf $e_i \in \Edge^{\greyt}(\Gamma)$ on which there is a weighting
$\rho(e_i) = 0$ resp. $\infty$, we declare $u:C \to X$ to be {\em
  equivalent} to the pseudoholomorphic quilted treed disk
$u' : C' \to X$ obtained by replacing the label $x^{\greyt}$ at $e_i$
by $x^{\blackt}$ resp. $x^{\whitet}$ and adding a segment
$C' = C \cup [-\infty,\infty)$ so that $u$ is constant on
$[-\infty,\infty)$, and so represents a trajectory from $x^{\greyt}$
to $x^{\blackt}$ resp. $x^{\whitet}$ above that leaf as in the
unquilted case in Figure \ref{triv0}.  In other words, forgetting
constant infinite segments and replacing them by the appropriate
weights gives equivalent pseudoholomorphic treed disks. For any
combinatorial type $\Gamma$ of quilted disks we denote by
$\ol{\M}_\Gamma(L,D)$ the compactified moduli space of equivalence
classes of adapted quilted pseudoholomorphic treed disks.

The moduli space of adapted pseudoholomorphic quilted disks breaks
into components depending on the limits along the root and leaf edges.
Denote by $\M_\Gamma(L,D,\ul{x}) \subset \ol{\M}_{\Gamma}(L,D)$ the
moduli space of isomorphism classes of adapted pseudoholomorphic
quilted treed disks with boundary in $L$ and limits $\ul{x}$ along the
root and leaf edges, where $\ul{x} = (x_0,\ldots, x_n) \in \cI(L)$
satisfies the requirement:
\begin{enumerate} 
\item {\rm (Label axiom)} 
\begin{enumerate} 
\item {\rm (Incoming unit)}  If $x_0 = x^{\greyt}$ resp.  $x_0 = x^{\whitet}$,
  then there is a single leaf reaching $x^{\greyt}$ resp. $x^{\whitet}$
  and no interior marking (in which case the moduli space will be a
  point). 
\item {\rm (Outgoing unit)} If $x_i = x^{\greyt}$ resp. $x^{\whitet}$
  resp. $x^{\blackt}$ for some $i \geq 1$ then the $i$-th leaf is
  required to be weighted resp.  forgettable resp. unforgettable and
  the limit along this leaf is required to be $x_M$.
\item {\rm (Non-units)} If
  $x_i \notin \{ x^{\greyt}, x^{\whitet}, x^{\blackt} \} $, then
  the $i$-th leaf is required to be unforgettable.
\end{enumerate}
\item {\rm (Outgoing edge axiom)} The outgoing edge $e_0$ is weighted
  (resp. forgettable) only if there is a single leaf $e_1$,
  which is weighted (resp. forgettable) with the same weight
  $\rho(e_0) = \rho(e_1)$ and the domain $C$ has no interior leaves
  (so there is a single quilted disk $S \subset C$ with no interior
  leaves.)
\end{enumerate}

In order to obtain moduli spaces with the expected boundary, we
introduce a coherence condition. A collection
$\ul{P}^{01} = (P_\Gamma^{01})$ of perturbation morphisms is {\em
  coherent} if $P_\Gamma^{01}$ is compatible with the morphisms of
moduli spaces as before:
\begin{enumerate} 
\item {\rm (Collapsing edges/making an edge or weight finite or non-zero)}
  If $\Gamma'$ is obtained from $\Gamma$ by collapsing an edge or
  edges, then $P_{\Gamma}^{01}$ is the pullback of $P_{\Gamma'}^{01}$.

\item {\rm (Cutting edges) } If $\Gamma'$ is obtained from $\Gamma$ by
  cutting an edge or a collection of edges of infinite length, then
  $P_\Gamma^{01}$ is the pushforward of $P_{\Gamma'}^{01}$.
Suppose that $\Gamma'$ is the union of a quilted type
  $\Gamma_1$ and a non-quilted type $\Gamma_0$, then $P_{\Gamma'}^{01}$
  is obtained from $P_{\Gamma_1}^{01}$ and $P_{\Gamma_0}^0$ as
  follows: Let 
$\pi_k: \ol{\M}_{\Gamma'} \cong \ol{\M}_{\Gamma_1} \times
  \ol{\M}_{\Gamma_0} \to \ol{\M}_{\Gamma_k}$
  denote the projection on the $k$-factor, so that 
$\ol{\U}_{\Gamma'} \cong \pi_1^* \ol{\U}_{\Gamma_1} \cup \pi_0^*
\ol{\U}_{\Gamma_0} .$
Then we require that $P_{\Gamma'}$ is equal to the pullback of
$P_{\Gamma_1}^{01}$ on $\pi_1^* \ol{\U}_{\Gamma_1}$ and to the
pullback of $P_{\Gamma_0}^{0}$ on $\pi_0^* \ol{\U}_{\Gamma_0}$.

  Similarly, if $\Gamma'$ is the union of a non-quilted type
  $\Gamma_1$ and quilted types $\cup_{i} \Gamma_{0,i}$ and quilted
  types with interior markings, then $P_{\Gamma'}^{01}$ is equal to
  the pullback of $P_{\Gamma_1}^{1}$ on $\pi_1^* \ol{\U}_{\Gamma_1}$
  and to the pullback of $P_{\Gamma_{0,i}}^{01}$ on
  $\pi_{0,i}^* \ol{\U}_{\Gamma_{0,i}}$.

  The case of constant quilted types requires special treatment.  If
  any of the types $\Gamma_{0,i}$ have no interior markings and a
  single weighted leaf
  $e \in \Edge(\Gamma), \rho(e) \in (0,\infty)$ then we label the
  corresponding leaf of $\Gamma_1$ with the same weight, by
  our (Cutting edges) construction.  This guarantees that the moduli
  spaces are of expected dimension.

\item {\rm (Locality axiom)} For any spherical vertex $v$, $P_\Gamma$
  restricts to the pull-back of a perturbation $P_{\Gamma,v}$ on the
  image of $\pi^* \U_{\Gamma(v)}$ in $\U_\Gamma$.  (Note that the
  perturbation $P_{\Gamma,v}$ is allowed to depend on $\Gamma$, not
  just $\Gamma(v)$.)  Furthermore, the restriction of $P_\Gamma$ to
  the disk components and boundary edges in $\U_\Gamma$ is the
  pull-back of a perturbation datam $P_{\Gamma_\circ}$ from
  $\U_{\Gamma_\circ}$.

\item {\rm (Forgettable edges)} Whenever some weight parameter
  $\rho(e)$ attached to a leaf $e \in \Edge_{\rightarrow}(\Gamma)$ is
  equal to infinity, then the $P_\Gamma^{01}$ is pulled back under the
  forgetful map $f: \U_\Gamma \to \U_{f(\Gamma)}$ forgetting the first
  leaf $e_i \in \Edge_{\rightarrow}^{\whitet}(\Gamma)$ and stabilizing
  from the perturbation morphism $P_{f(\Gamma)}^{01}$ given by
  (Forgetting tails).

\end{enumerate} 

The moduli space of pseudoholomorphic adapted quilted treed disks has
compactness and transversality properties similar to those for
unquilted disks.  A perturbation morphism
$\ul{P}^{01} = (P_\Gamma^{01})$ is {\em stabilizing} if it satisfies a
condition analogous to that in Definition \eqref{stabilized}: the
divisors contain no non-constant pseudoholomorphic spheres and the
intersection of any pseudoholomorphic sphere with the divisor contains
at least three points.  For a comeager subset of perturbation
morphisms $\ul{P}^{01}$ extending those chosen for unquilted disks,
the uncrowded moduli spaces $\M_\Gamma(L,D)$ of expected dimension at
most one are smooth and of expected dimension.  For sequential
compactness, it suffices to consider a sequence $u_\nu: C_\nu \to X$
of quilted treed disks of fixed combinatorial type $\Gamma_\nu$
constant in $\nu$.  Coherence of the perturbation morphism implies the
existence of a stable limit $u: C \to X$ which we claim is adapted.
In particular, the (Leaf Property) is justified as follows.  For each
component $C_i \subset C$, the almost complex structure
$J_\Gamma | C_i$ is constant near the almost complex submanifold
$D_{\delta_\Gamma^{01}\circ d(C_i)}$.  Suppose that $C_i$ is a
component of the limit of some sequence of components $C_{i,\nu}$ of
$C_\nu$.  Coherence for the parameter $\delta_{01}^\nu$ implies that
$D_{\delta_\Gamma^{01}\circ d(C_i)}$ is the limit of the divisors
$D_{\delta_{\Gamma_\nu}^{01} \circ d(C_{i,\nu})}$.  Then local
conservation of intersection degree\footnote{Since
  $\delta_{01}^\Gamma$ is constant on each disk or sphere the union
\[ D^{\delta_{01}^\Gamma} = \cup_{z \in \S_\Gamma} \left( \{z \} \times 
D^{\delta_{01}^\Gamma \circ d(z)} \right) \]
is an almost complex submanifold of $\S_\Gamma \times X$.  In
particular, the intersection multiplicity of $u :C \to X$ with
$D^{\delta_{01}^\Gamma}$ at $z_i$ is positive. } implies that any
component of $u^{-1}(D_{\delta_\Gamma^{01}\circ d(C_i)})$ contains a
limit point of some markings
$z_{i,\nu} \in u^{-1}_\nu(D_{\delta_{\Gamma_\nu}^{01} \circ
  d(C_{i,\nu})})$.
For types of index at most one, each component of
$u^{-1}(D_{\delta_\Gamma^{01}\circ d(C_i)})$ is a limit of a unique
component of
$u^{-1}_\nu( D_{\delta_{\Gamma_\nu}^{01} \circ d(C_{i,\nu})})$,
otherwise the intersection degree would be more than one which is a
codimension two condition.  Since non-trivial sphere bubbling is a
codimension two condition and ghost bubbling is impossible unless two
markings come together, this implies that
$u^{-1}(D_{\delta_\Gamma^{01}\circ d(C_i)}) = \{ z_i \}$ is also a
marking. As before, we call a perturbation morphism $\ul{P}^{01}$
admissible if it is coherent, regular, and stabilizing.

The condition \eqref{inva_axio} implies the following properties of
the moduli spaces of constant quilted disks.  By the argument in the
proof of Proposition \ref{equalpert}, in the $0$-dimensional strata
all of the quilted disks are mapped to points.  The $1$-dimensional
strata will be of two types: First, there are $1$-dimensional families
of quilted trajectories where every quilted disk is constant. Second,
there may be $1$-dimensional families for which only the quilting
parameter on a single non-constant quilted disk varies from $-\infty$
to $\infty$.

\section{Morphisms of Fukaya algebras}

Morphisms of Fukaya algebras are defined using the moduli spaces of
quilted disks in the previous section. Given an 
admissible perturbation morphism $\ul{P}^{01}$ from $\ul{P}^0$ to
$\ul{P}^1$, define for each $n \ge 0$ 
\begin{multline} \label{phin}
 \phi^n: CF(L;\ul{P}^0)^{\otimes n} \to CF(L;
 \ul{P}^1) \\ (x_1,\ldots,x_n) \mapsto \sum_{x_0,u \in
   {\M}_\Gamma(L,D,x_0,\ldots,x_n)_0} (-1)^{\heartsuit} \eps(u) (
 \sigma(u)!)^{-1} q^{ E(u)} y(u) x_0
\end{multline}
where the sum is over quilted disks in strata of dimension zero with
$x_1,\ldots,x_n$ incoming labels.

\begin{remark} {\rm (Lowest energy terms)} For $x \in \cI(L)$, the
  element $\phi^1(x^{\greyt})$ resp. $\phi^1(x^{\blackt})$
  resp. $\phi^1(x^{\whitet})$ has a term containing $x^{\greyt}$
  resp. $x^{\blackt}$ resp. $x^{\whitet}$ coming from the count of
  quilted treed disks $u: C\to X$ with no interior leaves.  Such a
  configuration consists of a quilted treed disk $C = S \cup T$ with
  only one disk component $S \cong B$ that is quilted and mapped to a
  point $u(S) = \{ x_M \}$.
\end{remark}

\begin{remark} \label{qtypes} The codimension one strata
  $\M_\Gamma(L,D)$ in the moduli spaces of pseudoholomorphic quilted
  treed disks $\ol{M}(L,D)$ are of several possible types: either
  there is one (or a collection of) edge $e \subset T$ of length
  $\ell(e)$ infinity, there is one (or a collection of) edge $e$ of
  length $\ell(e)$ zero, or equivalently, boundary nodes, or there is
  an edge $e$ with zero or infinite weight $\rho(e)$.  The case of an
  edge of zero or infinite weighting $\rho(e)$ is equivalent to
  breaking off a constant trajectory, and so may be ignored.  In the
  case of edges of infinite length(s) $l(e)$, then either $\Gamma$ is
\begin{enumerate} 
\item {\rm (Breaking off an uncolored tree)} a pair $\Gamma_1 \sqcup
  \Gamma_2$ consisting of a colored tree $\Gamma_1$ and an uncolored
  tree $\Gamma_2$ as in Figure \ref{bunquilt}; necessarily the breaking must be a leaf of
  $\Gamma_1$; or
\begin{figure}[ht]
\includegraphics[height=1.5in]{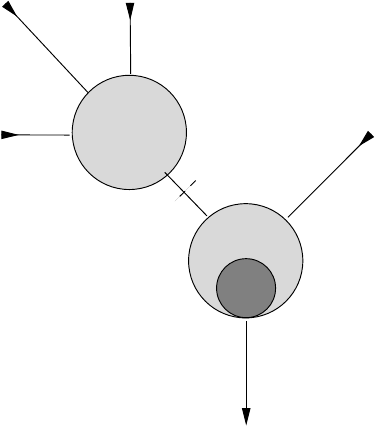}
\caption{Breaking off an unquilted treed disk} 
\label{bunquilt}
\end{figure}
\item {\rm (Breaking off colored trees)} a collection consisting of an
  uncolored tree $\Gamma_0$ containing the root $e_0$ and a collection
  $\Gamma_1,\ldots,\Gamma_r$ of colored trees attached to each of its
  $r$ leaves as in Figure \ref{bcoll}.  Such a stratum $\M_\Gamma$ is
  codimension one because of the (Balanced Condition) which implies
  that if the length $\ell(e)$ of any edge $e$ between $e_0$ to $e_i$
  is infinite for some $i$ then the path from $e_0$ to $e_i$ for any
  $i$ has the same property.

\begin{figure}[ht]
\includegraphics[height=1.5in]{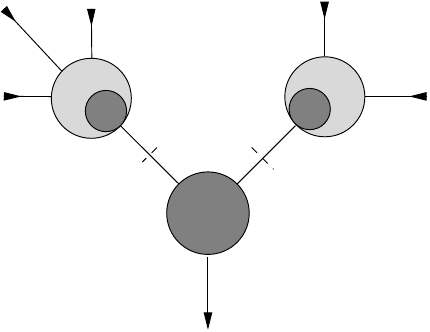}
\caption{Breaking off a collection of quilted disks}
\label{bcoll} 
\end{figure}

\end{enumerate} 
\noindent In the case of a zero length(s), one obtains a fake boundary component
with normal bundle $\R$, corresponding to either deforming the edge(s)
$e$ to have non-zero length $\ell(e)$ or deforming the node(s).  This
ends the Remark. 
\end{remark} 

\begin{theorem} \label{samedegree} {\rm (\ainfty morphisms via
    quilted disks)} For any admissible collection $\ul{P}^{01}$ of
  perturbations morphisms from $\ul{P}^0$ to $\ul{P}^1$, the
  collection of maps $\phi = (\phi^n)_{n \ge 0}$ constructed in
  \eqref{phin} is a convergent unital \ainfty morphism from
  $CF(L,\ul{P}^0)$ to $CF(L,\ul{P}^1)$.
\end{theorem} 

\begin{proof} The statement is an algebraic consequence of the
  description of the boundary of $\ol{\M}_{n,1}(L,D)_1$ in Remark
  \ref{qtypes}.  By counting the ends of the one-dimensional moduli
  spaces we obtain the relation \eqref{ident} but with the union over
  type replaced by the set of types of quilted treed disks with $n$
  leaves and $m$ interior markings.  The true boundary strata
  $\M_\Gamma(L,D)$ are those described in Remark \ref{qtypes} and
  correspond to the terms in the axiom for \ainfty morphisms
  \eqref{faxiom}.  We refer to \cite{ainfty} for the sign computation
  in a slightly different, but equivalent, context.

  The assertion on the strict units is a consequence of the existence
  of forgetful maps for infinite values of the weights.  By assumption
  the $\phi^n$ products of type $\Gamma$ that involve $x^{\whitet}$ as inputs involve
  counts of quilted treed disks $u: C \to X$ using perturbations
  $P_\Gamma$ that are pulled back under the forgetful morphisms
  $\ol{\U}_\Gamma \to \ol{\U}_{f(\Gamma)}$ forgetting the first leaf
  $e_i$ labelled with the strict unit $x^{\whitet}$.  Since forgetting
  that semi-infinite edge gives a configuration $u':C' \to X$ of
  negative expected dimension, if non-constant, the only
  configurations contributing to these terms must be the constant maps
  $u: C \to X$ with a single quilted disk $S$ on which $u$ is constant
  and a single leaf $T_{e_1}$ mapping to $x_M$ and root edge $T_{e_0}$
  mapping to $x_M$ as well.  Hence
  $ \phi^1(x^{\whitet}) = x^{\whitet}$ and
  $\phi^n( \ldots, x^{\whitet}, \dots) = 0, n \ge 2 .$
 \end{proof}

 In the following we show that in the case that the incoming and
 outgoing perturbations are the same, the corresponding morphism may
 be taken to be the identity.  This fact is used for the proof of
 invariance of the homotopy type of the Fukaya algebra in Corollary
 \ref{hequiv} below.

\begin{proposition} \label{equalpert} Suppose that $\ul{P}^0 =
  \ul{P}^1$ is an admissible perturbation datum for treed disks.  For
  each type $\Gamma$ of quilted treed disk, let $\Gamma'$ denote the
  corresponding type of unquilted treed disk obtained by forgetting
  the quilting and collapsing unstable components.  Pulling back
  $P_{\Gamma',0} = P_{\Gamma',1}$ to a perturbation morphism
  $P^{01}_\Gamma$ for quilted treed disks gives an admissible
  perturbation morphism for quilted disks such that the corresponding
  \ainfty morphism is the identity.
\end{proposition} 

\begin{proof}  
  The claim follows by choosing quilt-independent perturbations
  $\ul{P} = (P_\Gamma)$ as in \eqref{qindep} on page \pageref{qindep}.
  Given a non-constant quilted treed disk $u: C \to X$ contributing to
  $\phi^n$ in the moduli space $\M_\Gamma(L,D)_0$ of expected dimension
  $0$, one obtains an unquilted treed disk $u' : C' \to X$ in a
  stratum $\M_{\Gamma'}(L,D)_{-1}$ where $\Gamma'$ is the type obtained
  from $\Gamma$ by forgetting to quilting, which is a
  contradiction. Therefore, the only configurations contribution to
  $\phi^n$ are the constant configurations.  Hence
  $\phi^1 = \on{Id}: CF(L) \to CF(L)$ is the identity and all other
  maps $\phi^n: CF(L)^{\otimes n} \to CF(L), n > 0$ vanish.
\end{proof} 

\section{Homotopies} 
\label{homsec}

The morphism of Fukaya algebras constructed in Theorem
\ref{samedegree} is a homotopy equivalence by an argument using {\em
  twice-quilted disks}.  A twice-quilted disk $(C,Q_1,Q_2)$ is defined
in the same way as once-quilted disks, but with two interior circles
$Q_1,Q_2 \subset C$ that are either equal $Q_1 = Q_2$ or with the
second $Q_2 \subset \on{int}(Q_1)$ contained inside the first, say
with radii $\rho_1 < \rho_2$.  The moduli space of twice-quilted treed
disks is a cell complex constructed in a similar way to the space of
once-quilted treed disks.  Denote the moduli space with $n$ leaves and
two quiltings by $\ol{\M}_{n,2}$.  As before the moduli space admits a
stratification by combinatorial type.  The combinatorial type of a
twice-quilted disk $(C,Q_1,Q_2)$ with boundary markings is a tree
\[\Gamma =
(\Ver(\Gamma), \Edge(\Gamma), (h \times t): \Edge(\Gamma) \to
\Ver(\Gamma) \cup \{ \infty \} )\] 
equipped with subsets
\[\Ver^1(\Gamma), \Ver^2(\Gamma) \subset \Ver(\Gamma)\] 
corresponding to the quilted components; the set
\[\Ver^{12}(\Gamma) := \Ver^1(\Gamma) \cap
\Ver^2(\Gamma)\] 
corresponds to the twice-quilted components. 
The ratios 
\[\lambda_S (v) = \rho_2(v)/\rho_1(v), \quad v \in
\Ver^{12}(\Gamma)\]
of the radii of the interior circles with radii $\rho_2(v),\rho_1(v),
v \in \Ver^{12}(\Gamma)$ are required to be equal for each
twice-quilted disk in the configuration, if the configuration has
twice-quilted components:
\begin{equation} \label{ratioeq} \lambda_S(v_1) = \lambda_S({v_2}), \quad \forall v_1,v_2 \in
  \Ver^{12}(\Gamma) .\end{equation}
The stratification of the moduli space of twice-quilted disks
$\ol{\M}_{n,2}$ by type is a cell decomposition with cells in
bijection with certain expressions involving formal functions $f,g$
and inputs $x_1,\ldots, x_n$.  For example, in Figure \ref{twicecol}
the moduli space of twice-quilted stable disks $\ol{\M}_{2,2}$ is
shown without trees; it is a pentagon whose vertices correspond to the
expressions
$f(g(x_1x_2))$, $f(g(x_1)g(x_2))$, $f(g(x_1))f(g(x_2))$, $((fg)(x_1))
((fg)(x_2))$, $(fg)(x_1x_2) .$
\begin{figure}[ht]
\includegraphics[height=2.75in]{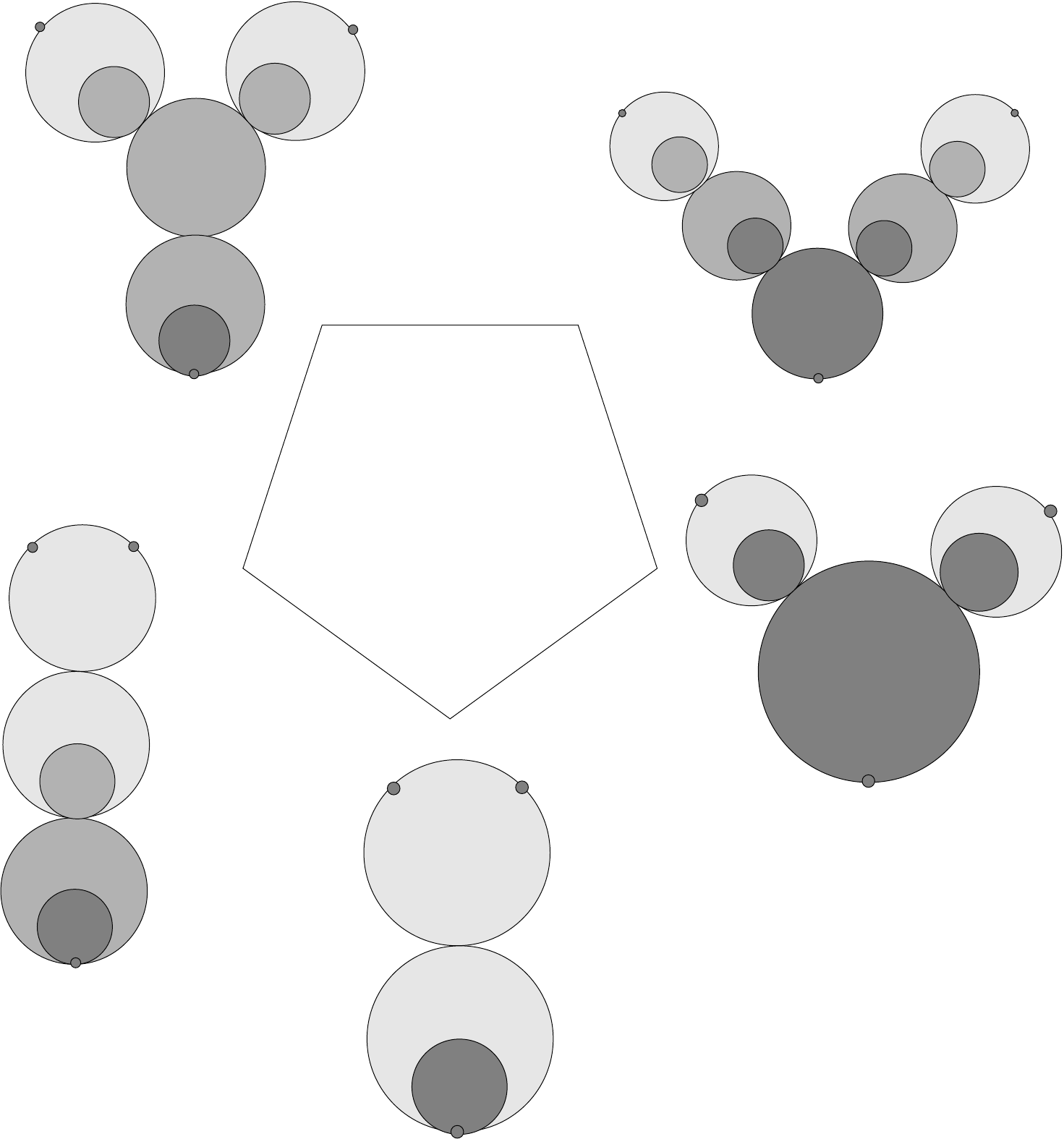}
\caption{Twice-quilted disks}
\label{twicecol} 
\end{figure} 
\noindent One can also consider a moduli space of twice-quilted disks with
interior markings; these are required to lie on components $v \in
\Ver(\Gamma)$ which are at least as far away from the root edge as the
vertices $v \in \Ver^2(\Gamma)$.  That is, if $v \in \Ver(\Gamma)$
lies on a non-self-crossing path from the root edge to an element in
$\Ver^2(\Gamma)$ then $v \in \Ver^2(\Gamma)$.  Once one allows
interior markings, these can bubble off onto {\em twice-quilted
  spheres}, which are marked spheres $(C,(z_0,\ldots,z_n))$ equipped
with {\em two} isomorphisms $C - \{ z_0 \} \to \C$.

As in the quilted case, there is a moduli space of {\em treed}
twice-quilted spheres which assigns lengths to the interior and
boundary nodes.  The lengths of paths between colored vertices satisfy
the balanced condition \eqref{balanced} for each color:
 For any two vertices of the same color $v_1,v_2 \in \Ver^{k}(\Gamma)$, 
\begin{equation} \label{balanced2}
\sum_{e \in P_+(v_1,v_2)} \ell(e) = 
\sum_{e \in P_-(v_1,v_2)} \ell(e) 
 \end{equation}
where $P(v_1,v_2)$ is the (finite length) oriented non-self-crossing
path from $v_1$ to $v_2$ and $P_+(v_1,v_2)$ resp. $P_-(v_1,v_2)$ is
the part of the path pointing towards resp. away from the root edge.
In particular this implies that for two vertices of different colors $v_1\in \Ver^{1}(\Gamma)$
and $v_2 \in \Ver^{2}(\Gamma)$ for which there is a (finite length) oriented 
non-self-crossing path $P(v_1,v_2)$ from $v_1$ to $v_2$, let
\begin{equation} \label{lambdaT}
\lambda_T(v_1,v_2) = \sum_{e \in P(v_1,v_2)} \ell(e).
\end{equation}
Then $\lambda_T(v_1,v_2)$ is independent of the choice
$v_1 \in \Ver^{1}(\Gamma)$ and $v_2 \in \Ver^{2}(\Gamma)$.  

Suppose that divisors and perturbations for unquilted and once-quilted
disks have already been chosen.  That is, there are given compatible
almost complex structures $J_0,J_1,J_2$, metrics $G_0,G_1,G_2$,
divisors $D_0,D_1,D_2$ and perturbation systems
$\ul{P}^0,\ul{P}^1,\ul{P}^2$ for unquilted disks.  Furthermore, there
are given paths $J_{01}^t,J_{12}^t, J_{02}^t$ of compatible almost
complex structures from $J_0$ to $J_1$, $J_1$ to $J_2$ and $J_0$ to
$J_2$, paths $D_{01}^t$ from $D_0$ to $D_1$, and $D_{12}^t$ from $D_1$
to $D_2$, and $D_{02}^t$ from $D_0$ to $D_2$.  (In our application we
are particularly interested in the case $D_0 = D_2$ and the constant
path $D_{02}^t = D_0 = D_2$.) Suppose there are given perturbation
data $\ul{P}^{ij}$ for once-quilted disks giving rise to morphisms
\[\phi_{ij}: CF(L,\ul{P}^i) \to CF(L,\ul{P}^j),
\quad 0 \leq i < j \leq 2 .\]
We have in mind especially the case that $D_0 = D_2$, $D_{01}^t =
D_{12}^{1-t}$, and $D_{02}^t$ is the constant path.  In this case one
may take $\ul{P}^{12,t} = \ul{P}^{01,1-t}$ and $\ul{P}^{02}$ the
perturbation system pulled back by the forgetful map forgetting the
quilting as in Proposition \ref{equalpert}.

To construct the moduli spaces of twice-quilted treed disks we extend
the families of stabilizing divisors over the universal twice-quilted
treed disks.  Choose a {\em domain-dependent parameter} for twice
quilted disks: a smooth map
\[\delta^{012}: \triangle \equiv \{(t_1,t_2) \in
[-\infty,\infty]^2 \ | \ t_2 \leq t_1\} \to [0,2] .\]
For every point $z \in C$ of a twice quilted treed disk $C$, let
$d(z)=(t_1,t_2) \in \triangle$ with 
$t_1$ being the signed distance of $z$ to the lowest quilted
components $S_v, v \in \Ver^1(\Gamma)$ of $C$ and $t_2$ being the
signed distance of $z$ to the highest quilted components $S_v, v \in
\Ver^2(\Gamma)$  of $C$.

\begin{definition}\label{doub_pert_morp}
A perturbation $P^{012}_\Gamma$ for twice-quilted treed disks from
quilted perturbation systems $\ul{P}^{01} \times \ul{P}^{12}$ to
$\ul{P}^{02}$ consists of
 \begin{enumerate} 
     \item a domain-dependent parameter $\delta^{012}_\Gamma$ that
       agrees 
\begin{itemize}
\item 
with $\delta^{01}_{\Gamma_{01}}$ on $[-\infty,\infty] \times
  \{-\infty\}$,
\item with $\delta^{12}_{\Gamma_{12}}$ on $\{-\infty\} \times
  [-\infty,\infty]$ and
\item with $\delta^{02}_{\Gamma_{02}}$ on $\{ (t_1,t_2) \in
  [-\infty,\infty]^2 | t_1=t_2 \}$, 
\end{itemize} 
where $\Gamma_{ij}, 0 \leq i \leq j \leq 2$ is the corresponding type
of once-quilted disk;

\item a smooth family of metrics $G^{012}_\Gamma$ constant equal to
  $G^0$ resp. $G^1$ resp. $G^2$ on a neighborhood of the endpoints for
  which $d=(-\infty,-\infty)$ resp. $d=(\infty,-\infty)$
  resp. $d=(\infty,\infty)$ and that agrees 
\begin{itemize} 
\item with $G^{01}_{\Gamma_{01}}$ on $[-\infty,\infty] \times
  \{-\infty\}$, 
\item with $G^{12}_{\Gamma_{12}}$ on $\{+\infty\} \times
             [-\infty,\infty]$ and
\item with $G^{02}_{\Gamma_{02}}$ on $\{ (t_1,t_2) \in
  [-\infty,\infty]^2 | t_1=t_2 \}$, 
\end{itemize} 
where $\Gamma_{ij}, 0 \leq i \leq j \leq 2$ is the corresponding type
of once-quilted disk;

\item a domain-dependent Morse function $F^{012}_\Gamma$ equal to
  $F^0$ resp. $F^1$ resp. $F^2$ in a neighborhood of the endpoints at
  $d=(-\infty,-\infty)$ resp. $d=(\infty,-\infty)$
  resp. $d=(\infty,\infty)$ and that agrees with
  $F^{01}_{\Gamma_{01}}$ resp. $F^{12}_{\Gamma_{12}}$ resp.
  $F^{02}_{\Gamma_{02}}$ on the once quilted \label{qtreep}
  \llabel{qtree} disk components of type $\Gamma_{01}$
  resp. $\Gamma_{12}$ resp. $\Gamma_{02}$ containing the root
  resp. the leaves resp. where the quilting radii coincide,

\item a domain-dependent almost-complex structure $J^{012}_\Gamma$ 
    such that for every surface component $C_i$,  
    equal to $J_{\delta^{012}_\Gamma \circ d(z)}$ on 
    $D^{\delta^{012}_\Gamma \circ d(z)}$, in a neighborhood of the spherical
    nodes, the interior markings and on the boundary of $C_i$, 
that is 
\begin{itemize} 
\item equal to $J^0_{\Gamma_0}$, resp. $J^1_{\Gamma_1}$
  resp. $J^2_{\Gamma_2}$ on the unquilted components of type
  $\Gamma_0$ resp. $\Gamma_1$ resp. $\Gamma_2$ at
  $d=(-\infty,-\infty)$ resp. $d=(\infty,-\infty)$
  resp. $d=(\infty,\infty)$ and 
\item agrees with $J^{01}_{\Gamma_{01}}$ resp. $J^{12}_{\Gamma_{12}}$
  resp.  $J^{02}_{\Gamma_{02}}$ on the once quilted treed disk
  components of type $\Gamma_{01}$ resp. $\Gamma_{12}$
  resp. $\Gamma_{02}$ containing the root resp. the leaves resp. where
  the quilting radii coincide.
\end{itemize} 
\item In cases that the perturbations being connected are identical,
  one can also require the following invariance property: A
  perturbation datum is {\em quilting-independent}
  if \label{doub_inva_axio} $G^{012}_\Gamma$, $F_\Gamma^{012}$, and
  $J_\Gamma^{012}$ are pull-backs under the forgetful morphism
  forgetting the quiltings on each once-quilted or twice-quilted disk.
\end{enumerate}
\end{definition}

The perturbations above allow the definition of pseudoholomorphic
twice-quilted treed disks.  Given a treed twice-quilted treed disk $C$
of type $\Gamma$, one obtains perturbation data for $C$ by pull-back
from the stabilization $C^{\on{st}}$ of $C$, which may be identified
with a fiber of the universal twice-quilted disk
$\U_\Gamma \to \M_\Gamma$.  A {\em pseudoholomorphic twice-quilted
  treed disk} is a twice-quilted disk $C$ together with a map
$u:C \to X$ that is $J_\Gamma^{012}$-holomorphic on surface
components, a $F_\Gamma^{012}$-Morse trajectory on boundary tree
segments with respect to the metrics
$G^{01}_\Gamma, G^{12}_\Gamma, G^{02}_\Gamma$, and constant on the
interior part of the tree.  Stable and adapted twice-quilted treed
disks are defined as in the once-quilted case.  In particular, each
interior marking $z_i$ maps to the divisor
$D^{\delta^{012}_\Gamma(z_i)}$.  Suppose that the perturbations
$\ul{P} = (P_\Gamma^{012})$ satisfy coherence and stabilized
conditions similar to those for quilted disks.  The moduli spaces
$\ol{\M}_\Gamma(L,D)$ of adapted stable twice-quilted disks are then
compact for each uncrowded combinatorial type $\Gamma$ of expected
dimension $\dim(\M_\Gamma(L,D))$ at most one.  The property
\eqref{doub_inva_axio} of Definition \ref{doub_pert_morp} ensures that
the latter zero dimensional spaces
$\M_\Gamma(L,D), \dim(\M_\Gamma(L,D)) = 0$ will not contain
non-constant (either once or twice) quilted disks.  The one
dimensional strata $\M_\Gamma(L,D)$ may involve at most one
non-constant once-quilted disk $S_v \subset S$ and if so,
$\M_\Gamma(L,D)$ is a constant family over which the latter quilting
radius varies freely.

In order to obtain transversality the fiber products involved in the
definition of the universal twice-quilted disks in \eqref{ratioeq}
must be perturbed, using {\em delay functions}, as in Seidel
\cite{se:bo} and Ma'u-Wehrheim-Woodward \cite{ainfty}.
\llabel{ainftyref} \label{ainftyrefp}  We define a map incorporating
both the condition on ratios and distances between quilted components
as follows.  Identify $\M_{1,0,2}$ with $[0,\infty]$ as in Figure
\ref{oneleaf}.  For $n \ge 1, m \ge 0$ denote by
\[ \lambda: \M_{n,m,2} \to \M_{1,0,2} \cong [0, \infty ] \]
the forgetful morphism forgetting all but the first marking; note that
on the interval consisting of only once-quilted disks, $\lambda$ is
essentially equivalent to the map $\lambda_T$ of \eqref{lambdaT} while
on the interval with twice-quilted disks $\lambda$ is given by the map
$\lambda_S$ of \eqref{ratioeq}.  Note that $\lambda$ is also defined
in the case $n = 0$, by combining the maps \eqref{lambdaT} and
\eqref{ratioeq}.
\begin{figure}[ht]
\begin{picture}(0,0)%
\includegraphics{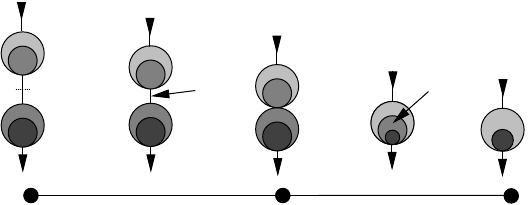}%
\end{picture}%
\setlength{\unitlength}{4144sp}%
\begingroup\makeatletter\ifx\SetFigFont\undefined%
\gdef\SetFigFont#1#2#3#4#5{%
  \reset@font\fontsize{#1}{#2pt}%
  \fontfamily{#3}\fontseries{#4}\fontshape{#5}%
  \selectfont}%
\fi\endgroup%
\begin{picture}(4001,1552)(-16950,-2025)
\put(-15423,-1220){\rotatebox{360.0}{\makebox(0,0)[lb]{\smash{{{$\lambda_T$}%
}}}}}
\put(-13644,-1173){\rotatebox{360.0}{\makebox(0,0)[lb]{\smash{{{$\lambda_S$}%
}}}}}
\end{picture}%
\caption{Moduli of treed twice-quilted disks with one leaf}
\label{oneleaf}
\end{figure} 
\noindent  Let
$\Gamma$ be a combinatorial type of twice-quilted disks.  Define
$\ol{\M}^{\pre}_\Gamma$ as the product of moduli spaces for the
vertices,
\[ \ol{\M}_\Gamma^{\pre} = \prod_{v \in \Ver(\Gamma)} \ol{\M}_v .\]
Let $k$ denote the number of twice-quilted vertices and 
\[ \lambda_\Gamma: \ol{\M}_\Gamma^{\pre} \to \R^{k}, \quad (r_v) \mapsto
\prod_{v \in \Ver^{(12)}(\Gamma)} \lambda(r_v) \]
the map combining the forgetful maps for the twice-quilted components.
Then $ \ol{\M}_\Gamma = \lambda_\Gamma^{-1}(\Delta) $ where
$\Delta \subset \R^{k}$ is the diagonal.  A {\em delay function} for
$\Gamma$ is a collection of smooth functions depending on
$r \in \ol{\M}^{\pre}_\Gamma$
\[ \tau_{\Gamma} = \left( \tau_e \in C^\infty(\ol{\M}^{\pre}_\Gamma)
\right)_{e\in \Edge(\Gamma)}.\]
Letting $\lambda_i := \lambda(r_{v_i})$ where $\lambda(r_{v_i})$ is the 
ratio of the radii circles for $r_{v_i}$, the {\em delayed evaluation
  map} is
\begin{eqnarray}
\lambda_{\tau_\Gamma} : \prod\limits_{v\in \Ver(\Gamma)} \ol{\M}_{v} &
\to & \R^k \label{delayeval} \\ (r_v, u_v)_{v\in \Ver{\Gamma} } &
\mapsto & \left( \lambda_i \exp\left( \sum_{e\in p_i} \tau_e(r)\right)
\right)_{i=1,\ldots,k} \nonumber
\end{eqnarray}
where the sum is the sum of delays along each path $p_i$ to a
twice-quilted disk component.  Call $\tau_\Gamma$ {\em regular} if the
delayed evaluation map $\lambda_{\tau_\Gamma}$ is transverse to the
diagonal $\Delta \subset \R^k$.  Given a regular delay function
$\tau_\Gamma$, we define
\begin{equation} \label{delayedfiber} 
\ol{\M}_\Gamma := \lambda_{\tau_\Gamma}^{-1}(\Delta) . \end{equation}
For a regular delay function $\tau_\Gamma$, the delayed fiber product
has the structure of a smooth manifold, of local dimension
\begin{equation} \label{localdim}
 \dim \M_\Gamma = 1-k + \sum_{v\in \Ver \Gamma} \dim \M_v   \end{equation}
where $k$ is the number of twice-quilted disk components.  A
collection $\{\tau^d\}_{d\geq 1}$ of delay functions is {\em
  compatible} if the following properties hold.  Let $\Gamma$ be a
combinatorial type of twice-quilted disk and $v_0,\ldots,v_k$ the
vertices corresponding to disk components.
\begin{enumerate}

\item {\rm (Subtree property)} Suppose that the root component $v_0$
  is not a twice-quilted disk.  Let $\Gamma_1, \ldots, \Gamma_{|v_0|-1}$
  denote the subtrees of $\Gamma$ attached to $v_0$ at its leaves; then $\Gamma_1, \ldots, \Gamma_{|v_0|-1}$ are combinatorial
  types for nodal twice-quilted disks.  Let $r_i$ be the component of $r
  \in \M^{\pre}_\Gamma$ corresponding to $\Gamma_i$.  We require that
  $ \tau_{\Gamma}(r) \big\lvert_{\Gamma_i} = \tau_{\Gamma_i}(r_i)$,
  i.e., for each edge $e$ of $\Gamma_i$, the delay function
  $\tau_{\Gamma,e}(r)$ is equal to $\tau_{\Gamma_i,e}(r_i)$.  See
  Figure \ref{subtree}.  

\begin{figure}[ht]
\begin{picture}(0,0)%
\includegraphics{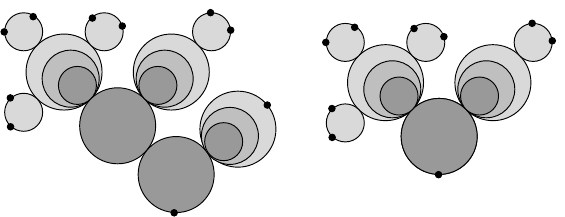}%
\end{picture}%
\setlength{\unitlength}{3947sp}%
\begingroup\makeatletter\ifx\SetFigFont\undefined%
\gdef\SetFigFont#1#2#3#4#5{%
  \reset@font\fontsize{#1}{#2pt}%
  \fontfamily{#3}\fontseries{#4}\fontshape{#5}%
  \selectfont}%
\fi\endgroup%
\begin{picture}(4453,1667)(472,-3792)
\put(3595,-2967){\makebox(0,0)[lb]{\smash{{{$\tau_{\Gamma',e}$}%
}}}}
\put(1096,-2856){\makebox(0,0)[lb]{\smash{{{$\tau_{\Gamma,e}$}%
}}}}
\end{picture}%

\caption{The (Subtree Property)} 
\label{subtree}
\end{figure}

\item {\rm (Refinement property)} Suppose that the combinatorial type
  $\Gamma^\prime$ is a refinement of $\Gamma$, so that there is
  a surjective morphism $f: \Gamma^\prime \to \Gamma$ of trees; let
  $r$ be the image of $r'$ under gluing.
We require that $\tau_\Gamma\big\lvert_U$ is determined by
$\tau_{\Gamma^\prime}$ as follows: for each $e \in \Edge(\Gamma)$, and
$r \in U$, the delay function is given by the formula
\[ \tau_{\Gamma,e}(r) = \tau_{\Gamma',e} + \sum_{e'}
\tau_{\Gamma',e'}(r')\]
where the sum is over edges $e'$ in $\Gamma'$ that are collapsed under
gluing and $e$ is the next-furthest-away edge from the root
vertex.  See Figure \ref{refinement1}.
\begin{figure}[ht]
\begin{picture}(0,0)%
\includegraphics{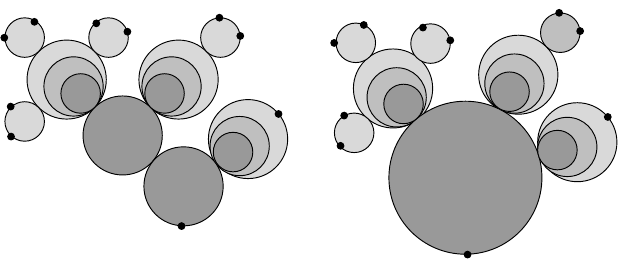}%
\end{picture}%
\setlength{\unitlength}{3947sp}%
\begingroup\makeatletter\ifx\SetFigFont\undefined%
\gdef\SetFigFont#1#2#3#4#5{%
  \reset@font\fontsize{#1}{#2pt}%
  \fontfamily{#3}\fontseries{#4}\fontshape{#5}%
  \selectfont}%
\fi\endgroup%
\begin{picture}(4944,2004)(472,-4074)
\put(1361,-3350){\makebox(0,0)[lb]{\smash{{{$\tau_{\Gamma,e_1}$}%
}}}}
\put(999,-2875){\makebox(0,0)[lb]{\smash{{{$\tau_{\Gamma,e_2}$}%
}}}}
\put(1597,-2967){\makebox(0,0)[lb]{\smash{{{$\tau_{\Gamma,e_3}$}%
}}}}
\put(3580,-2988){\makebox(0,0)[lb]{\smash{{{$\tau_{\Gamma,e_2}+\tau_{\Gamma,e_1}$}%
}}}}
\put(4297,-2842){\makebox(0,0)[lb]{\smash{{{$\tau_{\Gamma,e_3}+\tau_{\Gamma,e_1}$}%
}}}}
\end{picture}%
\caption{The (Refinement property), first case} 
\label{refinement1}
\end{figure}
\noindent In the case that the collapsed edges connect twice quilted components
with unquilted components, this means that the delay functions are
equal for both types, as in the Figure \ref{casetwo}.

\item {\rm (Core property)} Given a combinatorial type $\Gamma$, we
  say that the {\em core} $\Gamma_0$ of $\Gamma$ is the combinatorial
  type obtained by removing all vertices above the colored vertices.
  If two combinatorial types, say $\Gamma$ and $\Gamma^\prime$, have
  the same core $\Gamma_0$, let $r,r'$ be disks of type $\Gamma$
  resp. $\Gamma^\prime$ and $r_0,r'_0$ the disks of type $\Gamma_0$
  obtained by removing the components except for those corresponding
  to vertices of $\Gamma_0$.  If $r_0 = r'_0$ then
  $\tau_{\Gamma,e}(r) = \tau_{\Gamma^\prime,e}(r')$.  (That is, the
  delay functions depend only on the region between the root vertex
  and the bicolored vertices.)
\end{enumerate}
A collection of compatible delay functions is {\em positive} if, for
every vertex $v \in \Gamma_0$ with $k$ leaves labeled in
counterclockwise order by $e_1, \ldots, e_{k}$, the associated delay
functions satisfy 
\begin{equation} \label{poscond} \tau_{e_1} < \tau_{e_2} < \ldots <
  \tau_{e_{k}} .\end{equation} 

\begin{figure}[ht]
\begin{picture}(0,0)%
\includegraphics{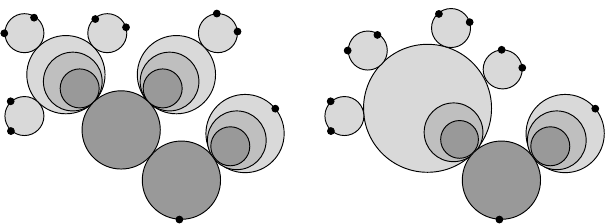}%
\end{picture}%
\setlength{\unitlength}{3947sp}%
\begingroup\makeatletter\ifx\SetFigFont\undefined%
\gdef\SetFigFont#1#2#3#4#5{%
  \reset@font\fontsize{#1}{#2pt}%
  \fontfamily{#3}\fontseries{#4}\fontshape{#5}%
  \selectfont}%
\fi\endgroup%
\begin{picture}(4843,1730)(472,-3837)
\put(1110,-2865){\makebox(0,0)[lb]{\smash{{{$\tau_{\Gamma',e_2}$}%
}}}}
\put(1668,-2950){\makebox(0,0)[lb]{\smash{{{$\tau_{\Gamma',e_3}$}%
}}}}
\put(1350,-3346){\makebox(0,0)[lb]{\smash{{{$\tau_{\Gamma',e_1}$}%
}}}}
\put(3769,-3351){\makebox(0,0)[lb]{\smash{{{$\tau_{\Gamma',e_1}$}%
}}}}
\end{picture}%

\caption{The (Refinement property), second case} 
\label{casetwo} 
\end{figure}

For each type there exist regular, positive delay functions compatible
with a collection of given choices on the boundary.  For any type
$\Gamma$ the (Subtree property) implies that all the delay functions
in $\tau_\Gamma := \tau_\Gamma(L)$ {\em except} those for the finite
edges adjacent to $v_0$, the root component, are already fixed.  We
wish to construct regular delay functions for the finite edges
$e \in \Edge(\Gamma)$ adjacent to the root vertex
$v_0 \in \Ver(\Gamma)$ in a way that is also compatible with the
(Refinement property).  We ensure compatibility by proceeding
inductively on the number of edges connecting to the root vertex $v_0$
as follows.  Suppose that regular delay functions are given for types
$\Gamma$ corresponding to strata of $\M_{n',m,2}$ for $n' < n$, as
well as for types $\Gamma'$ appearing in the (Refinement property) for
$\Gamma$.  We construct inductively on $d \ge 2$ a collection of
regular delay functions $\tau_\Gamma := \tau_\Gamma(L)$ for all
combinatorial types $\Gamma$ whose root vertex has $2 \leq d \leq n$
incoming edges.  We may assume that $\Gamma$ has no vertices ``beyond
the twice-colored vertices'' since by the (Core property) the delay
functions $\tau_\Gamma$ are independent of the additional vertices.
Consider a combinatorial type $\Gamma$ whose root vertex has $d$
leaves.  There is an open neighborhood $U$ of
$\partial \M^{\pre}_\Gamma $ in $\ol{\M}^{\pre}_\Gamma$ in which the
delay functions $\tau_\Gamma$ for the leaves $e$ adjacent to the root
vertex $v \in \Ver(\Gamma)$ are already determined by the
compatibility condition (Refinement property) and the inductive
hypothesis.  We need to show that we may extend the $\tau_i$ over the
interior of $\ol{\M}^{\pre}_\Gamma$.  To set up the relevant function
spaces let $l\ge 0$ be an integer and let $f$ be a given $C^l$
function on $\ol{U}$.  Let $C_{f}^l(\M^{\pre}_\Gamma)$ denote the
Banach manifold of functions with $l$ bounded derivatives on
$\M^{\pre}_\Gamma$, equal to $f$ on $\ol{U}$.  Let
$\Gamma_i, i =1,\ldots, d$ be the trees attached to the root vertex
$v_0$.  Consider the evaluation map
\bea & & \on{ev}: \M_{{\Gamma_1}}\times \ldots \times
\M_{{\Gamma_{d}}} \times \M_{v_0} \times \prod_{i=1}^{d} C^l_{
  \tau_i}(\M^{\pre}_\Gamma) \to \R^{d-1}\\ & & \left((r_1, u_1),
\ldots, (r_n, u_d), (r_0,u_0), \tau_1, \ldots, \tau_d\right) \mapsto
\\ & & \left( \lambda_{\Gamma_j}(r_j) \exp( \tau_j(r)) -
\lambda_{\Gamma_{j+1}}(r_{j+1}) \exp( \tau_{j+1}(r)) \right)_{j=1}^{d-1}
\eea
where $ r = (r_0,\ldots, r_d)$.  Note that $0$ is a regular value.
The Sard-Smale theorem implies that for $l$ sufficiently large the
regular values of the projection
\bea \Pi : \on{ev}^{-1}(0) \to \prod_{i=1}^{d} C^l_{
  \tau_i}(\M^{\pre}_\Gamma) \eea
form an open dense set.  Taking the intersection over $l$ sufficiently
large gives that the set of smooth regular delay functions is
comeager.  Both the positivity condition \eqref{poscond} and the
regularity condition are open conditions given an energy bound
$E(u) <E$.  Therefore the set of smooth, positive, compatible, delay
functions that are regular for a given energy bound is non-empty and
open.  Taking the intersection of these sets over all possible energy
bounds $E \in \Z$ we obtain a comeager set of delay functions for
which all moduli spaces $\M_\Gamma(L,D)$ are regular.  By induction,
there exists a smooth, positive, compatible, regular delay function
$\tau_\Gamma$.

Given a collection $\ul{\tau} = (\tau_\Gamma)$ of regular compatible
delay functions and a collection of perturbation data
$\ul{P}^0,\ul{P}^1,\ul{P}^2, \ul{P}^{01}, \ul{P}^{12}$ and
$\ul{P}^{02}$ for unquilted and once-quilted disks, perturbations
$\ul{P}^{012} = (P^{012}_\Gamma)$ for twice-quilted disks are
constructed inductively using the gluing construction and the
Sard-Smale theorem.  For each stratum $\M_\Gamma(L,D)$, we first use
the construction of the previous paragraph to find regular delay
functions for the boundary strata $\M_{\Gamma'}(L,D)$, then (after
replacing the fiber products in the boundary strata by the {\em
  delayed} fiber products of boundary strata using the delay functions
chosen) extend the perturbations by the gluing construction.  A
Sard-Smale argument shows that for perturbations in a comeager subset,
the uncrowded moduli spaces of expected dimension at most one are
regular of expected dimension.

\begin{theorem} \label{twicequilted} {\rm (\ainfty homotopies via
    twice-quilted disks)} Given admissible perturbation systems
  $\ul{P}^{01},\ul{P}^{12},\ul{P}^{02}$ defining morphisms
\[\phi_{ij}:
CF(L,\ul{P}^i) \to CF(L,\ul{P}^j), \quad 0 \leq i
< j \leq 2 \] 
and an admissible perturbation system $\ul{P}^{012}$ for
twice-quilted disks, counting treed pseudoholomorphic twice-quilted disks
defines a convergent \ainfty homotopy between $\phi_{02}$ and
$\phi_{01} \circ \phi_{12}$.
\end{theorem} 

\begin{proof} The homotopy is defined by combining homotopies
  constructed from small variations of ratio.  Consider the map
  $ \psi: \ol{\M}_{n,m,2} \to [1,\infty] $ giving the ratio
  $ \rho_1/\rho_2$ of radii of the inner circles of the twice-quilted
  disks.  For generic values of $\lambda$ the moduli space
  $\ol{\M}_{n,m,2}^\lambda(L,D)= \psi^{-1}(\lambda)$ is smooth of
  expected dimension.  We let $\phi_{02}^\lambda$ denote the \ainfty
  morphism obtained by counting twice-quilted disks with ratio of
  radii $\lambda \in [1,\infty)$.  After fixing an energy bound $E$
  and a number of leaves $n$ we may divide the interval $[0,1]$ into
  subintervals $[\lambda_i,\lambda_{i+1}], i = 0,\ldots, k-1$ so that
  each is sufficiently small so that there is a single singular value
  $\lambda \in [\lambda_i,\lambda_{i+1}]$, contained in the interior
  of the interval, for which there exist twice-quilted disks of
  expected dimension zero of energy at most $E$ and $n$ leaves and no
  such disks with fewer number of leaves.  Define
  $\cT_{02}^{\lambda_1,\lambda_{i+1},\leq E}$ by counting such
  twice-quilted disks,
\begin{multline} \label{twiceq}
 (\cT_{02}^{\lambda,\leq E})^n: CF(L;\ul{P}^0)^{\otimes n}
  \to CF(L; \ul{P}^2) \\ (x_1,\ldots,x_n) \mapsto
  \sum_{x_0,u \in\ol{\M}_\Gamma^\lambda(L,D;x_0,\ldots,x_n)_0}
  (-1)^{\heartsuit} \eps(u) (\sigma(u)!)^{-1} q^{E(u)} y(u)
  x_0 \end{multline} 
where the sum is over combinatorial types $\Gamma$ of twice-quilted
disks.  The difference
\[ (\phi_{02}^{\lambda_i} -
\phi_{02}^{\lambda_{i+1}})^n(x_1,\ldots,x_n) \]
is a count of configurations $u: C \to X$ either involving an
unquilted disk breaking off $C_1 \subset C$, or a collection of
twice-quilted treed disks $C_1,\ldots, C_r \subset C$ with
$i_1,\ldots,i_r$ leaves breaking off from an unquilted treed disk, see
\cite[Section 7]{ainfty}. For degree reasons, because of the fiber
product with the diagonal exactly one of these twice-quilted disks
lies in the moduli space of expected dimension zero, while the rest
have index one.  We suppose that the twice-quilted configuration in
the expected-dimension-zero moduli space is the $i+1$-st twice-quilted
treed disk attached to the unquilted treed disk.  Using positivity of
the delay functions one obtains that the moduli space of twice quilted
disks $\lambda_i$ and $\lambda_{i+1}$ and expected dimension zero are
cobordant:
\[ \begin{array}{llll}
 \ol{\M}_{i_j,m_j,2}^{\lambda + \tau_j}(L,D)_0 &\sim&
\ol{\M}_{i_j,m_j,2}^{\lambda_{i+1} + \tau_j}(L,D)_0 & j > i+2 \\ 
\ol{\M}_{i_j,m_j,2}^{\lambda + \tau_j}(L,D)_0 &\sim&
\ol{\M}_{i_j,m_j,2}^{\lambda_i + \tau_j}(L,D)_0 & j \leq i
                                                  .\end{array} \]
It follows that
\[ \phi_{02}^{\lambda_i} - \phi_{02}^{\lambda_{i+1}} =
  \mu^1_{\phi_{02}^{[\lambda_i,\lambda_{i+1}]}}(\cT_{02}^{\lambda_i,\lambda_{i+1},\leq E}) .\]
  The facets of $\ol{\M}_{n,m,2}$ with ratio $\lambda = 1$ or
  $\lambda = \infty$ correspond to either to terms in the definition
  of composition of \ainfty maps
  $\phi_{12} \circ \phi_{01}: CF(L,\ul{P}^0) \to CF(L,\ul{P}^2)$, to
  the components contributing to
  $\phi_{02} : CF(L,\ul{P}^0) \to CF(L,\ul{P}^2)$ or to terms
  corresponding to the bubbling off of some markings on the boundary.
  Composition using \eqref{composehom} produces a homotopy
\[ \TT_{02}^{\leq E} := \mu^2 ( \TT_{02}^{\lambda_{k-1},\lambda_k,\leq E},
\mu^2 ( \ldots \mu^2(\TT_{02}^{\lambda_{1},\lambda_{2}},
\TT_{02}^{\lambda_{0},\lambda_1, \leq E} ) ))\]
between $\phi_{12} \circ \phi_{01}$ and the identity modulo terms
involving powers $q^E$.  Taking the limit $E \to \infty$ defines a
homotopy
\[\TT_{02} := \lim_{E \to \infty} \TT_{02}^{\leq E} \] 
between $\phi_{02}$ and $\phi_{12} \circ \phi_{01}$.  Convergence,
that is, that $\T^0_{02}(1)$ has coefficients with positive
$q$-valuation, holds since any contributing configuration must contain
a non-trivial disk.
\end{proof} 

\begin{corollary}\label{hequiv}  For any two admissible convergent collections of perturbation data
  $\ul{P}^0, \ul{P}^1$ the Fukaya algebras $CF(L,\ul{P}^0)$
  and $CF(L,\ul{P}^1)$ are homotopy equivalent via strictly
  unital convergent maps and homotopies in the sense of Lemma
  \ref{uptogauge}.
\end{corollary} 

\begin{proof} Using Theorem \ref{twicequilted} and taking $D_{02}^t$
  to be constant and $\ul{P}^{02}$ to be pulled back by the map
  forgetting the quilting, one obtains a homotopy between the
  composition of morphisms
  $ \phi_{01}: CF(L,\ul{P}^0) \to CF(L,\ul{P}^1)$,
  $ \phi_{10}: CF(L,\ul{P}^1) \to CF(L,\ul{P}^0) $
  and the identity morphism as in Proposition \ref{equalpert}.  Since
  the perturbation data $P_\Gamma$ for a type $\Gamma$ with an
  infinite weight $\rho(e)$ is a pull-back under the morphism
  forgetting the first infinite weight, the maps $\phi_{01}^k$
  involving an identity insertion all vanish except $\phi_{01}^1$, for
  which $\phi_{01}^1(e_0) = e_1$.  Hence $\phi_{01}$ is strictly
  unital.  \end{proof}

\section{Stabilization}
\label{stabilize}

In this section we complete the proof of homotopy invariance of the
Fukaya algebras constructed above in the case that the algebras are
defined using divisors are not of the same degree or built from
homotopic sections.  For this we need to recall some results about
existence of a Donaldson hypersurface transverse to a given one.
Recall from \cite[Lemma 8.3]{cm:trans} that for a constant $\eps > 0$,
two divisors $D,D'$ intersect {\em $\eps$-transversally} if at each
intersection point $x \in D \cap D'$ their tangent spaces
$T_x D, T_x D'$ intersect with angle at least $\eps$.  A result of
Cieliebak-Mohnke \cite[Theorem 8.1]{cm:trans} states that there exists
an $\epsilon>0$ such that given a divisor $D$, there exists a divisor
$D'$ of sufficiently high degree $\eps$-transverse to $D$.  Moreover,
for any $\theta>0$, $\omega$-tamed almost complex structures
$\theta$-close to $J$ making $D,D'$ almost complex exist (provided
that the degrees are sufficiently large).

We apply the result of the previous paragraph as follows.  Suppose
that $D^0,D^1$ are stabilizing divisors for $L$, possibly of different
degrees.  By the previous paragraph, there exists a pair $D^{0,'}$
$D^{1,'}$ of higher degree stabilizing divisors built from homotopic
unitary sections over $L$ that are $\epsilon$-transverse to $D^0$ and
$D^1$, respectively.  Let $\ul{P}_0', \ul{P}_1'$ be perturbation
systems for $D^{0,'}, D^{1,'}$.  We have already shown that
\begin{equation} \label{shown} 
 CF(L,\ul{P}_0') \cong
 CF(L,\ul{P}_1') \end{equation}
are convergent, strictly-unital homotopy equivalent.  It remains to show:

\begin{theorem} \label{stab} For any admissible perturbation systems
  $\ul{P}_k,\ul{P}_k', k = 0,1$ as above, the Fukaya algebras
  $CF(L,\ul{P}_k)$ and $CF(L,\ul{P}_k')$ are convergent,
  strictly-unital homotopy equivalent.
\end{theorem} 

\begin{proof}[Sketch of proof]  
Denote the subset  of almost complex structures close to $J$ 
preserving $TD^{k,'}$ as well by 
\[ \J^*(X,D^k \cup D^{k,'}; J_{D^0},\theta,E) \subset
\J^*(X,D^k,J,\theta,E) .\] 
By \cite[Corollary 8.20]{cm:trans},
there exists a path-connected, open, dense set in $\J^*(X,D^k \cup
D^{k,'}; J_{D^k},\theta,E)$ with the property that for any $J \in
\J^*(X,D^k \cup D^{k,'}; J_{D^k},\theta,E)$, neither $D^k$ nor
$D^{k,'}$ contain any non-constant pseudoholomorphic spheres of energy at
most $E$, and each pseudoholomorphic sphere meets both $D^k$ and $D^{k,'}$
in at least three points.  Fix such an almost complex structure
$J_{D^k,D^{k'}}$, and an associated perturbation system $\ul{P}_k''$
using the divisor $D^k$ such that the almost complex structures
$J_{k,\Gamma}'' $ are equal to $J_{D^k,D^{k,'}}$ on $D^k \cup
D^{k,'}$.  The argument in the previous section (keeping the divisor
constant but changing the almost complex structures) shows that the
associated Fukaya algebras
$ CF(L, \ul{P}_k) \cong CF(L, \ul{P}_k'') $
are homotopy equivalent.  Similarly, choose a perturbation system
$\ul{P}_k'$ using the divisor $D^{k,'}$.  We claim that the Fukaya
algebras
$ CF(L, \ul{P}_k') \stackrel{?}{\cong}CF(L,
\ul{P}_k'') $
are homotopy equivalent.  To see this, we define adapted stable maps
adapted to the pair $(D^k, D^{k,'})$: a map is adapted if each
interior marking maps to either $D^k$ or $D^{k,'}$, and the first
$n_k$ markings map to $D^k$ and the last $n_k'$ markings map to
$D^{k,'}$.  A perturbation datum morphism $\ul{P} = (P_{\Gamma})$ is
coherent if it is compatible if with the morphisms of moduli spaces as
before: (Cutting edges) axiom, (Collapsing edges/Making an edge or
weight finite or non-zero) axiom, and satisfies the (Infinite weights)
axiom and (Product) axiom, where now on the unquilted components above
resp. below the quilted components the perturbation system is required
to depend only on the first $n_k$ resp. last $n_k'$ points mapping to
$D^k$ resp. $D^{k,'}$ (that is, pulled back under the forgetful map
forgetting the first $n_k$ resp. last $n_k'$ markings).  Then the same
arguments as before, together with the existence of the homotopy from
\cite{wx:partly}, produce the required homotopy equivalence.  Putting
everything together we have homotopy equivalences
$ 
 CF(L, \ul{P}_k) \cong CF(L, \ul{P}_k'') 
\cong  CF(L, \ul{P}_k')  .$
Applying \eqref{shown} completes the proof.
\end{proof}

\begin{corollary} \label{diffdegree} For any stabilizing divisors
  $D^0,D^1$ and any convergent admissible perturbation systems
  $\ul{P}_0,\ul{P}_1$, the Fukaya algebras $CF(L,\ul{P}_0)$
  and $CF(L,\ul{P}_1)$ are convergent-homotopy-equivalent.
\end{corollary} 

\begin{proof} Since homotopy equivalence of \ainfty algebras is an
  equivalence relation, see Definition \ref{morphisms}
  \eqref{compfun}, combining Theorems \ref{samedegree} and \ref{stab}
  gives the result.
\end{proof}

\chapter{Fukaya bimodules} 

Floer cohomology in the Morse model is invariant under Hamiltonian
perturbation, as in the various models described in Fukaya-Oh-Ohta-Ono
\cite{fooo}.  In this section we construct, by counting treed
Hamiltonian-perturbed pseudoholomorphic strips, a {\em Fukaya
  bimodule} for pairs of Lagrangians equipped with a Hamiltonian
perturbation such that the perturbed Lagrangians intersect cleanly.
The Fukaya bimodule is isomorphic to the bimodule for the Fukaya
algebra if the Lagrangians are identical.  The homotopy type of the
Fukaya bimodule is independent of the choice of Hamiltonian.  In
particular, if a Lagrangian brane is displaceable by a Hamiltonian
diffeomorphism then the Floer cohomology for any element of the
Maurer-Cartan moduli space vanishes.  This shows that the version of the 
Floer cohomology used in this paper is strong
enough to deduce standard conclusions about Hamiltonian displaceability.

\section{\ainfty bimodules} 

We introduce the following definitions on \ainfty bimodules,
with conventions similar to those of Seidel in \cite{seidel:sub}.  Let
$A_0,A_1$ be strictly unital \ainfty algebras graded \llabel{gradedby}
by $\Z_g$ where $g$ is an even integer.  An $\Z_g$-graded {\em \ainfty
  bimodule} is a $\Z_g$-graded vector space $M$ equipped with
operations
\[ \mu^{d|e} : A_0^{\otimes d} \otimes M \otimes A_1^{\otimes e} \to
M[1-d-e] \]
that satisfy the following relations among homogeneous elements $a_{i,k}, m$:
\begin{multline} \label{bimod}
 0 =  \sum_{i,k} (-1)^{\aleph} \mu^{d-i+1|e}( a^0_{1}, \lldots,
\mu_0^{i}(a^0_{k},\lldots, a^0_{k+i-1}), a^0_{k+i}, \lldots, a^0_{d},
m,a^1_{1},\lldots, a^1_{e})) \\ + \sum_{j,k} (-1)^{\aleph} \mu^{d|e-j+1}(a^0_{1},\lldots,
a^0_{d},m,a^1_{1},\lldots, \mu_1^{j}(a^1_{k},\lldots, a^1_{k+j-1}),\lldots,
a^1_{e})  \\ + \sum_{i,j} (-1)^{\aleph} \mu^{d-i|e-j}( a^0_{1},\lldots, a^0_{d-i},
\mu^{i|j}(a^0_{d-i+1},\lldots, a^0_{d},m,a^1_{1},\lldots,a^1_{j}),
\lldots, a^1_{e}) .\end{multline}
Here we follow Seidel's convention \cite{seidel:sub} of denoting by
$(-1)^{\aleph}$ the sum of the reduced degrees to the left of the
inner expression, except that $m$ has ordinary (unreduced) degree.  A
{\em morphism} $\phi$ of \ainfty-bimodules $M_0$ to $M_1$ of degree
$|\phi|$ is a collection of maps
\[ \phi^{d|e} : A_0^{\otimes d} \otimes M_0 \otimes A_1^{\otimes e}
\to M_1[|\phi|-d-e] \]
satisfying a splitting axiom
\begin{multline} \label{morphism} 
 \sum_{i,j} (-1)^{|\phi| \aleph}
\mu_1^{d|e}(a^0_{1},\lldots,a^0_{d-i}, \phi^{i|j}(a^0_{d-i+1},\lldots,
a^0_{d},m,a^1_{1},\lldots,a^1_{j}), a^1_{j+1},\lldots,a^1_{e}) \\ 
+ \sum_{i,j} (-1)^{|\phi| + 1 + \aleph} \phi^{d|e}(a^0_{1},\lldots,a^0_{d-i} ,
\mu_0^{i|j}(a^0_{d-i+1},\lldots,a^0_{d},m,a^1_{1},\lldots,a^1_{j} ),
a^1_{j+1},\lldots,a^1_{e}) \\
 + \sum_{i,j} (-1)^{|\phi | + 1 + \aleph}
\phi^{d|e}(a^0_{1}, \lldots, a^0_{j-1}, \mu_0^{i}(a^0_{j},\lldots, a^0_{j+i-1} ), \lldots,
a^0_{d},m,a^1_{1},\lldots,a^1_{e}) \\
+ \sum_{i,j} (-1)^{|\phi| + 1 + \aleph} \phi^{d|e}(a^0_{1},\lldots,
a^0_{e},m,a^1_{1},\lldots, \mu_1^{j}(a^1_{j},\lldots,a_{j+i-1}),
\lldots,a^1_{e}) = 0 .\end{multline}
Composition of bimodule morphisms
  $\phi: M_0 \to M_1, \psi: M_1 \to M_2$ is defined by
  \begin{multline} (\phi \circ \psi)^{d|e}(a^0_{1},\lldots, a^0_{d},m, a^1_{1},\lldots,a^1_{e})  \\
    = \sum_{i,j} (-1)^{|\psi| \aleph}
    \phi^{i|j}(a^0_{1},\lldots,a^0_{i},
    \psi_{d-i|e-j}(a^0_{i+1},\lldots, a^0_{d},m,
    a^1_{1},\lldots,a^1_{e-j}),a^1_{e-j+1}, \lldots,a^1_{e})
    .\end{multline}
  A {\em homotopy} of morphisms $\psi_0,\psi_1: M_0 \to M_1$ of degree
  zero is a collection of maps $(\phi^{d|e})_{d,e \ge 0}$ such that
  the difference $\psi_1 - \psi_0$ is given by the expression on the
  left hand side of \eqref{morphism}.  Each of these notions has an
  extension to the strictly unital case.  Suppose that $A_0,A_1$ are
  strictly unital with strict units $e_0,e_1$.  A bimodule $M$ is {\em
    strictly unital} $\mu^{1,0}(e_0,\cdot)$ and $\mu^{0,1}(\cdot,e_1)$
  are the identity, and all other operations involving $e_0$ and $e_1$
  vanish.  A morphism resp. homotopy is {\em strictly unital} if all
  operations involving the identities $e_0$ and $e_1$ vanish.
  As usual any \ainfty algebra $A$ is an \ainfty bimodule over itself.
  The bimodule operations are related to the structure maps by
\begin{multline} \label{change}
 \mu^{d|e}(a^0_{1},\lldots, a^0_{d},m,a^1_{1},\lldots,a^1_{e}) \\ =
(-1)^{1 + \diamondsuit } \mu^{d+ e + 1}(a^0_{1},\lldots,
a^0_{d},m,a^1_{1},\lldots,a^1_{e}), \quad \diamondsuit = \sum_{j=1}^d
(|a^0_{j}| + 1),
 \end{multline}
see Seidel \cite[2.9]{seidel:susp}.

For bimodules satisfying a convergence property the deformed
composition maps with infinitely many insertions are well-defined.
Given \ainfty algebras $A_0, A_1$, denote by ${MC}(A_k), k \in \{ 0, 1
\}$ the corresponding Maurer-Cartan solution spaces from
\eqref{mcspace}.  For each $b \in {MC}(A_k)$ the cohomology group is
denoted $H(\mu^{k,1}_b)$.  Define for odd elements $b_0 \in A_0, b_1
\in A_1$ the maps
\begin{multline} 
  \mu^{d|e,b_0|b_1}: A_0^{\otimes d} \otimes M \otimes A_1^{\otimes e}
  \to M, \\ (a^0_{1},\lldots, a^0_{d}, m,a^1_{1} \lldots, a^1_{e})
  \mapsto \sum_{i_0,\lldots, i_d, j_0,\lldots, j_e} \mu^{d + d'|e + e'}
  (\underbrace{b_0, \lldots, b_0}_{i_0} , a^0_{1},
  \underbrace{b_0,\lldots, b_0}_{i_1}, \\ \lldots, 
\underbrace{b_0,\lldots,b_0}_{i_d},m,  \underbrace{b_1,\lldots,b_1}_{j_0}, a^1_{1} ,\lldots, \underbrace{b_1,\lldots,b_1}_{j_e} ) .\end{multline} 

\begin{proposition} Suppose that $M$ is a finitely generated
  $\Lambda$-module.  For any odd $b_0 \in A_0^+, b_1 \in A_1^+$, the
  operations $\mu^{d|e, b_0|b_1}$ define the structure of an \ainfty
  bimodule on $M$.  For $b_0 \in {MC}(A_0)$ and $b_1 \in {MC}(A_0)$ we
  have
\[ (\mu^{0|0, b_0| b_1} ) ^2 = 0 .\]
The cohomology $H(\mu^{0|0, b_0|b_1})$ is invariant up to homotopy in
the sense that if $M_0 \to M_1$ is a homotopy equivalence of \ainfty
bimodules then the induced map
$H(\mu_0^{0|0, b_0|b_1}) \to H(\mu_1^{0|0, b_0 | b_1})$ is an
isomorphism.
\end{proposition}

\begin{proof} The \ainfty bimodule axiom implies in particular that
  $ (\mu^{0|0, b_0 | b_1} )^2 (a)$ is equal to
  $ (-1)^{|a| } ( \mu^{1|0}(\mu_0^{b_0}(1), a) - \mu^{0|1}(a,
  \mu_0^{b_1}(1) ) $
  for any $a \in A_0$.  Since the elements $\mu_0^{b_k}(1)$ are
  multiples of strict identities, the claim on the square zero
  follows.  It follows from the \ainfty bimodule morphism axiom in
  \eqref{morphism} that for any (convergent) morphism $\phi^{n,m}$
  from $M_0$ to $M_1$ the sum
\[ \phi^{b_0|b_1} = \sum_{n,m \ge0} \phi^{n|m}(b_0,\lldots, b_0; b_1,\lldots b_1) \]
is a chain map; a similar argument using \ainfty bimodule homotopies
implies that homotopic bimodules have isomorphic cohomology groups.
\end{proof}

\section{Treed strips}
\label{treedstrips}

The construction of Fukaya bimodules will be similar to the treatment
in Fukaya-Oh-Ohta-Ono \cite{fooo}, but using the Morse model.  Our
perturbation systems uses Donaldson hypersurfaces to stabilize the
domains as in Cieliebak-Mohnke \cite{cm:trans}.  For this purpose
we develop universal curves over moduli spaces of marked strips as
follows.

\begin{definition} 
\begin{enumerate}
\item {\rm (Marked strips)} A {\em marked nodal strip} is a marked
  nodal disk with two boundary markings.  Let $S$ be a connected
  $(2,n)$-marked nodal disk with markings $\ul{z}$.  We write
  $\ul{z} = (z_-, z_+, z_1,\lldots, z_n)$ where $z_\pm$ are the
  boundary markings and $z_1,\lldots,z_n$ are the interior markings.
  We call $z_-$ (resp. $z_+$) the {\em incoming} (resp. {\em
    outgoing}) marking. Let $S_1,\lldots, S_m$ denote the ordered {\em
    strip components} of $S$ connecting $z_-$ to $z_+$; the remaining
  components are either {\em disk components}, if they have boundary,
  or {\em sphere components}, otherwise.  The nodal strip is stable if
  it is stable as a marked disk, as in Definition \ref{ndisk}. 
Let $\ti{w}_i := S_i \cap S_{i+1}$ denote the {\em intermediate node}
connecting $S_i$ to $S_{i+1}$ for $i = 1,\lldots, m-1$.  Let
$\ti{w}_0 = z_-$ and $\ti{w}_m = z_+$ denote the incoming and outgoing
markings. 
\item {\rm (Strip coordinates)} Given a nodal strip $S$ denote the
  curve obtained by removing the nodes connecting strip components and
  the incoming and outgoing markings
  \begin{equation} \label{strips} S^\times := S - \{ \ti{w}_0,\lldots,
    \ti{w}_m \}.\end{equation}
  Each strip component in $S^\times$ may be equipped with coordinates
\[ \phi_i : S_i^\times := S_i - \{ \ti{w}_i,\ti{w}_{i-1} \} \to \R \times [0,1]
, \quad i =1,\lldots, m \]
satisfying the conditions that if $j$ is the standard complex
structure on $\R \times [0,1]$ and $j_i$ is the complex structure on
$S_i$ then 
\[\phi_i^* j = j_i, \quad  \lim_{z \to \ti{w}_i} \pi_1
\circ \phi_i(z) = -\infty, \quad \displaystyle \lim_{z \to \ti{w}_{i+1}}
\pi_1 \circ \phi_i(z) = \infty\] 
where $\pi_i$ denotes the projection on the $i^{th}$ factor.  We
denote by 
\begin{equation} \label{timecoord}
f: S^\times \to [0,1], \quad z \mapsto \pi_2 \circ
\phi_i(z) \end{equation} 
the continuous map induced by the {\em time} coordinate on the strip
components. The time coordinate is extended to nodal marked strips $S$
by requiring constancy on every connected component of
$S^\times - \bigcup_i S^\times_i$.
\item {\rm (Partition of the boundary)} The boundary of any marked
  strip $S$ is partitioned as follows.  For $b \in \{ 0, 1 \}$ denote
  \begin{equation} \label{Sb} (\partial S^\times)_b := f^{-1}(b)
    \cap \partial S \quad \text{so that}  \quad 
\partial S^\times = (\partial S)_0 \cup (\partial S)_1. \end{equation}
That is, $(\partial S)_b$ is the part of the boundary from $z_-$ to
$z_+$, for $b = 0$, and from $z_+$ to $z_-$ for $b = 1$.  An example
of a stable strip is shown in Figure \ref{stabstrip}. 

\begin{figure}
\begin{center} 
\includegraphics[width=3in]{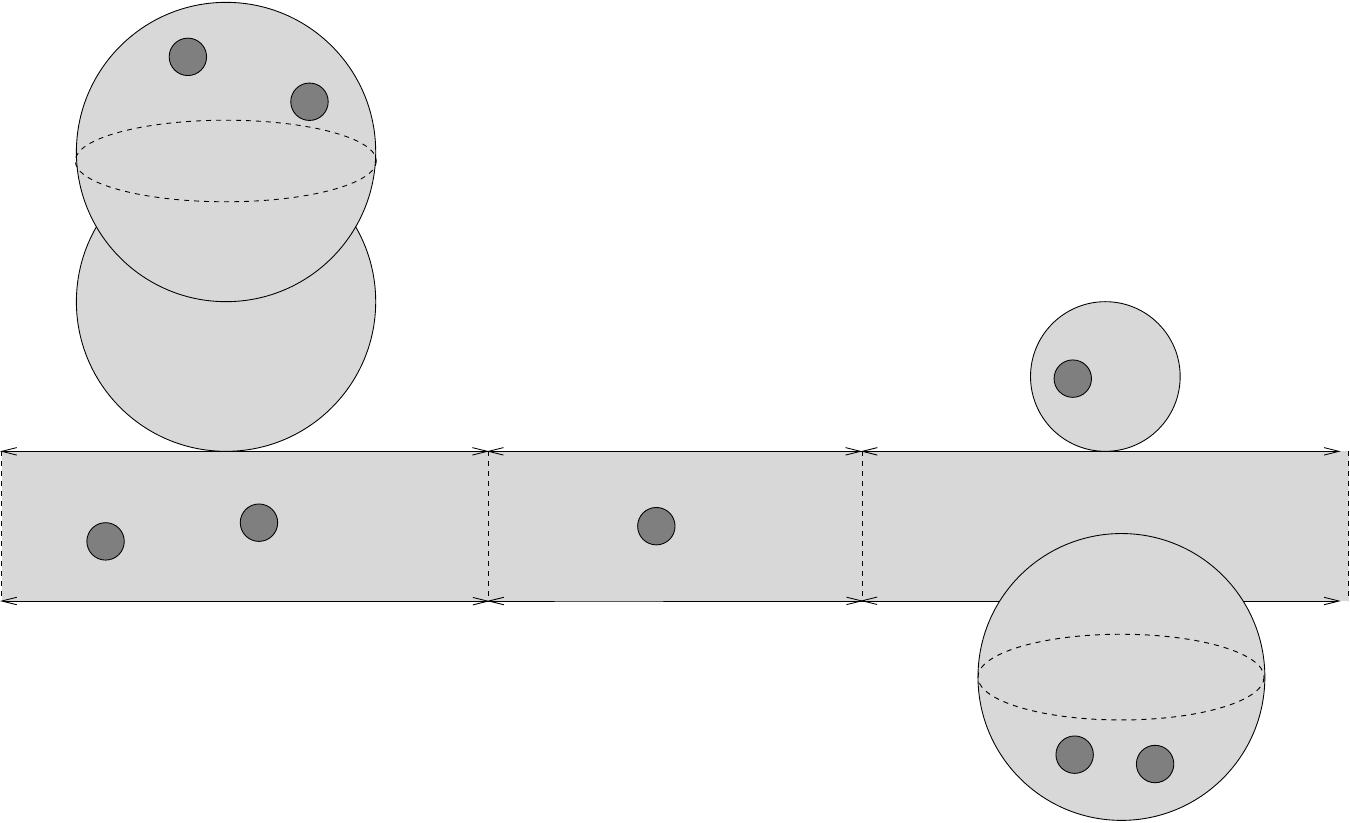}
\end{center}
\caption{A stable marked strip} 
\label{stabstrip} 
\end{figure} 

\item {\rm (Treed strips)} A {\em treed strip} is a space
  $C = S \cup T$ obtained from a marked strip $(S_0,\ul{z})$ (that is,
  a marked disk with two distinguished boundary markings) by replacing
  each boundary node and marking of $S_0$ with a possibly broken edge
  $e$ of some length $\ell(e)$ and number of breakings $b(e)$; denote
  by $T$ the union of such edges.
\item {\rm (Combinatorial types)} 
The {\em combinatorial type} $\Gamma$
  of a treed strip $C$ is defined as in the case of tree disks, but
  there is a distinguished incoming semi-infinite edge
  $e_- \in \Edge(\Gamma)$ and a outgoing edge $e_+ \in \Edge(\Gamma)$.
\item {\rm (Isomorphisms)} An {\em isomorphism} of treed marked strips
  $\phi: C \to C'$ is an isomorphism of the underly treed marked disks
  preserving the incoming and outgoing edges.  A treed strip
  $C = S \cup T$ is {\em stable} if the treed pseudoholomorphic marked
  disk $S$ marked by the intersections $S \cap T$ is stable.
\end{enumerate}
\end{definition} 

We introduce the following notation for moduli spaces.  Let
$\ol{\M}_{n_0|n_1,m}$ denote the moduli space of isomorphism classes
of stable treed strips with $n_0, n_1$ boundary markings (using the
partition of the boundary) and $m$ interior markings. For $\Gamma$ a
connected type we denote by
${\M}_{\Gamma} \subset \ol{\M}_{n_0|n_1,m}$ the moduli space of stable
strips of combinatorial type $\Gamma$ and $\ol{\M}_\Gamma$ its
closure.  Each moduli space $\ol{\M}_\Gamma$ is naturally a manifold with corners,
with local charts obtained by a standard gluing construction.
Generally $\ol{\M}_{n_0|n_1,m}$ is the union of several
top-dimensional strata.  For $\Gamma$ disconnected, $\ol{\M}_\Gamma$
is the product $\prod_i \ol{ \M}_{\Gamma_i}$ of moduli spaces for the
connected components $\Gamma_i$ of $\Gamma$.

\begin{definition} 
\begin{enumerate} 
\item {\rm (Universal strip)} By the universal strip we mean the
  complement of the intermediate nodes in the universal treed disk.
  That is, $\ol{\U}_\Gamma$ is the union of the curves $C^\times$ from
  \eqref{strips}. On the universal strip, the time coordinates
  \eqref{timecoord} on the strip components extend to a map
\begin{equation} \label{allbutone} f: \ol{\U}_\Gamma \to [0,1
  ] \end{equation}
On the subsets $f^{-1}(0), f^{-1}(1)$ with time coordinate equal to
zero or one, we have additional maps measuring the distance to the
strip components given by summing the lengths of the connecting edges:
\[ \ell_b: f^{-1}(b) \to [0,\infty], 
\quad z \mapsto \sum_{e \in \Edge(z)} \ell(e), \quad b \in \{ 0, 1
\} \]
where $\Edge(z)$ is the set of edges corresponding to nodes between
$z$ and the strip components.  Thus any point on the universal strip
$z \in \ol{\U}_\Gamma$ which lies on a disk component has $f(z) \in \{
0, 1 \}$.
\item {\rm (Partition by dimension) } The universal treed strip can be
  written as the union of one-dimensional and two-dimensional parts
\[ \ol{{\U}}_\Gamma = \ol{{\S}}_\Gamma \cup \ol{{\T}}_\Gamma \] 
so that $\ol{\S}_\Gamma \cap \ol{\T}_\Gamma$ is the set of points on
the disks or spheres meeting the edges of the tree.  Denote by
\[ \ol{\T}_\Gamma^b = f^{-1}(b) \cap \ol{\T}_\Gamma \]
the part of the tree corresponding to the minimum resp. maximum values and 
\[ \ol{\T}_\Gamma^{01} = \ol{\T}_\Gamma - \ol{\T}_\Gamma^0 -
\ol{\T}_\Gamma^1 .\]
\end{enumerate} 
\end{definition} 
\section{Hamiltonian perturbations}

The Fukaya bimodules in this paper will be defined using the Morse
model, by a combination of Hamiltonian-perturbed pseudoholomorphic
maps on the strip components and Morse trajectories on the edges.  We
introduce notations for Hamiltonian perturbations and associated
perturbed Cauchy-Riemann operators.  Let $S$ be a nodal strip.  Let
\[K \in \Omega^1(S,\partial S; C^\infty(X))\] 
be a one-form with values in smooth functions vanishing on the tangent
space to the boundary, that is, a smooth map $TS \times X \to \R$
linear on each fiber of $TS$, equal to zero on $TS | \partial S$.
Denote by 
\[\widehat{K} \in \Omega^1(S,\partial S; \Vect(X))\] 
the corresponding one form with values in Hamiltonian vector fields.
The {\em curvature} of the perturbation is
\begin{equation} \label{RK}
 R_K = \d K + \{ K, K \}/2 \in \Omega^2(S, C^\infty(X)) \end{equation} 
where $ \{ K, K \} \in \Omega^2(S, C^\infty(X))$ is the two-form
obtained by combining wedge product and Poisson bracket, see
McDuff-Salamon \cite[Lemma 8.1.6]{ms:jh}.  Given a map $u : S \to X$,
define
\begin{equation} \label{cr} 
\olp_{J,K} u := (\d u + \widehat{K} )^{0,1} \in \Omega^{0,1}(S, u^*
TX).\end{equation}
The map $u$ is {\em (J,K)-holomorphic} if $\olp_{J,K} u = 0 $.

For Hamiltonian-perturbed pseudoholomorphic maps the action and energy
are related by an equality involving a curvature correction term.
Suppose that $C$ is equipped with a compatible metric and $X$ is
equipped with a tamed almost complex structure and perturbation $K$.
The {\em $K$-energy} of a map $u: C \to X$ is
\[ E_K(u) :=  \hh \int_C | \d u + \widehat{K}(u) |_J^2 \]
where the integral is taken with respect to the measure determined by
the metric on $C$ and the integrand is defined as in \cite[Lemma
  2.2.1]{ms:jh}.  If $\olp_{J,K}u = 0$, then the $K$-energy differs
from the symplectic area $A(u) := \int_C u^* \omega$ by a term
involving the curvature from \eqref{RK}:
\begin{equation} \label{energyarea}
 E_K(u) = A(u) + \int_C R_K(u) .\end{equation}
In particular, if the curvature vanishes then the area is
non-negative.  In the case of a strip 
\[C = \R \times [0,1] = \{ (s,t) \ | \ s \in \R, t \in [0,1] \}.\]
and Hamiltonian perturbation $H \in C^\infty([0,1] \times X)$ let $K$
denote the perturbation one-form $K = - H \d t$ and let $E_H := E_K$.
If $u: \R \times [0,1] \to X$ has limits $x_\pm : [0,1] \to X$ as
$s \to \pm \infty$ then the energy-area relation \eqref{energyarea}
becomes
$ E_H(u) = A(u) - \int_{[0,1]} (x_+^* H - x_-^* H) \d t .$

Floer trajectories are solutions to the perturbed pseudoholomorphic
equation with a translational notion of equivalence.  Let
$\olp_{J,H} = \olp_{J,K}$ be the corresponding perturbed
Cauchy-Riemann operator from \eqref{cr}.  A map $u: C \to X$ is {\em
  $(J,H)$-holomorphic} if $\olp_{J,H} u =0$.  A {\em perturbed
  pseudoholomorphic strip} for Lagrangians $L_0,L_1$ is a finite
energy $(J,H)$-holomorphic map $u: \R \times [0,1] \to X$ with
$\R \times \{ b \} \subset L_b, b =0,1$.  An {\em isomorphism} of
Floer trajectories $u_0,u_1: \R \times [0,1] \to X$ is a translation
$\phi: \R \times [0,1] \to \R \times [0,1]$ in the $\R$-direction such
that $\phi^* u_1 = u_0$.  Denote by
\[\M(L_0,L_1) := \Set{ u: \R \times [0,1] \to X \ | \begin{array}{c}
  \olp_{J,H} u = 0 \\ u(\R \times \{b \}) \subset L_b, b \in \{ 0,1 \}
  \\ E_H(u) < \infty \end{array} } / \R \]
the moduli space of isomorphism classes of Floer trajectories of
finite energy, with its quotient topology.

Floer trajectories with Hamiltonian perturbation correspond to
unperturbed pseudoholomorphic maps with a perturbed boundary
condition.  Let $H \in C^\infty([0,1] \times \R, X)$ be a
time-dependent Hamiltonian and let
$J \in \Map([0,1],\J_{\tau}(X,\omega))$ be a time-dependent almost
complex structure.  Suppose that $L_0,L_1$ are Lagrangians such that
$\varphi_1(L_0) \cap L_1$ is transversal.  There is a bijection between
$(J_t,H_t)$-holomorphic Floer trajectories $u: \R \times [0,1] \to X$
with boundary conditions $L_0,L_1$ and
$(\varphi_{1-t}^{-1})^* J_t$-holomorphic Floer trajectories with boundary
conditions $\varphi_1(L_0),L_1$ obtained by mapping each $(L_0,L_1)$
trajectory $(s,t) \mapsto u(s,t)$ to the
$(\varphi_1(L_0),L_1)$-trajectory given by
$(s,t) \mapsto \varphi_{1-t}(u(s,t))$ \cite[Discussion after
(7)]{fhs:tr}.

In order to relate bimodules defined using different perturbations we
also consider Hamiltonian perturbations on surfaces with strip-like
ends.  A {\em surface with strip like ends} consists of a surface with
boundary $S$ equipped with a complex structure $j: TS \to TS$, and a
collection of embeddings
  \[ \kappa_e : \pm(0,\infty) \times [0,1] \to S, \quad e = 0,\lldots,
  n \]
  such that $\kappa_e^* j$ is the standard almost complex structure on
  the strip, and the complement of the union of the images of the maps
  $\kappa_e$ is compact.  Any such surface has a canonical
  compactification $\ol{S}$ with the structure of a compact surface
  with boundary obtained by adding a point at infinity along each
  strip like end and taking the local coordinate to be the exponential
  of $ \pm 2 \pi i \kappa_e $.

  Hamiltonian-perturbed pseudoholomorphic maps are in bijection with
  holomorphic sections of the trivial bundle with respect to a
  non-product almost complex structure.  Given a surface with
  strip-like ends let $ E := S \times X $ denote the product
  considered as a fiber bundle over $S$ with fiber $X$.  Following
  \cite[(8.1.3)]{ms:jh}, let $\pi_X: E \to X$ denote the projection on
  the fiber. In local coordinates $s,t$ on $S$ define $K_s,K_t$ by
  $ K = K_s \d s + K_t \d t$.  Let
\[ \omega_{E} = \pi_X^* \omega - \pi^*_X \d K_s \wedge \d s - \pi_X^* \d
K_t \wedge \d t + (\partial_t K_s - \partial_s K_t) \d s \wedge \d t
.\]
The form $\omega_{E}$ is closed, restricts to the two-form
$\omega$ on any fiber, and defines the structure of a symplectic fiber
bundle on $E$ over $S$.  Consider the splitting
$TE \cong \pi_X^* TX \oplus (S \times \R^2) .$
Let $j_S: T {S} \to T{S}$ denote the standard complex
structure on ${S}$.  Define an almost complex structure on $E$ by
\[ J_E : T E \to TE, \quad (v,w) \mapsto ( ( J \hat{K} - \hat{K} j_{S} )w +
Jv, j_{S} w ) \]
where $\hat{K} \in \Omega^1({S},\Vect(X))$ is the
Hamiltonian-vector-field-valued one-form associated to $K $. A smooth
map $u: {S} \to X$ is $(J,K)$-holomorphic if and only if the
associated section $(\on{id} \times u): {S} \to E$ is
$J_E$-holomorphic \cite[Exercise 8.1.5]{ms:jh}.  Let
$ \ti{L}_i = (\partial {S})_i \times L_i$ the fiber-wise Lagrangian
submanifolds of $E$ defined by $L_i$.  Then $u: {S} \to X$ has
boundary conditions in $(L_i, i = 1,\lldots, m)$ if and only if
$\on{id} \times u: {S} \to E$ has boundary conditions in
$(\ti{L}_i, i = 1,\lldots, m)$. 

\section{Clean intersections} 

In this section we extend the above results to the case that the union
of the cleanly-intersecting Lagrangians is rational using stabilizing
divisors.  We have in mind especially the case that the two
Lagrangians are rational and equal.  In the particular case of
diagonal boundary conditions, we show that the Floer cohomology is the
singular cohomology with Novikov coefficients.  

We recall some terminology for clean intersections.  A pair
$L_0,L_1 \subset X$ of submanifolds intersect {\em cleanly} if
$L_0 \cap L_1$ is a smooth manifold and
$ T(L_0 \cap L_1) = TL_0 \cap TL_1 .$ Floer homology for clean
intersections was constructed in Pozniak \cite{po:cl} and also
Schm\"aschke \cite[Section 7]{schmaschke} under certain monotonicity
assumptions, with the Floer differential counting pseudoholomorphic
strips.  The definition of Floer cohomology in the clean intersection
case is a count of configurations of pseudoholomorphic strips, disks,
spheres, and Morse trees as in, for example, Biran-Cornea
\cite[Section 4]{bc:ql}.  Let $L_k, k \in \{ 0,1 \}$ be compact
Lagrangian branes equipped with Morse-Smale pairs $(F_k, G_k)$.  In
particular, the critical set $\crit(F_k) \subset L_k$ of each Morse
function is finite.  Let
\[ \ul{P}_k = (P_{\Gamma,k} = (J_{\Gamma,k}, F_{\Gamma,k})),k \in \{
  0,1 \} \] 
  be regular perturbation data giving rise to moduli spaces of
  pseudoholomorphic treed disks with boundary in $L_k$.  Consider a a
  Hamiltonian with flow
\[ H_{01} \in C^\infty(X \times [0,1]), \quad \varphi^{01}_t: X \to X
\] 
so that $\varphi^{01}_1(L_0) \cap L_1$ is a clean intersection and
$\varphi^{01}_1(L_0) \cup L_1$ is rational.  That is, some power of
the line-bundle-with-connection $\widetilde{X}$ is trivializable over
$L_0 \cup L_1$.  This includes the special cases of equality
$\varphi^{01}_1(L_0) = L_1$ if $L_0$ is rational.  Let
  \[ F_{01}: \varphi^{01}_1(L_0) \cap L_1 \to \R \]
be a Morse function.  By the Morse lemma, the critical set
\[\cI(L_0,L_1;H_{01}) := \crit(F_{01}) = \{l \in L_0 \cap L_1 \ | \ \d F_{01}(l) = 0 
\} \]
is necessarily finite.  Choose a generic metric $G_{01}$ on
$\varphi^{01}_1(L_0) \cap L_1$, and let
\[ \varphi^{01}_t: L_0 \cap L_1 \to L_0 \to L_1 \]
be the time $t$ flow of $-\grad(F_{01}) \in \Vect(L_0 \cap L_1)$.
Denote the stable and unstable manifolds of $F_{01}$:
\[ W_x^\pm = \Set{ l \in L_0 \cap L_1 \ | \lim_{t \to \pm \infty}
  \varphi^{01}_t(l) = x } .\]
We assume that $(F_{01},G_{01})$ is {\em Morse-Smale}, that is, the
stable and unstable manifolds meet transversally
\[ T_l W_x^+ + T_l W_y^- = T_l (\varphi^{01}_1(L_0) \cap L_1), \quad
\forall l\in W_x^+ \cap W_y^-, \ x,y \in \crit(F) .\]
The critical set admits a natural grading map
\[i: \cI(L_0,L_1;H_{01}) \to
\Z_{g} \]
obtained by adjusting the Maslov index of paths from $T_xL_0$ to
$T_xL_1$ in $T_x X$ for any $x \in \crit(F_{01})$ by the index
$i(x) = \dim(W_x^-)$.

The space of Floer cochains is then generated by the finite set of
critical points on the given Morse function on the clean intersection:
\[ CF(L_0,L_1;H_{01}) = \bigoplus_{x \in \cI(L_0,L_1;H_{01})} \Lambda
{x} \]
with $\Z_g$-grading induced by
the grading on $\cI(L_0,L_1;H_{01})$.

The structure maps for the Fukaya bimodule count configurations
containing perturbed pseudoholomorphic strips and gradient segments
for the Morse function on the intersection.  Given a stable strip
$C_0$ with boundary markings $z_-,z_+$ let $w_1,\lldots,w_k \in C_0$
denote the nodes appearing in any non-self-crossing path between $z_-$
and $z_+$.  Define a {\em treed strip}
\[C =  C_0 \sqcup \bigsqcup_{i = 1}^k [0,\ell(w_i)] / \sim \] 
by replacing each node $w_i$ by a segment $T_i \cong [0,\ell(w_i)]$ of
length $\ell(w_i)$.  Denote by 
\[T = T_1 \cup \lldots \cup T_k \quad S = \ol{C - T} \] 
the {\em tree resp. surface part} of $C$.  A perturbed pseudoholomorphic strip is then
a map from $C = S \cup T$ that is $J$-holomorphic on the surface part
and a $F_{01}$ resp. $F_0$ resp. $F_1$ gradient trajectory on each
segment in $T$.  See Figure \ref{treedstrip}.
\begin{figure}[h!] 
\begin{picture}(0,0)%
\includegraphics{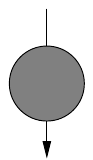}%
\end{picture}%
\setlength{\unitlength}{3947sp}%
\begingroup\makeatletter\ifx\SetFigFont\undefined%
\gdef\SetFigFont#1#2#3#4#5{%
  \reset@font\fontsize{#1}{#2pt}%
  \fontfamily{#3}\fontseries{#4}\fontshape{#5}%
  \selectfont}%
\fi\endgroup%
\begin{picture}(891,1224)(3974,-3373)
  \put(4557,-2414){\makebox(0,0)[lb]{\smash{{\SetFigFont{12}{14.4}{\rmdefault}{\mddefault}{\updefault}{$\varphi^{01}_1(L_0)
            \cap L_1$}%
        }}}}
  \put(4570,-3279){\makebox(0,0)[lb]{\smash{{\SetFigFont{12}{14.4}{\rmdefault}{\mddefault}{\updefault}{$\varphi^{01}_1(L_0)
            \cap L_1$}%
        }}}}
  \put(3539,-2800){\makebox(0,0)[lb]{\smash{{\SetFigFont{12}{14.4}{\rmdefault}{\mddefault}{\updefault}{$\varphi^{01}_1(L_0)$}%
        }}}}
  \put(4850,-2782){\makebox(0,0)[lb]{\smash{{\SetFigFont{12}{14.4}{\rmdefault}{\mddefault}{\updefault}{$L_1$}%
        }}}}
\end{picture}%
\caption{A treed strip with Lagrangian boundary conditions}
\label{treedstrip}
\end{figure}
Fix {\em thin parts} of the universal curves: a neighborhood
$\ol{{\T}}_{\Gamma}^{\thin}$ of the endpoints and a neighborhood
$\ol{{\S}}_{\Gamma}^{\thin}$ of the markings and nodes.  In the
regularity construction, these neighborhoods must be small enough so
that either a given fiber is in a neighborhood of the boundary, where
transversality has already been achieved, or otherwise each segment
and each disk or sphere component in a fiber meets the complement of
the chosen thin parts.  For an integer $l \ge 0$ a {\em
  domain-dependent perturbation} of $F$ of class $C^l$ is a $C^l$ map
\begin{equation} \label{FGam2}  F_{\Gamma}: \ol{{\T}}_{\Gamma}
 \times ( \varphi_1^{01}(L_0) \cap L_1) \to \R \end{equation}
equal to the given function $F$ away from the endpoints:
\[ 
F_{\Gamma,01} | \ol{{\T}}_{\Gamma}^{\thin,01} = \pi_2^* F_{01}, \quad
F_{\Gamma,k} | \ol{{\T}}_{\Gamma}^{\thin,k} = \pi_2^* F_{k}, \quad k
\in \{ 0, 1 \}\]
where $\pi_2$ is the projection on the second factor in \eqref{FGam2}.
A {\em domain-dependent almost complex structure} of class $C^l$ for
treed disks of type $\Gamma$ is a map from the two-dimensional part
$\ol{{\S}}_{\Gamma} $ of the universal curve $\ol{{\U}}_{\Gamma}$ to $
\J_\tau(X)$ given by a $C^l$ map
\[ J_{\Gamma} : \ \ol{{\S}}_{\Gamma} \times X \to \End(TX) \]
gives as the product of pull-backs of maps 
\[ J_{\Gamma(v)} : \ \ol{{\S}}_{\Gamma(v)} \times X \to \End(TX) \]
and equal to the given $J_D$ near the nodes and boundary:
\[ J_\Gamma | \ol{{\S}}_\Gamma^{\thin} = 
\pi_2^* J_D .\]
A {\em perturbed pseudoholomorphic strip} for the pair $(L_0,L_1)$
consists of a treed disk $C$ and a map $u: C = S \cup T \to X$ such
that
\vskip .1in
\begin{enumerate} 
\item[] {\rm (Boundary condition)} The Lagrangian boundary condition
  holds $u ( (\partial C )_b ) \subset L_b$ for $b \in \{ 0,1 \}$.
\vskip .1in
\item[] {\rm (Surface equation)} On the surface part $S$ of $C$ the map
  $u$ is $J$-holomorphic for the given domain-dependent almost complex
  structure: if $j$ denotes the complex structure on $S$ then
\[ J_{\Gamma,u(z),z} \ \d u_S = \d u_S \ j. \] 
\vskip .1in
\item[] {\rm (Boundary tree equation)} On the boundary tree part $T
  \subset C$ the map $u$ is a collection of gradient trajectories:
  \[ \dds u |_{T_b} = -\grad_{ F_{\Gamma, b,(s,u(s))} }(u |_{T_b}) \]
  where $s$ is a local coordinate with unit speed and $b$ is one of
  the symbols $0$, $1$ or $01$.  Thus for each edge $e \in
  \Edge_{-}(\Gamma)$ the length of the trajectory is given by the
  length $u |_{e \subset T}$ is equal to $\ell(e)$.
\end{enumerate} 
For each combinatorial type $\Gamma$, the moduli space
$\M_\Gamma(L_0,L_1)$ of treed strips is locally cut out as the zero
set of a Fredholm map as in \eqref{linop}.

To obtain regular moduli spaces we introduce a stabilizing divisor.
Given a stabilizing divisor $D \subset X - (L_0 \cup L_1)$, a stable
strip $u:C \to X$ is {\em adapted} if and only if {\rm (Stable surface
  axiom)} $C$ is a stable marked strip; and {\rm (Leaf axiom)} each
interior leaf $T_e$ lies in $u^{-1}(D)$ and each component of
$u^{-1}(D)$ contains an interior leaf $T_e$.  Let $\ol{\M}(L_0,L_1,D)$
denote the set of isomorphism classes of stable adapted Floer
trajectories to $X$, and by $\M_\Gamma(L_0,L_1,D)$ the subspace of
combinatorial type $\Gamma$.  Compactness and transversality
properties of the moduli space of Floer trajectories in the case of
clean intersection, including exponential decay estimates, can be
found in \cite{wo:gdisk} and \cite{schmaschke}. The necessary gluing
result can be found in Schm\"aschke \cite[Section 7]{schmaschke}.

In order to define the structure maps in the Fukaya bimodules, we
assume that the union of Lagrangians is equipped with the following
further structure:
\begin{enumerate}
\item {\rm (Local system)} Let $y_0,y_1$ denote the local
  systems on $L_0,L_1$ assumed as part of their brane structures.  We
  suppose that $L_0 \cup L_1$ is equipped with a local system
  $\rho \in \cR(L_0 \cup L_1) $ restricting to the given local systems
  on $y_0,y_1$.  Such a local system is determined by an
  isomorphism of the corresponding flat line bundles on the
  intersection $L_0 \cap L_1$, but this isomorphism is not unique.
\item {\rm (Relative spin structure)} We suppose that the embedding
  $L_0 \cup L_1 \to X$ is equipped with a relative spin structure,
  restricting to the given relative spin structures on $L_0,L_1$. 
\end{enumerate}

Given relative spin structures as above, the construction of
orientations proceeds along the lines of Fukaya-Oh-Ohta-Ono
\cite[Chapter 10]{fooo}, see also Wehrheim-Woodward \cite{orient}.  We
may ignore the constraints at the interior markings
$z_1,\lldots, z_m \in \on{int}(S)$, since the tangent spaces to these
markings and the linearized constraints $\d u(z_i) \in T_{u(z_i)}D$
are even dimensional and oriented by the given complex structures.  At
any regular element $(u:C \to X) \in \M(L_0,L_1,D)$ the tangent space
to the moduli space of treed disks is the kernel of the linearized
operator
\[ T_u \M(L_0,L_1,D) \cong \ker(\ti{D}_u) .\]
The operator $\ti{D}_u$ is canonically homotopic via family of
operators $\ti{D}_u^t, t \in [0,1]$ to the operator
$0 \oplus D_u \oplus \dds$ given by zero on the variations of complex
structure on the domain, $D_u$ on the variations of map $S \subset C$,
and $\dds$ on the variations of map on the tree part $T \subset C$.
The deformation $\ti{D}_u^t, t \in [0,1]$ of operators induces an
family of determinant lines over the interval, necessarily trivial,
and so (by taking a connection on this family) an identification of
determinant lines
\begin{equation} \label{split1} \det(T_{u} \M(L_0,L_1,D)) \to \det(T_C
  \M_\Gamma ) \otimes \det(D_{{u}}) \end{equation}
well-defined up to isomorphism.  (Here $D_u$ denotes the linearized
operator subject to the constraints which require the attaching points
of edges mapping to critical points to map to the corresponding
unstable manifolds of the Morse function.)  In the case of nodes of
$S$ mapping to intersection points $x \in L_0 \cap L_1$ the
determinant line $\det(D_{u})$ is oriented by ``bubbling off
one-pointed disks'', see \cite[Theorem 44.1]{fooo} or \cite[Equation
(36)]{orient}.  That is, for each intersection point
$x \in L_0 \cap L_1$ choose a path of Lagrangian subspaces
\begin{equation} \label{gammax} \gamma_x:[0,1] \to \Lag(T_{x} X),
  \quad \gamma_x(0) = T_x L_0 \quad \gamma_x(1) = T_x L_1 .
\end{equation} 
Let $S$ be the unit disk with a single boundary marking
$1 \in \partial S$.  The path $\gamma_x$ defines a totally real
boundary condition on $S$ on the trivial bundle with fiber $T_x X$,
equipped with a Cauchy-Riemann operator $D_x$ acting on the space of
$W^{k,p,-\delta}$ sections of $T_x X$ with totally real boundary
conditions, for some Sobolev weighting $\delta > 0$ sufficiently small
so that the operator $D_x$ is Fredholm.  Let $\det(D_x)$ denote the
determinant line for the Cauchy-Riemann operator $D_x$ with boundary
conditions $\gamma_x$.  Denote by
\[ i(x) = \dim(\ker(D_x)) - \dim(\coker(D_x)) \in \Z \]
the index of the operator $D_x$.  Let $\DD^+_{x} = \det(D_{x}^+)$ and
let $\DD^-_{x}$ be the tensor product of the determinant line
$\det(D_k^-)$ for the once-marked disk with
$\det(T_x L_0) \cong \det(T_x L_1)$.  Because the once-marked disks
with boundary conditions $\gamma_{x_k}$ and $\gamma_{\ol{x_k}}$ glue
together to a trivial problem on the disk with index $T_{x_k}L $,
there is a canonical isomorphisms
\[\DD^-_{x_k} \otimes \DD^+_{x_k} \to \R .\]
The orientations for the intersection points are {\em coherent} if the
above isomorphism is orientation preserving with respect to the
standard orientation on $\R$.  The orientation for a treed strip $u$
is determined by an isomorphism \label{familytransversality2}
 \begin{equation} \label{split2} \det(D_{{u}}) \cong \DD^+_{{x}_0}
   \otimes \DD^-_{x_1} \otimes \lldots \otimes \DD^-_{x_d}.
 \end{equation}
 The isomorphism \eqref{split2} is determined by degenerating the
 surface with strip-like ends to a nodal surface. Thus each end
 $\eps_e, e \in \mE(S_i)$ of a component $S_i$ with a node $q_k$
 mapping to a intersection point is replaced by a disk $S_{i^\pm(k)}$
 with one end attached to the rest of the surface by a node $q_k^\pm$.
 After combining the orientations on the determinant lines on
 $S_{i^\pm(k)}$ with coherent orientations on the tangent spaces
 to the stable manifolds $W_{x_k}^\pm$ in the case of broken edges or
 semi-infinite edges, one obtains an orientation on the determinant
 line of the parameterized operator $\det(\ti{D}_u)$ and so
 orientations on the regularized moduli spaces $\M_\Gamma(L_0,L_1,D)$.
 In particular the zero-dimension component of the moduli space
 inherits an orientation map
\[ \eps:  \M(L_0,L_1,D)_0 \to \{ +1, -1 \} \]
comparing the constructed orientation to the canonical orientation of
a point.  

The structure maps in the Fukaya bimodules are defined by weighted
counts of perturbed-holomorphic treed strips.  The weighting is by
powers $q^{A(u)}$ of the formal variable $q$ with exponent given by
the vertical symplectic area $A(u)$ for $u \in \M(L_0,L_1)$; while
these powers are non-negative for any particular Fukaya bimodule the
homotopy invariance maps will require possibly negative powers, so
that the homotopy type is only invariant over the Novikov field
$\Lambda$ rather than the Novikov ring.  As before let $\sigma(u)$
denote the number of interior markings; $\eps(u)$ is the orientation
sign while $y(u)$ is the holonomy of the local system.  Define
\begin{equation} \label{coboundary3} CF(L_0,L_1;H_{01}) = \bigoplus_{l 
    \in \cI(L_0,L_1;H_{01})} \Lambda l \end{equation}
Counting Floer trajectories defines operations
\[ \mu^{d|e}: CF(L_0)^{\otimes d} \otimes CF(L_0,L_1;H_{01}) \otimes
CF(L_1)^{\otimes e} \to CF(L_0,L_1;H_{01}) \]
by 
\begin{multline} {x^0_{1}} \otimes \lldots \otimes {x^0_{d}} \otimes
  {x} \otimes {x^1_{1}} \otimes \lldots \otimes {x^1_{e}} \\ \mapsto
  \sum_{ n,u \in
    {\M}_{d|e,n}(x^0_{1},\lldots,x^0_{d},x,x^1_{1},\lldots,x^1_{e},y)_0}
  (-1)^{\heartsuit + \diamondsuit} \eps(u) y(u) (n!)^{-1} q^{A(u)}
  {y} \end{multline}
where $\diamondsuit$ is defined in \eqref{change}

\begin{theorem} \label{bimodthm} For admissible collections
  $\ul{P} = (P_\Gamma)$ of perturbation data the maps
  $(\mu^{d|e})_{d,e \ge 0}$ induce on $CF(L_0,L_1;H_{01})$ the
  structure of a strictly unital \ainfty
  $(CF(L_0), CF(L_1))$-bimodule.
\end{theorem} 

As in the case of \ainfty algebras, the relation \eqref{bimod} follows
from a study of the ends of the one-dimensional moduli space.  The
configurations $u : C \to X$ representing endpoints of the
one-dimensional moduli spaces arise when one of the segments $e$ in
the configuration becomes length $\ell(e)$ infinity, and so broken.
The first term in \eqref{bimod} arises when the broken edge $e$
connects a strip $S_1 \subset S$ to a treed disk configuration
$S_2 \subset S$ attached at a point where $t = 1$; the second term
when a broken edge $e, \ell(e) = \infty$ connects a strip $S_2$ to a
treed disk configuration $S_1$ attached at a point where $t = 0$; and
the third when the broken edge $e, \ell(e) = \infty$ connects two
strips $S_1,S_2$.  The sign computation is similar to that for \ainfty
algebras, and left to the reader.

\section{Morphisms}

In the remainder of this section we show that up to \ainfty homotopy
the Fukaya bimodule is independent of all choices, including the
choice of Hamiltonian perturbation.  In particular, if a Lagrangian is
Hamiltonian displaceable then its Fukaya bimodule is homotopy
equivalent to the trivial bimodule.  The necessary morphisms of
\ainfty bimodules are given by counting {\em parametrized treed
  strips}.  Let $C$ be a disk with distinguished boundary markings
$z_-,z_+ \in \partial C$.  A {\em parametrization} of $(C,z_-,z_+)$ is
a holomorphic isomorphism
\[ \phi: C - \{ z_-, z_+ \} \to \R \times [ 0,1 ] .\]
A {\em $(d|e,n)$-marking} of a parametrized strip is a collection of
$d$ resp. $e$ resp $n$ markings on $\R \times \{ 0 \}$ resp.
$\R \times \{ 1 \}$ resp. in $\R \times (0,1)$. In the treed version
of this moduli space, the $(d|e,n)$-marking becomes a collection of
leaves, $d+e$ on the boundary and $n$ in the interior.  The moduli
space of $(d|e,n)$-leafed treed parametrized strips has a natural
compactification $\ol{\M}_{d|e,n,1}$ by stable treed parametrized
strips, in which unparametrized strips are allowed to bubble off the
ends.  This moduli space is equipped with a continuous map
$ \ol{\M}_{d|e,n,1} \to \ol{\M}_{d|e,n} $ forgetting the
parametrization on quilted strip components.  The fibers of the
forgetful map are canonically oriented so that the positive
orientation corresponds to moving the quilting to the left, and so the
orientation of $\ol{\M}_{d|e,n}$ induces an orientation on
$\ol{\M}_{d|e,n,1}$.

Regularization of the space of pseudoholomorphic parametrized treed
strips uses a divisor in the total space of the fibration
$\R \times [0,1] \times X$.  Let
\[J_e \in \J(X,\omega), \quad e \in \{ 0,1 \} \]
be compatible almost complex structures stabilizing for
$\varphi_{H_e,1}(L_0) \cup L_1$.  Let $\varphi_{H_e,1-t}^* J_e$ denote the
corresponding time-dependent almost complex structures, and
$\sigma_{k,e} : X \to \widetilde{X}^k$ are asymptotically
$J_e$-holomorphic, uniformly transverse sequences of sections with the
property that 
\[D'_{e} = \sigma_{k,e}^{-1}(0), \quad e \in \{ 0 ,1 \} \]
are stabilizing for $\varphi_{H_e,1}(L_0) \cup L_1$ for $k$
sufficiently large.  The pull-backs $\phi_{e,1-t}^* \sigma_{k,e}$ are
then $\phi_{e,1-t}^* J_e$-holomorphic, and any $(J_e,H_e)$-holomorphic
strip with boundary in $(L_0, L_1)$ meets
$\phi_{e,1-t}^{-1}({S} \times D'_{e})$ in at least one point.

These divisors extend over the parametrized strip, so that any
Hamiltonian-perturbed pseudoholomorphic strip meets the extended
divisor.  Denote by $\ti{E} \to E$ the pull-back of
$\widetilde{X} \to X$ to the fibration, equipped with the almost
complex structure induced by the given almost complex structure on
$E$.  Recall the following from Charest-Woodward \cite[Theorem
6.1]{cw:traj}:

\begin{lemma}  
  Let ${S}$ be a surface with strip-like ends, let $L_b \subset X$ be
  rational Lagrangians associated to the boundary components
  $(\partial {S})_b$, and suppose that stabilizing divisors $D'_e$ for
  the ends $e = 1,\lldots,n$ of ${S}$ have been chosen as zero sets of
  asymptotically holomorphic sequences of sections $\sigma_{e,k}$ for
  $k$ sufficiently large.  There exists a asymptotically
  $J_E$-holomorphic, uniformly transverse sequence
  $\sigma_k: E \to \ti{E}^k$ with the property that for each end $e$,
  the pull-back $ \kappa_e^* \sigma_k( \cdot + s, \cdot, \cdot)$
  converges in $C^\infty$ uniformly on compact subsets to
  $\phi_{e,1-t}^* \sigma_{e,k}$ as $s \to \pm \infty$.  The zero set
  $D_E = \sigma_k^{-1}(0)$ is approximately holomorphic for $k$
  sufficiently large, asymptotic to
  $(1 \times \phi_{e,1-t})^{-1} ( \R \times [0,1] \times D'_e)$ for
  each end $e = 1,\lldots, n$, and stabilizing for holomorphic disks
  or spheres in the fibers of $E \to S$. \end{lemma}

\begin{proof} 
  We include the proof for completeness.  Let
  $\H = \{ z \in \C \ | \ \on{Im}(z) \ge 0 \}$ denote complex
  half-space.  Let $\ol{E}$ be the fiber bundle over $\ol{{S}}$
  with fiber $X$ defined by gluing together
  $ U_0 = {S} \times X$ and $U_e= \H \times X ,e = 1,\dots, n$
  using the transition maps $\kappa_e \times \phi_{H_e,1\pm t}$ on
  $\H - \{ 0 \} \cong \R \times [0,1]$ from $U_0$ to $U_e$.  Denote
  the projections over $U_e$ to $X$ by $\pi_{X,e}$.  The two-forms
  $\omega_{E}$ on $U_0$, and $\pi_{X,e}^* \omega$ on $U_e$ glue
  together to a two-form $\omega_E$ on $E$, making $\ol{E}$ into a
  symplectic fiber bundle.  The fiber-wise symplectic form $\omega_E$
  above may be adjusted to an honest symplectic form on $E$ by adding
  a pull-back from the base, and furthermore adjusted so that the
  boundary conditions have rational union.  For the first claim, since
  $\omega_E$ is fiber-wise symplectic, there exists a symplectic form
  $\nu \in \Omega^2( \ol{{S}})$ with the property that
  $\omega_E + \pi^* \nu$ is symplectic, where
  $\pi: \ol{E} \to \ol{{S}}$ is the projection.  The almost complex
  structure $J_E$ is compatible with $\omega_E + \pi^* \nu$, and equal
  to the given almost complex structures $J_B \oplus J_e$ on the ends.
  For the second claim, let $\ti{{S}} \to \ol{{S}}$ be a
  line-bundle-with-connection whose curvature is
  $\nu \in \Omega^2(\ol{{S}})$.  Let $\ti{E} \to E$ be a
  line-bundle-with-connection whose curvature is
  $\omega_E + \pi^* \nu$.  Denote by $\ti{L}_i$ the closure of the
  image of $(\partial {S})_i \times L_i$ for $i = 0,1$.  Fix
  trivializations of $\ti{E}$ over
  $\ti{L}_{e,0} \cap \ti{L}_{e,1} \cong L_{e,0} \cap L_{e,1}$ for each
  end $e$.  By assumption, the line bundle $\ti{E}$ is trivializable
  over $L_i$, hence also $\ti{L}_i$ by parallel transport along the
  boundary components.  Let $e_k, k = 0,1$ be ends connected by a
  connected boundary component labelled by $L_i$.  For any
  $p_k \in \phi_{H_{e},1}(L_0) \cap L_1$ the parallel transport
  $T(p_0,p_1) \in U(1)$ from $p_0$ to $p_1$ is independent of the
  choice of path.  Indeed, any two paths differ up to homotopy by a
  loop which has trivial holonomy by assumption.  After perturbation
  of the connection and curvature on $\ti{E}$, we may assume that the
  parallel transports $T(p_0,p_1)$ are rational for all choices of
  $(p_0,p_1)$.  After taking a tensor power of $\ti{E}$, we may assume
  that the parallel transports $T(p_0,p_1)$ are trivial, hence
  $\ti{E}$ admits a covariant constant section $\tau$ over the union
  $\ti{L}_e$.

  Donaldson's construction \cite{don:symp} implies the existence of a
  symplectic hypersurface in the total space of the fibration.  We
  show that the hypersurface $D_E \subset E$ may be taken to equal the
  pullback of one of the given ones $D'_e, e \in \{ 0, 1 \}$ on the
  ends, as follows.  Let $\sigma_{m,b}: X \to \widetilde{X}^k$ be an
  asymptotically holomorphic sequence of sections concentrating on
  $L_b$.  Let $\sigma_{e,k}: X \to \widetilde{X}^k$ be an
  asymptotically holomorphic sequence of sections concentrating on
  $\phi_{H_e,1}(L_{e,0}) \cup L_{e,1}$ and $\sigma_{e,m}$ an
  asymptotically holomorphic sequence of sections concentrating on
  $L_{e,m}$, $m \in \{ 0, 1 \}$, both asymptotic to the given
  trivializations on the Lagrangians themselves.  For each point
  $z \in \ol{{S}}$, let $\sigma_{z,k}: \ol{{S}} \to \ti{{S}}^k$ denote
  the Gaussian asymptotically holomorphic sequence of sections of
  $\ti{{S}}^k$ concentrated at $z$ as in \eqref{approxhol}.  We may
  assume that the images of the strip-like ends are disjoint.  Let
  $V_i$ be disjoint open neighborhoods
  $(\partial {S})_i - \cup_e \on{Im}(\kappa_e)$ in ${S}$.  For each
  $p = (z,x) \in \ti{L}_i$, let $ \sigma_{x,k}$ be either equal to
  $\sigma_{e,k}$, for $z \in \on{Im}(\phi_e)$ or otherwise equal to
  $\sigma_{i,k}$ if $b$ lies in $V_i$.  Let $P_k$ be a set of points
  in $\partial \ol{{S}}$ such that the balls of $g_k$-radius $1$ cover
  $\partial \ol{{S}}$ and any two points of $P_k$ are at least
  distance $2/3$ from each other, where $g_k$ is the metric defined by
  $k \nu$.  The desired asymptotically-holomorphic sections are
  obtained by taking products of asymptotically-holomorphic sections
  on the two factors: Write
  $\theta(z,x) \tau(z,x) = \sigma_{x,k}(x) \boxtimes \sigma_{z,k}(z)$
  so that $\theta(x,z) \in \C$ is the scalar relating the two
  sections.  Define
\begin{equation} \label{approxhol} \sigma_k = \sum_{p \in P_k}
  \theta(z,x)^{-1} \sigma_{x,k} \boxtimes \sigma_{z,k} .\end{equation}
Then by construction the sections \eqref{approxhol} are asymptotically
holomorphic, since each summand is.

Recall from \cite{don:symp} that a sequence $(s_k)_{k \ge 0}$ is {\em
  uniformly transverse} to $0$ if there exists a constant $\eta$
independent of $k$ such that for any $x \in X$ with $|s_k(x)| < \eta$,
the derivative of $s_k$ is surjective and satisfies
$| \nabla s_k(x)| \ge \eta$.  The sections $\sigma_k$ are
asymptotically $J_E$-holomorphic and uniformly transverse to zero over
$\partial \ol{{S}} \times X$, since the sections $\sigma_{i,k}$ and
$\sigma_{e,k}$ are uniformly transverse to the zero section. Hence
$\sigma_k$ is also uniformly transverse over a neighborhood of
$\partial {S}$.  Pulling back to $E$ one obtains an asymptotically
holomorphic sequence of sections of $\ti{E} | {S}$ that is informally
transverse in a neighborhood of infinity, that is, except on a compact
subset of $E$.  Donaldson's construction \cite{don:symp} although
stated only for compact manifolds, applies equally well to non-compact
manifolds assuming that the section to be perturbed is uniformly
transverse on the complement of a compact set.  The resulting sequence
$\sigma_{E,k}$ is uniformly transverse and consists of asymptotically
holomorphic sections asymptotic to the pull-backs of $\sigma_{e,k}$ on
the ends.  The divisor $D_E = \sigma_{E,k}^{-1}(0)$ is approximately
holomorphic for $k$ sufficiently large and equal to the given divisors
$ \pm (0,\infty) \times D_e$ on the ends, by construction, and
concentrated at $L_b$, over each boundary component
$(\partial {S})_b$.  The stabilizing properties follow. 
\end{proof}

A perturbation scheme similar to the one for Floer trajectories makes
the moduli spaces transverse.  We restrict to the case
$S \cong \R \times [0,1]$ Choose a tamed almost complex structure
$J_E \in \J_\tau(E, \omega_E + \pi_B^* \nu)$ leaving $D_E$ invariant,
so that $D_E$ contains no holomorphic spheres, each holomorphic sphere
meets $D_E$ in at least three points, and each disk with boundary in
$\ti{L}_{0} \cup \ti{L}_{1}$ meets $D_E$ in at least one point.  Since
$D_E$ is only approximately holomorphic with respect to the product
complex structure, the complex structure $J_E$ will not necessarily be
of split form, nor will the projection to ${S}$ necessarily be
$(J_E,j_{S})$-holomorphic away from the ends.  Furthermore, choose
domain-dependent perturbations $F_{\Gamma}$ of the Morse functions
$F_e$ on $\phi_{H_e,1}(L_0) \cap L_1$, so that $F_\Gamma$ is a
perturbation of $F_e$ on the segments that map to
$\phi_{H_e,1}(L_0) \cap L_1$.

Domain-dependent perturbations give a regularized moduli space of
adapted treed strips.  These are maps
$u: C \cong \R \times [0,1] \to E$, homotopic to sections, with the
given Lagrangian boundary conditions and mapping the positive
resp. negative end of the strip to the positive resp. negative end of
$E$.  As before, there is a compactified moduli space
$\ol{\M}(L_0,L_1,\ul{D},\ul{P})$ with a single parametrized strip
component, and also including disk, sphere, and unparametrized strip
components.  For admissible perturbations $\ul{P}$ given as
collections $(P_{\Gamma,E} = (F_{\Gamma,E},J_{\Gamma,E}))$ on $E$, the
moduli space $\ol{\M}(L_0,L_1,\ul{D},\ul{P},(x_e))$ of perturbed maps
to $E$ with boundary in $\ti{L}_0 \cup \ti{L}_1$ and limits
$(x_e), e \in \{ 0 , 1 \} $ has zero and one-dimensional components
that are compact and smooth with the expected boundary.  In particular
the boundary of the one-dimensional moduli spaces
${\M}_1(L_0,L_1,\ul{D},\ul{P}, (x_e))$ are $0$-dimensional strata
$\M_\Gamma(L_0,L_1,\ul{D},\ul{P}, (x_e))$ corresponding to either a
perturbed pseudoholomorphic strip bubbling off on end, or a disk
bubbling off the boundary.

Counting rigid elements of the moduli spaces of parametrized treed
strips defines morphisms between bimodules.  Given two Hamiltonian
perturbations $H_{01}', H_{01}''$ such that the intersections
$\varphi^{01,'}_1(L_0) \cap L_1$ and $\varphi^{01,''}_1(L_0) \cap L_1$
are clean, consider a Hamiltonian perturbation
$H = H_{s} \d s + H_{t} \d t$ equal to $H_{01}' \d t$ for $ s \gg 0$
and to $H_{01} \d t''$ for $s \ll 0$.  Define operations
\begin{multline} 
  \phi^{d|e}: CF(L_0)^{\otimes d} \otimes CF(L_0,L_1;H_{01}') \otimes
  CF(L_1)^{\otimes e} \to CF(L_0,L_1;H_{01}'') \\ {z^0_{1}}
  \otimes \ldots \otimes {z^0_{d}} \otimes {z} \otimes
  {z^1_{1}} \otimes \ldots \otimes {z^1_{e}} \mapsto \\ \sum_{ u
    \in
    \ol{\M}_{d|e,1}(z^0_{1},\ldots,z^0_{d},z,z^1_{1},\ldots,z^1_{e},y)_0}
  (-1)^{\heartsuit + \diamondsuit} \eps(u) y(u) (\sigma(u)!)^{-1}
  q^{A(u)} {y}
\end{multline} 
where the cochain groups involved are defined using $\Lambda$
coefficients; the values $A(u)$ are not necessarily positive because
of the additional term in the energy-area relation \eqref{energyarea}.

\begin{proposition} \label{phi00} For admissible collections of
  perturbation data, the maps $ (\phi^{d|e})_{d,e \ge 0}$
  form a strictly unital morphism of \ainfty bimodules from
  $CF(L_0,L_1;H_{01}')$ to $CF(L_0,L_1;H_{01}'')$.
\end{proposition}

\begin{proof} The proof is essentially the same as that for \ainfty
  algebra morphisms in Theorem \ref{samedegree}.  The true boundary
  components of $\ol{\M}_{d|e,1}(\ul{z}_0,\ul{z}_1,y)_1$ consist of
  configurations with a broken segment in which either treed strip has
  broken off or a treed disk has broken off.  The former case
  corresponds to one of the first two terms in \eqref{morphism} while
  the latter corresponds to the last two terms.  The first term in
  \eqref{morphism} (in which $\mu$ appears before $\phi$) has an
  additional sign coming from the definition of the orientation on
  $\ol{\M}_{d|e,1}$ as an $\R$-bundle over $\ol{\M}_{d|e}$, so that
  the orientation of the fiber corresponds to composing
  parametrization with a translation in the positive direction; this
  means that the positive orientation on the gluing parameter for the
  boundary components corresponding to terms of the first type becomes
  identified with the negative orientation on these fibers, giving
  rise to the additional sign.  The degree of the morphism is zero,
  which causes those contributions to the sign in \eqref{morphism} to
  vanish.  This leaves the contributions from $\aleph$, which are
  similar to those dealt with before and left to the reader.
 \end{proof}

 \section{Homotopies}
\label{twice} 

In this section we show that the morphisms of bimodules introduced
above satisfy a gluing law of the type found in topological field
theory.  The construction uses a moduli space of {\em twice}
parametrized treed strips defined as follows.  Let $C$ be a
holomorphic strip, that is, disk with markings $z_-, z_+$ on the
boundary.  A {\em twice-parametrization} is pair
$\ul{\phi} = (\phi_0,\phi_1)$ of holomorphic isomorphisms
\[ \phi_k: C - \{ z_-, z_+ \} \to \R \times [ 0,1 ], \quad k \in \{ 0
, 1 \} .\]
An {\em isomorphism} of a twice-parametrized disks
$(C_l, \ul{\phi}_l), l \in \{ 1, 2 \}$ is a holomorphic isomorphism
$C_1$ to $C_2$ that intertwines parametrizations
$\phi_{k,l}, k \in \{ 0,1 \}$ for $l \in \{ 1,2 \}$.  Each
parametrization is equivalent to an affine structure on
$C - \{z_+ \} $, and in this way a twice-parametrized strip is
equivalent to a twice-quilted disk in the language of Chapter
\ref{homsec}, but with a distinguished incoming boundary marking.  In
particular, the moduli space of twice-parametrized treed strips with
additional boundary markings and interior markings has a compact
moduli space with locally toric singularities.

In the compactification, parametrized or unparametrized strips can now
bubble off on either end.  Suppose that we have chosen Hamiltonian
perturbations $H_{01}^a, a = 0,1,2$ so that the intersections
$\phi_t^a(L_0) \cap L_1$ are clean where $\phi_t^a$ denotes the time
$t$ flow of $H_{01}^a$.  We extend the Hamiltonian perturbation over
the universal twice-parametrized strip and choose generic
domain-dependent almost complex structures.  The weighted counts of
twice-parametrized strips define operations
\begin{multline} 
  \tau^{d|e}: CF(L_0)^{\otimes d} \otimes CF(L_0,L_1;H_{01}^0) \otimes
  CF(L_1)^{\otimes e} \to CF(L_0,L_1;H_{01}^2)[-1] \\
  (z^0_{1},\ldots,z^0_{d},z,\ldots,z^1_{1},\ldots,z^1_{e},y) \\
  \mapsto \sum_{ u \in
    {\M}_{d|e,2}(z^0_{1},\ldots,z^0_{d},z,z^1_{1},\ldots,z^1_{e},y)_0}
  (-1)^{\heartsuit + \diamondsuit} \eps(u) y(u) (\sigma(u)!)^{-1}
  q^{A(u)} {y}.
\end{multline}

\begin{theorem} Suppose that admissible perturbation data
  $\ul{P}^{ab} = (P_\Gamma^{ab})$ for once-parametrized strips have been
  chosen giving rise to morphisms of \ainfty bimodules
\[\phi_{ab}: CF(L_0,L_1;H_{01}^a) \to CF(L_0,L_1;H_{01}^b ) , \quad 0
\leq a < b  \leq 2 .\]
For any admissible collection of perturbations $\ul{P} = (P_\Gamma)$
extending the given collections $\ul{P}^{ab} = (P_\Gamma^{ab})$ over
the universal curves over moduli spaces of twice-parametrized strips,
the resulting operations $(\tau^{d|e})_{d,e \ge 0}$ define a homotopy
of morphisms of \ainfty bimodules
\[ \phi_{12} \circ \phi_{01} \simeq_{\tau} \phi_{12}  \in 
\Hom(CF(L_0,L_1;H_{01}^0), CF(L_0,L_1;H_{01}^2)) .\]
\end{theorem}

\begin{proof} 
  The statement is a consequence of the description of
  twice-parametrized strips, similar to that of the existence of
  homotopies between \ainfty functors defined by quilted disks in
  Theorem \ref{twicequilted}.  The components of the boundary of the
  moduli space of twice-quilted strips
  $
  \ol{\M}_{d|e,2}(z^0_{1},\ldots,z^0_{d},z,z^1_{1},\ldots,z^1_{e},y)_1$
  that do {\em not} involve the parametrized strip breaking, or the
  shaded region disappearing, correspond the terms on the left-hand
  side of \eqref{morphism}.  The strata where the twice-parametrized
  strip breaks into once-parametrized strips correspond to the
  contributions to the composition $\phi_{12} \circ \phi_{01}$, while
  the components where the parametrized strip vanishes give the
  identity morphism of \ainfty modules.  The sign computation is the
  same.
\end{proof}

For any $b_0 \in {MC}(L_0), b_1 \in {MC}(L_1)$ we denote by
$ HF(L_0,L_1; H_{01}; b_0,b_1)$ the cohomology of the operator
$\mu^1_{b_0,b_1}$.  It follows as in Lemma \ref{uptogauge} that any
morphism of \ainfty bimodules corresponding to difference choices of
perturbations $H_{01}, H_{01}'$ induces a map
$$ HF(L_0,L_1; H_{01}; b_0,b_1) \to HF(L_0,L_1; H_{01}'; b_0, b_1
) $$
of cohomology groups, functorially.  As a result,
$HF(L_0,L_1; H_{01}; b_0,b_1)$ is independent of the choice of the
choice of Hamiltonian perturbation.  On the other hand, for the case
$L_0 = L_1 = L $ the \ainfty bimodule $CF(L_0,L_1;0)$ is equal to the
span of $x \neq x^{\whitet}, x^{\greyt}$ in $CF(L)$ itself considered
as an \ainfty bimodule over $CF(L)$, so that using Lemma
\ref{nochange} we have 
$$ HF(L,L; b; b) \cong HF(L,b) .$$

\begin{corollary} \label{nondispl} If a compact rational Lagrangian
  brane $L$ in a compact rational symplectic manifold $X$ is
  Hamiltonian displaceable, then $HF(L,b)$ vanishes for every
  $b \in {MC}(L)$.
\end{corollary} 

\begin{proof} Let $\phi: X \to X$ be a Hamiltonian diffeomorphism with
  $\phi(L) \cap L = \emptyset$.  Then the homotopy equivalences
  $CF(L) \simeq CF(L,L) \simeq CF(L,\phi(L)) = \{ 0 \}$ imply that
  $HF(L,b) = \{ 0 \}$ for all $b \in {MC}(L)$.
\end{proof}

\chapter{Broken Fukaya algebras}

In this section we introduce a version of the Fukaya algebra for
Lagrangians in {\em broken} symplectic manifolds along the lines of
{\em symplectic field theory} as introduced by
Eliashberg-Givental-Hofer \cite{egh:sft}.  The symplectic manifold is
degenerated, through neck stretching, to a {\em broken} symplectic
manifold.  A sequence of pseudoholomorphic curves with respect to the
degenerating almost complex structure converges to a {\em broken
  pseudoholomorphic map}: a collection of pseudoholomorphic curves in
the pieces as well as maps to the neck region.  In the version that we
consider here introduced by Bourgeois \cite{bourg:mb}, these
components are connected by gradient trajectories of possibly finite
or infinite length.

\section{Broken curves} 

First we describe the kind of domains that appear in the particular
kind of broken limit we will consider, following
Bourgeois-Eliashberg-Hofer-Wysocki-Zehnder \cite{bo:com}.

\begin{definition} \label{brokencurves} Let $n,m,s \ge 0$ be integers.
A {\em level of a broken curve} with $n$ boundary
  markings, $m$ interior markings, and $s$ sublevels consists of
\begin{enumerate} 
\item a sequence 
$ C= (C_1,\lldots, C_s) $ 
of treed nodal curves with boundary, called {\em sublevels}; in our
situation only the first piece $C_1$ is allowed to have non-empty
boundary:
$ \partial C_2 = \lldots \partial C_s = \emptyset .$
\item interior markings 
$\ul{z}_i^\pm \subset C_i, \quad i = 1,\lldots, s .$
\item a collection of {\em (possibly broken) finite and semi-infinite
  edges attached to boundary points} that are intervals 
$I_1,\lldots, I_m$
connecting boundary points in components of $C_i$ for
$i = 1,\lldots, s$ with lengths $\ell(I_i)$; since in our situation
only $C_1$ is allowed to have non-empty boundary, these edges only
connecting components of $C_1$; and
\item a sequence of {\em intervals}
$I^i_1,\lldots, I^i_{s(i)}$
attached to interior points in $C_i$ with finite lengths
$\ell(I^i_k)$ independent of $k$;
 \end{enumerate}
 Out of the data above one constructs a topological space $C$ by
 removing the nodes and gluing in the intervals
 $I^i_1,\lldots, I^i_{s(i)}$.
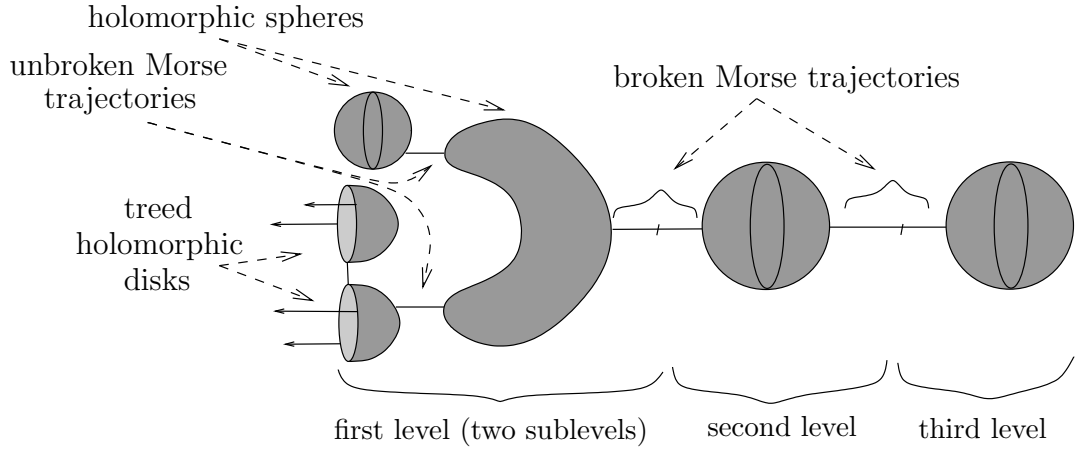
\begin{figure}[ht]
\begin{picture}(0,0)%
\includegraphics[height=2in]{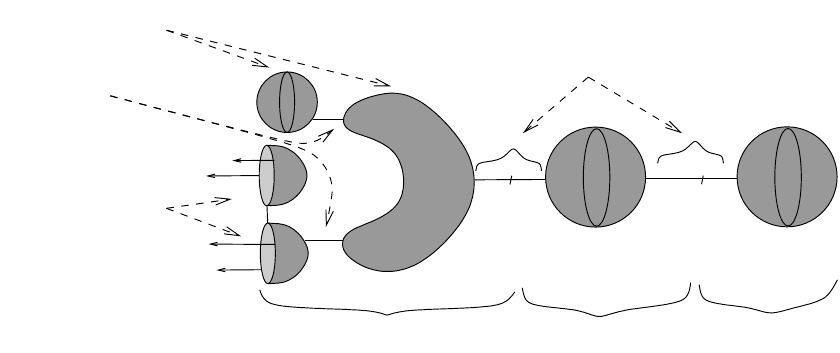}%
\end{picture}%
\setlength{\unitlength}{3947sp}%
\begingroup\makeatletter\ifx\SetFigFont\undefined%
\gdef\SetFigFont#1#2#3#4#5{%
  \reset@font\fontsize{#1}{#2pt}%
  \fontfamily{#3}\fontseries{#4}\fontshape{#5}%
  \selectfont}%
\fi\endgroup%
\begin{picture}(6711,2809)(545,-2261)
\put(2098,-2201){\makebox(0,0)[lb]{first level (two sublevels) }%
}
\put(4442,-2180){\makebox(0,0)[lb]{\smash{{{second level}%
}}}}
\put(5567,-2196){\makebox(0,0)[lb]{\smash{{{third level}%
}}}}
\put(4151, -214){\makebox(0,0)[lb]{\smash{{{broken Morse trajectories}%
}}}}
\put(1276,-1261){\makebox(0,0)[lb]{\smash{{{disks}%
}}}}
\put(976,-1086){\makebox(0,0)[lb]{\smash{{{holomorphic}%
}}}}
\put(1051,89){\makebox(0,0)[lb]{\smash{{{holomorphic spheres}%
}}}}
\put(560, -286){\makebox(0,0)[lb]{\smash{{{unbroken Morse}%
}}}}
\put(774,-421){\makebox(0,0)[lb]{\smash{{{trajectories}%
}}}}
\put(1276,-911){\makebox(0,0)[lb]{\smash{{{treed}%
}}}}
\end{picture}%
\caption{A broken disk} 
\label{bdisk} 
\end{figure} 
A {\em broken curve} with $k$ levels is obtained from $k$ curves at a
single level $C_1,\lldots,C_k$ by gluing together the endpoints of the
semi-infinite interior edges at infinity, as in Figure \ref{bdisk}.
The {\em combinatorial type} of a broken curve $C$ is the graph
$\Gamma$ obtained by gluing together the types
$\Gamma_1,\ldots,\Gamma_k$ of the levels, with the partition into
sub-graphs remembered.  A broken disk $C$ is defined similarly to a
broken curve but the first level $C_1$ of $C$ is a disjoint union of
treed disk and sphere components with connected boundary and segments
attached, while the other pieces have surface parts that are spheres;
furthermore, the combinatorial type $\Gamma(C)$ is a tree.  Weightings
are defined as in Definition \ref{wdef}: each semi-infinite edge
$e \in \Edge_{\rightarrow}(\Gamma(C))$ is assigned a weight
$\rho(e) \in [0,\infty]$.  A broken curve $C$ is {\em stable} if it
has only finitely many automorphisms $\phi: C \to C$, except for
automorphisms of infinite length segments $C_i \cong \R$ with one
weighted end $e_- \in \mE(C)$ and one unweighted end $e_+ \in \mE$.
\end{definition}  

\section{Broken maps} 

A broken map is a map from a broken curve into a broken symplectic
manifold, defined as follows.

\begin{definition} \label{bsymp} 
\begin{enumerate} 
\item \label{bsympz} {\rm (Broken symplectic manifold)} Let $X$ be a
  compact rational symplectic manifold.  Let $Z \subset X$ be a
  coisotropic hypersurface separating $X \backslash Z$ into components
  $X_{\subset}^\circ,X_{\supset}^\circ$.  Suppose that $Z$ has null
  foliation
\[\ker(\omega | Z) \subset TZ\] 
that is a circle fibration over a symplectic manifold $Y$; that is,
there exists a {\em Reeb vector field} $v \in \Vect(Z)$ taking values
in the null foliation whose flow defines an action of $S^1$ on $Z$ so
that $Y := Z/S^1 $ is a smooth manifold.  By the coisotropic embedding
theorem, a neighborhood of $Z$ in $X$ is symplectomorphic to
$(-\eps,\eps) \times Z$ for $\eps > 0$ small equipped with a
symplectic structure arising from a connection.  By a construction of
Lerman \cite{le:sy2} the unions \llabel{unions} \label{unionsp}
\[ X_\subset :=  X_\subset^\circ \cup Y, 
\quad X_\supset := X_\supset^\circ \cup Y \]
have the structure of symplectic submanifolds.  Let $N_\pm \to Y$
denote the normal bundle of $Y$ in $X_\subset,X_\supset$, and $N_\pm
\oplus \ul{\C}$ the sums with the trivial bundle $\ul{\C} = Y \times
\C$.  Denote by $\P(N_+ \oplus \ul{\C}) \cong \P(N_- \oplus \ul{\C}) $
the projectivized normal bundle, where the isomorphism is induced from
\[ \P(N_+ \oplus \ul{\C}) \cong \P( (N_+ \oplus \ul{\C}) \otimes N_-)) = \P(
 N_- \oplus \ul{\C}) .\]
 The {\em broken symplectic manifold} arising from the triple $
 (X_{\subset}, X_{\supset}, Y) $ is the topological space 
 \[\XX = X_{\subset} \cup_Y X_{\supset} \]
 obtained by identifying the copies of $Y$ in $X_{\subset}$ and
 $X_{\supset}$.  Thus $\XX$ is a stratified space and the link of $Y$
 in $\XX$ is a disjoint union of two circles.  The space $Y$ comes
 equipped with an isomorphism of normal bundles
 \begin{equation} \label{niso} (TX_\subset)_Y/TY \cong (
   (TX_{\supset})_T/TY)^{-1} .\end{equation}
\item {\rm (Multiply broken symplectic manifold)} For an integer $m
  \ge 1$ define the {\em $m-1$-broken symplectic manifold}
\[ \XX[m] = X_{\subset} \cup_Y \P(N_\pm \oplus \ul{\C}) \cup_Y \P(N_\pm \oplus \ul{\C}) \cup_Y
  \lldots \cup_Y X_{\supset} \]
where there are $m-2$ copies of $\P(N_\pm \oplus \ul{\C})$ called {\em broken
  levels}.  Define
\begin{equation} \label{pieces} \XX[m]_0 = X_{\subset}, \ \ \XX[m]_1 = \P(N_\pm \oplus \ul{\C}), \ \ ,\lldots,
\XX[m]_m = X_{\supset} .\end{equation}
There is a natural action of $\C^\times$ on $\P(N \oplus\C)$ given by
scalar multiplication on each projectivized normal bundle:
\[ \C^\times \times \P(N_\pm \oplus \ul{\C}) \to \P(N_\pm \oplus \ul{\C}), \quad
(z,[n,w]) \mapsto z [n,w] := [zn,w] .\]
The fixed points of the $\C^\times$ action are the divisors at $0$ and $\infty$:
\[ \P(N_\pm \oplus \ul{\C})^{\C^\times} = \{ [n,0] \} \cup \{ [0,w] \} \]
where $n$ reps. $w$ ranges over vectors in $N_\pm$ resp. $\ul{\C}$
\item An almost complex structure $J \in \J(\R \times Z)$ is of {\em
  cylindrical form} if there exists an almost complex structure $J_Y$
  on $Y$ such that the projection $\pi_Y : \R \times Z \to Y$ is
  almost complex and $J$ is invariant under the $\C^\times$-action on
  $\R \times Z$ induced from the embedding in $\P(N_\pm \oplus \ul{\C})$
  given by 
\[ \C^\times \times \P(N_\pm \oplus \ul{\C}) \to \P(N_\pm \oplus \ul{\C}), \quad
  s \exp( i t) (s_0,z) = (s_0 + s, \exp(it)z) .\]
That is,
\[D \pi_Y J = J_Y D \pi_Y, \quad  J \in \J(\R \times Z)^{\C^\times} .\]
Denote by $\J^{\on{cyl}}(\R \times Z)$ and for any $\eps > 0$ denote
by $\J^{\on{cyl}}((0,\infty) \times Z)$ the image of
$\J^{\on{cyl}}(\R \times Z)$ under restriction.  Denote by
\[ \J(\XX) = \J(X_{\subset}^\circ) \times_{\J^{\on{cyl}}((0,\infty)
  \times Z)} \J(X_{\supset}^\circ) \]
the fiber product consisting of tamed almost complex structures of
cylindrical form on the neck region.  (Note that this definition differs from
the one in \cite[Section 2.1]{bo:com}, which does not require the
invariance of the almost complex structure under the Reeb flow and so
does not suffice for our purposes.)  \llabel{reebflow}  \label{reebflowp}
\end{enumerate} 
\end{definition} 

\begin{definition} \label{bmaps} {\rm (Broken maps)} Let
  $\XX = X_{\subset} \cup_Y X_{\supset} $ be a broken symplectic
  manifold as above, and $L \subset X_{\subset}$ a Lagrangian disjoint
  from $Y$.  Let $J \in \J(\XX)$ be an almost complex structure on of
  cylindrical form, $H$ be a Morse function on $Y$ and $(F,G)$ a
  Morse-Smale pair on $L$.  A {\em broken map} to $\XX$ with boundary
  values in $L$ consists of:
\begin{enumerate} 
\item {\rm (Broken curve)} a broken curve $C = (C_0,\lldots, C_p)$;
\item {\rm (Broken map)} a map $u: C \to \XX$, that is, a collection of
  maps  (notation from \eqref{pieces}) 
\[u_k: C_k \to \XX[p]_k, \quad k = 0,\lldots, p ;\]
\item {\rm (Framings at infinity)} for each pair of interior nodes
  $w_\pm \in C$ connected by an interior gradient trajectory
  $T_e \subset C$, a framing $\tau(w_\pm): T_{w_\pm} S \to \C$ of the
  tangent spaces $T_{w_\pm} S$ to the surface $S$;
\end{enumerate}
satisfying the following conditions:
\begin{enumerate} 
\item {\rm (Pseudoholomorphicity)} On the two-dimensional part $S
  \subset C$, the map $u$ is $J$-holomorphic, that is, $ \olp_{J}
  (u |_S) = 0 .$
\item {\rm (Gradient flow in the Lagrangian)} On the one-dimensional
  part $T_{\white} \subset C$ connecting boundary nodes, $u$ is a segment of
  a gradient trajectory on each interval component for the Morse
  function $F$ on $L$:
\[ \left(\ddt + \grad_F \right) (u | T_{\white}) = 0 .\]
\llabel{lastline} \label{lastlinep}
\item {\rm (Intersection multiplicity)} If a pseudoholomorphic map $u:
  C \to X$ has isolated intersections with an almost complex
  codimension two submanifold $Y \subset X$ then at each point $z \in
  u^{-1}(Y)$ there is a positive {\em intersection multiplicity}
  $\mu(u,z) \in \Z_{> 0}$ describing the winding number of a small loop
  counterclockwise around $Y$:
\[ \mu(u,z) = [u( z + r \exp( i
\theta)) |_{\theta \in [0,2\pi]}] \in \pi_1(U - (U \cap Y)) \cong \Z\]
where $ U$ is a contractible open neighborhood of $z$ and $r$ is
sufficiently small so that $u( z + r \exp(  i \theta)) \in U$ for
all $\theta \in [0,2 \pi]$.
\item {\rm (Gradient flow in the manifold)} On the one-dimensional
  part $T_{\black} \subset C$ connecting interior nodes, $u$ is a segment of
  a gradient trajectory on each interval component
\[ \left(\ddt + \grad_H \right) (u | T_{\black}) = 0 ;\]
\end{enumerate} 
and satisfying the following:
\begin{enumerate} 
\item {\rm (Matching condition for multiplicities)} For any pair of
  nodes $w_\pm$ of the domain $C$ connected by an interior trajectory,
  the intersection multiplicities $\mu(w_\pm)$ of the map $u$ with the
  hypersurface $Y$ are equal:
  \[ \mu(w_-) = \mu(w_+) .\]
\item {\rm (Matching condition for framings)} The local framings
  satisfy the following condition: Choose local coordinates $z_\pm$ on
  a neighborhood of $w_\pm$ inducing the given framings $\tau(w_\pm)$
  on $T_{w_\pm}C$.  Let $X_\pm \subset \XX[k]$ denote the components
  receiving the components $S_\pm \subset C$ containing $w_\pm$ and
  choose local coordinates $(z_1^\pm,\ldots, z_n^\pm)$ so that
  $Y = \{ z_1^\pm = 0 \}$ and choose local coordinates on $X_\pm$
  compatible with the isomorphism of normal bundles \eqref{niso}.
  Then in local coordinates on $X_\pm$ the map
  $z_\pm \mapsto u_1(z_\pm)$ has leading order term
  \begin{equation} \label{fmatch} u_1(z_\pm) \sim z_\pm^{\mu(w_\pm)}
    . \end{equation}
\end{enumerate}
We also define the following notions:
\begin{enumerate} 
\item {\rm (Isomorphisms of broken maps)} An isomorphism between
  broken maps
\[u_i: C_i \to \XX[k], \ i  \in \{  0,1 \}\] 
is an isomorphism of domains $\phi: C_0 \to C_1$ together with an
element $g \in (\C^\times)^{k-1}$ such that $u_1 \circ \phi = g u_0$,
and so that the framings $\tau_1(w_\pm)$ are equivalent up to
simultaneous rotation:
$\tau_1(w_\pm) = \zeta D_{w_\pm} \phi \tau_0(w_\pm)$ for some
$\zeta \in \C^\times$.  Note that the leading order condition
\eqref{fmatch} determines the framings $\tau(w_\pm)$ up to an
$\mu(w_\pm)$-order root of unity.  Since two framings related by
simultaneous rotation are considered equivalent, there are
$\mu(w_\pm)$ inequivalent framings allowed at each node connected by
an internal edge.
\item{} {\rm (Combinatorial type)} The {\em combinatorial type}
  $\Gamma$ of a broken map $u: C \to \XX$ is the combinatorial type of
  the underlying curve $C$, but with the additional data of the
  homology class $u_{i,*}[C_i]$ of each component $C_i$ (as a
  labelling of the vertices) and the intersection multiplicities
  $m(z_i) \in \Z_{\ge 0}$ with the stabilizing divisor $D$ at each
  attaching point $z_i$ of an edge or interior node.  Let $\Gamma$ be
  a type with $n$ leaves (corresponding to trajectories of the Morse
  function on the Lagrangian) and $l$ broken Morse trajectories on the
  degenerating divisor.  An {\em admissible labelling} for a $\Gamma$
  is a collection $\ul{l} \in \cI(L)^{n+1}$ such that whenever the
  corresponding label is $x^{\greyt}$ resp. $x^{\greyt}$
  resp. $x^{\whitet}$ or the corresponding leaf has weight $0$
  resp. $[0,\infty]$ resp. $\infty$.
\end{enumerate}
Denote by $\ol{\M}(L,\XX,\DD)$ the union over types, and by 
$\M(L,\XX,\DD)$ the locus of types formally of top dimension where 
there are at most two levels $C_0,C_1 \subset C$ and each edge 
$T_e, e \in \Edge(\Gamma)$ not connecting two different levels has 
finite and non-zero length.    This ends the Definition.
\end{definition} 

\begin{remark} In contrast to the unbroken case, configurations 
  $u: C \to \XX$ with two levels are not positive codimension since 
  there is no gluing construction which produces a broken map.  On the 
  other hand, configurations $u$ with a neck piece $u_i$ may be glued 
  to broken maps in two different ways, depending on whether that neck 
  piece $u_i$ is glued with the piece $u_0$ or $u_p$ mapping to 
  $X_\subset$ or $X_\supset$.\end{remark}

\label{weighted} 

Broken maps may be viewed as pseudoholomorphic maps of curves with
cylindrical ends, by the removal of singularities argument explained
in Tehrani-Zinger \cite[Lemma 6.6]{tz}.  This leads to a natural
notion of convergence in which the moduli space of broken maps of any
given combinatorial type is compact.  The compactness statement is
essentially a special case of compactness in symplectic field theory
\cite{bo:com}, \cite{abbas:com}, although the particular set-up here has not been considered before.  First we recall terminology for
the type of cylindrical ends we consider.  First we introduce notation
for the symplectic manifolds with cylindrical ends: Let $X^\circ_\pm$
denote the manifold obtained by removing the divisor $Y$, or more
generally, for the intermediate pieces
$\P(N_\pm \oplus \ul{\C})^\circ \cong \R \times Z$ the manifold
obtained by removing the divisors at zero and infinity, isomorphic to
$Y$.  We identify a neighborhood of infinity in $\P(N_\pm \oplus C)$
with $\R_{ > 0 } \times Z$ with the almost complex structure induced
from a connection on $Z$ and the given almost complex structure on
$Y$.

Recall that the notion of Hofer energy for symplectizations of contact
manifolds with fibrating null-foliations, which is a special case of a
more general definition for stable Hamiltonian structures in
\cite{bo:com}.  Let $X = \R \times Z$, where $Z$ is equipped with
closed two-form $\omega_Z \in \Omega^2(Z)$ with fibrating
null-foliation and connection form $\alpha \in \Omega^1(Z)$.

\begin{definition}  {\rm (Action and energy)}  
\begin{enumerate} 
\item {\rm (Horizontal energy)} The {\em horizontal energy} of a
  holomorphic map $ u = (\phi,v): (C,j) \to (\R \times Z,J)$ is
  (\cite[5.3]{bo:com})
\[ E^h(u) = \int_C v^* \omega_Z .\]
\item {\rm (Vertical energy)} The {\em vertical energy} of a
  holomorphic map $ u = (\phi,v): (C,j) \to (\R \times Z,J)$ is
  (\cite[5.3]{bo:com})
\begin{equation} \label{alphaen}
E^v(u) = \sup_{\zeta} \int_C (\zeta \circ \phi) \d \phi \wedge v^*
\alpha \end{equation}
where the supremum is taken over the set of all non-negative
$C^\infty$ functions 
\[\zeta: \R \to \R, \quad  \int_\R \zeta(s) \d s = 1 \]
with compact support.
\item {\rm (Hofer energy)} The {\em Hofer energy} of a holomorphic map
  $ u = (\phi,v): (C,j) \to (\R \times Z,J)$ is (\cite[5.3]{bo:com})
  is the sum
\[ E(u) = E^h(u) + E^v(u) .\]
\item {\rm (Generalization to manifolds with cylindrical ends)}
  Suppose that $X^\circ$ is a symplectic manifold with cylindrical end
  modelled on $\R_{> 0} \times Z$.  \llabel{Zfix} \label{Zfixp} The
  vertical energy $E^v(u)$ is defined as before in \eqref{alphaen}.
  The Hofer energy $E(u)$ of a map $u: C^\circ \to X^\circ$ from a
  surface $C^\circ$ with cylindrical ends to $X^\circ$ is defined by
  dividing $X^\circ$ into a compact piece $X^{\on{com}}$ and a
  cylindrical end \llabel{Rfix} \label{Rfixp} diffeomorphic to
  $\R_{> 0} \times Z$.  Then we set
\[ E(u) = E(u | X^{\on{com}} ) + E(u | \R_{\ge 0} \times Z) .\]
\end{enumerate} 
\end{definition} 

A compact, Hausdorff moduli space of broken maps may be obtained by
imposing an energy bound as well as a stability condition.  As in
previous cases, the moduli space so obtained will be regularized later
using Cieliebak-Mohnke perturbations \cite{cm:trans}, in the cases of
low expected dimension.

\begin{definition} {\rm (Stable broken maps)}  
  A broken map $u :C \to \XX[k]$ is {\em stable} if it has only
  finitely many automorphisms $\phi$, except for automorphisms of
  infinite length segments $C_i \cong \R$ with one weighted end and
  one unweighted end.  This means in particular at least one component
  at each level $u_i: C_i \to \XX[p]$ is not a trivial cylinder.
\end{definition}

\begin{theorem} \label{sftcompact}
  (c.f. Bourgeois-Eliashberg-Hofer-Wysocki-Zehnder \cite{bo:com}) Any
  sequence of finite energy stable broken pseudoholomorphic maps
  $u_\nu: C_\nu \to \XX[k]^\circ$ with bounded Hofer energy
  $\sup_{\nu} E(u_\nu) < \infty $ has a Gromov convergent subsequence
  to a stable limit, and any such convergent sequence has a unique
  limit.  \label{equant}
\end{theorem} 

\begin{proof}[Sketch of proof]
We will not give a complete proof but
  rather indicate how the proof can be adapted from the published
  treatments of sft compactness in
  Bourgeois-Eliashberg-Hofer-Wysocki-Zehnder \cite{bo:com}, Abbas
  \cite{abbas:com}, and Cieliebak-Mohnke \cite{cm:com}.  In the
  case that the almost complex structure is domain-independent, and
  preserves the horizontal subspace, Theorem \ref{sftcompact} with the
  Hofer energy bound is essentially a special case of the compactness
  result in symplectic field theory \cite[Section 5.4]{bo:com} (with
  further details and corrections in Abbas \cite{abbas:com} and
  alternative approach given in \cite{cm:com}) with the additional
  complication of Lagrangian boundary conditions.  Since the
  Lagrangian $L$ is compact in $X^\circ_\subset$, the Lagrangian
  boundary conditions do not affect any of the arguments.  Thus the
  question is to extend these results to domain-dependent almost
  complex structures, and the essential point is that our
  domain-dependent almost complex structures are domain-independent
  near the punctures.

  Our particular setup corresponds to the case of relative stable maps
  in Ionel-Parker \cite{io:rel} and Li-Ruan \cite{liruan:surg}, as
  explained in Bourgeois et al. \cite[Remark 5.9]{bo:com}.  In
  particular, asymptotic convergence follows from asymptotic
  convergence for holomorphic maps to $Y$; energy quantization for
  disks in $X_{\subset}$ implies energy quantization for finite energy
  holomorphic maps of half-cylinders to $X_{\subset}^\circ$, where the
  boundary of the cylinder maps to the Lagrangian $L$.  Energy
  quantization for holomorphic maps of spheres to $Y$ implies energy
  quantization for maps of holomorphic spheres to
  $\P(N_\pm \oplus \ul{\C})$: there exists a constant $\hbar > 0$ such
  that any holomorphic map $\P(N_\pm \oplus \ul{\C})$ with non-trivial
  projection to $Y$ has energy at least $\hbar$.  Matching of
  intersection multiplicities is \cite[Remark 5.9]{bo:com},
  Tehrani-Zinger \cite[Lemma 6.6]{tz}: By removal of singularities,
  there is a one-to-one correspondence between finite energy
  holomorphic curves in $X^\circ_\pm$ resp. $\XX[k]^\circ$ and those
  in $X_\pm$ resp. $\XX[k]$ that are not contained in the divisor $Y$
  resp. divisors at zero and infinity.  Thus the intersection
  multiplicity is the degree of the cover of the Reeb orbit at
  infinity.  It suffices to show that on each tree or surface part of
  the domain, a subsequence converges to some limit in the Gromov
  sense.  Since each domain is stable, each surface part has a unique
  hyperbolic metric so that the boundary is totally geodesic, see
  Abbas \cite[I.3.3]{abbas:com}.  Denote by
  $r_\nu: C_\nu \to \R_{> 0}$ the injectivity radius.  The argument of
  Bourgeois et al. \cite[Chapter 10]{bo:com}, see also Abbas
  \cite{abbas:com}, shows that after adding finitely many sequences of
  points to the domain we may assume that the domain $C_\nu$ converges
  to a limit $C$ such that the first derivative
  $\sup | \d u_\nu |/r_\nu$ is bounded with respect to the hyperbolic
  metric on the surface part, and with respect to the given metric on
  the tree part.  Thus there exists a limiting map
  $u: C^\times \to \XX$ on the complement $C^\times$ of the nodes so
  that on compact subsets of the complement of the nodes a subsequence
  of $u_\nu$ converges to $u$ in all derivatives.  Removal of
  singularities and matching conditions then follows from the
  corresponding results for holomorphic maps: the matching condition
  for nodes mapping into the cylindrical end is simply the matching
  condition for the maps to $Y$, in addition to matching of
  intersection degrees which is immediate from the description as a
  winding number.  Convergence on the tree part of the domain follows
  from uniqueness of solutions to ordinary differential equations.
  The extension to sequences with bounded area is \cite[Lemma
  9.2]{bo:com}, or rather, the extension of that Lemma to curves with
  Lagrangian boundary conditions for Lagrangians not meeting the neck
  region, for which the proof is the same.
\end{proof}

\begin{remark} \label{expdec} {\rm (Exponential decay)} Any
  holomorphic map $u: C^\circ \to X^\circ$ with finite Hofer energy
  converges {\em exponentially fast} to a Reeb orbit along each
  cylindrical end: In coordinates $s,t$ on the cylindrical end
  diffeomorphic to $\R_{> 0} \times Z$, there exists a constant $C$
  and constants $s_0,s_1 > 0$ such that for $s > s_1$ the distance
  $ \dist( u(s,t), ( \mu (s - s_0),\gamma(t))) $ is bounded by
  $C \exp( - s) $.  This follows from the correspondence with
  holomorphic maps to the compactification in e.g.  Tehrani-Zinger
  \cite[Lemma 6.6]{tz}.
\end{remark}

\section{Broken perturbations} 

In order to achieve transversality we introduce stabilizing divisors
satisfying a compatibility condition with the degeneration and
introduce domain-dependent almost complex structures and Morse
functions.  By a broken divisor we mean a divisor that arises from
degeneration of a divisor in the original manifold via neck
stretching.

\begin{definition} \label{bdiv}  {\rm (Broken divisors)}  
A {\em broken divisor} for the broken almost complex manifold $\XX :=
X_{\subset} \cup_Y X_{\supset}$ consists of a pair 
\[\DD = (D_{\subset},D_{\supset}), \quad D_{\subset} \subset X_\subset, \ D_{\supset}
\subset X_\supset\] 
of codimension two almost complex submanifolds
$D_{\subset}, D_{\supset}$ such that each intersection
\[ D_{\subset} \cap Y = D_{\supset} \cap Y = D_Y \]  
is a codimension two almost complex submanifold $D_Y$ in $Y$.  Given a
broken divisor $\DD =
(D_{\subset},D_{\supset})$ as above we obtain a divisor
\[D_N := \P( N_\pm | D_Y \oplus \ul{\C}) \subset \P(N_\pm \oplus \ul{\C}) .\]
We suppose that each $[D_{\subset},D_{\supset}]$ is dual to a large
multiple of the symplectic class on $X_\pm$, that is, 
\[ [D_{\subset}] =
k [\omega_{\subset}], \quad [D_{\supset}] = k[\omega_\supset] .\]  
Then 
\[ [D_N] = k \pi_Y^* [\omega_Y] \]
where $\pi_Y$ is projection onto $Y$, and as a result does not
represent a multiple of any symplectic class on
$\P(N_\pm \oplus \ul{\C})$.  Thus the divisor $D_N$ can be disjoint
from non-constant holomorphic spheres in $\P(N_\pm \oplus \ul{\C})$,
namely the fibers.  However, holomorphic spheres whose projections to
$Y$ are non-constant automatically intersect $D_N$.
\end{definition}  

As in the unbroken case, transversality uses almost complex structures
equal to a fixed almost complex structure on the stabilizing divisor.
We introduce the following notations.  For a symplectic manifold
$X^\circ$ with cylindrical end, denote by $\J(X^\circ)$ the space of
tamed almost complex structures on $X^\circ$ that are of cylindrical
form on the end.  Given
$J_{\DD} = ( J_{D_{\subset}}, J_{D_{\supset}} )\in \J(\XX)$, denote by
$\J(\XX , J_{\DD})$ the space of almost complex structures in
$\J(\XX)$ that agree with $J_{\DD}$ on $D_{\subset},D_{\supset}$:
\begin{equation}\label{bstab} \J(\XX, J_{\DD}) = \Set{
    (J_\subset,J_\supset) \in \J(\XX)  \ | \ J_\subset
    |D_{\subset} = J_{D_\subset}, \ J_\supset | D_{\supset} =
    J_{D_\supset} } .\end{equation} 
Fix a tamed almost complex structure $J_{\DD}$ such that
$D_{\subset},D_{\supset}$ contains no non-constant
$J_{\DD}$-holomorphic spheres of any energy and any holomorphic sphere
meets $D_{\subset},D_{\supset}$ in at least three points, as in
\cite[Proposition 8.14]{cm:trans}.  By \cite[Proposition
8.4]{cw:traj}, for any energy $E> 0$ there exists a contractible open
neighborhood $\J^*(\XX,J_{\DD}, E)$ of $J_{\DD}$ agreeing with
$J_{\DD}$ on $D_\supset, D_\subset$ with the property that
$D_{\subset},D_{\supset}$ still contains no non-constant holomorphic
spheres and any holomorphic sphere of energy at most $E$ meets
$D_{\subset},D_{\supset}$ in at least three points.  Denote the fiber
product
\[ \J^*(\XX,J_\DD,E) := \J^*(X_{\subset},J_{D_{\subset}}, E) \times_{\J^{\on{cyl}}((0,\infty) \times Z)}
\J^*(X_{\supset},J_{D_{\supset}}, E) .\]

Given a broken divisor define perturbation data for a broken
symplectic manifold as before, but we also perturb the Morse function
on the separating hypersurface $Y = X_\subset \cap X_\supset$.  For
base almost complex structures $J_{D,\pm}$ agreeing on $Y$, a {\em
  perturbation datum} for type $\Gamma$ of broken maps is a datum
\[ P_\Gamma = (J_\Gamma,F_\Gamma,G_\Gamma,H_\Gamma) \]
where 
\[ J_\Gamma: \ \ol{\S}_\Gamma \to \J(\XX, J_{\DD}), \quad
F_\Gamma: \ol{\T}_{\Gamma,\white} \to C^\infty(L) \]
\[ G_\Gamma: \ \ol{\T}_{\Gamma,\white} \to \G(L), \quad H_\Gamma:
\ol{\T}_{\Gamma,\black} \to C^\infty(Y) \]
and $J_\Gamma$ is equal to the given almost complex structures
$J_{D_{\subset}}, J_{D_{\supset}}$ on $D_{\subset},D_{\supset}$ and
satisfies
the (Locality Axiom).  

In order to define perturbation data we note that the domain of a
broken curve is not necessarily stable because of Morse trajectories
of infinite length.  However, given a broken curve $C$ we obtain a
stable broken curve $f(C)$ by collapsing unstable components and a
perturbation system $P_\Gamma$ for curves of such type by pulling back
$P_{f(\Gamma)}$ under the stabilization map $C \to \U_{f(C)}$.  Given
a broken map $u: C \to \XX[k]$, a {\em trivial cylinder} is a map to
some intermediate piece $\P(N_\pm \oplus \C)$ projecting to a constant
map to $Y$.  Given a type $\Gamma$ of broken disk and a perturbation
datum $P_\Gamma$, an {\em adapted broken map} is a map
$u: C \to \XX[k]$ from a broken weighted treed disk $C$ to $\XX[k]$
for some $k$ such that each interior leaf $e \subset C$ maps to $\DD$
and each component of $u^{-1}(\DD) $ contains an interior leaf
$e \subset T$.  An adapted broken map $u: C = S \cup T \to \XX$ is
{\em stable} if each level $S_i$ of the surface part $S$ the union of
at least one stable curve $S_{i,k} \subset S_i$ and a collection of
trivial cylinders $S_{i,l} \cong S^2, D\pi \circ \d u |S_{i,l} =0 $.
The coherence, regularity, and stabilizing conditions on perturbation
data from the unbroken case in Definition \ref{coherent} generalize
naturally.  A perturbation system is {\em admissible} if it satisfies
these three conditions. 

The following generalizes the compactness and transversality results
for Fukaya algebras to the broken case:

\begin{theorem} \label{broken} Let $\Gamma$ be an uncrowded type of
  adapted pseudoholomorphic broken treed disk of expected dimension at
  most one and suppose that admissible perturbation data $P_{\Gamma'}$
  have been chosen for all boundary strata
  $\ol{\U}_{\Gamma'} \subset \ol{\U}_\Gamma$.  There exists a comeager
  subset of the space of admissible perturbation data $P_\Gamma$ equal
  to the given perturbation data on lower-dimensional strata such that
\begin{enumerate} 
\item 
{\rm (Transversality)} every element of $\M_\Gamma(\XX,L,\DD)$ is
  regular;
\item {\rm (Compactness)} the closure $\ol{\M}_\Gamma(\XX,L,\DD)$ is
  compact and contained in the adapted uncrowded locus;
\item 
{\rm (Tubular neighborhoods)} each uncrowded stratum
$\M_\Gamma(\XX,L,\DD)$ of dimension zero has a tubular neighborhood of
dimension one in any adjoining uncrowded strata of one higher
dimension;
\item {\rm (Orientations)} the uncrowded strata $\M_\Gamma(\XX,L,\DD)$
  of formal dimension at most one are equipped with orientations
  satisfying the standard gluing signs for inclusions of boundary
  strata as in the unbroken case; in particular we denote by
  $\eps(u) \in \{ \pm 1 \}$ the orientation sign associated to
  elements $u$ of the zero-dimensional moduli spaces
  $\M(\XX,L,\DD)_0$.
\end{enumerate} 
\end{theorem}

\begin{proof}[Sketch of Proof] We describe the differences in the
  proof from the unbroken case in Theorem \ref{compthm} and Theorem
  \ref{main}.  Let $u_\nu: C_\nu \to \XX$ be a sequence of adapted
  maps of the given combinatorial type $\Gamma$.  By Theorem
  \ref{sftcompact}, there exists a limit $u: C \to \XX$ in the Gromov
  sense that is a broken map.  We claim that the limit is adapted and
  uncrowded.  The proof of this claim is the same as that of Theorem
  \ref{compthm}: The stabilization condition on the divisor in
  \eqref{bstab} implies that any sphere bubble resp. disk bubble $C_i$
  appearing in the limit $C$ has finitely many but at least three
  resp. one interior intersection points $u^{-1}(\DD) \cap C_i$ with
  the stabilizing divisor $\DD$.  Furthermore, by preservation of
  intersection multiplicity $\# u(C).(\DD) $ with the divisor, each
  maximal ghost component $C_i \subset C$ mapping to the divisor $\DD$
  must contain at least one endpoint $z_j \in C_i$ of an interior
  leaf.  Such a component $C_i$ must be adjacent either to at least
  two non-ghost components $C_j,C_k$, a single non-ghost component
  $C_j$, or adjacent to two tree segments $T_j,T_j \subset C$.  Strata
  of maps with a component with a point with intersection multiplicity
  two, or mapping the node to the divisor are codimension at least
  two, and so these strata do not occur in the limit.  Hence the $C_i$
  contains at most one marking, so the limit is of uncrowded type.

  Transversality as in the unbroken case in Theorem \ref{main} is an
  application of Sard-Smale as in Charest-Woodward \cite{cw:traj} on
  the universal space of maps.  However, the divisor $\pi^{-1}(D_Y)$
  does not intersect every non-trivial map to $\P(N_\pm \oplus \C)$ in
  finitely many points, that is, the divisor on the neck pieces does
  not represent a multiple of the symplectic class.  Any map contained
  in the divisor must be a multiple of a fiber and these maps are
  automatically transversal.  We introduce a universal moduli space as
  follows.  Let $\Gamma$ be a type of broken map $u: C \to \XX$, and
  $\Gamma_0$ the stabilization of the underlying type of curve $C$.
  We begin by covering the universal treed broken disk
  $\U_{\Gamma_0} \to \M_{\Gamma_0}$ by local trivializations
  $\U^i_{\Gamma_0} \to \M^i_{\Gamma_0}, i = 1,\lldots, N$.  For each
  local trivialization consider a moduli space defined as follows.
  Let $\Map^{k,p}_\Gamma(C,\XX,L,\DD)$ denote the space of maps of
  class $k\ge 1,p \ge 1, kp > 2$ (with respect to some connections on
  $C, X_\subset, X_\supset)$ making $L$ totally geodesic) mapping the
  boundary of $C$ into $L$, the interior markings into $\DD$, and
  constant (or constant after projection to $Y$, if the component maps
  to a neck piece) on each disk with no interior marking.  Let
  $\ol{\U}_{\Gamma_0}^{\thin}$ be a small neighborhood of the nodes
  and attaching points in the edges $\ol{\U}_{\Gamma}$, so that the
  complement in each edge and surface component is open.  Let
  $\cP_\Gamma^l(\XX,L,\DD)$ denote the space of perturbation data
  $P_\Gamma = (J_\Gamma,F_\Gamma,G_\Gamma,H_\Gamma)$ of class $C^l$
  for $\Gamma_0$ equal to the given pair $(J,F,G,H)$ on
  $\ol{\U}_{\Gamma_0}^{\thin}$, and such that the restriction of
  $P_{\Gamma}$ to $\ol{\U}_{\Gamma'}$ is equal to $P_{\Gamma'}$, for
  each boundary type $\Gamma'$.  For any broken curve $C$ of type
  $\Gamma$ we obtain perturbation data on $C$ by identifying the
  stabilization $C^{\on{st}}$ of $C$ with a fiber of the universal
  tree disk $\ul{\U}_{\Gamma_0}$.  Let $l \gg k$ be an integer and
\[ \bB^i_{k,p,l,\Gamma} := {\M}^i_{\Gamma_0} \times
\Map^{k,p}_\Gamma(C,\XX,L,\DD) \times {\cP}^l_\Gamma(\XX,L,\DD) .\]
Consider the map given by the local trivialization
\[ {\M}^{\univ,i}_{\Gamma_0} \to \J(S), \ m \mapsto j(m).\]
Let $S^{\on{nc}} \subset S$ be the union of disk and sphere components
on which the map is non-constant.  Let $m(e)$ denote the function
giving the intersection multiplicities with the stabilizing divisor,
defined on edges $e$ corresponding to intersection points, and
consider the fiber bundle $\E^i = \E^i_{k,p,l,\Gamma}$ over  
$\bB^i_{k,p,l,\Gamma}$ given by  
\begin{multline} \label{CST} (\E^i_{k,p,l,\Gamma})_{m, u, J } \subset  
\Omega^{0,1}_{j,J,\Gamma}(S^{\on{nc}}, (u |_S)^* T\XX)_{k-1,p}
\\ \oplus \Omega^1(T_{\white}, (u |_{T_{\white}})^* TL)_{k-1,p} \oplus \Omega^1(T_{\black}, (u  
|_{T_{\black}})^* TY)_{k-1,p}
\end{multline}
the space of $0,1$-forms with respect to $j(m),J$ that vanish to  
order $m(e)-1$ at the node or marking corresponding to each contact  
edge $e$.  The Cauchy-Riemann and shifted gradient operators applied  
to the restrictions $u_S$ resp. $u_T$ of $u$ to the two resp. one  
dimensional parts of $C = S \cup T$ define a $C^q$ section  
\begin{multline} \label{olp3} 
\olp_\Gamma: \bB_{k,p,l,\Gamma}^i \to \cE_{k,p,l,\Gamma}^i,\\ \quad (m,u, 
J,F_\Gamma,H_\Gamma) \mapsto \left(\olp_{j(m),J} u_S , \left( \dds +
    \grad_{F_\Gamma} \right)u_{T_{\white}}, \left( \dds +
\grad_{H_\Gamma} \right)u_{T_{\black}} \right) \end{multline}
where 
\begin{equation} \label{olp4} 
\olp_{j(m),J} u := \hh (J \d u_S - \d u_S j(m)),
\end{equation} 
and $s$ is a local coordinate with unit speed.  The {\em local  
  universal moduli space} is  
\[{\M}^{\univ,i}_{\Gamma}(\XX,L,\DD) = \olp^{-1} \bB^i_{k,p,l,\Gamma} \]
where $\bB^i_{k,p,l,\Gamma}$ is embedded as the zero section.  This  
subspace is cut out transversally: by \cite[Lemma 6.5, Proposition  
6.10]{cm:trans}, the linearized operator is surjective on the
two-dimensional part of the domain mapping to $X_\pm$ on which $u$ is
non-constant, while at any point $z$ in the interior of an edge in $C$
with $\d u(z) \neq 0$ the linearized operator is surjective by a
standard argument.  Furthermore, the matching conditions at the nodes
are cut out transversally, by an inductive argument given in the
unbroken case.  Furthermore, for any map to a neck piece $u: \P^1 \to
\P(N_\pm \oplus \C)$ whose projection to $Y$ is non-constant, the
linearized operator is surjective.  Let $\eta \in \Omega^{0,1}(u^* T
(\P(N_\pm \oplus \ul{\C}))$ be a one-form on one of the intermediate
broken pieces $S_i$ such that $\eta$ lies in the cokernel of the
universal linearized operator
\[D_{u,J}(\xi,K) = D_u \xi + \hh K Du j \]
defining the tangent space to the universal moduli space.  Variations
of tamed almost complex structure of cylindrical type are
$J$-antilinear maps
\[K: T \P(N_\pm \oplus \ul{\C}) \to T \P(N_\pm \oplus \ul{\C})\] 
that vanish on the vertical subbundle and are $\C^\times$-invariant.
Since the horizontal part of $D_z u$ is non-zero at some $z \in C$, we
may find an infinitesimal variation $K$ of almost complex structure
{\em of cylindrical type} by choosing $K(z)$ so that $K(z) D_zu j(z)$
is an arbitrary $(j(z),J(z))$-antilinear map from $T_z C$ to
$T_{u(z)} \P(N_\pm \oplus \ul{\C})$.  Choose $K(z)$ so that
$K(z) D_zu j(z)$ pairs non-trivially with $\eta(u(z))$ and extend
$K(z)$ to an infinitesimal almost complex structure $K$ by a cutoff
function.  Finally, suppose that $u: C_i \to \P(N_\pm \oplus \ul{\C})$
is a component that projects to a point in $Y$.  Such maps are
automatically multiple covers of a fiber of
$\P(N_\pm \oplus \ul{\C}) \to Y$ and so automatically regular.  By the
implicit function theorem, ${\M}^{\univ,i}_{\Gamma}(\XX,L,\DD)$ is a
Banach manifold of class $C^q$, and the forgetful morphism
\[\varphi_i: {\M}^{\univ,i}_{\Gamma}(\XX,L,\DD)_{k,p,l} \to
\cP_{\Gamma}(\XX,L,\DD)_l \]
is a $C^q$ Fredholm map.  Let
${\M}^{\univ,i}_{\Gamma}(\XX,L,\DD)_d \subset {\M}^{\univ,i}_{\Gamma}(\XX,L,\DD) $
denote the component on which $\varphi_i$ has Fredholm index $d$. By
the Sard-Smale theorem, for $k,l$ sufficiently large the set of
regular values $\cP^{i,\reg}_{\Gamma}(\XX,L,D)_l$ of $\varphi_i$ on
${\M}^{\univ,i}_{\Gamma}(\XX,L,\DD)_d$ in
$\cP_{\Gamma}(\XX,L,\DD)_l$ is comeager.  Let
\[ \cP^{l,\reg}_{\Gamma}(\XX,L,\DD)_l = \cap_i
\cP^{i,l,\reg}_{\Gamma}(\XX,L,\DD)_l .\]
A standard argument shows that the set of smooth domain-dependent
$\cP^{\reg}_{\Gamma}(\XX,L,\DD)$ is also comeager.  Fix
$(J_\Gamma,F_\Gamma) \in \cP^{\reg}_{\Gamma}(\XX,L,\DD)$.  By elliptic
regularity, every element of ${\M}^i_{\Gamma}(\XX,L,\DD) $ is
smooth. The transition maps for the local trivializations of the
universal bundle define smooth maps
$ {\M}^i_{\Gamma}(\XX,L,\DD) |_{ {\M}^i_{\Gamma} \cap
  {\M}^j_{\Gamma}} \to {\M}^j_{\Gamma}(\XX,L,\DD) _{{\M}^i_{\Gamma}
  \cap {\M}^j_{\Gamma}} .$
This construction equips the space
\[ {\M}_{\Gamma}(\XX,L,\DD) = \cup_i {\M}^i_{\Gamma}(\XX,L,\DD) \]
with a smooth atlas.  Since $\M_{\Gamma_0}$ is Hausdorff and
second-countable, so is $\M_\Gamma(\XX,L,\DD)$ and it follows that
$\M_\Gamma(\XX,L,\DD)$ has the structure of a smooth manifold.

Existence of orientations and tubular neighborhoods for codimension
one strata involving broken Morse trajectory is similar to that for
the unbroken case.  However, for strata corresponding to a trajectory
of length zero, there is a new gluing result necessary which is proved
in Chapter \ref{getting}.
\end{proof} 

\begin{remark} {\rm (True and fake boundary components)} 
  The formally-codimension-one strata of $\ol{\M}(L,\XX,\DD)$ are of
  the following types:
\begin{enumerate} 
\item Strata of maps 
\[ u: C = C_0 \cup C_1 \to \XX[1] \] 
such that one component $C_0' \subset C_0$ has a boundary edge $e$ of
length $\ell(e)$ zero.  See Figure \ref{bdisk3}.
\begin{figure}[ht]
\begin{picture}(0,0)%
\includegraphics{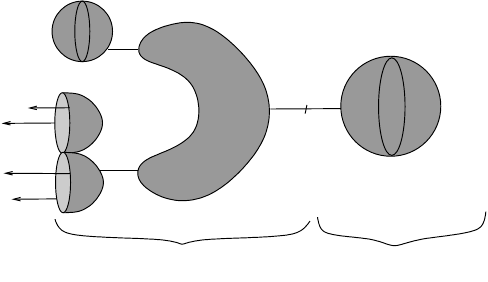}%
\end{picture}%
\setlength{\unitlength}{3947sp}%
\begingroup\makeatletter\ifx\SetFigFont\undefined%
\gdef\SetFigFont#1#2#3#4#5{%
  \reset@font\fontsize{#1}{#2pt}%
  \fontfamily{#3}\fontseries{#4}\fontshape{#5}%
  \selectfont}%
\fi\endgroup%
\begin{picture}(3899,2248)(2182,-2261)
\put(2598,-2201){\makebox(0,0)[lb]{\smash{{{first level (two sublevels) }%
}}}}
\put(4942,-2180){\makebox(0,0)[lb]{\smash{{{second level}%
}}}}
\end{picture}%
\caption{Broken disk with a boundary node}
\label{bdisk3}
\end{figure}
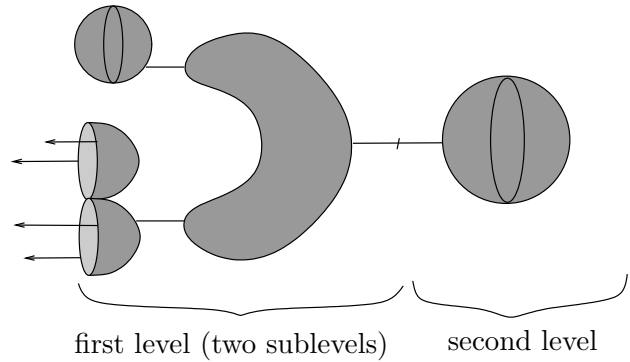
\item Strata of maps 
  \[ u: C = C_0 \cup C_1 \cup C_2 \to
  \XX[2] \]  
  that is, with three levels $C_0,C_1,C_2$ connected by broken
  gradient trajectories of $H$.  See Figure \ref{bdisk}.
\item Strata of maps $u: C = C_0 \cup C_1 \to \XX[1]$ with two broken
  segments.  See Figure \ref{bdisk4}.
\end{enumerate} 

\begin{figure}[ht]
\begin{picture}(0,0)%
\includegraphics{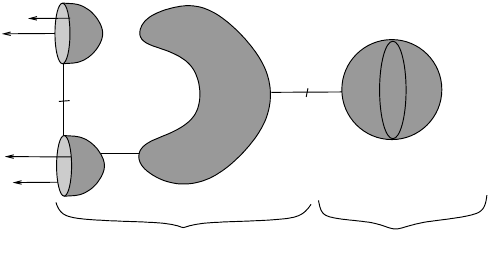}%
\end{picture}%
\setlength{\unitlength}{3947sp}%
\begingroup\makeatletter\ifx\SetFigFont\undefined%
\gdef\SetFigFont#1#2#3#4#5{%
  \reset@font\fontsize{#1}{#2pt}%
  \fontfamily{#3}\fontseries{#4}\fontshape{#5}%
  \selectfont}%
\fi\endgroup%
\begin{picture}(3906,2103)(2175,-2261)
\put(2598,-2201){\makebox(0,0)[lb]{\smash{{{first level (two sublevels) }%
}}}}
\put(4942,-2180){\makebox(0,0)[lb]{\smash{{{second level}%
}}}}
\end{picture}%
\caption{Broken disk with a broken boundary trajectory}
\label{bdisk4}
\end{figure}
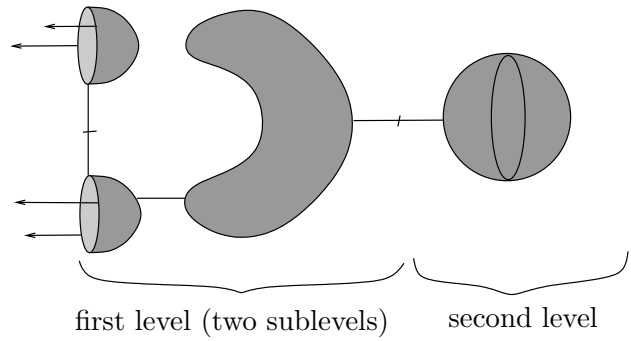

Of these three types, the first two are {\em fake} boundary types in
the sense that the strata $\M_\Gamma(\XX,L,\DD)$ do not represent
points in the topological boundary of the union of 
top dimensional 
strata $\ol{\M}(\XX,L,\DD)$. 
 In the first type, one can either make
the length $\ell(e)$ of the gradient trajectory $e$ of $H$ finite and
non-zero or deform the node $w \in C_i \cap C_j$ connecting the two
disk components $C_i,C_j \subset C$; this shows that any stratum
$\M_\Gamma(\XX,L,\DD)$ of this type $\Gamma$ is in the closure
$\ol{\M}_{\Gamma_k}(\XX,L,\DD), k \in \{ 1, 2 \}$ of two strata of top
dimension.  In the second case one can make the length
$\ell(e_k), k \in \{ 1 ,2 \}$ of either the first gradient trajectory
$e_1$ connecting the first and second levels of $u$ or the second
$e_2$ connecting the second and third levels of $u$ finite, but not
both (since the total length $\ell(e_1) + \ell(e_2)$ must be
infinite).  The last type is a true boundary component since the only
deformation is that which deforms the length $\ell(e)$ of the
trajectory $u| T_e$ to a finite real number $\ell(e) < \infty$.
\end{remark} 

Denote by $\M(\XX,L,\DD,\ul{l})_d$ the locus of maps expected 
dimension equal to $d$.  In particular $\M(\XX,L,\DD,\ul{l})_0$
denotes the space of {\em rigid maps} that are expected dimension zero 
in top-dimensional strata.  Using the regularized moduli spaces of
broken maps we define the composition maps of the broken Fukaya
algebra
\[ \mu^{n}: CF(\XX,L)^{\otimes n} \to CF(\XX,L) \]
on generators by 
\begin{equation} \label{bfuk} \mu^{n}({l_1},\lldots,{l_n}) = \sum_{u
    \in {\M}(\XX,L,\DD,\ul{l})_0} (-1)^{\heartsuit} (\sigma(u)!)^{-1}
  y(u) q^{E(u)} \eps(u) {l_0} \end{equation}
where $ \heartsuit = {\sum_{i=1}^n i|l_i|} $.

\begin{theorem} \label{yields2} {\rm (Broken Fukaya algebra)} For any
  admissible perturbation system $\ul{P} = (P_\Gamma)$ the maps
  $(\mu^{n})_{n \ge 0}$ satisfy the axioms of a convergent \ainfty
  algebra $CF(\XX,L)$ with strict unit.  If the perturbations
  $P_\Gamma$ for types $\Gamma$ with one leaf satisfy the forgetful
  axiom \eqref{fp} then the maps $\mu^1$ satisfy the weak divisor
  axiom \eqref{diveq}.  The homotopy type of $CF(\XX,L)$ is
  independent of all choices up to convergent homotopy equivalence.
\end{theorem} 

The statement of the theorem follows from the transversality and
compactness properties of the moduli space of adapted maps of expected
dimension at most one in Theorem \ref{broken}.  The parts of the map
in the pieces of $\XX$ other than $X_{\subset}$ do not affect the sign
computation, since all these components are spheres.  To prove
homotopy invariance one may either repeat the arguments of Chapter
\ref{hinv} replacing quilted treed disks with broken quilted treed
disks (which are obtained from quilted disks by replacing each
component curve with a broken curve).  Alternatively, in Theorem
\ref{same} below we show that the broken Fukaya algebra is homotopy
equivalent to an unbroken one.

\section{Broken divisors} 

In the rest of the chapter we show that broken stabilizing divisors
exist.  The result is an analog of a relative version of Bertini's
theorem:

\begin{lemma} Let $X$ be a smooth complex projective variety equipped
  with an ample line bundle $\mE$ and $i:Y \hookrightarrow X$ a smooth
  subvariety of codimension one.  After replacing $\mE$ with a tensor
  power, the following holds: For a given $s_Y \in H^0(i^* \mE )$
  cutting out a smooth divisor $D_Y$ on $Y$ there exists a section
  $s \in H^0(\mE)$ restricting to $s_Y$ and cutting out a smooth
  divisor on $X$.
\end{lemma} 

\begin{proof} Let $\mE(Y)$ denote the sheaf of sections vanishing on
  $Y$.  The exact sequence of sheaves
  $ 0 \to \mE(Y) \to \mE \to i_* i^* \mE \to 0 $ induces a long exact
  sequence of cohomology groups including the sequence
\[ 0 \to H^0(\mE(Y)) \to H^0(\mE) \to H^0(i_* i^* \mE ) \to
H^1(\mE(Y)) \to \lldots .\]
By Kodaira vanishing (see for example Griffiths-Harris \cite[Theorem
5, p. 159]{gh}) $H^1(\mE(Y))$ vanishes for sufficiently positive $\mE$
and furthermore $\mE(Y)$ is generated by its global sections.  By the
long exact sequence $H^0(\mE) \to H^0(i_* i^* \mE )$ is surjective.
Generation by global sections implies that the set of $s$ restricting 
to $s_Y$ transverse to the zero section is open and dense.  Compare 
with Bertini \cite[II.8.18]{ha:ag}.
\end{proof}

The symplectic version of relative Bertini is obtained by a
modification of Donaldson's argument in \cite{don:symp}.  Let
$\widetilde{X} \to X$ be a line-bundle with connection $\alpha$ over
$X$ whose curvature two-form $\curv(\alpha)$ satisfies
$\curv(\alpha)= (2\pi/i) \omega$.  Since our symplectic manifolds are
rational we may always assume the existence of such a line bundle
after taking a suitable integer multiple $k \omega$ of the symplectic
form $\omega$.

\begin{definition} {\rm (Asymptotically holomorphic sequences of sections)}  
Let $(s_k)_{k \ge 0}$ be a sequence of sections of $\widetilde{X}^k \to X$.
\begin{enumerate} 
\item The sequence $(s_k)_{k \ge 0}$ is {\em asymptotically
  holomorphic} if there exists a constant $C$ and integer $k_0$ such
  that for $k \ge k_0$,
\begin{equation} \label{asymhol}
 | s_k | + | \nabla s_k| + | \nabla^2 s_k | \leq C, \quad |\olp s_k| +
 | \nabla \olp s_k| \leq C k^{-1/2} .\end{equation} 
\item The sequence $(s_k)_{k \ge 0}$ is {\em uniformly transverse} to
  $0$ if there exists a constant $\eta$ independent of $k$ such that
  for any $x \in X$ with $|s_k(x)| < \eta$, the derivative of $s_k$ is
  surjective and satisfies $| \nabla s_k(x)| \ge \eta$.  
\end{enumerate} 
In both definitions the norms of the derivatives are evaluated using
the metric $g_k = k \omega( \cdot , J \cdot)$.
\end{definition}

\begin{theorem} \label{ddexists} Suppose that
  $X_{\subset}, X_{\supset}, Y$ as above are equipped with line
  bundles $\ti{X}_{\subset}, \ti{X}_{\subset}, \ti{Y}$ with
  connections with curvatures have cohomology classes
  $ [\omega_\subset], [\omega_\supset], [\omega_Y]$ such that
\[ \ti{X}_\subset |_Y \cong \ti{Y} \cong \ti{X}_{\supset} |_Y.\]%
For $k \gg 0$ there exist approximately holomorphic codimension two
submanifolds $D_{\subset,k},D_{\supset,k} \subset X_\pm$ representing
$k[\omega_{\subset}], k[\omega_{\supset}]$ such that
\[D_{\subset,k} \cap Y = D_{\supset,k} \cap Y = D_{Y,k} \] 
for some sequence $D_{Y,k}$ that is also 
is asymptotically holomorphic represents $k[\omega_Y]$. 
\end{theorem}

The proof will be given after two lemmas below.

\begin{lemma} \label{extlem} {\rm (Extension of asymptotically
    holomorphic sequences)} Let $X$ be an integral symplectic manifold
  equipped with a compatible almost complex structure,
  $\widetilde{X} \to X$ a line bundle with connection whose curvature
  is the symplectic form and $Y \subset X$ an almost complex (hence
  symplectic) submanifold.  Denote by $\widetilde{Y} \to Y$ the
  restriction of $\widetilde{X}$ to $Y$.  Given any asymptotically
  holomorphic sequence $s_{Y,k}$ of sections of
  $\widetilde{Y}^k \to Y$, there exists an asymptotically holomorphic
  sequence $s_{k}$ of sections of $\widetilde{X}^k \to X$ such that
  $s_{k,\pm} | Y = s_{Y,k}$.
\end{lemma} 

\begin{proof} We construct an extension by multiplying the given
  sections over the hypersurface by Gaussians in the normal direction.
  We may identify $X$ near $Y$ with the normal bundle $N$ of $Y$ in
  $X$ on a neighborhood $U$ of the zero section.  Let $J_N: N \to N$
  be the complex structure, $\pi: N \to Y$ denote the projection, $P$
  the frame bundle, and $k = \rank(N)$.  Identify
  $N \cong P \times_{U(k)} \C^k$ via the associated bundle
  construction.  The linearization $\widetilde{X}$ on a neighborhood
  of $Y$ admits an isomorphism
  $ \widetilde{X}^k | U \cong \pi^* \widetilde{Y}^k $, using parallel
  transport along the normal directions to $Y$.  We may assume that
  the connection on $\ti{X}$ in the normal direction is induced via
  the associated bundle construction from a one-form
  \[ \alpha_k = \sum_{i=1}^{k} \frac{k}{4} ( z_i \d \ol{z}_i -
  \ol{z}_i \d z_i) \in \Omega^1(\C^k) \]
  where $z_1,\ldots, z_k$ are coordinates on $\C^k$.  Let
  $\phi: N \to \R_{\ge 0}$ denote the norm function, induced from the
  function $(z_1,\ldots, z_k) \to \frac{1}{2} \sum_{i=1}^k |z_i|^2$ on
  $\C^k$.    Define a Gaussian
  sequence
  \[ s_{k,\pm} = (\pi^* s_{Y,k}) \exp( - k \phi^2/4) \]
  in a neighborhood of the divisor $Y$.  After multiplication by a
  cutoff function supported in a neighborhood of size $k^{-1/6}$ of
  $Y$, the section extends by zero to all of $X$.  The bound
\[ | s_k | + | \nabla s_k| + | \nabla^2 s_k | \leq C \]
follows immediately from the fact that the derivatives of the Gaussian
are bounded, and the derivatives are with respect to the Levi-Civita
connection for the metric $g_k$.  The bound
\[ |\olp s_k|
+ | \nabla \olp s_k| \leq C k^{-1/2} \]
follows from the fact that the Gaussian is holomorphic to leading
order as in \cite[(10)]{don:symp}: Consider the splitting $TN$ into
the vertical part $T^{\on{ver}}N \cong N$ and its orthogonal
complement $T^{\on{hor}}N$.  The difference between the almost complex
structures $J_0 = \pi^* J_Y \oplus J_N$ and $J$ is the graph of a map
 as on \cite[Section
2]{don:symp} so that 
\[
\mu: \Lambda^{1,0} TN \to \Lambda^{0,1} N, \quad \olp_{J} s =
\olp_{J_0} s - \mu ( \partial_{J_0} s ) .\]
We may then estimate the failure of $s_k$ to be holomorphic by
\begin{eqnarray*} \olp_J s_k &=&  (\olp_J - \olp_{J_0}) s_k + \olp_{J_0} s_k \\
&=&  - \mu ( \partial_{J_0} s_k) + (\olp_{ \pi^* J_Y} \pi^* s_{Y,k} )
    e^{- k \phi^2/4}
  + \pi^* s_{Y,k} \frac{\mu}{2} ( k \alpha^{1,0} ) e^{- k \phi^2/4}
  .\end{eqnarray*}
Taking norms with respect to the metric induced by $J, k\omega$ we have
\begin{equation} \label{first0}  | \olp_J s_k| \leq C k^{-1/2} 
   \phi^2 e^{- k  \phi^2/4}  + \sup
| \olp_{J_Y} s_{Y,k} | \leq C k^{-1/2}.
 \end{equation} 
Similarly 
\begin{eqnarray}
| \nabla \olp_J s_k | &\leq& C ( |
  \nabla \mu| | \partial_{J_0} s_k | + | \mu| | \nabla ( k \alpha^{1,0}
  e^{- k \phi^2/4})  | | s_{Y,k} | +
  \sup | \nabla \olp_{\pi^* J_Y} \pi^* s_{Y,k}|  \nonumber \\
  \label{second0}  &\leq& C k^{-1/2} ( \phi + \phi^3 ) e^{- k \phi^2/4} + C \sup | \nabla 
  \olp_Y s_{Y,k}| \leq C k^{-1/2} \qedhere 
\end{eqnarray} 
\end{proof} 

\begin{lemma} \label{conclem} Continuing the assumptions of the
  previous lemma, for any $p \in X - Y $ with $d_k(p,Y)\ge k^{-1/2}$,
  with $\codim(Y) = 2$, there exists an approximately section
  $s_{p,k}$ satisfying the estimates \eqref{first0} and
  \eqref{second0} with the property that $s_{p,k}$ vanishes on $Y$.
\end{lemma} 

\begin{proof} We modify the construction of perturbations so that the
  sections on the hypersurface are unchanged.  If
  $d_k(p,Y) \ge k^{-1/6}$, then the previously chosen locally Gaussian
  section satisfies the required properties. Indeed in this case
  $s_{k,p}$ vanishes on $Y$ since $Y$ is outside the support of the
  cutoff function.  So it suffices to assume that $p$ is in the
  intermediate region
\begin{equation} \label{intermediate}
d_k(p,Y) \in (k^{-1/2}, k^{-1/6}) .\end{equation}   
So fix $p' \in Y$ and choose a local Darboux chart $(z_1,\lldots,z_n)$
near $p'$ so that $Y$ is described locally by $z_1 = 0$ and
$\olp z_1(p')= 0$.  Let $p$ lie in this Darboux chart, satisfying the
estimate \eqref{intermediate}. Let $p_1 \neq 0$ denote the first
coordinate of the point $p$.  Given an approximately holomorphic
sequence $s'_{p,k}$ with sufficiently small support (for example, a
Gaussian $s(z) = \exp( - k |z - p|^2)$) the section
$s_{p,k}(z) = s'_{p,k} (z) z_1/p_1$ is also approximately holomorphic,
uniformly in $p$ as long as $|p_1| > k^{-1/2}$.  Indeed the bound
\[ | s_{k,p} | + | \nabla s_{k,p}| + | \nabla^2 s_{k,p} | \leq C \]
is immediate, and uniform if $|p_1| > k^{-1/2}$ since $s'_{p,k}$ is
Gaussian in $k^{1/2} z_1$.  The bound 
\[ |\olp s_{k,p}| + | \nabla \olp s_{k,p}| \leq C k^{-1/2} \]
follows from the fact that $z_1/p_1$ is holomorphic to leading order;
see Auroux \cite[Proof of Proposition 3]{auroux:remark} where similar
approximately holomorphic sections were used to simplify Donaldson's
construction \cite{don:symp}. 
\end{proof} 

\begin{proof}[Proof of Theorem \ref{ddexists}]
  We first perturb on the hypersurface, and then use the special
  perturbations in the previous lemma to perturb away from the
  hypersurface so that the restriction is unchanged.  Let
  $\XX = X_{\subset} \cup_Y X_{\supset}$ be a broken symplectic
  manifold.  The sections $s_{\subset,k}, s_{\supset,k}$ from Lemma
  \ref{extlem} are already asymptotically holomorphic and uniformly
  transverse in a neighborhood of size $Ck^{-1/2}$ around $Y$.  Recall
  that in Donaldson's construction \cite{don:symp}, one has for each
  $k$ a collection of subsets $V_0 \subset \lldots V_N = X$, where $N$
  is independent of $k$, and one shows that given a combination of the
  local Gaussians that is approximately holomorphic and transverse
  section $s_{k,i}$ on $V_i$, that one can adjust the coefficients of
  the locally Gaussian functions so that the section is approximately
  holomorphic and transverse over $V_{i+1}$.  Here we may use sections
  $s_{p,k}$ in Lemma \ref{conclem} for $p$ of distance at least
  $k^{-1/2}$ to achieve transversality off of $Y$.  More precisely, in
  Donaldson's construction \cite[p. 681]{don:symp} taking $V_0$ to be,
  rather than empty, a neighborhood of size $k^{-1/2}$ around $Y$.
  Thus only the initial step of Donaldson's construction is different.
\end{proof}

\begin{proposition} \label{expdim2} For any type $\Gamma$ with
  $k_{\white} \ge 1$ disk components of the surface part
  resp. $k_{\black}$ sphere components with $l$ levels joined by $e$
  cylindrical ends, $m_{\white}$ Morse trajectories in $L$,
  $m_{\black}$ interior edges, and limits $\ul{l}$ along the $n$
  semi-infinite edges mapping to the Lagrangian, the expected
  dimension of the moduli space $\ol{\M}_\Gamma(\XX,L,\DD,\ul{l})$ of
  adapted broken maps of combinatorial type $\Gamma$ limits $\ul{l}$
  with no tangency conditions at the divisor is given by
\begin{multline} 
  \dim T_{[u]} \ol{\M}_\Gamma(\XX,L,\DD,\ul{l}) = (k_{\white} -
  m_{\white}) \dim(L) +(k_{\black} - m_{\black}) \dim(X) + I(u) -
  \dim(W_{l_0}^+) \\ - \sum_{i=1}^n \dim(W_{l_i}^-) +n - 3 - \sum_{i=1}^e (\dim(Y) + 4s_i)  - 2(l-2) \end{multline}
where $s_i+1$ are the multiplicities of $u$ at the intersection points
with $Y$.
\end{proposition} 

\begin{proof}   By construction we have an isomorphism of the tangent
  space 
with the kernel of the linearized operator and the cokernel vanishes:
\[    T_{[u]} \ol{\M}_\Gamma(\XX,L,\DD,\ul{l})  \cong
\ker(\ti{D}_u), 
\quad \{ 0 \} = \coker(\ti{D}_u) .\]
We apply Riemann-Roch for Cauchy-Riemann operators on surfaces with
boundary, \cite[Appendix]{ms:jh} to compute the index
\[\on{Ind}(\ti{D}_u) =   \dim T_{[u]} \ol{\M}_\Gamma(\XX,L,\DD,\ul{l}) .\]
The operator $\ti{D}_u$ is the direct sum of operators $D_{u,S}$ on
the surface part $S \subset C$, $D_{u,T}$ on the tree part
$T \subset C$, and an operator on the tangent space to the moduli
space of treed disks that can be deformed to zero without changing the
index.  Write $S$ as the union of a disk components
$S_{\white,i},i = 1,\ldots, k_{\white}$ and sphere components
$S_{\black,j}, j = 1,\ldots, k_{\black}$.  By Riemann-Roch
\[ \Ind(D_{u,S}) = k_{\white} \dim(L) + k_{\black} \dim(X) + I(u) .\]
Each boundary resp. interior edge has index $\dim(L)$ resp. $\dim(X)$,
since a gradient trajectory is determined by its value at any point.
However, there are two matching conditions at the ends of any such
edge, so that the semi-infinite edges contribute corrections
$ \dim(W_{l_i}^-)$ appearing from the constraints from the
semi-infinite edges, while the boundary resp. interior edges at any
fixed level contribute corrections $ \dim(L)$ resp.  $\dim(X)$.  The
tangent space to the moduli space of stable disks contributes $n - 3$.
The edges connecting levels contribute
$ \sum_{i=1}^e (\dim(Y) + 4s_i)$ from the matching and tangency
conditions for the Morse trajectories in $Y$.  The factor $2(l-2)$ is
the dimension of the group of fiber-wise automorphisms of the neck
pieces mapping to $\P(N_\pm \oplus \C)$.
\end{proof}

\section{Reverse flips} 

Now we specialize to the case that the symplectic manifold is obtained
by a small simple reverse flip or blow-up.  A symplectic manifold $X$
is obtained from a {\em smooth reverse flip} if the local model
$\ti{V}$ in \eqref{sflip} has positive weights $\mu_i$ all equal to
$1$, that is,
\[ (\mu_i > 0) \implies (\mu_i = 1) . \]   
We say that $X$ is obtained from a small reverse {\em simple} flip or
blow-up if and only if the local model $\ti{V}$ in \ref{sflip} has all
weights equal to $\pm 1$.  In this case there exists an embedded
projective space $\P^{n_+ - 1} $ in $X$ and a tubular neighborhood of
$\P^{n_+ - 1}$ in $X$ symplectomorphic to a neighborhood of the zero
section in $\mO(-1)^{\oplus n_-}$ where $n_+ + n_- = n + 1$ and
$\mO(-1) \to \P^{n_+ - 1}$ is the tautological bundle.

We apply a degeneration argument so that after degeneration, the mmp
transition is given by variation of symplectic quotient.  Let
$\mO(-1)_1^{n_-}$ denote the unit sphere bundle in $\mO(-1)^{n_-}$.
Then $\mO(-1)_1^{n_-}$ is circle-fibered coisotropic fibering over the
projectivized bundle
$\P(\mO(-1)^{n_-}) = \P^{n_+ - 1} \times \P^{n_- - 1} . $ The variety
$X$ degenerates to a broken manifold $(X_{\subset}, X_{\supset})$
where $X_{\subset}$ is a toric variety
  \[X_{\subset} \cong \on{Bl}_E \P(\mO(-1)^{n_-} \oplus \ul{\C}) \]
  where $E \cong \P^{n_- - 1}$ is the exceptional locus of the flip.
  The other piece is a disjoint union
  \[ X_{\supset} = X_{\supset}' \sqcup X_{\supset}'', \quad
  X_{\supset} ' = \P(\mO(-1)^{n_-} \oplus \ul{\C}) , \quad
  X_{\supset}'' = \on{Bl}_E X . \]
  We call $X_{\subset}$ resp. $X_{\supset}$ the {\em exceptional}
  resp. {\em remainder} piece of $X$.  The space $X_{\subset}$ may
  also be realized as a symplectic quotient
  \begin{equation} \label{realized} X_{\subset} = ( \C^{n_-}_{-1}
    \oplus {\C}^{n_+}_{+1} \oplus {\C} ) \qu (U(1) \times_{\Z_2}
    U(1)) \end{equation}
  where the action on $\C^{n_-}_{-1}$ is with weight $(-1,1)$ and on
  $\C^{n_+}_{+1}$ with weight $(+1,1)$, and on the last factor of $\C$
  with weight $(0,-2)$.  The flip is obtained by variation of git
  quotient in the above local model.  The unstable locus changes from
  $\{ 0\} \oplus {\C}^{n_+} $ to $\C^{n_-} \oplus \{ 0 \}$.  It
  follows that under the flip the exceptional locus in $X_-$
\[ (\{ 0\} \oplus
{\C}^{n_+})) \qu \C^\times \cong \P^{n_+ - 1}\]
is replaced by the exceptional locus in $X = X_{+}$
\[ (\C^{n_-} \oplus \{ 0 \}) \qu \C^\times \cong \P^{n_- - 1} .\]

The toric piece is a symplectic quotient as explained in Chapter \ref{trans}.  Denote by
$T = (S^1)^{n+2} / (S^1 \times S^1)$
the residual torus acting on $X_{\subset}$.  The canonical moment map
for the action of $(S^1)^{n+2}$ on $\C^{n+2}$ induces a moment map
$\Phi_\subset : X_{\subset} \to \t^\dual .$ 
The moment polytope may be written in terms of the normal vectors
$\nu_k$ as
\[ P_{\subset} := \Phi_{\subset}(X_\subset) = \{ \mu \ |\ \lan \mu,
\nu_k \ran \ge c_k, \ k = 1,\lldots, n+2 \} \]
where the normal vectors $\nu_k$ are the projections of minus the
standard basis vectors in $\R^{n+2} = \on{Lie}((S^1)^{n+2})$: Using
the parametrization
\[(S^1)^{n} \cong T , \quad (z_1,\lldots, z_{n}) \mapsto
[z_1,\lldots, z_{n},1,1]\]
and letting $\eps_j$ be minus the standard basis vectors, we have
\begin{equation} \label{firstnus}
 \nu_1 := \eps_1, \ \nu_2 := \eps_2, \lldots, \nu_{n} := \eps_{n}
 \in \t^\dual .\end{equation}
On the other hand, from the description of the weights in
\eqref{realized} we have
\begin{equation} \label{lastnus}
 \nu_{n+1} := \eps_1 + \lldots + \eps_{n_-} - \eps_{n_- +1} - \lldots -
 \eps_{n}, \ \nu_{n+2} := - \eps_1 + \lldots - \eps_{n_-}
 .\end{equation}
We assume 
\[ c_1 = c_2 = \ldots c_{n} = 0, \ c_{n+1} = \eps, c_{n+2} \gg 0 \]
where the constant $\eps > 0$ represents the size of the exceptional
divisor.

The regular Lagrangian described in Definition \ref{tori} and referred
to in Theorem \ref{result} is a toric moment fiber: Define
\begin{equation} \label{regulartor} \lambda = \eps(1,\lldots, 1)/ (n_+
  - n_-), \quad L = \Phi_{\subset}^{-1}(\lambda) .\end{equation}
By the Cho-Oh classification in Proposition \ref{chooh}, $L$ is
regular in the sense of \ref{tori}.  For example, if
$n_+ = 2, n_- = 1$ then the corresponding transition is a blow-down of
curve in a surface.  The moment polytope has normal vectors
$\eps_1, \eps_2, \eps_1 - \eps_2, - \eps_1$.  The moment polytope is
\[ \{ (\lambda_1,\lambda_2 ) \ | \ \lambda_1 \ge 0, \lambda_2 \ge 0, 
\lambda_1 - \lambda_2 \ge \eps, \lambda_1 \leq c_4 \} \]
where $c_4 \gg 0$.  Then $L$ is the fiber over $(\eps,\eps)$.  See
Figure \ref{blowup} where the point $\lambda$ is shown as a shaded dot
inside the moment polytope, which is a trapezoid.
\begin{figure} 
\includegraphics[height=1in]{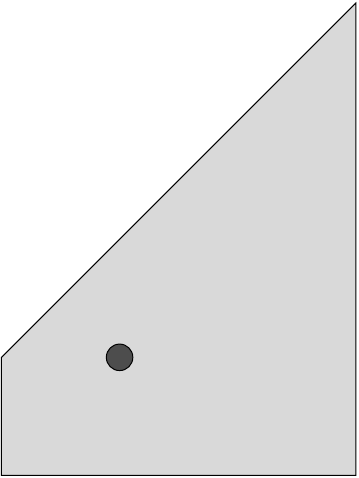}
\caption{The polytope for the blow-up of the projective plane}
\label{blowup}
\end{figure} 

In order to study the broken Fukaya algebra we introduce a Morse
function on the separating hypersurface of standard form.  The pieces
$X_{\subset},X_{\supset}$ are joined by the disconnected hypersurface
\[Y = Y_0 \cup Y_\infty \quad Y_0 \cong Y_\infty \cong \P^{n_+ - 1}
\times \P^{n_- - 1} \]
  namely the divisors in $X_\subset$ at $0$ and $\infty$.  Let
  \[ H: Y_0 \cong Y_\infty \to \R , \quad y \mapsto \lan \Phi_Y(y),
  \xi \ran
\] 
be a Morse function obtained as the pairing of the moment map
$\Phi_Y : Y \to \R^{n_- + n_+}$ with a generic vector $\xi$, which we
take to be $((1,\lldots, n_-), (1,\lldots, n_+))$.  For
$1 \leq i \leq k$ let
\[ [\eps_i] := [0,\lldots, 1, 0 ,\lldots, 0] \in \P^{k-1} \]
denote the point whose homogeneous coordinates are all zero except for
the $i$-th coordinate.  The critical points of $H$ are the fixed
points for the torus action
\begin{equation} \label{critH} \crit(H) = \Set{ ([\eps_{i_-}],
    [\eps_{i_+}]) \in \P^{n_- - 1} \times \P^{n_+ - 1}, \quad i_\pm
    \leq n_\pm } .\end{equation}
Consider the one-parameter subgroup
$\C^\times \to (\C^\times)^{n_- + n_+}$ generated by $\xi$ inside the
standard torus of dimension $n_- \times n_+$ acting on
$\P^{n_- - 1} \times \P^{n_+ - 1}$.  The stable manifolds consist of
those points that flow to $([\eps_{i_-}],[\eps_{i_+}])$ under the
$\C^\times$-action:
\[ W_{([\eps_i], [\eps_j])}^- = \Set{ (z_1,z_2) \in \P^{n_- - 1}
  \times \P^{n_+ - 1} \ | \ \lim_{z \to 0} z(z_1,z_2) =
  ([\eps_{i_-}],[\eps_{i_+}]) } .\]
The Morse cycles are those points $(z_-,z_+)$ such that the homogeneous
coordinates of $z_\pm$ above index $i_\pm$ vanish.  That is, 
\[ W_{([\eps_{i_-}], [\eps_{i_+}])}^- \cong \P^{i_- - 1} \times
\P^{i_+ - 1}, \quad \text{for} \ i_- \leq n_-, i_+ \leq n_+ .\]
For example, when $n_- = 2, n_+ = 1$ the transition corresponds to the
\label{lastline}  
blow-down of a curve in a surface, $Y_0 \cong Y_\infty \cong \P^1$ and
any Morse cycle is either all of $\P^1$ or a point.

In order to study the broken Fukaya algebra, we choose the standard
complex structure on the toric piece.  This suffices to achieve
transversality:

\begin{proposition} \label{std}There exist admissible perturbation
  data $\ul{P} = (P_\Gamma)$ for the broken manifold
  $\XX = (X_{\subset},  X_{\supset}) $ with the properties that
\begin{enumerate}
\item the almost complex structure $J_{\subset,\Gamma} $ on
  $X_{\subset} \subset \XX$ is domain-independent and equal to the standard
  $T_\C$-invariant complex structure from \eqref{tgit}; and
\item the Morse function $H_\Gamma$ on $Y_0 \cup Y_\infty$ is
  domain-independent and equal to a component of the moment map on
  $Y_0 \cong Y_\infty \cong \P^{n_+ - 1} \times \P^{n_- - 1}$ as in
  \eqref{critH}.
\end{enumerate} 
\end{proposition}

We first prove a result on holomorphic spheres in toric varieties:

\begin{lemma} \label{reglem} Let $X$ be a smooth compact toric variety 
  with torus-invariant prime boundary divisors $D_i, i = 1,\ldots, N$. 
  Let $D_i, i \in I$ be a collection of prime boundary divisors with 
  the following property:
\begin{itemize} 
\item[] For each $D_i, i \in I$ there exists a collection of linearly 
  equivalent divisors $D_j, j \in I(i)$ with the property that 
  $\cap D_j, j \in I(i) = 0$. 
\end{itemize} 
Let $u: \P^1 \to X$ be a holomorphic sphere not contained in the union
$\cup_{i \notin I} D_i$.  Then $ H^1(\P^1, u^* TX) = \{ 0 \}$ 
and the 
evaluation map at any point is a submersion.
\end{lemma}

\begin{proof} We write $X$ as a git quotient of $\C^N$, so that the 
  $i$-th factor of $T\C^N$ descends to $\mO(D_i)\subset TX$.  Thus up 
  to the addition of a trivial vector bundle we have 
$$ TX \cong \bigoplus_{i=1}^N  \mO(D_i) $$
and it suffices to show that the higher cohomology of each 
$H^1(\P^1, u^* \mO(D_i))$ vanishes.  If $i \notin I$, then $u(\P^1)$
is not contained in $D_i$ and so the intersection number 
$ u(\P^1).D_i = \deg( u^* \mO(D_i))$ is positive.  On the other hand,
if $i \in I$ then $D_i$ is linearly equivalent to some 
$D_j, j \in I(i)$ not containing $u(\P^1)$, so $\deg( u^* \mO(D_i))$
is positive in this case as well.  Hence $H^1( \P^1, u^* \mO(D_i))$
vanishes, as claimed. 
\end{proof}

\begin{remark} 
  The toric piece $X_{\subset}$ is Fano. Indeed by Kleiman's criterion 
  \cite{kleiman} it suffices to check that the anticanonical degree of 
  holomorphic curves $v: S \to X_{\subset}$ is positive,
  $v. K^{-1} > 0$.  Since any such holomorphic curve degenerates to a 
  union of rational torus-invariant holomorphic curves, it suffices to 
  check the condition on torus-invariant rational curves.  By Remark 
  \ref{toricpiece} $X_{\subset}$ is a $\P^1$-bundle over 
  $\P^{n_- - 1} \times \P^{n_+ -1}$.  The irreducible invariant 
  holomorphic curves lie in either a fiber or in the divisors 
  $D_0, D_\infty$ isomorphic to $\P^{n_- -1} \times \P^{n_+ - 1}$. The 
  restriction of the canonical bundle to $D_0, D_\infty$ is 
  $\mO(n_- \pm 1) \otimes \mO(n_+ \mp 1)$, and the claim follows. 
\end{remark}

\begin{proof}[Proof of Proposition \ref{std}] By the git presentation
  \eqref{realized} the exceptional piece $X_{\subset}$ is a toric
  variety.  Any holomorphic sphere contained in a broken configuration
  mapping to $X_{\subset}$ is not contained in $Y_0 \cup Y_\infty $.
  The remaining prime boundary divisors in $X_{\subset}$ are those
  which fiber over a prime boundary divisor of
  $Y_0 \cong Y_\infty \cong \P^{n_+ - 1} \times \P^{n_- - 1}$.  These
  are of two types: the prime boundary divisors in $\P^{n_+ - 1}$ and
  those in $\P^{n_- - 1}$.  In each case, each prime boundary divisor
  $D_i$ is linearly equivalent to all the remaining prime boundary
  divisors of its type, and the intersection of all such boundary
  divisors is empty.  By Lemma \ref{reglem}, any holomorphic sphere in
  $X_\subset$ not contained in $Y = Y_0 \cup Y_\infty$ 
is regular and 
  the evaluation map at any point is a submersion.

  We claim that perturbation data with almost complex structure equal
  to the standard almost complex structure on the toric piece and
  standard Morse function on the separating hypersurface are
  admissible. \llabel{singledisk} As in Cho-Oh \cite{chooh:toric}, all
  holomorphic disks (necessarily given by the Blaschke products in
  \eqref{blaschke}) are regular, and by the previous paragraph, all
  the spheres in the middle piece are regular as well.
 The same holds after
  imposing constraints at the interior leaves, that is, requiring the
  markings to map to stable manifolds of $H$ as long as the interior
  markings lie on distinct points on the disk component, since the
  Blaschke description \eqref{blaschke} shows that the moduli space of
  such disks is transversally cut out by varying the positions of the
  roots $a_{i,j}$.  Note that the only rigid configurations are those
  with a single interior special point on the disk component, since
  the roots $a_{i,j}$ may always be chosen arbitrarily.
  Configurations $u: C \to \XX$ where two interior nodes $w_1,w_2$ lie
  on a ghost bubble $S_v \subset S \subset C$ are {\em not}
  transversally cut out, but automatically deform to configurations
  $u': C\to \XX$ with the two nodes $w_1',w_2'$ at distinct points of
  the disk component $S_{v'} \subset S$, and so never lie on a
  component $\M_\Gamma(\XX,L)$ of the moduli space of broken treed
  disks of expected dimension $\dim(\M_\Gamma(\XX,L)) \leq 1$ at most
  one.  Thus the moduli spaces of expected dimension at most one are
  regular on the piece mapping to $X_\subset$.  A repeat of the
  previous arguments shows that generic admissible such perturbation
  data on $X_\supset$ make the moduli spaces $\M_\Gamma(\XX,L)$ of
  expected dimension at most one regular.
\end{proof}

The following is the main result of this section; it states that the
broken Floer theory of the Lagrangian torus is unobstructed and
non-trivial.

\begin{theorem} \label{unobs} Let $L \subset X_{\subset}$ be a regular
  Lagrangian brane that is a Lagrangian torus orbit in the toric piece
  $X_{\subset}$.  If $\ul{P} = (P_\Gamma)$ is a collection of
  admissible perturbations such that each almost complex structure is
  constant equal to the standard complex structure on $X_{\subset}$
  and the Morse function on $Y_0 \cong Y_\infty$ for the nodes
  attaching to the disk components is the standard one on
  $Y_0 \cong Y_\infty \cong \P^{n_- - 1} \times \P^{n_+ - 1}$ then
  \begin{enumerate}
  \item the broken Fukaya algebra $CF(\XX,L)$ is a projectively flat
    for any local system ${y}$ and some $b(y)$ depending on ${y}$,
    that is, $\mu_{y}^{0,b(y)}(1)$ is a multiple of the strict
    identity and so 
$(m_{1,y}^{b(y)})^2 = 0$; and
  \item there exist local systems
    ${y}_i \in \RR(L) , i =1,\ldots, n_+ - n_- $ such that the Floer
    cohomology is non-vanishing:
 $ H(m_{1,Y}^{b(y)}) = H(L) \neq 0 .$
\end{enumerate} 
\end{theorem} 

\begin{proof} The argument for weak unobstructedness of the Floer
  theory is a dimension count using the classification of disks in
  \eqref{blaschke}.  Let
  \begin{multline}
u = (u_{\subset}, u_1,\ldots, u_k, u_{\supset}): C_\subset \cup
  C_1 \cup \ldots \cup C_k \cup C_\supset \\
\to X_\subset \cup \P(N_\pm
  \oplus \C) \oplus \ldots \P(N_\pm \oplus \C) \cup X_{\supset} \end{multline}
  be a broken disk contributing to $\mu_0(1)$, that is, with no
  boundary leaves, with intersection multiplicities $\ul{s}$ with the
  hypersurface $Y$ and mapping to the unstable manifolds of
  $x_i, i =1,\ldots, m$ at the intersection points with $Y$.  Each
  such constraint raises the required Maslov index of $u_\subset$ by
  $2$, so that 
  \begin{equation} \label{expdimeq} I(u_\subset) + \sum_{i=1}^m (2 -
    \deg(x_i) - 2 s_i ) + \deg(l) - 2 = 0 \end{equation}
  where $m$ is the number of points mapping to $Y$.  On the other
  hand, the requirement that $u_\subset(z_k)$ meet $\P^i \times \P^j $
  at a given point $z_k \in S_\subset$ implies that for any $l > i$,
  the component $u_l$ has a zero at $z_k$; the number of such
  constraints is the degree of $x_i$.  Furthermore, if $u_0$ is the
  component whose vanishing corresponds to intersections with $Y$ then
  the intersection multiplicity condition also requires that $u_0$
  have a zero of order $s_i$ at $z_k$.  Each of these zeroes
  $u_\subset^{-1}(Y)$ contributes two to the Maslov index
  $I(u_\subset)$ of $u_\subset$.  For the moduli space to contain some
  element $[u]$ it must hold that 
\begin{equation}\label{bigI} I(u_\subset) \ge \sum_{i=1}^m (\deg(x_i) + 2s_i)
  .\end{equation}
Combining \eqref{bigI} with \eqref{expdimeq} it follows that 
\[ \deg(l) \leq 2 - 2m \]
with $\deg(l) = 2$ only if $m = 0$ and $I(u_\subset) = 0$.  From the
Blaschke classification \eqref{blaschke} the requirement
$I(u_\subset) = 0$ implies that $u_\subset |S_\subset$ is constant.
No such disks contribute to $\mu^0_{y}(1)$ since any configuration
contributing to $\mu^0_{y}(1)$ must have at least one non-constant
component.  Hence, $\deg(l) < 2 - 2m \leq 0$.  Since $\deg(l) \ge 0$,
we have $\deg(l) =0$.  By assumption, there is a unique element
$l \in \cI^{\on{geom}}(L)$ of degree zero, the geometric unit
$l = x^{\blackt}$.  It follows that $CF(\XX, L)$ is projectively flat,
that is, $\mu^0_{y}(1)$ is a multiple of the geometric unit
$x^{\blackt}$ for any local system ${y}$.  As in \ref{Wx}, for similar
reasons only configurations with no disks contribute to the sum
\eqref{diff} which becomes
$ \mu^{1}_{y}(x^{\greyt}) = x^{\whitet} - x^{\blackt}$.  Hence if
$\mu^1_{y}(0) = W({y}) x^{\blackt}$ then
$W({y})x^{\greyt} \in \tilde{MC}(L)$.

Local systems for which the Floer theory is non-trivial are found by
computing the critical points of the potential.  The leading order
terms in the potential $\WW(y)$ are as in the toric case and can be
read off from \eqref{firstnus}, \eqref{lastnus}: If $\eps$ denotes the
size of the symplectic class in the exceptional locus
$H^2(\P^{n_- -1 }) \cong \R$ then
\[ \WW(y_1,\lldots, y_{n-1}) = q^\eps \left(y_1 + \lldots + y_{n-1} +
  \frac{y_1 \lldots y_{n_+}}{y_{n_+ + 1} \lldots y_{n-1}}  \right) +
  \text{higher order} \]
where the higher order terms have $q$-exponent at least $\eps$.
Indeed, each of these terms corresponds to a Maslov index two disk in
the toric piece $X_\subset$ with a single exception: In the case of a
blow-up $n_- = 1, n_+ = n$ there is also a contribution from the disk
meeting the exceptional divisor $\P^{n_+ - 1}$ which, in the broken
limit, consists of a disk in $X_{\subset}$ meeting $Y_0$ and a
holomorphic sphere mapping to a fiber of
$X_{\supset}' \to \P^{n_+ -1}$; the corresponding contribution to the
potential is $q^{\eps} y_1 \ldots y_{n-1}$.  The leading order potential
\[ \WW_0(y_1,\lldots,y_{n-1}) = q^{\eps} \left( y_1 + \lldots +  
y_{n-1} + \frac{y_1 \lldots
  y_{n_+}}{y_{n_+ + 1} \lldots y_{n-1}} \right) \]
has partial derivatives for $i \leq n_+$ resp. $I > n_+$ 
\[
q^{-\eps} y_i \partial_{y_i} \WW_0(y_1,\lldots, y_{n-1}) = y_i + \
\text{ resp. - } \   \frac{y_1
  \lldots y_{n_+}}{y_{n_+ + 1} \lldots y_{n-1}} .\]
Setting all partial derivatives equal to zero we obtain 
\[ y_1 = \lldots = y_{n_+} = - y_{n_+ + 1} = \lldots = - y_{n-1} \]
and 
$ y_1^{n_+ - n_-} = (-1)^{n_-} .$
Hence $\WW_0(y)$ has a non-degenerate critical point at certain roots
of unity.  One can solve for a critical point of
$\WW = \WW_0 + \ \textrm{higher order}$ using the Theorem \ref{ift} by
varying the local system.  For a comeager subset of perturbations, the
one-leaf moduli spaces admit forgetful maps as in Lemma \ref{fexist}.
Thus the critical points give local systems $y \in \RR(L)$ for the
broken Fukaya algebra $CF(\XX,L)$ for which there exists Maurer-Cartan
solutions $b \in MC(L)$ for which the Floer cohomology $HF(L,b)$ is
non-vanishing by Proposition \ref{critW}.
\end{proof} 

\chapter{The break-up process}

In this Chapter we show that the Fukaya algebra is homotopy equivalent
to the broken Fukaya algebra obtained in the sft limit, combined with
a family of broken Fukaya algebras in which the sum of the lengths of
the Morse trajectories connecting the pieces in the building varies.

\section{Varying the length}

We first introduce a version of the broken Fukaya algebra where the
gradient segments in the separating hypersurface sum to a finite,
prescribed length parameter.  The broken Fukaya algebras
$CF(\XX,L,\varsigma)$ for various choices of length parameter
$\varsigma$ will be \ainfty homotopy equivalent.  For $\varsigma = 0$,
we will relate the broken Fukaya algebra with the unbroken one, while
the case $\varsigma = \infty$ is computationally more tractable since
any gradient trajectory of infinite length must pass through a
critical point.

\begin{definition} \label{brokencurves2} Let $n,m,s \ge 0$ be
  integers.  A {\em $\varsigma$-broken disk} with $s$ sublevels
  consists of a treed disk $C = S \cup T $ and a decomposition of the
  surface part $S = S_0 \cup \ldots \cup S_s$ into {\em sublevels} of
  possibly disconnected surfaces $S_0,\ldots, S_s$ where only the
  first piece $S_0$ is allowed to have non-empty boundary, that is,
  $ \partial S_1 = \lldots = \partial S_s = \emptyset, $ and
  satisfying the following conditions:
\begin{enumerate} 
\item {\rm (Same Length Axiom)} the edges $T_e \subset T$ connecting
  adjacent components $S_i$ to $S_{i+1}$ have the same length
  $\ell_i$;
\item {\rm (Non-adjacent sublevel axioms)} there are no edges
connecting non-adjacent components $S_i, S_j, |i - j| \ge 2$; 
\item {\rm (Total distance axiom)} for any component $S_i$ define the
  {\em distance} to $S_0$ by summing the lengths of the segments
  connecting $S_0$ and $S_i$:
  \[ \dist(S_0,S_i) = \sum_{j=1}^{i-1} \ell_i .\]
  Then the distance between $S_0$ and the last component $S_s$ is
  $\varsigma$.
\end{enumerate} 
\end{definition}  

To achieve regular perturbations of the moduli spaces we introduce the
notation of broken disks adapted to a broken Donaldson hypersurface.
Let $\M_\Gamma^\varsigma(\XX, L, \DD)$ denote the moduli space of
stable adapted $\varsigma$-broken weighted treed disks with boundary
in $L$ of type $\Gamma$.  Counting adapted broken maps with lengths
$\varsigma$ as in \eqref{bfuk} defines a $\varsigma$-broken Fukaya
algebra $CF(\XX,L,\varsigma)$, by the same arguments as in the case
$\varsigma = \infty$ treated in the previous Chapter.  To prove that
the broken Fukaya algebra is independent of the length, we construct
morphisms $CF(\XX,L,\varsigma_0) \to CF(\XX,L,\varsigma_1)$ using
treed quilted disks where the total distance from the first to last
level varies over the quilting.  We also choose a function which
specifies the length of the gradient trajectories connecting the
pieces of the broken map.  Choose a continuous function
\[ \ell: \U_{n,1} \to [\varsigma_0,\varsigma_1] \]
on the moduli space of quilted disks as follows.  Require
that
\[ ( d(z) = \infty ) \implies ( \ell(z) = \varsigma_0  ) \]
\[ ( d(z) = -\infty ) \implies (\ell(z) = \varsigma_1 ) \]
and furthermore $\ell$ is constant on any disk or sphere component in
a fiber of $\U_{n,1}$.  A {\em quilted
  $[\varsigma_0,\varsigma_1]$-broken disk} is obtained from a quilted
disk by replacing each disk or sphere with a broken disk or sphere
with the property that on any broken disk or sphere component $C_i$,
the distance from the first to last sublevel is $\ell(C_i)$.  Given
perturbation data $\ul{P}^0$ and $\ul{P}^1$ with respect to metrics
$G^0,G^1 \in \G(L)$ over unquilted treed disks for $D^0$ resp. $D^1$,
a {\em perturbation morphism} $\ul{P}^{01}$ is defined as in Chapter \ref{qdisks}.  A {\em pseudoholomorphic quilted
  $[\varsigma_0,\varsigma_1]$-broken treed disk} $u: C \to X$ of
combinatorial type $\Gamma$ is a continuous map from a quilted treed
$[\varsigma_0,\varsigma_1]$-broken disk $C$ that is smooth on each
component, $J_\Gamma^{01}$-holomorphic on the surface components,
$F_\Gamma^{01}$-Morse trajectory with respect to the metric
$G_\Gamma^{01}$ on each boundary tree segment of disk type
$e \in \Edge_{\white}(T)$, and an $H_\Gamma^{01}$-gradient trajectory
on the tree segments of sphere type $e \in \Edge_{\black}(\Gamma)$.
The moduli space of quilted broken disks $\ol{\M}_{n,1}(\XX,L)$ then
has the same transversality and compactness property as in the
unquilted case.

Counting broken quilted holomorphic disks defines homotopy
equivalences of the broken Fukaya algebras for various choices.  Given
an admissible perturbation morphism $\ul{P}^{01}$ from $\ul{P}^0$ to
$\ul{P}^1$, define
\begin{multline}
  \phi^n: CF(L;\ul{P}^0,\varsigma_0)^{\otimes n} \to
  CF(L; \ul{P}^1,\varsigma_1) \\ (x_1,\lldots,x_n) \mapsto
  \sum_{x_0,u \in {\M}^{[\varsigma_0,\varsigma_1]}_\Gamma(L,D,x_0,\lldots,x_n)_0}
  (-1)^{\heartsuit} \eps(u) ( \sigma(u)!)^{-1} q^{ E(u)}
  y(u) x_0
\end{multline}
where the sum is over quilted disks in strata of dimension zero with
$x_1,\lldots,x_n$ incoming labels.  Similarly counting twice-quilted
disks gives homotopy equivalences between the various morphisms.
Thus the homotopy type of $CF(\XX,L,\varsigma)$ and non-vanishing of
the broken Floer cohomology is independent of all choices (including
the choice of $\varsigma$) up to homotopy equivalence. 

\section{Breaking a symplectic manifold} 

Next we study the relationship between the broken Fukaya algebra and
the unbroken Fukaya algebra.  We recall some terminology from
Bourgeois-Eliashberg-Hofer-Wysocki-Zehnder \cite{bo:com}.

\begin{definition} \label{neckstretched}
\begin{enumerate}
\item {\rm (Neck-stretched manifold)} Let $X$ be a closed almost
  complex manifold, and $Z \subset X$ a separating real hypersurface
  of the form in Definition \ref{bsymp} \eqref{bsympz}.  Let $X^\circ$
  denote the manifold with boundary obtained by cutting open $X$ along
  $Z$.  Let $Z', Z''$ denote the resulting copies of $Z$.  For any
  $\tau > 0$ let
\[ X^\tau = X^\circ \cup \bigcup_{ Z'' = \{ - \tau \} \times Z, \{ \tau \}
  \times Z = Z' } [-\tau,\tau] \times Z \]
obtained by gluing together the ends $Z', Z''$ of $X^\circ$ using a
neck $[-\tau,\tau] \times Z$ of length $2\tau$. 
\end{enumerate} 
\end{definition} 

To achieve compactness and gluing, we restrict to almost complex
structures that are of standard form on the neck region as in the
following:

\begin{definition} \label{cform} {\rm (Cylindrical almost complex structures)} Let $\pi_Y: Z \to Y $ be a circle bundle over a symplectic
  manifold $Y$ and $X = \R \times Z$.  Let 
\[\pi_\R: X \to \R, \quad  \pi_Z: X
  \to Z, \quad \pi: X \to Y\] 
  the projections onto factors resp. onto $Y$ and
  $\ker(D\pi) \subset TX$ the vertical subspace.  Let
  $\omega_Z = \pi_Y^* \omega_Z \in \Omega^2(Z)$ denote the pullback of
  the symplectic form $\omega_Y$ to $Z$.  Then
\begin{equation} \label{V}
 V = \ker (D \pi_Y) \oplus \ker (D \pi_Z) \subset TX \end{equation}
is a rank two subbundle complementary to 
\[ H = \ker(\alpha) \subset \pi_Z^* TZ \subset TX .\]
The $\R$ action by translation on $\R$ and $U(1)$ action on $Z$
combine to a smooth $\C^\times \cong \R \times U(1)$ action on $X$.
An almost complex structure $J$ on $X$ is called {\em cylindrical} if
and only if
\begin{enumerate} 
\item $ J \ppt = v $, with $t$ the coordinate on $\R$ and
\item $J$ is invariant under the $\C^\times$-action. 
\end{enumerate} 
\end{definition} 

\begin{remark} Any cylindrical almost complex structure induces an
  almost complex structure $J_Y$ on $Y$ by projection, so that
  $ D \pi (J w) = J_Y D \pi w $ for any $w \in TX$.  Since $J$
  preserves the vertical component $\ker (D \pi)$ this formula defines
  $J_Y D \pi w$ independent of the choice of $w$.  There are
  isomorphisms of complex vector bundles
\begin{equation} \label{splitT}
 TX \cong H \oplus V, \quad H \cong \pi_Y^* TY, \quad V \cong X
 \times \C .\end{equation}
\end{remark} 

\begin{remark} \label{stretchJ} The neck-stretched manifold $X^\tau$
  is diffeomorphic to $X$ by a family of diffeomorphisms given on the
  neck region by a map $(-\tau,\tau) \times Z \to (-\eps,\eps) \times
  Z$ equal to the identity on $Z$ and a translation in a neighborhood
  of $\pm \tau$.  Given an almost complex structure $J$ on $X$ that is
  of cylindrical form on $(-\eps,\eps) \times Z$, we obtain an almost
  complex structure $J^\tau$ on $X^\tau$ by using the same cylindrical
  almost complex structure on the neck region.  Via the diffeomorphism
  $X^\tau \to X$ described above, we obtain an almost complex
  structure on $X$ also denoted $J^\tau$.
\end{remark}

\section{Breaking perturbation data}

The next lemmas allow us to develop a perturbation scheme for broken
pseudoholomorphic maps compatible with degeneration.  

\begin{lemma} \label{gluelem} {\rm (Symplectic sums for pairs)}
  Suppose that $X_{\subset},X_{\supset}$ are symplectic manifolds both
  containing a symplectic hypersurface
  $Y \subset X_{\subset},X_{\supset}$ such that the normal bundles
  $N_\subset,N_\supset$ of $Y$ in $X_{\subset},X_{\supset}$ are
  inverse, that is, there exists an isomorphism
  $N_\subset \cong N_\supset^{-1}$.  Suppose furthermore that
  $D_{\subset},D_{\supset} \subset X_{\subset},X_{\supset}$ are
  codimension two symplectic submanifolds such that
  $D_{\subset},D_{\supset} \cap Y = D_Y$ for some $D_Y \subset Y$.
  Then the symplectic sum
\[D = D_{\subset} \#_{D_Y} D_{\supset}\] 
is naturally a symplectic submanifold of
$X = X_{\subset} \#_Y X_{\supset}$ preserved by a compatible almost
complex structure of cylindrical form on the neck region.
\end{lemma} 

\begin{proof} The statement is an application of symplectic local
  models as in Lerman \cite{le:sy2} and Gompf \cite{go:ss}.  Choose a
  metric on $X_{\subset},X_{\supset}$ near $Y$ so that
  $D_{\subset},D_{\supset}$ is totally geodesic in
  $X_{\subset},X_{\supset}$, as in \cite[Lemma 6.8]{milnor:hcobord};
  this can be done in stages so that the metrics on $Y$ (considered as
  submanifolds of $X_{\subset},X_{\supset}$) agree.  The geodesic
  exponential map
  $N_{\subset} \to X_{\subset},N_\supset \to X_{\supset}$ identifies a
  neighborhood of the zero section in $N_\subset, N_\supset$
  restricted to $ D_Y$ with a neighborhood of $D_Y$ in
  $D_\subset, D_\supset$.  Choose a unitary connection $\alpha_+$ on
  $N_\subset$, let $\alpha_-$ denote the dual connection on
  $N_\supset$ and let $\rho: N \to \R_{\ge 0}$ denote the norm
  function.  Let $\pi: N \to Y$ denote the projection.  The two forms
\[\pi^* \omega_Y + \d (\alpha_+, \rho) \in \Omega^2(N_\subset), \quad  \pi^* \omega_Y -
\d (\alpha_-, \rho) \in \Omega^2(N_\supset) \]
are symplectic in a neighborhood of the zero section and the punctured
normal bundles $N_\subset^\circ, N_\supset^\circ$ may be glued
together to form the symplectic sum.  

To globalize this procedure we show that the symplectic normal forms
may be chosen so that the identification of the divisors is preserved.
The constant rank embedding theorem in Marle \cite{ma:so} identifies a
neighborhood of $D_Y$ in $D$ with $N | D_Y$ symplectically.  Consider
the identification of $N$ with a neighborhood of $Y$ on
$X_{\subset},X_{\supset}$ which maps $N |D_Y $ to $D = D_{\subset}$ or
$D = D_{\supset}$.  Let $\omega_1$ denote the pull-back of the
symplectic form on $X_{\subset},X_{\supset}$ and $\omega_0$ the
symplectic form induced from a connection on $N$.  We have
$\omega_1 - \omega_0 | (D \cup Y) = 0$.  The intersection $D \cap Y$
is transverse, so there exists a deformation retract of a neighborhood
$U$ of $Y$ to $(D \cup Y) \cap U$ in $X_{\subset}$ or $X_{\supset}$.
The standard homotopy formula produces a one-form
$\beta \in \Omega^1(U)$ such that
\[\d \beta = \omega_1 - \omega_0, \quad \beta | ((D \cup Y)
\cap U ) = 0 .\]
Then $\omega_t = t \omega_1 + (1-t) \omega_0$ satisfies
$\ddt \omega_t = \d \beta$.  Define a vector field
\[v_t \in \Vect(U), \quad \iota(v_t) \omega_t = \beta .\]  
Then $v_t$ vanishes on $(D \cup Y) \cap U$ and defines a
symplectomorphism from $\omega_0 $ to $\omega_1$ in a neighborhood of
$Y$ equal to the identity on $D$.  Thus the gluing of local models
induces an identification of the divisors as well.  The almost complex
structure on $N$ preserves $N|D_Y$ and induces a cylindrical almost
complex structure on the symplectic sum preserving $D$.
\end{proof}

\begin{lemma}  Suppose that $L \subset X_{\subset}$ is a rational Lagrangian. 
There exists a family of divisors $D^\tau$ on $X$ so that the image of
$L$ in $X$ is exact in the complement of $D^\tau$, and $D^\tau$ degenerates
as $t \to 0$ to a broken symplectic manifold $\DD$ where
$D_{\subset},D_{\supset}$ are divisors in $X_{\subset},X_{\supset}$ so
that $L$ is exact in the complement of $D_\subset$.
\end{lemma}

\begin{proof} We first construct a suitable Donaldson hypersurface in
  the broken symplectic manifold.  Let $D_\subset \subset X_{\subset}$
  be a divisor making $L$ exact in the complement, and such that the
  intersection $D_\subset \cap Y = D_Y$ is a Donaldson hypersurface
  for $Y$.  We may assume that $D_\subset, D_\supset$ are the zero set
  of some element in an asymptotically holomorphic sequence of
  sections covariant constant on $L$.  By Theorem \ref{ddexists},
  there exists a Donaldson hypersurface $D_\supset \subset
  X_{\supset}$ with $Y \cap D_\supset = D_Y$.  We now apply the gluing
  procedure of Lemma \ref{gluelem}.  Since the normal bundles
  $N_{\subset}, N_{\supset}$ of $Y$ in $X_{\subset},X_{\supset}$ are
  inverses, the restrictions $N_{\subset}, N_{\supset} | D_Y$ are also
  inverses.  Then $(Y,D_Y)$ has a symplectic tubular neighborhood of
  the form $(N_{\subset} , N_\subset | D_Y)$ resp $(N_\supset$ resp.
  $ N_{\supset} | D_Y)$, where the latter has symplectic structure
  induced by the choice of connection as in the Lemma \ref{gluelem}.
  Furthermore, $D$ is almost complex with respect to every element of
  the family $J^\tau$ by construction, and is the zero set of some
  element in an asymptotically holomorphic sequence of sections of
  $\ti{X}^k$.  The connections on the bundles $\ti{X}_\subset,
  \ti{X}_\supset^k$ glue together to a connection on $\ti{X}^k$ making
  $\omega$ exact in the complement of $D$, and for which the section
  defining $D$ is covariant constant on $L$. Then $L$ is exact in the
  complement $X - D$ by a standard argument using Stokes' theorem, c.f. 
\cite[Section 3]{cw:traj}. 
\end{proof} 

We now construct a system of perturbations for the breaking process.
A {\em breaking disk} is a treed holomorphic disk $u: C \to X^{\tau}$
with a {\em breaking parameter} $\tau \in [0,\infty]$ such that if
$\tau < \infty$ resp.  $\tau = \infty$ then the disk is unbroken
resp. broken.  The compactified moduli space of breaking disks of type
$\Gamma$ is denoted $\ol{\M}_\Gamma$ with a universal curve
$\ol{\U}_\Gamma$ that has surface part $\ol{\S}_\Gamma$ and tree part
$\ol{\T}_\Gamma$.  Breaking quilted disks and twice-quilted disks can
be defined similarly.

\begin{definition}  {\rm (Perturbation data)} 
  A {\em perturbation datum} for breaking curves of type $\Gamma$ is a
  datum $ P_\Gamma = (J_\Gamma,F_\Gamma,G_\Gamma,H_\Gamma)$ consisting
  of maps
\[ J_{\Gamma}: \ol{\S}_\Gamma \to \J^\tau(X), 
\quad
\left( \begin{array}{l}  
F_{\Gamma}:
\ol{\T}_{\Gamma,\white} \to C^\infty(L) \\  G_\Gamma:
\ol{\T}_{\Gamma,\white} \to \G(L) \end{array} \right), \quad H_\Gamma:
\ol{\T}_{\Gamma,\black} \to C^\infty(Y) .\]
\end{definition} 

In particular, suppose that $J_\Gamma$ is a collection of
domain-dependent almost complex structures for disks equal to a fixed
cylindrical structure $J_{\R \times Z}$ in a tubular neighborhood of
the separating hypersurface $Z \subset X$.  Then by gluing in lengths
of neck $\tau$ we obtain a family $J_\Gamma^\tau$ of domain-dependent
almost complex structures for breaking disks.  

Associated to any morphism of combinatorial types is a morphism of
perturbation data, as before.  However in this case there is a new
kind of morphism $\Gamma \to \Gamma'$ of combinatorial types
corresponding to {\em gluing of cylindrical end symplectic manifolds},
that is, making the neck length $\tau$ in \ref{bsymp} finite.  In this
case a collection of perturbations $\ul{P} = (P_\Gamma)$ is {\em
  coherent} if, in addition to the usual conditions in the coherence
axioms, we require that the restriction of $P_{\Gamma'}$ to
$\U_\Gamma$ is equal to $P_\Gamma$ whenever a neck length $\tau$ is
set to infinity.  Given a map $u: C = S \cup T \to X$ we obtain an
element
\begin{multline} 
  \olp u = \left(  J_\Gamma \circ \d (u | S) \circ j + \d (u | S) ,
    \dds u | T_{\black}  +
  \grad_{F_\Gamma} u | T_{\black} ,
  \dds u | T_{\white}  + \grad_{H_\Gamma} u | T_{\white}  \right)\\
  \in \Omega^{0,1}(S, (u|S)^* TX) \times \Omega^1(T_{\black}, (u|T_{\black})^*
  TX) \times \Omega^1(T_{\white}, (u |T_{\white})^* TL). \end{multline}
 We say that $u$ is $P_\Gamma$-holomorphic if $\olp u = 0$.

\begin{definition} 
  Given a perturbation datum an {\em adapted breaking map
    $u: C = S \cup T \to X$} of type $\Gamma$ is a breaking disk $C$
  with breaking parameter $\tau \in [0,\infty)$ together with an
  adapted map $C \to X$ that is $P_\Gamma$-holomorphic for
  $\tau < \infty$, or an adapted map $C \to \XX$ for $\tau = \infty$
  that is a broken $P_\Gamma$-holomorphic map. 
\end{definition}

For generic breaking perturbations, the moduli space of adapted
breaking maps is cut out transversally, by a repeat of the Sard-Smale
arguments described in the proof of Theorem \ref{main}.  

A theorem analogous to Theorem \ref{sftcompact} holds in the case of a 
breaking symplectic manifold.  Compare with, for example,
\cite[Theorem 1.1]{cm:com}.  The proof is similar, but allowing for 
domain-dependent almost complex structures away from the nodes. 

\begin{theorem} \label{sftcompact2} Let $\ul{P} = (P_\Gamma)$ be a 
  coherent collection of perturbation data for breaking disks.  Let 
  $\tau_\nu \to \infty$ be a sequence of neck lengths.  Any sequence 
  of adapted stable maps $u_\nu: C_\nu \to X^{\tau_\nu}$ holomorphic 
  with respect to $J_\Gamma$ with bounded energy converges, after 
  passing to a subsequence, to a stable adapted broken map 
  $u: C \to \XX[k]$, for some $k$, with the same index. 
\end{theorem} 

\section{Getting back together} 
\label{getting} 

We have tubular neighborhoods of the strata as in Theorem \ref{main},
in particular, a tubular neighborhood of the boundary corresponding to
the breaking.  This gluing result is similar but slightly different
from that in Bourgeois-Oancea \cite{bourg:auto}.  Much more
complicated gluing theorems in symplectic field theory have been
proved in Hutchings-Taubes \cite{ht:gl2} and Hofer-Wysocki-Zehnder
(see e.g. \cite{ho:sc}) both of which involve obstructions arising
from multiple branched covers of Reeb orbits.  Here any such cover
corresponds to a fiber of a bundle with projective line fibers, and so
one has transversality automatically (although one also has to achieve
transversality with the diagonal at the nodes).  Recall from
Definition \ref{neckstretched} that $C^{\delta_1,\ldots,\delta_k}$
resp. $ X_\delta := X^{|\ln(\delta)|}$ are obtained from $C, \XX$ by
gluing in necks of length $|\ln(\delta_i)|$ resp. $|\ln(\delta)|$ at
each node of $C$ resp. the divisor $Y$.

\begin{theorem} \label{gluing2} \label{bij} Suppose that
  $u: C \to \XX$ is a $0$-regular broken map with multiplicities
  $\mu_1,\ldots,\mu_k$ at the separating hypersurface $Y \subset \XX$.
  Then there exists $\delta_0 > 0$ such that for each gluing parameter
  $\delta \in (0,\delta_0) $ there exists an unbroken map
  $u_\delta: C^{\delta/\mu_1,\ldots, \delta/\mu_k} \to X_{\delta}$,
  with the property that $u_\delta$ depends smoothly on $\delta$ and
  $\lim_{\delta \to 0} [u_\delta] = [u]$.  For any energy bound $E >0$
  there exists $\delta_0$ such that for $ \delta < \delta_0$ the
  correspondence $[u] \mapsto [u_\delta]$ defines a bijection between
  the moduli spaces $\M^{< E}_\Gamma(X_{\delta},L)$ and
  $\M^{< E}_\Gamma(\XX,L)$ for any combinatorial type $\Gamma$.
\end{theorem}

\begin{proof} 
  As for other gluing theorems in pseudoholomorphic curves the proof
  is an application of a quantitative version of the implicit function
  theorem for Banach manifolds.  The steps are: construction of an
  approximation solution; construction of an approximate inverse to
  the linearized operator; quadratic estimates; application of the
  contraction mapping principle, and surjectivity of the gluing
  construction.  We describe the proof of this gluing result in the
  following simple case: Let $C$ be a broken curve with two sublevels
  $C_+,C_-$.  A standard gluing procedure creates, for any small {\em
    gluing parameter} $\delta \in \C$, a curve $C^\delta$ obtained by
  removing small disks $U_\pm$ around the node in the surface part
  $w \in S \subset C$ and gluing using a map given in local
  coordinates by $z \mapsto \delta/z$.  That is, $C^\delta$ is
  obtained by replacing $S$ with
  $S^\delta - (U_+ - U_-)/ (z \sim \delta/z)$ and leaving the tree
  part the same.

  \vskip .1in \noindent {\em Step 0: Fix local trivializations of the
    universal treed disk and the associated families of complex
    structures and metrics on the domains and targets.}  Let
  $u_-: C_- \to X_- := X_\subset$ and $u_+: C_+ \to X_+ := X_\supset$
  with $C_\pm = S_\pm \cup T_\pm$ be treed holomorphic disks agreeing
  at a point $u_+(w_{+-}) = u_-(w_{-+}) \in Y$.  Let $\Gamma_\pm$
  denote the combinatorial types of $u_\pm$ and let
\[{{\U}}_{\Gamma_\pm}^i \to {{\M}}_{\Gamma_\pm}^i \times S_\pm, , i =
1,\lldots, l \]
be local trivializations of the universal treed disk, identifying each
nearby fiber with $(C_\pm^\circ,\ul{z},\ul{w})$ such that each point
in the universal treed disk is contained in one of these local
trivializations.  We may assume that $\M_{\Gamma_\pm}^i$ is identified
with an open ball in Euclidean space so that the fiber $C_\pm^\circ$
corresponds to $0$.  Similarly, we assume we have a local
trivialization of the universal bundle near the glued curve giving
rise to a family of complex structures
\begin{equation} \label{localtriv2} {\M}_{\Gamma}^i \to \J(S^\delta) \end{equation}
of complex structures on the two-dimensional locus
$S^\delta \subset C^\delta$ that are constant on the neck region.  We
consider metrics on the punctured curves $C_\pm^\circ$ that are
cylindrical on the neck region given as the image of the end
coordinates 
\[ \eps^C_\pm:  \pm [0,\infty) \times S^1 \to C_\pm^\circ .\]
Introduce cylindrical ends for $X_- := X_\subset, X_+ := X_\supset$ so
that the embeddings
\[ \eps^X_\pm: \ \mp [0,\infty) \times Z \]
are isometric.  Both the glued target $X_{\delta^\mu}$ and glued domain
$C^\delta$ are defined by removing the part of the end with
$|s| > |\ln(\delta)| $ and identifying
$(s,t) \sim (s- |\ln(\delta)|,t)$ for
$(s,t) \in  (0, |\ln(\delta)|) \times S^1 $ resp.
$(s,t) \sim (s- |\mu\ln(\delta)|,t)$ for
$(s,t) \in  (0, |\ln(\delta)| \times Z) $.

{\vskip .1in \noindent \em Step 1: Define an approximate solution by
  gluing together the two solutions using a cutoff function.} 
 Choose  
a cutoff function  
\begin{equation} \label{beta} 
\beta \in C^\infty(\R, [0,1]), \quad  
\begin{cases} 
\beta(s) = 0 & s \leq 0 \\ \beta(s) = 1 &  s \ge  
1 \end{cases}. \end{equation}
We may suppose by Remark \ref{expdec} and a shift in coordinates that
the maps $u_\pm$ are asymptotic to $(\mu s, t^\mu z)$ for some
$z \in Z$.  The maps $u_\pm^\pm$ considered locally as maps to
$X^\circ$ are asymptotic to the trivial cylinder $(\mu s,t^\mu z) $:
\[ \lim_{s \to \infty} d(u_+(s,t), (\mu s, t^\mu z)) = \lim_{s \to
  -\infty} d(u_-(s,t), ( \mu s, t^\mu z)) = 0 .\]
Denote by $ \exp_x: T_x X \to X$ geodesic exponentiation with respect
to the given cylindrical metric on the neck region.  Let $\zeta_\pm$
be the section related by geodesic exponentiation in cylindrical
coordinates to the maps $u_\pm$
\[ u_\pm(s,t) = \exp_{(\mp \mu s,t^\mu z)} (\zeta_\pm(s,t)) .\] %
Define $u^{\pre}_\delta$ to be equal to $u_\pm$ away from the neck
region, while on the neck region of $C^\delta$ with coordinates $s,t$
define
\begin{multline} \label{preglued} u^{\pre}_\delta(s,t) = \exp_{(\mu 
    s,t^\mu z)} ( \zeta^\delta(s,t)), \\ \quad 
\zeta^\delta(s,t) = \beta(-s) \zeta_-(- s + |\ln(\delta)|/2, t) +
  \beta(s ) \zeta_+(s - |\ln(\delta)|/2,t)) .
\end{multline}
In other words, one translates $u_+,u_-$ by some amount $|\ln(\delta)|$,
and then patches them together using the cutoff function and geodesic
exponentiation. 

{\vskip .1in \noindent \em Step 2: Define a map between suitable
  Banach spaces whose zeroes describe pseudoholomorphic curves near to
  the approximate solution.}  Denote by
\[  (s,t) \in [-|\ln(\delta)|/2, |\ln(\delta)|/2] \times S^1 \] 
the coordinates on the neck region in $C^\delta$ created by the
gluing.  Let $\lambda \in (0,1)$ be a {\em Sobolev weight}.  Define a
{\em Sobolev weight function}
\[ \kappa^\delta_\lambda: C^\delta \to [0,\infty), \quad 
 (s,t) \mapsto 
  \beta(|\ln(\delta)|/2 - |s|) p \lambda ( |\ln(\delta)|/2 - |s| ) \]
where $ \beta(|s| - |\ln(\delta)|/2) p \lambda (|s| - |\ln(\delta)|/2) 
) $ is by definition zero on the complement of the neck region.
We will also use similar weight functions on the punctured curves 
\[ \kappa_\lambda^{\pm}: C_\pm^\circ \to [0,\infty), \quad (s,t)
\mapsto \beta(|s|) p \lambda |s| \]
Pseudoholomorphic maps near the pre-glued solution are cut out locally
by a smooth map of Banach spaces.  Given a smooth map
$u: C^\delta \to X$, let 
\[ \Omega^0(C^\delta, u^* TX, T_{w_\pm} Y) = \Omega^0(S^\delta,
(u|S)^*TX, T_{w_\pm} Y) \oplus \Omega^0(T, (u|T)^* TL) \]
the space of infinitesimal deformations of the map, preserving the
condition $u(w_\pm) \in Y$.  Given an element and sections 
\[ m \in \M_\Gamma^i, \quad \xi_S: S^\delta \to u_S^* TX, \quad \xi_T
: T \to u_T^* TL, \quad \xi = (\xi_S,\xi_T) \]
define as in Abouzaid \cite[5.38]{ab:ex} a norm based on the
decomposition of the section into a part constant on the neck and the
difference:
\begin{multline} \label{1pl2} 
\Vert (m,\xi) \Vert_{1,p,\lambda}^p :=
   \Vert m \Vert^p + \Vert \xi_S \Vert^p_{1,p,\lambda}
 + \Vert \xi_T \Vert^p_{1,p,\lambda} \\ 
\Vert \xi_S \Vert^p_{1,p,\lambda} := 
\Vert (\xi_S(0,0)) \Vert^p +
    \int_{C^\delta} ( \Vert \nabla \xi_S \Vert^p  \\  + \Vert \xi_S - \beta(|
    \ln(\delta)|/2 - |s|) \cT^u ( \xi_S(0,0) ) \Vert^p ) 
    \exp( \kappa_\lambda^\delta) \d \Vol_{C^\delta} 
\end{multline}
where $\cT^u$ is parallel transport from $u^{\pre}(0,t)$ to 
$u^{\pre}(s,t)$ along $u^{\pre}(s',t)$.  Denote by 
$\Omega^0(C^\delta, u^* TX, T_w Y)_{1,p,\delta}$  the space of 
$W^{1,p}_{\loc}$ sections with finite norm \eqref{1pl2}; these are 
sections whose difference from a covariant-constant $TY$-valued 
section on the neck has an exponential decay behavior governed by the 
Sobolev constant $\lambda$.  Pointwise geodesic exponentiation defines 
a map (using Sobolev multiplication estimates)
\begin{equation}
  \exp_{u_\delta^{\pre}}: \Omega^0(C^\delta, (u_\delta^{\pre})^*
  TX_{\delta^\mu})_{1,p,\lambda} \to \Map_{1,p}(C^\delta,X_{\delta^\mu}) \end{equation}
where $\Map_{1,p}(C^\delta,X_{\delta^\mu})$ denotes maps of class
$W_{1,p}^{\loc}$ from $C^\delta$ to $X_{\delta^\mu}$.  Similarly for
the punctured surfaces we have Sobolev norms
\begin{multline} \label{1pl22} \Vert (m,\xi) \Vert_{1,p,\lambda} :=
   \Vert m \Vert^p + \Vert \xi_S \Vert_{1,p,\lambda}
+ \Vert \xi_T \Vert_{1,p,\lambda}, \\ 
\Vert \xi_S \Vert_{1,p,\lambda} :=  \Vert \xi(0,0) \Vert^p +
    \int_{C^\delta} ( \Vert \nabla \xi \Vert^p +  \\
  \left. \Vert \xi_S - \beta(|s|) \cT^u \xi(0,0) \Vert^p ) \exp(
    \kappa_\lambda^\pm) \d \Vol_{C^\circ_\pm} \right)^{1/p}.
\end{multline}
Geodesic exponentiation defines maps 
\begin{equation}
  \exp_{u_\delta^{\pre}}: \Omega^0(C^\circ_\pm, (u_\delta^{\pre})^*
  TX)_{1,p,\lambda} \to \Map_{1,p,\lambda}(C^\circ_\pm,X_\pm^\circ) \end{equation}
where, by definition, $\Map_{1,p,\lambda}(C^\circ_\pm,X_\pm^\circ)$ is
the space of $W_{1,p}^{\loc}$ maps from $C^\circ_\pm$ to $X_\pm$ that
differ from a Reeb orbit at infinity by an element of
$\Omega^0(C^\circ_\pm,X_\pm^\circ)_{1,p,\lambda}$ (which may vary at
infinity because of the inclusion of constant maps on the end in the
Banach space).  Since the surface part $S^\delta$ satisfies a uniform
cone condition and the metrics on $X_{\delta^\mu}$ are uniformly
bounded, one has uniform Sobolev embedding estimates and
multiplication estimates.

Results on elliptic operators on manifolds with cylindrical ends in
Lockhart-McOwen \cite{loc:ell} imply that the linearized
Cauchy-Riemann operators are Fredholm. For the closed manifolds
$C_\pm$, we have linearized operators
\[ D_{u_\pm}:  \Omega^0(C_\pm, u_\pm^* TX)_{1,p} \to
\Omega^{0,1}(C_\pm, u_\pm^* TX)_{0,p} \]
where by definition $\Omega^0(C_\pm, u_\pm^* TX)_{1,p}$ consists of
those sections $\xi$ satisfying various constraints such as tangency
to the divisor at infinity $\xi(w_\pm) \in TY$, see \eqref{CST}.  In the case of the
cylindrical end manifolds $u_\pm^\circ: C_\pm^\circ \to X$, the
assumption $\lambda \in (0,1)$ on the Sobolev decay constant implies
that the linearized operators
\[ D_{u_\pm}^{\circ} : 
\Omega^0(C_\pm^\circ,
u_\pm^* TX^\circ_\pm)_{1,p,\lambda} 
\to \Omega^{0,1}(C_\pm^\circ,
u_\pm^* TX^\circ_\pm)_{1,p,\lambda} \]
are Fredholm.  As in \cite[4.18]{ab:ex} there is an inclusion
\begin{equation} \label{extmap} \ker(D_{u_\pm}^\circ) \to
  \ker(D_{u_\pm}) 
\end{equation}
defined by removal of singularities at infinity: The space
$\Omega^0(C_\pm^\circ, u_\pm^* TX^\circ_\pm)_{1,p,\lambda} $ includes
into $\Omega^0(C_\pm, u_\pm^* TX_\pm)_{1,p,\lambda} $ because of the
exponential decay estimates while conversely elliptic regularity
implies any element of $ \ker(D_{u_\pm})$ is smooth near $w_\pm$ and
so decays exponentially in logarithmic coordinates; note that a change
to the constant term on the neck corresponds to a change to the
derivative of the map at the point at infinity. 
 Since any translation 
on the neck extends to a diffeomorphism of the domain, any element 
$\xi_\pm \in \ker(\ti{D}_{u_\pm})$ is equivalent mod 
$\ker(\ti{D}_{u_\pm})$ to a section taking values in $T_y Y$ at 
infinity.  
By the regularity assumption the fiber products 
\begin{equation} \label{trancut} \ker(\ti{D}_{u_-^\circ}) \times_{T_y 
    Y} \ker(\ti{D}_{u_+^\circ}) \cong \ker(\ti{D}_{u_-}) \times_{T_y 
    Y} \ker(\ti{D}_{u_+}) 
\end{equation}
are transversally cut out. 

The space of pseudoholomorphic maps near the pre-glued solution is cut
out locally by a smooth map of Banach spaces.  For a $0,1$-form
$\eta \in \Omega^{0,1}( C^\delta, u^* TX)$ define
\[ \Vert \eta \Vert_{0,p,\lambda} = \left( \int_{C^\delta} \Vert \eta 
\Vert^p \exp( \kappa_\lambda^\delta) \d \Vol_{C^\delta} \right)^{1/p} .\]
Parallel transport using an almost-complex connection defines a map
\[ \cT_{u_\delta^{\pre}}(\xi) :\ \Omega^{0,1}(C^\delta,
(u_\delta^{\pre})^*TX)_{0,p,\lambda} \to \Omega^{0,1}(C,
(\exp_{u_\delta^{\pre}}(\xi))^*TX)_{0,p,\lambda} .\]
Define
\begin{multline}
 \cF_\delta: \M_\Gamma^i \times \Omega^0(C^\delta, (u_\delta^{\pre})^*
 TX)_{1,p} \to \Omega^{0,1}(C^\delta, (u_\delta^{\pre})^* TX)_{0,p}
 \\ (m_S,m_T,\xi_S,\xi_T) \mapsto \left( \cT_{u_\delta^{\pre}}
(\xi_S)^{-1} \olp_{J_\Gamma,
   j(m_S)} \exp_{u_\delta^{\pre}} (\xi_S), \
 ( \ddt +
 \grad_{m_T}(H_\Gamma))(\xi_T) \right) 
 .\end{multline}
Treed pseudoholomorphic maps close to $u_\delta^{\pre}$ correspond to
zeroes of $\cF_\delta$.  In addition, because we are working in the
adapted setting, our curves $C^\delta$ have a collection of interior
leaves $e_1,\lldots, e_n$.  We require that the constant maps
on the interior edges have values 
\[ (\exp_{u_\delta^{\pre}} (\xi) )(T_{e_i}) \subset D, \quad i =
1,\lldots, n .\]
By choosing local coordinates near the attaching points
$z_i = e_i \cap S$, these constraints may be incorporated into the map
$\cF_\delta$ to produce a map
\begin{multline} 
 \cF_\delta: \M_\Gamma^i \times 
\Omega^0(C_\de, 
 (u_\delta^{\pre})^* TX)_{1,p} 
\\ \to \Omega^{0,1}(C^\delta,
 (u_\delta^{\pre})^* TX)_{0,p}
 \oplus \bigoplus_{i=1}^n T_{u(z_i)}
 X/T_{u(z_i)} D \end{multline}
whose zeroes correspond to {\em adapted} pseudoholomorphic maps near
the preglued map $u_\delta^{\pre}$.

{\vskip .1in \noindent \em Step 3: Estimate the failure of the approximate
  solution to be an exact solution.}  The one-form $ \cF_\delta(0) $
has contributions created by the cutoff function as well as the
difference between $J_{u_\pm}$ and $J_{u_\delta^{\pre}}$:
\begin{eqnarray*} \Vert \cF_\delta(0) \Vert_{0,p,\lambda} 
&=& 
\Vert \olp \exp_{(\mu s,t^\mu z)} ( \beta(-s) \zeta_-(- s +
  |\ln(\delta)|/2, t) \\ && + \beta(s ) \zeta_+(s - |\ln(\delta)|/2,t))
  \Vert_{0,p,\lambda}\\ 
&=& \Vert ( D \exp_{(\mu s,t^\mu z) } ( \d \beta(-s) \zeta_-(- s +
 |\ln(\delta)|/2, z) \\ && + \d \beta(s ) \zeta_+(s -
  |\ln(\delta)|/2,t)) + 
   \\ &&
 ( \beta(-s) \d \zeta_-(- s +
  |\ln(\delta)|/2, z) \\ && + \beta(s ) \d \zeta_+(s - |\ln(\delta)|/2,t))
  )^{0,1} \Vert_{0,p,\lambda}.
\end{eqnarray*}
Holomorphicity of $u_\pm$ implies an
estimate
\begin{multline} \nonumber
 \Vert ( ( \beta(-s) \d \zeta_-(- s + |\ln(\delta)|/2, z) + \beta(s
) \d \zeta_+(s - |\ln(\delta)|/2,t)) )^{0,1} \Vert_{0,p,\lambda} \\ \leq
C e^{ - |\ln(\delta)| (1- \lambda)} = C \delta^{1 - \lambda } ,\end{multline}
c.f. Abouzaid \cite[5.10]{ab:ex}.  Similarly from the terms involving
the derivatives of the cutoff function and exponential convergence of
$\zeta_\pm $ to $0$ we obtain an estimate
\begin{equation} \label{zeroth} \Vert \cF_\delta(0) \Vert_{0,p,\lambda} < C \exp( - |
\ln(\delta)| ( 1- \lambda)) = C \delta^{1 - \lambda} \end{equation}
with $C$ independent of $\delta$.  

{\vskip .1in \noindent \em Step 4: Construct a uniformly bounded right
  inverse for the linearized operator of the approximate solution from
  the given right inverses of the pieces.}  Given an element
$\eta \in \Omega^{0,1}( C^\delta, (u^{\pre})^* T(\R \times Z))_{0,p}$,
one obtains elements
\[ \ul{\eta} = (\eta_-, \eta_+) \in
\Omega^{0,1}(C_\pm^\circ, u_\pm^* TX_\pm^\circ) \] 
by multiplication with the cutoff function and parallel transport
$\cT^{u_\pm}$ to $u_\pm$ along the path
$\exp_{(\mu s, t^\mu z)} ( \rho (\zeta^\delta(s,t) + (1- \rho)
\zeta_\pm(s,t))), \rho \in [0,1] $.   Define 
\[ \eta_+ = \cT^{u_+}  \beta(s - 1/2) \eta, \quad \eta_- = \cT^{u_-}
\beta(1/2 - s)
\eta .\]
Since the fiber product \eqref{trancut} is transversally cut out, for
any $ \eta_\pm \in \Omega^{0,1}(C,u_\pm^*T X)_{0,p,\lambda} $ there
exists
\[ (\xi_+,\xi_-) \in \Omega^0(C_\pm^\circ,u^*
TX_\pm^\circ)_{1,p,\lambda}, \quad \ti{D}_{u_\pm^\circ} \xi_\pm =
\eta_\pm, \quad \xi_+(w_{+-}) = \xi_-(w_{-+})\]
where $w_{\pm \mp} \in C_\pm$ are the nodal points considered as the
points at infinity in $C_\pm^\circ$ and furthermore equal at infinity
to an element
\[ \xi_\infty \in T_{u_\pm(\infty)} Y \]
and at the markings constrained to lie at the divisor satisfying the
constraints
\[ \xi_\pm(z_i) \in T_{u_\pm(z_i)} D_i . \] 
Define $Q^\delta \eta$ equal to $( \xi_-, \xi_+)$ away from
$[- |\ln(\delta)|/2, |\ln(\delta)|/2] \times Z$ and on the neck region
by patching these solutions together using a cutoff function
\begin{multline} Q^\delta \eta := \beta \left( - s - \qq |\ln(\delta)|
    \right)
(  (\cT^{u_-})^{-1} \xi_- - \cT^u \xi_\infty) \\ +
  \beta \left(  s + \qq | \ln(\delta)| \right)  ( ( \cT^{u_+})^{-1}
  \xi_+ - \cT^u \xi_\infty) \\ + \cT^u
  \xi_\infty \in \Omega^{0,1}(C^\delta, (u^{\pre}_{\delta})^*
  TX)_{1,p,\lambda}
    \end{multline}
    where $\cT^\delta_\pm$ denotes parallel transport from $u_\pm$ to
    $u^\pre$
 along the path 
\[ \exp_{(\mu s, t^\mu z)} ( \rho (\zeta^\delta(s,t) + (1- \rho) 
    \zeta_\pm(s,t))), \rho \in [0,1] . \]
   Since 
\[  \eta = (\cT^{u_-})^{-1}  \eta_- + (\cT^{u_+})^{-1} \eta_+ \] 
we have
\begin{eqnarray*} 
 \Vert \ti{D}_{u_{\pre}^\delta} Q^\delta \eta - \eta \Vert_{1,p,\lambda}
 &=& \Vert \ti{D}_{u^{\pre}_\delta} Q^\delta \eta -  (\cT^{u_-})^{-1} \ti{D}_{u_-^\delta}
 \xi_- \\ &&  - (\cT^{u_+})^{-1} \ti{D}_{u_+^\delta} \xi_+
 \Vert_{1,p,\lambda} \\ &\leq& C \exp( (1 - \lambda) | \ln(\delta)/4|) 
 \Vert 
 \eta \Vert_{0,p,\lambda}  \\ && + C \Vert \d \beta( s - |\ln(\delta)|/4) 
 Q^\delta_- \ul{\eta} \Vert_{0,p,\lambda} \\ && + C \Vert \d \beta( -s - |\ln(\delta)|/4) Q^\delta_+
 \ul{\eta}\Vert_{0,p,\lambda} .\end{eqnarray*}
Here the first term arises from the difference between
$ \ti{D}_{u^\pre_\delta} $ and $(\cT^{u_\pm} )^{-1} \ti{D}_{u_\pm}
\cT^{u_\pm}$ and the second from the derivative $\d
\beta$ of the cutoff function
$\beta$.  The difference in the exponential factors
\[ \kappa_\lambda^\pm = \kappa_\lambda^\delta \exp( \pm 2s \lambda), \quad 
\mp s \ge |\ln(\delta)|/2 \] 
in the definition of the Sobolev weight functions implies that
possibly after changing the constant $C$, we have
\[ \Vert \d \beta( s - |\ln(\delta)|/4) Q^\delta_\pm \eta
\Vert_{1,p,\lambda} < C e^{-  \lambda |\ln(\delta)|/2} =
C\delta^{\lambda/2} .\]
Hence one obtains an estimate as in Fukaya-Oh-Ohta-Ono
\cite[7.1.32]{fooo}, Abouzaid \cite[Lemma 5.13]{ab:ex}: for some
constant $C > 0$, for any $\delta > 0$
\begin{equation} \label{first} \Vert \ti{D}_{ u^{\pre}_\delta } Q^\delta -
  \on{Id} \Vert < C \min( \delta^{\lambda/2} , \delta^{(1 - \lambda)/4})
  .\end{equation}
It follows that for $\delta$ sufficiently large an actual inverse may
be obtained from the Taylor series formula
\[ \ti{D}_{u^{\pre}_\delta}^{-1} = ( Q^\delta \ti{D}_{u^{\pre}_\delta})^{-1} Q^\delta =
\sum_{k \ge 0} (I - Q^\delta \ti{D}_{u^{\pre}_\delta})^k Q^\delta .\]

{\vskip .1in \noindent \em Step 5: Obtain a uniform quadratic estimate for the
  non-linear map}.  After redefining $C > 0$ we have for all
$m,\xi,m_1, \xi_1$
\begin{equation} \label{second} \Vert D_{m,\xi} \cF_\delta (m_1,\xi_1)
  - \ti{D}_{u^{\pre}_\delta} (m_1, \xi_1) \Vert \leq C \Vert m, \xi
  \Vert_{1,p,\lambda} \Vert m_1, \xi_1
  \Vert_{1,p,\lambda}. \end{equation}
To prove this we require some estimates on parallel transport.  Let
\[ \cT_{z}^{{\delta},x}(m,\xi): \Lambda^{0,1} T_z^* {C}_\delta
\otimes T_x X \to \Lambda^{0,1}_{{j}^{{\delta}}(m)} T_z^* {C}_\delta
\otimes T_{\exp_x(\xi)} X \]
denote pointwise parallel transport.  Consider its derivative
\[ D\cT_{z}^{{\delta},x}(m,\xi,m_1,\xi_1;\eta) = \nabla_t |_{t = 0}
\cT_z^{\delta,x} m + t m_1, \xi + t\xi_1) \eta .\]
For a map $u: C \to x$ we denote by $D\cT_{u}^\pre$ the corresponding
map on sections at $u^\delta_{\pre}$.  By Sobolev multiplication (for
which the constants are uniform because of the uniform cone condition
on the metric on $C^\delta$ and uniform bounds on the metric on
$X_{\delta^\mu}$) there exists a constant $c$ such that
\begin{equation} \label{Psiest} \Vert
  D\cT_{{u}}^{\delta}(m,\xi,m_1,\xi_1; \eta ) \Vert_{0,p,\lambda} \leq
  c \Vert (m,\xi) \Vert_{1,p,\lambda} \Vert (m_1, \xi_1)
  \Vert_{1,p,\lambda} \Vert \eta \Vert_{0,p,\lambda}.
\end{equation}
Differentiate the equation
\[ \cT_{{u}}^{{\delta}} (m,\xi) \cF_{\delta}(m,\xi) =
\olp_{{j}^\delta(m)}(\exp_{{u}^\delta_{\pre}}(\xi))) \]
with respect to $(m_1,\xi_1)$ to obtain
\begin{multline}
 D\cT_{{u}}^\delta(m,\xi,m_1,\xi_1, \cF_{\delta}(m,\xi) ) +
 \cT_{{u}}^{\delta}(m,\xi)( D \cF_{\delta}(m,\xi,m_1,\xi_1)) = \\ (D
 \olp )_{j^\delta(m),\exp_{u^\delta}(\xi)} (Dj^\delta
 (m,m_1),D\exp_{{u}^\delta} (\xi,\xi_1)) .\end{multline}
The pointwise inequality
\[ | \cF_\delta(m,\xi) | < c | \d
{\exp_{{u}^{\pre}_{\delta}(z)}(\xi)} | < c ( | \d {u}^{\pre}_{\delta}
| + | \nabla \xi | )
\]
holds for $m,\xi$ sufficiently small.  Together with the estimate \eqref{Psiest} we obtain 
a pointwise estimate
\[ | \cT_{u}^\delta(\xi)^{-1}
D\cT_{{u}_\pre^{\delta}}(m,\xi,m_1,\xi_1,\cF_{\delta}(m,\xi)) |
  \leq c (| \d {u}^\delta_{\pre} | + | \nabla \xi |) \, | ( m,\xi ) | \, |
  (\xi_1,m_1) | .\]
Hence
\begin{multline} \Vert  \cT_{u}^\delta(\xi)^{-1}
D\cT_{{u}}^\delta(m,\xi,m_1,\xi_1,\cF_\delta(m,\xi))
\Vert_{0,p,\lambda} \\ \leq c ( 1+ \Vert \d {u}^\delta
\Vert_{0,p,\lambda} + \Vert \nabla \xi \Vert_{0,p,\lambda} ) \Vert
(m,\xi) \Vert_{L^\infty} \Vert (\xi_1,m_1) \Vert_{L^\infty}
.\end{multline}
It follows that 
\begin{equation} \label{firstclaim}
  \Vert \cT_{u}^\delta(\xi)^{-1}
  D\cT_{{u}}^\delta(m,\xi,m_1,\xi_1,\cF_{\delta}(m,\xi))
  \Vert_{0,p,\lambda} \leq c \Vert (m,\xi) \Vert_{1,p,\lambda} \Vert
  (m_1,\xi_1) \Vert_{1,p,\lambda}
\end{equation}
since the $W^{1,p}$ norm controls the $L^\infty$ norm by the uniform
Sobolev estimates.  Then, as in McDuff-Salamon \cite[Chapter
  10]{ms:jh}, Abouzaid \cite{ab:ex} there exists a constant $c > 0$
such that for all $\delta$ sufficiently small, after another
redefinition of $C$ we have
\begin{multline} \label{secondclaim} \Vert \cT_{{u}}^\delta(\xi)^{-1}
\ti{D}_{\exp_{u^\delta_{\pre}}(\xi)}(D_m{j}^{{\delta}}(m_1),D_{\exp_{{{u}^\delta_{\pre}}} (\xi)} \xi_1)) - \ti{D}_{u_{\pre}^\delta} (m_1,\xi_1)
\Vert_{0,p,\lambda} \\ \leq C \Vert m,\xi \Vert_{1,p,\lambda} \Vert
m_1,\xi_1 \Vert_{1,p,\lambda} .\end{multline}
Combining these estimates completes the proof of claim \eqref{second}.

 {\vskip .1in \noindent \em Step 6: Apply the implicit function
   theorem to obtain an exact solution.}  Recall Floer's version of
 the Picard Lemma, \cite[Proposition 24]{floer:monopoles}).  The
 set-up is as follows.  Let $f : V_1 \to V_2$ be a smooth map between
 Banach spaces that admits a Taylor expansion
 $f(v) = f(0) + df(0)v + N(v)$ where $df(0): V_1 \to V_2 $ has a right
 inverse $G:V_2 \to V_1$ satisfying the uniform bound
\[ \Vert GN(u) - GN(v) \Vert \leq C( \Vert u\Vert  + \Vert v \Vert)\Vert 
u - v \Vert \]
for some constant $C$. Let $B(0, \eps)$ denote the $\eps$-ball 
centered at $0 \in V_1$ and assume that
\[\Vert Gf(0) \Vert \leq 1/8C .\]
The conclusion of the lemma is that  for $\eps< 1/4C$, the zero-set of $f^{-1}(0) \cap B(0,\eps)$ is a 
smooth submanifold of dimension $\dim(\Ker(df(0)))$ diffeomorphic to 
the $\eps$-ball in $\Ker(df(0))$. 

Floer's Picard lemma together with the estimates \eqref{zeroth},
\eqref{first}, \eqref{second} produce a unique solution
$m(\delta),\xi(\delta)$ to the equation
$\cF_\delta(m(\delta),\xi(\delta)) = 0$ for each $\delta$, such that
the maps $u(\delta) := \exp_{u_\delta^{\pre}}(\xi(\delta))$ depend
smoothly on $\delta$.  Note that Picard lemma itself does not give
that the maps $u_\delta$ are distinct, since each $u_\delta$ is the
result of applying the contraction mapping principle in a different
Sobolev space.

 {\vskip .1in \noindent \em Step 7: Surjectivity of the gluing
   construction.}  We show that the gluing construction gives a
 bijection.  Consider a  family of adapted maps 
\[ [u'_\delta: C(\delta) \to X_{\delta}], \quad  |\ln(\delta)| \to
\infty .\]
By the compactness result Theorem \ref{sftcompact}, there exists a
subsequence that converges to a broken map $u: C \to \XX$.  Note that
we do not require the domain $C(\delta)$ to be given by the same
gluing parameter as the gluing parameter $\delta$ for $X_\delta$.  By
definition of Gromov convergence the curve $C(\delta)$ is obtained
from $C$ using a gluing parameter $\delta_C$, which is a function of
the gluing parameter $\delta$ for the breaking of target to $\XX$ and
converges to zero as $\delta \to 0$.

 To prove the bijection we must show that any such family of maps is
 in the image of the gluing construction.  Since the implicit function
 theorem used to construct the gluing gives a unique solution in a
 neighborhood, it suffices to show that the maps $[u'_\delta]$ are
 close, in the Sobolev norm used for the gluing construction, to the
 approximate solution $u^\pre_\delta$ defined by \eqref{preglued}.

 We think of the map on the neck region is being composed of a
 horizontal and vertical part.  For the horizontal part of the map
 $\pi \circ u_\delta': C(\delta) \to Y$, the necessary estimate is a
 consequence of the standard result on pseudoholomorphic cylinders of
 small energy, see for example Frauenfelder-Zemisch \cite[Lemma
 3.1]{totreal}.  Denote by $A(l)$ the annulus
$$  A(l) =  [- l/2, l/2| \times S^1 .$$
Since there is no area loss in the limit, for any $\eps > 0$ there
exists $\delta' > \delta_C$ such that the restriction of
$\pi \circ u'_\delta$ to the annulus $A(|\ln(\delta')|/2)$ satisfies
the energy estimate of \cite[Lemma 3.1]{totreal}.  Thus
\begin{multline} \label{longcyl} \pi u'_\delta(s,t) = \exp_{\pi
    u^{\pre}(s,t)} \xi^h(s,t), \quad \Vert \xi^h(s,t) \Vert \leq \eps
  ( e^{ s - |\ln(\delta')|/2 } + e^{|\ln(\delta')|/2 - s} ) \\
  \quad s \in [-|\ln(\delta')|/2, |\ln(\delta')|/2] .\end{multline}
A similar estimate holds for the higher derivatives $D^k \xi^h(s,t)$
by elliptic regularity, for any $k \ge 0$.

For the vertical component we wish to compare the given family of maps
with the trivial cylinder, c.f. B. Parker \cite[Lemma
5.13]{parker:compact}. For $l < |\ln(\delta')|$, but still very large,
consider the $\C^\times$-bundle $P \to A(l)$ obtained from
$\R \times Z \to Y$ by pull-back under $\pi \circ u'_\delta | A(l)$.
The connection on $Z$ induces the structure of a holomorphic
$\C^\times$-bundle on $P$, which is necessarily holomorphically
trivializable.  That is, there exists a $\R \times S^1$-equivariant
diffeomorphism
\[ (\pi \times \varpi): P \to A(l) \times \R \times S^1 \]
mapping the complex structure $(u^{\pre}_\delta)^* J_\Gamma$ on $P$ to
the standard complex structure $J_{(\R \times S^1)^2} $ on the
right-hand-side.  We claim that the holomorphic trivialization may be
chosen to differ from the one given by parallel transport from the
trivial cylinder by an estimate similar to \eqref{longcyl}.  To see
this, note that the bundle $Z \to Y$ is an $S^1$ bundle and the almost
complex structure on $\R \times Z$ is induced from the almost complex
structures on the base and fiber and a connection, given as a one-form
in $\Omega^1(Z)^{S^1}$.  Over any subset $U \subset Y$ we may
trivialize $P|U \cong U \times S^1$ using geodesic exponentiation from
the fiber.  Since $u$ takes values in $\pi^{-1}(U)$, the pull-back
connection in the given trivialization is a one-form
$\alpha \in \Omega^1(U)$.  Any other trivialization of
$u^* (\R \times Z)_U$ is then given by a $\C^\times$-valued
gauge-transformation
\[ \exp(\zeta): U \to \C^\times , \zeta = (\zeta_s,\zeta_t) .\]  
The trivialization is holomorphic if the complex gauge transform of
the connection is trivial.  Thus we wish to solve an inhomogeneous
Cauchy-Riemann equation of the form
\[ \alpha = \alpha_s \d s + \alpha_t \d t = \exp(\zeta)^{-1} \d
\exp(\zeta) = \d \zeta_s + * \d \zeta_t.\]
Write the connection and infinitesimal gauge transformation in terms of its Fourier expansion 
\[
 \alpha(s,t) = \sum_{n \in \Z} \alpha_n(s) \exp( i n t) , \quad 
 \zeta(s,t) = \sum_{n \in \Z} \zeta_n(s) \exp( i n t) \]
 where we identify $S^1 \cong [0,2\pi]/ (0 \sim 2\pi)$.  The Fourier
 coefficients $\zeta_n, n \in \Z$ of $\zeta$ satisfy an equation
\begin{equation} \label{feq}
 \left(\dds - n \right)  \zeta_n(s) = \alpha_n(s) .\end{equation}
An explicit solution of \eqref{feq} is given by integration 
\[ \zeta_n(s) \exp(- n (s - s_0(n))) = \int_{s_0(n)}^s \alpha_n(s') \d
s' \]
so that the solution $\zeta_n(s)$ vanishes on $s_0(n)$.  We make a
careful choice of the Dirichlet condition $\zeta_n(s_0(n)) =0 $ for
the $n$-th Fourier coefficient $\zeta_n$ so that the solution
$\zeta(s,t)$ satisfies the same exponential decay condition
\eqref{longcyl} as the connection $\alpha$.  Define
\[ s_0(n) = \begin{cases}    l/2  & n > 0 \\
                                         0                     & n = 0
                                         \\ 
                                       - l/2  & n <
                                       0 \end{cases}.   \]
                                     Now the estimate on the neck
                                     region \eqref{longcyl} implies by
                                     integration
\begin{eqnarray*} 
  \Vert \zeta_n(s) \Vert  &=&  
\left\Vert \exp(-n (s - s_0(n))) \int_{s_0(n)}^s \int_{t \in S^1}
                              \alpha(s') \exp( -  i n t) \d t \d s'/
                              2\pi   \right\Vert 
\\ &\leq&  (1/2\pi) \begin{cases}  
\eps  (l/2 + s)  \exp( - ( |\ln(\delta')|/2 + s) )  & n < 0 \\ 
\eps \exp( -  (|\ln(\delta')| /2 - |s|)  & n = 0  \\
\eps  (l/2 - s)  \exp( - ( |\ln(\delta')|/2 - s ) )  & n > 0  
\end{cases}
.\end{eqnarray*}  
For $l$ sufficiently large absorb the prefactor $ (l/2 - |s|) $ at the
cost of weakening the exponential decay constant to some
$\rho \in (\lambda,1)$:
\[ (|\ln(l(\delta))|/2 - |s|) \exp( - (|\ln (\delta')|/2 - |s|) )
\leq \exp \left( - \rho (|\ln (\delta')|/2 - |s|) \right) .\]
Thus for $k = 0$ and any $\eps > 0$ we have for $l$ sufficiently large
the exponential decay holds:
\[ \Vert \zeta_{A(l)}\Vert_{k,2} \leq \eps \exp( - \rho( |
\ln(\delta')| /2 - l)) ) .\]
The same arguments applied to the uniform bound on the $k$-th
derivative proves the same estimate for the Sobolev $k,2$-norm for any
$k \ge 0$.  By Sobolev embedding one obtains a $C^{k - 2}$-estimate
for $\zeta(s,t)$ of the form: For any $\eps > 0$ there exists
$ l = l(\delta)$ sufficiently large so for $(s,t) \in A(l-1)$,
\[ \sup_{m \leq k-2} |D^m \zeta(s,t)| \leq C \eps \exp( - \rho
(|\ln(\delta')|/2 - |s|)) \]
where $C$ is a uniform-in-$\delta$ Sobolev embedding constant.  Thus
the holomorphic trivialization of the $\C^\times$-bundle $P$ is
exponentially small over the middle of the cylinder as claimed.

Having constructed a holomorphic trivialization, we may now compare
the given holomorphic cylinder with the trivial cylinder.  Write
\[ \varpi(p) = (\varpi_s(p), \varpi_t(p) ) \in \R \times S^1, \quad 
\forall p \in P . \]  
Since the complex structure is constant in the local trivialization
the difference between the given map and the trivial cylinder
\begin{equation} \label{diffmaps} (s,t) \mapsto (\mu s,t^\mu )^{-1}
  \varpi(u'_\delta(s,t)) = (\varpi_s(u'(s,t)) - \mu s, t^{-\mu}
  \varpi_t(u'_\delta(s,t)))
\end{equation}  
is also holomorphic.  By uniform convergence of $u'_\delta$ to $u$ on
compact sets, we have 
\[ (\pm |\ln(\delta)|/2 + \mu s,t^\mu)^{-1} u'_\delta( \pm
|\ln(\delta_C)|/2 + s, t) \to (0,1) \]
as $s \to \mp \infty$ in cylindrical coordinates on $X_\pm^\circ$.
Thus the difference 
\[ (\mu s, t^{\mu} )^{-1} \varpi( u'_\delta(s,t)) \]
is holomorphic and converges uniformly in all derivatives to the
constant map $ \pm ( |\ln(\delta)|/2 - \mu ( |\ln(\delta_C)|))$ on the
components of $A(l) - A(l-1) \cong [0,1] \times S^1$ as $\delta \to 0$
and $l \to \infty$.  Writing
\[ \partial (\mu s, t^{\mu} )^{-1} \varpi( u'_\delta(s,t)) =
\xi''_\delta(s,t) \]
the map $\xi''_\delta(s,t)$ is also holomorphic in $s,t$ and converges
to zero uniformly on the ends of the cylinder.  It follows from the
annulus lemma \cite[3.1]{totreal} that for any $\eps$, there exists $l$ sufficiently large
so that
\begin{equation} \label{integrand} \Vert \xi''_\delta(s,t) \Vert \leq
  \eps ( e^{ s - l/2 } + e^{- l/2 - s} ) .\end{equation}
In particular 
\[ u'_\delta(s,t) = (\mu (s - s_0), t^\mu t_0^{-1}) \xi_\delta'(s,t)
.\]
for some $(s_0(\delta),t_0(\delta)) \in \R \times S^1$ converging to
$(0,1)$ as $\delta \to 0$.  In particular the difference of lengths
\[ \mu |\ln(\delta_C)| - | \ln(\delta)| \to 0 \]
converges to zero.  That is, the gluing parameters $\delta_C$ for the
domains of $u_\delta$ satisfy $\delta \delta_C^{-\mu} \to 1$ as
$\delta \to 0$.

We now complete the proof that the given family of solutions is close
to the pre-glued solution.  Choose $\eps > 0$.
We write 
\
\[ C(\delta) = \exp_{C^{\delta_C}(m_\delta')}, \quad u'_\delta(s,t) = \exp_{u^\pre_\delta(s,t)} \xi_\delta'(s,t) \]
and claim that 
\[ \Vert (m'_\delta,\xi'_\delta) \Vert_{1,p,\lambda}^p< \eps \] 
for $\delta$ sufficiently small.  By assumption $m'_\delta$ converges 
to zero so for $\delta $ sufficiently small 
\begin{equation} \label{mest} \Vert m'_\delta \Vert^p < \eps / 2.  \end{equation}
Abusing notation we write $ \Vert \xi'_\delta |_{A(l(\delta))} \Vert_{1,p,\lambda} $ for the expression obtained by replacing the integral over $C^\delta$ in \eqref{1pl2} with $A(l(\delta))$ so that
\[ \Vert \xi'_\delta \Vert_{1,p,\lambda}^p = \Vert \xi'_\delta |_{A(l(\delta))} \Vert^p_{1,p,\lambda} +
 \Vert \xi'_\delta |_{ C^\delta -
    A(l(\delta)) } \Vert_{1,p,\lambda}^p .\] 
By uniform convergence
of $u'_\delta$ on compact sets, there exists $l(\delta)$ with
$|\ln(\delta)| - |\ln(l(\delta))| \to \infty$ such that
\begin{equation} \label{onehand} \Vert \xi'_\delta |_{ C^\delta -
    A(l(\delta)) } \Vert_{1,p,\lambda} < \eps/4 . \end{equation}
Since each holomorphic trivialization $\varpi_i$ differs from the trivialization of $P_i | U$ by an exponentially small factor on the middle of the neck, we have
\[ \Vert \xi'_\delta(0,0) \Vert < \eps/8 . \]
Write the trivial cylinder as a geodesic exponentiation from the
preglued solution
\[ (\ul{\mu}s, t^{\ul{\mu}}) = \exp_{u^\pre_\delta}( \xi_\delta^{\triv}(s,t)) .\]
The restriction of $\xi'_\delta$ to the neck region $A(l)$ has
$1,p,\lambda$-norm given by integrating the product of
\eqref{integrand} with the exponential weight function
$\kappa_\lambda^\delta$
\begin{eqnarray*}
 \Vert \xi'_\delta |_{A(l(\delta))} \Vert^p_{1,p,\lambda} 
&\leq& 
2^{p-1} \left( \Vert  \xi'_\delta  - \xi_\delta^{\triv} |_{A(l(\delta))} \Vert^p_{1,p,\lambda} 
+ 
\Vert \xi^{\triv}_\delta |_{A(l(\delta))} \Vert^p_{1,p,\lambda}  \right) \\ 
&\leq& 2^p \Vert   \xi''_\delta \Vert^p_{1,p,\lambda}
+  2^{p-1} \Vert  \xi^{\triv}_\delta |_{A(l(\delta))} \Vert^p_{1,p,\lambda}  
\\ 
&\leq&  2^p  \eps \left( e^{ - p \rho 
  ( | \ln( \delta') | - l ) + p \lambda ( | \ln(\delta)| - l )} /
 (\rho 
- \lambda) \right.
\\ &&  \left. +  e^{ p (\lambda - 1) ( | \ln(\delta') | - l)/2} / (1 - \lambda) \right) . 
\end{eqnarray*} 
For $ |\ln(\delta') - l|$ sufficiently large, the last expression is bounded by $\eps/4$, so 
\begin{equation} \label{otherhand} 
 \Vert \xi'_\delta |_{A(l)}
  \Vert_{1,p,\lambda}^p \leq \eps/4.\end{equation}
Combining \eqref{mest}, \eqref{onehand} and \eqref{otherhand}
completes the proof for the case of two levels joined by a single
node. 

The case of multiple levels joined by multiple nodes is similar.  We
choose gluing parameters $\delta_1' ,\ldots, \delta_l'$ assigned to
the hypersurfaces separating levels and gluing in necks
$Z \times [0, |\ln(\delta_i')|]$ to an obtain an unbroken symplectic
manifold $X_\delta$ with length $\sum_{i=1^l} | \ln(\delta_i')|$.  For
each node pair of nodes $w_\pm$ mapping to the $i$-th copy of $Y$ and
multiplicity $\mu(w_\pm)$ we assign a gluing parameter
$\delta(w_\pm) = (\delta_i')^{1/\mu(w_\pm)}$.  Then gluing in trivial
cylinders into the neck regions creates an approximate solution as
before.  Similar arguments to those for a single node show that the
approximate solution is nearby an actual family of solutions, and that
any family of solutions arises in this way.
\end{proof}

\begin{corollary} \label{eachE} Let $\ul{P}$ be an admissible
  perturbation system for broken disks.  For each $E > 0$, there
  exists $\tau_0$ such that if $\tau > \tau_0$ then
  $\M^{< E}(X,L,\ul{P}_\tau)_0$ is independent of $\tau$ and every
  element is regular.
\end{corollary} 

\begin{proof}  
  By Theorem \ref{sftcompact2}, any sequence $[u_\nu]$ in
  $\M(X_{\tau_\nu}, L)$ with bounded energy and $\tau_\nu \to \infty$
  has a subsequence that converges to an element of $\M(\XX,L)$.  
If 
  $\xi_\nu$ is a sequence of elements in the cokernel of 
  $\ti{D}_{u_\nu}$ with norm one then after passing to a subsequence 
  one obtains an element in the cokernel of the limiting linearized 
  operator $\ti{D}_u$.  Since $\ti{D}_u$ is surjective by assumption,
  $\ti{D}_{u_\nu}$ is surjective for sufficiently large $\nu$ as well. 
\end{proof} 

\begin{remark} After a further small perturbation we may assume that 
  the moduli spaces $\M^{< E}(X,L,\ul{P}^\tau)_{\leq 1}$ of expected 
  dimension at most one are also regular.  This perturbation may be 
  chosen sufficiently small so that $\M^{< E}(X,L,\ul{P}^\tau)_{0}$ is 
  unchanged, up to a $C^0$-small bijection, by the perturbation.  As a 
  results, the composition maps defined by counts of elements of 
  $\M^{< E}(X,L,\ul{P}^\tau)_{0}$ satisfy the \ainfty axiom. 
\end{remark}

\section{The infinite length limit}

Using the gluing result of the previous chapter, we identify the
broken theory with the infinite length limit of the unbroken theory.

\begin{proposition} \label{limequiv} The composition maps $\mu^{n,\tau}$
  of the \ainfty algebra $CF(L,\ul{P}_\tau)$ have a limit as
  $\tau \to \infty$:
  \begin{equation} \label{limexists} \mu^{n,\infty} := \lim_{\tau \to
      \infty} \mu^{n,\tau} .\end{equation}
  The limit $CF(X,L,\ul{P}_\infty)$ is convergent-\ainfty-homotopy
  equivalent to $CF(L,\ul{P}_\tau)$ for any finite $\tau$.
\end{proposition} 

\begin{proof} First we show that the limit in \eqref{limexists}
  exists.  For any energy bound $E$, the terms in $\mu^{n,\tau}$ of
  order at most $q^E$ are independent of $t$ for $\tau >\tau_0 $ by
  Corollary \ref{eachE}.  The \ainfty axiom mod $q^E$ follows from 
  the \ainfty axiom for $\mu^{n,\tau}$ modulo $q^E$; since this holds
  for any $E$, the \ainfty axiom holds on the nose.

  Second we construct a strictly-unital,
  convergent-homotopy-equivalence from the limit to the Fukaya algebra
  for any finite neck length.  Recall that a count of quilted disks
  defines a homotopy equivalences
\[ \phi_{\tau}:  CF(L,\ul{P}_\tau) \to CF(L,\ul{P}_{\tau+1}), 
\quad \psi_{\tau}: CF(L,\ul{P}_{\tau+1}) \to CF(L,\ul{P}_\tau) .\]
We claim that for any energy bound $E$, the terms in $\phi_\tau$ with
coefficient $q^{E(u)}, E(u) < E$ vanish for sufficiently large $\tau$
except for constant disks.  Indeed, otherwise there would exist a
sequence of breaking quilted disks with arbitrarily large $\tau$ in a
component of the moduli space with expected dimension zero and bounded
energy.  By
Theorem \ref{sftcompact2}, the limit would be a broken quilted disk in
a component of the moduli space of expected dimension $-1$, a
contradiction.

The claim implies that there exist limits of the successive
compositions of the homotopy equivalences.  Consider the composition 
\[ \phi_{(n)} := \phi_\tau \circ \phi_{\tau+1} 
\circ \lldots \circ \phi_{\tau + n} : 
CF(L,\ul{P}_\tau) \to 
CF(L,\ul{P}_{\tau+n +1}). 
\]
The bijection in Corollary \ref{eachE} implies that the limit
\[
\phi = \lim_{n \to \infty} \phi_{(n)}: CF(L,\ul{P}_\tau) \to \lim_{n
  \to \infty} CF(L,\ul{P}_{\tau + n}) \]
exists.  Similarly the limit 
\[ \psi = \lim_{n \to \infty} \psi_{(n)}, \quad \psi_{(n)} := \psi_\tau \circ \psi_{\tau+1} 
\circ \lldots \circ \psi_{\tau + n} \]
exists. 
Since the composition of strictly unital morphisms is strictly unital, 
the composition $\psi$ is strictly unital mod terms divisible by $q^E$
for any $E$, hence strictly unital. 

The limiting morphisms are also homotopy equivalences.  Let $h_n,g_n$
denote the homotopies satisfying
\[\phi_{(n)} \circ \psi_{(n)} - \Id = \mu^1(h_{(n)}), 
\quad \psi_{(n)} \circ \phi_{(n)} - \Id = \mu^1(g_{(n)}), \]
obtained as in Definition \ref{homotopy} from the homotopies relating
$\phi_{\tau} \circ \psi_{\tau}$ and $\psi_{\tau} \circ \phi_{\tau}$.
In particular, $h_{(n+1)},g_{(n+1)}$ differ from $h_{(n)},g_{(n)}$ by
expressions involving counting {\em twice-quilted} breaking disks.
For any $E > 0$, for $\tau$ sufficiently large all terms in
$h_{(n+1)} - h_{(n)}$ are divisible by $q^E$.  It follows that the
infinite composition
\[ h = \lim_{n \to \infty} h_{(n)}, \quad g= \lim_{n \to \infty}
g_{(n)} \]
exists and gives a homotopy equivalence $ \phi \circ \psi$ resp. $\psi
\circ \phi$ and the identities.  
\end{proof}

\begin{theorem}   \label{same} Suppose that 
$\ul{P}_\tau$ converges to an admissible perturbation
  system $\ul{P}$ for broken treed disks as above.  Then the limit
  $\lim_{\tau \to \infty} CF(X,L,\ul{P}_\tau)$ is equal to the broken
  Fukaya algebra $CF(\XX,L)$, and in particular $CF(X,L,\ul{P}_\tau)$
  is homotopy equivalent to $CF(\XX,L)$.
\end{theorem} 

\begin{proof} By Theorem \ref{bij}, for any given energy bound $E> 0$
  there exists a bijection between the moduli spaces
  $\M^{<E}(X,L,\ul{P}_\tau)_0$ and $\M^{<E}(\XX,L,\ul{P}_\infty)_0$.
  These moduli spaces define the structure coefficients of the Fukaya
  algebras $CF(L,\ul{P}_\tau)$ for $\tau$ sufficiently large resp. the
  structure coefficients of $CF(\XX,L)$.  The bijection preserves the
  area $A(u)$ of each map $u : C \to X$ as well as the homology class
  $[\partial u] \in H_1(L)$ of the restriction $\partial u$ of the map
  $u$ to the boundary of $S$.  Hence the bijection preserves the
  holonomies of the local system $y(u)$.  It follows from the
  construction of coherence property \eqref{coho} on the orientation
  of the stable and unstable manifolds $W_x^\pm$ of the gradient flow
  $-\grad(F)$ on the separating hypersurface $Y$ that the bijection
  $\M(X,L,\ul{P}_\tau)_0 \to \M(\XX,L,\ul{P}_\infty)_0$ is orientation
  preserving.  The last statement follows from Corollary
  \ref{limequiv}.
\end{proof} 

\begin{proof}[Proof of Theorem \ref{result}]
  We combine the computation in the sft limit with the homotopy
  equivalence above.  Let $X$ be obtained by a small reverse flip or
  blow-up, $L \subset X$ a regular Lagrangian near the exceptional
  locus.  Let $\ul{P}_\tau$ be as in Theorem \ref{same}.  The limit
$CF(\XX,L) = \lim_{\tau \to \infty} CF(X,L,\ul{P}_\tau)$ 
has non-empty Maurer-Cartan space $MC(\XX,L)$ and Floer cohomology
$HF(\XX,L,y_i, b(y_i))$ non-trivial for $n_+ - n_-$ local systems
$y_i, i = 1,\ldots, n_+ - n_-$ and weakly bounding cochains $b(y_i)$,
by Theorems \ref{unobs} and \ref{same}.  The same holds for any finite
$\tau$, by Theorem \ref{mcmapthm}.
\end{proof}

\section{Examples}

\label{flexamples} 
In this chapter we describe Lagrangians associated to mmp transitions
in each of the examples of Chapter \ref{regl}.  We remark that the
results of Palmer-Woodward \cite{pw} show that the non-triviality of
the Floer cohomology in Propositions \ref{tnt}, \ref{reglab},
\ref{reglabm} below hold for times far away from the critical time,
but before the next transition in the mmp.

\subsection{A standard Lagrangian in a Darboux chart}

Standard Lagrangians in Darboux charts in compact symplectic manifolds
are unobstructed, although their Floer cohomology is often trivial.
Let  $x \in X$ and let 
$ q_1,p_1,\lldots, q_n,p_n \in C^\infty(U) $
be Darboux coordinates on an open neighborhood $U$ of $x \in X$ centered at $0$.

\begin{proposition} 
The standard Lagrangian tori
$  L = \{ ( p_j^2 + q_j^2) = c_j, \ j = 1,\lldots, n \} $
for some small constants $c_1,\lldots, c_n > 0 $ have $ MC(L)$ non-empty
and trivial Floer cohomology $HF(L,b) \cong \{ 0 \}$ for any
$b \in MC(L)$.
\end{proposition} 

\begin{proof}   Choose a sphere around $x$ consider the degeneration of $X$ by neck-stretching along the sphere 
to a broken symplectic manifold $(X_\subset, X_\supset)$.    The Lagrangian $L$ becomes a toric moment fiber in $X_{\subset} \cong \P^{n-1}$ and so has
  non-empty Maurer-Cartan space $MC(L)$.  On the other hand, $L$ is
  displaceable for $c_j$ sufficiently small and so has trivial Floer
  cohomology $HF(L,b)$ for $b \in MC(L)$.
\end{proof}

\subsection{Toric symplectic manifolds} 

Continuing Chapter \ref{trans} the following describe
Floer-non-trivial torus orbits in toric varieties, which give special
cases of the results in \cite{fooo:toric1}.

\begin{proposition} \label{tnt} Suppose that $X$ is a compact
  symplectic toric manifold, $X_t$ an mmp running with $t_0$ a
  singular time, $x \in X_{t_0}$ a singular point, and
  $\lambda_0 = \Phi(t_0)$ the moment image of the singular point.
  Then for $t = t_0 - \eps$ for $\eps$ small, the Lagrangian
  $L = \Phinv_{t_0 - \eps}(\lambda)$ is regular and so has
  non-vanishing Floer cohomology for some local system $y \in \RR(L)$
  and weakly bounding cochain $b(y)$.
\end{proposition} 

\begin{example} {\rm (Blow-up of a product of projective lines)}
  The disks in the case of blow-up of $\P^1 \times \P^1$ are shown in
  Figure \ref{threeprim}.  The image of each disk is one-dimensional, since
  the angular direction in each disk is tangent to the level sets of
  the moment map.  On the other hand, the image of each disk (shown
  roughly as dotted lines in the figure) is non-linear since the map
  from $V$ to $X = V \qu G$ is non-linear.  In Figure \ref{threeprim},
  the areas of the three Maslov-index-two disks of smallest area are
  all equal.
\end{example} 

\begin{example} Continuing Example \ref{dp} of del Pezzo surfaces, the
  points giving Floer non-trivial tori in Theorem \ref{result} are
  darkly shaded in Figure \ref{corners}.  The fiber over the
  medium-shaded point is also Floer-non-trivial.
\end{example} 

\subsection{Polygon spaces} 

Continuing Chapter \ref{polygons}, regular labellings of triangulated
polygons give rise to Floer-non-trivial torus orbits in polygon space;
see Nishinou-Nohara-Ueda \cite{nish:degen} and Nohara-Ueda
\cite{no:degen} for another approach to Floer-non-trivial Lagrangian
tori in these spaces.

\begin{proposition} \label{reglab} Let $\mu \in\R_{\ge 0}^{n-3}$ be a 
  regular labelling, and $\mu(t)$ the family of labellings obtained by 
  replacing each $\mu_i$ with $\mu_i - t$ which becomes singular at 
  first time $t = t_i$.  Then for $\eps > 0$ sufficiently small and 
  $t \in (t_i - \eps, t_i)$, any labelling $\mu(t)$ has the property 
  that $\Psi^{-1}(\mu)$ has non-trivial Floer cohomology 
  $HF(\Psi^{-1}(\mu),y,b) \neq 0$ for some local system $y$ and weakly 
  bounding cochain $b$. 
\end{proposition}

\subsection{Moduli spaces of flat bundles on punctured spheres}
By similar arguments, Theorem \ref{result} implies the existence of  Floer non-trivial
tori in representation varieties:

\begin{proposition} \label{reglabm} Suppose that $\mu_1,\lldots,\mu_n$
  are generic labels and $\cR(\mu_1,\lldots,\mu_n)$ is the
  corresponding moduli space of flat $SU(2)$-bundles on the $n$-holed
  two-sphere $\Sigma$.  Let $\cP = \{ P \}$ be a pants decomposition
  of $\Sigma$ with ordered boundary circles
  $C_1,\lldots, C_{n-3} \subset \Sigma$ and suppose that
  $\lambda = (\lambda_1,\lldots,\lambda_{n-3})$ is a regular labelling
  with looseness $t_1$ given by the first transition time $t_1$ in the
  mmp.  Then for $t_1$ sufficiently small the Goldman Lagrangian
  $\Psi^{-1}(\lambda)$ has non-trivial Floer cohomology, 
$HF(\Psi^{-1}(\lambda)) \cong H(\Psi^{-1}(\lambda)) \neq 0 $.
\end{proposition} 

An example of a labelling giving a Floer non-trivial torus is shown in 
Figure \ref{pantses}.

\begin{figure}[ht]
\includegraphics[height=2in]{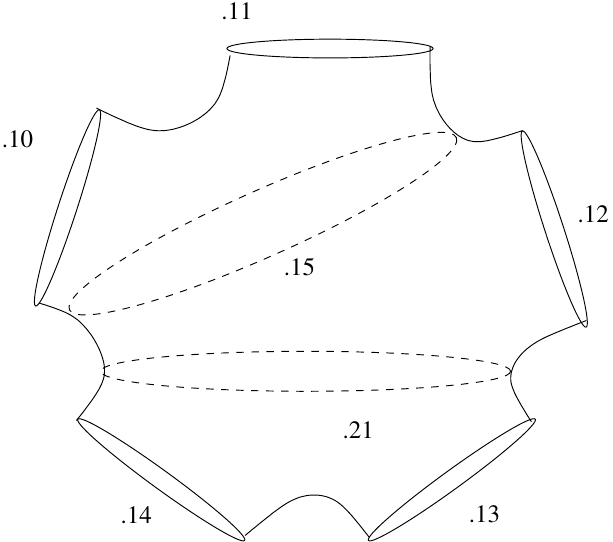}
\caption{A labelling giving a Floer-non-trivial torus}
\label{pantses}
\end{figure}

\bibliographystyle{amsalpha}

\begin{thebibliography}{10}
\bibitem{abbas:com}
C.~Abbas. 
\newblock {\em An introduction to compactness results in symplectic field 
  theory}. 
\newblock Springer, Heidelberg, 2014.


\bibitem{abouzaid:gen}
M.~Abouzaid. 
\newblock A geometric criterion for generating the {F}ukaya category. 
\newblock {\em Publ. Math. Inst. Hautes \'Etudes Sci.}, (112):191--240, 2010.

\bibitem{ab:ex}
M.~Abouzaid.
\newblock Framed bordism and {L}agrangian embeddings of exotic spheres.
\newblock {\em Ann. of Math. (2)}, 175(1):71--185, 2012.


\bibitem{as:qctb}
P.~{Acosta} and M.~{Shoemaker}.
\newblock {Quantum Cohomology of Toric Blowups and Landau-Ginzburg
  Correspondences}.
\newblock {\em Algebr. Geom.} 5 (2018), no. 2, 239--263. 
\href{http://www.arxiv.org/abs/1504.04396}{arXiv:1504.04396}

\bibitem{ag:ei}
S.~Agnihotri and C.~Woodward.
\newblock Eigenvalues of products of unitary matrices and quantum
              {S}chubert calculus,
\newblock {\em Math. Res. Lett.} (6): 817--836, 1998.


\bibitem{at:mo}
M.~F. Atiyah and R.~Bott.
\newblock The {Y}ang-{M}ills equations over {R}iemann surfaces.
\newblock {\em Phil. Trans. Roy. Soc. London Ser. A}, 308:523--615, 1982.



\bibitem{au:to}
M.~Audin.
\newblock {\em The Topology of Torus Actions on Symplectic Manifolds},
  volume~93 of {\em Progress in Mathematics}.
\newblock Birkh\"auser, Boston, 1991.

\bibitem{auroux:asym}
D.~Auroux. 
\newblock Asymptotically holomorphic families of symplectic submanifolds. 
\newblock {\em Geom. Funct. Anal.}, 7(6):971--995, 1997. 

\bibitem{auroux:remark} D.~Auroux.  \newblock A remark about 
  Donaldson's construction of symplectic submanifolds.  \newblock {\em 
    J. Symplectic Geom.} , 1:647--658, 2002. 

\bibitem{auroux:complement} D.~Auroux, D.~Gayet, and J.-P.~Mohsen. 
  \newblock Symplectic hypersurfaces in the complement of an isotropic 
  submanifold.  \newblock {\em Math. Ann.}, 321(4):739--754, 2001. 


\bibitem{bayer:blowups}
A.~Bayer.
\newblock Semisimple quantum cohomology and blowups.
\newblock {\em Int. Math. Res. Not.}, (40):2069--2083, 2004.

\bibitem{bchk}
C.~Birkar, P.~Cascini, C.~Hacon, J.~McKernan. 
\newblock Existence of minimal models for varieties of log general type 
\newblock {\em J. Amer. Math. Soc.} 23: 405--468, 2010. 



\bibitem{br:par}
I.~Biswas and N.~Raghavendra.
\newblock Determinants of parabolic bundles on {R}iemann surfaces.
\newblock {\em Proc. Indian Acad. Sci. Math. Sci.}, 103(1):41--71, 1993.

\bibitem{bc:rigid}
P.~Biran and O.~Cornea.  
\newblock Rigidity and uniruling for Lagrangian submanifolds.
\newblock {\em Geom. Topol.} 13:2881--2989, 2009.

\bibitem{bc:ql}
P.~Biran and O.~Cornea. 
\newblock Quantum structures for {L}agrangian submanifolds. 
\newblock 
\newblock \href{http://www.arxiv.org/abs/0708.4221}{arxiv:0708.4221}.

\bibitem{boardman:hom}
J.~M. Boardman and R.~M. Vogt.
\newblock {\em Homotopy invariant algebraic structures on topological spaces}.
\newblock Lecture Notes in Mathematics, Vol. 347. Springer-Verlag, Berlin,
  1973.

\bibitem{bo:le} D.~Borthwick, T.~Paul, and A.~Uribe.  \newblock
  Legendrian distributions with applications to relative {P}oincar\'e
  series.  \newblock {\em Invent. Math.}, 122(2):359--402, 1995.

\bibitem{boden:var}
H.~U. Boden and Y.~Hu.
\newblock Variations of moduli of parabolic bundles.
\newblock {\em Math. Ann.}, 301(3):539--559, 1995.

\bibitem{bondal:der}
A.~Bondal and D.~Orlov.
\newblock Derived categories of coherent sheaves.
\newblock In {\em Proceedings of the {I}nternational {C}ongress of
  {M}athematicians, {V}ol. {II} ({B}eijing, 2002)}, pages 47--56, Beijing,
  2002. Higher Ed. Press.

\bibitem{bcs:tdms}
L.~A. Borisov, L.~Chen, and G.~G. Smith.
\newblock The orbifold {C}how ring of toric {D}eligne-{M}umford stacks.
\newblock {\em J. Amer. Math. Soc.}, 18(1):193--215 (electronic), 2005.

\bibitem{bourg:auto}
F.~Bourgeois and A.~Oancea.
\newblock Symplectic homology, autonomous {H}amiltonians, and {M}orse-{B}ott
  moduli spaces.
\newblock {\em Duke Math. J.}, 146(1):71--174, 2009.

\bibitem{bo:com}
F.~Bourgeois, Y.~Eliashberg, H.~Hofer, K.~Wysocki, and E.~Zehnder. 
\newblock Compactness results in symplectic field theory. 
\newblock {\em Geom. Topol.}, 7:799--888 (electronic), 2003.

\bibitem{bourg:mb}
F.~Bourgeois.
\newblock A {M}orse-{B}ott approach to contact homology.
\newblock In {\em Symplectic and contact topology: interactions and
  perspectives ({T}oronto, {ON}/{M}ontreal, {QC}, 2001)}, volume~35 of {\em
  Fields Inst. Commun.}, pages 55--77. Amer. Math. Soc., Providence, RI, 2003.

\bibitem{br:ac}
M.~Brion and C.~Procesi. 
\newblock Action d'un tore dans une vari\'et\'e projective. 
\newblock In A.~Connes et~al., editors, {\em Operator Algebras, Unitary 
  Representations, Enveloping Algebras, and Invariant Theory}, volume~62 of 
  {\em Progress in Mathematics}, pages 509--539, Paris, 1989, 1990. 
  Birkh\"auser, Boston.

\bibitem{cm:com}
K.~Cieliebak and K.~Mohnke.
\newblock Compactness for punctured holomorphic curves.
\newblock {\em J. Symplectic Geom.}, 3(4):589--654, 2005.
\newblock Conference on Symplectic Topology.


\bibitem{charest:clust} F.~{Charest}.  \newblock {Source Spaces and 
  Perturbations for Cluster Complexes}.  \newblock 
\href{http://www.arxiv.org/abs/1212.2923}{arxiv:1212.2923}.

\bibitem{cw:traj}
F.~Charest and C.~Woodward.
\newblock Floer trajectories and stabilizing divisors.  
\newblock {\em J. Fixed Point Theory Appl. 19}  (2017), no. 2, 1165--1236.
\href{http://www.arxiv.org/abs/1401.0150}{arXiv:1401.0150}.

\bibitem{cm:trans}
K.~Cieliebak and K.~Mohnke.
\newblock Symplectic hypersurfaces and transversality in {G}romov-{W}itten
  theory.
\newblock {\em J. Symplectic Geom.}, 5(3):281--356, 2007.


\bibitem{coates:kt}
T. Coates, H. Iritani, Y. Jiang, and E. Segal. 
\newblock K-theoretic and categorical properties of toric {D}eligne-{M}umford stacks. 
\newblock preprint. 

\bibitem{cho} C.-H.~Cho. Products of Floer cohomology of torus fibers 
  in toric Fano manifolds.  Comm. Math. Phys.  260 (2005), 613--640.

\bibitem{chooh:toric}
C.-H.~Cho and Y.-G.~Oh. 
\newblock Floer cohomology and disc instantons of {L}agrangian torus fibers in 
  {F}ano toric manifolds. 
\newblock {\em Asian J. Math.}, 10(4):773--814, 2006. 



\bibitem{cl:clusters}
O.~Cornea and F. Lalonde. 
\newblock Cluster homology: An overview of the construction and results. 
\newblock {\em Electron. Res. Announc. Amer. Math. Soc.}, (12):1-12, 2006.

\bibitem{cox:toric}
D.~A. Cox, J.~B. Little, and H.~K. Schenck.
\newblock {\em Toric varieties}, volume 124 of {\em Graduate Studies in
  Mathematics}.
\newblock American Mathematical Society, Providence, RI, 2011.

\bibitem{de:ha}
T.~Delzant.
\newblock {H}amiltoniens p\'eriodiques et images convexes de l'application
  moment.
\newblock {\em Bull. Soc. Math. France}, 116:315--339, 1988.

\bibitem{don:symp}
S.~K. Donaldson.
\newblock Symplectic submanifolds and almost-complex geometry.
\newblock {\em J. Differential Geom.}, 44(4):666--705, 1996.

\bibitem{dr:pi}
J.-M. Drezet and M.~S. Narasimhan.
\newblock Groupe de {P}icard des vari\'et\'es de modules de fibr\'es
  semi-stables sur les courbes alg\'ebriques.
\newblock {\em Invent. Math.}, 97(1):53--94, 1989.

\bibitem{du:on}
J.~J. Duistermaat and G.~J. Heckman.
\newblock On the variation in the cohomology of the symplectic form of the
  reduced phase space.
\newblock {\em Invent. Math.}, 69:259--268. Addendum, Invent. Math. {\bf 72}
  (1983), 153--158, 1982.

\bibitem{do:va}
I.~V. Dolgachev and Y.~Hu. 
\newblock Variation of geometric invariant theory quotients. 
\newblock {\em Inst. Hautes \'Etudes Sci. Publ. Math.}, (87):5--56, 1998. 
\newblock With an appendix by Nicolas Ressayre.

\bibitem{egh:sft}
Y.~Eliashberg, A.~Givental, and H.~Hofer.
\newblock Introduction to symplectic field theory.
\newblock {\em Geom. Funct. Anal.}, (Special Volume, Part II):560--673, 2000.
\newblock GAFA 2000 (Tel Aviv, 1999).

\bibitem{floer:lag}
A.~Floer.
\newblock Morse theory for {L}agrangian intersections.
\newblock {\em J. Differential Geom.}, 28(3):513--547, 1988.

\bibitem{floer:monopoles} A. Floer. Monopoles on asymptotically flat 
  manifolds. In: Hofer H., Taubes C.H., Weinstein A., Zehnder E. (eds) 
  The Floer Memorial Volume. Progress in Mathematics, vol 133. 
  Birkhäuser Basel.

\bibitem{fhs:tr}
A.~{F}loer, H.~Hofer, and D.~Salamon.
\newblock Transversality in elliptic {M}orse theory for the symplectic action.
\newblock {\em Duke Math. J.}, 80(1):251--292, 1995.

\bibitem{fu:fl1}
K.~Fukaya.
\newblock {F}loer homology for $3$-manifolds with boundary {I}, 1999.
\newblock unpublished manuscript.

\bibitem{totreal}
U.~Frauenfelder and K.~Zehmisch. 
\newblock Gromov compactness for holomorphic discs with totally real boundary conditions. 
\newblock {\em  J. Fixed Point Theory Appl. }  17  (2015),  no. 3, 521--540. 

\bibitem{fuk:garc}
K.~Fukaya.
\newblock Morse homotopy, {$A\sp \infty$}-category, and {F}loer homologies.
\newblock In {\em Proceedings of {GARC} {W}orkshop on {G}eometry and {T}opology
  '93 ({S}eoul, 1993)}, volume~18 of {\em Lecture Notes Ser.}, pages 1--102,
  Seoul, 1993. Seoul Nat. Univ.

\bibitem{fooo} K.~Fukaya, Y.-G.~Oh, H.~Ohta, and K.~Ono.  \newblock
  {\em Lagrangian intersection {F}loer theory: anomaly and
    obstruction.}, volume~46 of {\em AMS/IP Studies in Advanced
    Mathematics}.  \newblock American Mathematical Society,
  Providence, RI, 2009.  Orientation chapter at
  \href{https://www.math.kyoto-u.ac.jp/~fukaya/bookchap9071113.pdf}{https://www.math.kyoto-u.ac.jp/$\sim$fukaya/bookchap9071113.pdf
    version 2007}.  Surgery chapter at
  \href{https://www.math.kyoto-u.ac.jp/~fukaya/Chapter10071117.pdf}{https://www.math.kyoto-u.ac.jp/~fukaya/Chapter10071117.pdf}.


\bibitem{fooo:toric1}
K.~Fukaya, Y.-G.~Oh, H.~Ohta, and K.~Ono. 
\newblock Lagrangian {F}loer theory on compact toric manifolds. {I}. 
\newblock {\em Duke Math. J.}, 151(1):23--174, 2010. 

\bibitem{fooo:toric2}
K.~Fukaya, Y.-G.~Oh, H.~Ohta, and K.~Ono. 
\newblock Lagrangian {F}loer theory on compact toric manifolds {II}: bulk 
  deformations. 
\newblock {\em Selecta Math. (N.S.)}, 17(3):609--711, 2011.


\bibitem{fooo:anti}
K.~Fukaya, Y.-G.~Oh, H.~Ohta, and K.~Ono. 
\newblock  Anti-symplectic involution and Floer cohomology. 
\newblock {\em  Geom. Topol.} 21 (2017), no. 1, 1–106.
\newblock \href{http://www.arxiv.org/abs/0912.2646}{arxiv:0912.2646}.

\bibitem{ganatra}
S.~Ganatra. 
\newblock 
Symplectic Cohomology and Duality for the 
Wrapped Fukaya Category. 
\newblock PhD Thesis, Massachusetts Institute of Technology, 2006. 

\bibitem{go:in}
W.~M. Goldman.
\newblock Invariant functions on {L}ie groups and {H}amiltonian flows of
  surface group representations.
\newblock {\em Invent. Math.}, 85:263--302, 1986.

\bibitem{go:ss}
R.~Gompf. 
\newblock A new construction of symplectic manifolds. 
\newblock {\em  Annals of Math.} 142:527--595, 1995.

\bibitem{gw:surject}
E.~{Gonz\'alez} and C.~{Woodward}.
\newblock {Quantum cohomology and toric minimal model programs}.
\newblock \href{http://www.arxiv.org/abs/1207.3253}{arXiv:1207.3253}.

\bibitem{gh} P.~Griffiths and J.~Harris.  \newblock Principles of
  algebraic geometry.  \newblock Reprint of the 1978 original.  Wiley
  Classics Library. John Wiley \& Sons, Inc., New York, 1994.

\bibitem{gu:bi}
V.~Guillemin and S.~Sternberg. 
\newblock Birational equivalence in the symplectic category. 
\newblock {\em Invent. Math.}, 97(3):485--522, 1989.

\bibitem{gu:sf}
V. Guillemin, E. Lerman, and S. Sternberg.
\newblock Symplectic fibrations and multiplicity diagrams,
\newblock Cambridge University Press, 1996.

\bibitem{gu:sy}
V.~Guillemin and S.~Sternberg.
\newblock {\em Symplectic Techniques in Physics}.
\newblock Cambridge Univ. Press, Cambridge, 1990.


\bibitem{hacon:flips}
C.~D. Hacon and J.~McKernan.
\newblock Flips and flops.
\newblock In {\em Proceedings of the {I}nternational {C}ongress of
  {M}athematicians. {V}olume {II}}, pages 513--539. Hindustan Book Agency, New
  Delhi, 2010.

\bibitem{hm:sark}
C.~D.~Hacon and J.~McKernan.
\newblock The {S}arkisov program.
\newblock {\em J. Algebraic Geom.}, 22(2):389--405, 2013.

\bibitem{ha:ag}
R.~Hartshorne.
\newblock Algebraic Geometry 
\newblock Graduate Texts in Mathematics, Springer-Verlag, New York,
1977.

\bibitem{ho:sc}
H.~Hofer, K.~Wysocki, and E.~Zehnder.
\newblock $sc$-smoothness, retractions and new models for smooth spaces.
\newblock {\em Discrete and Continuous Dyn. Systems}, 28:665--788, 2010.

\bibitem{ho:mi}
K.~Hori and C.~Vafa. 
\newblock Mirror symmetry. 
\newblock
\href{http://www.arxiv.org/abs/hep-th/0002222}{arxiv:hep-th/0002222}

\bibitem{ht:gl2}
M.~Hutchings and C.~H.~Taubes.
\newblock Gluing pseudoholomorphic curves along branched covered cylinders.
  {II}.
\newblock {\em J. Symplectic Geom.}, 7(1):29--133, 2009.

\bibitem{io:rel}
E.-N.~Ionel and T.~H. Parker.
\newblock Relative {G}romov-{W}itten invariants.
\newblock {\em Ann. of Math. (2)}, 157(1):45--96, 2003.

\bibitem{jw}
L.C.~Jeffrey and J.~Weitsman.
\newblock Toric structures on the moduli space of flat connections on a Riemann
   surface: volumes and the moment map
\newblock {\em Adv. Math.} 106:151--168, 1994.

\bibitem{kaw:der}
Y.~Kawamata.
\newblock Derived categories of toric varieties.
\newblock {\em Michigan Math. J.}, 54(3):517--535, 2006.

\bibitem{ke:le}
G.~Kempf and L.~Ness. 
\newblock The length of vectors in representation spaces. 
\newblock In K.~L{\o}nsted, editor, {\em Algebraic Geometry}, volume 732 of 
  {\em Lecture Notes in Mathematics}, pages 233--244, Copenhagen, 1978, 1979. 
  Springer-Verlag, Berlin-Heidelberg-New York.

\bibitem{ki:coh}
F.~C. Kirwan.
\newblock {\em Cohomology of Quotients in Symplectic and Algebraic Geometry},
  volume~31 of {\em Mathematical Notes}.
\newblock Princeton Univ. Press, Princeton, 1984.

\bibitem{kleiman} S. Kleiman.  \newblock Toward a numerical theory of
  ampleness.  \newblock {\em Ann. of Math.}  (2) 84 (1966), 293--344.

\bibitem{kl:poly}
A.~A. Klyachko.
\newblock Spatial polygons and stable configurations of points in the
  projective line.
\newblock In {\em Algebraic geometry and its applications (Yaroslavl\cprime,
  1992)}, pages 67--84. Vieweg, Braunschweig, 1994.

\bibitem{km:mmp}
J.~Koll{\'a}r and S.~Mori.
\newblock {\em Birational geometry of algebraic varieties}, volume 134 of {\em
  Cambridge Tracts in Mathematics}.
\newblock Cambridge University Press, Cambridge, 1998.
\newblock With the collaboration of C. H. Clemens and A. Corti, Translated from
  the 1998 Japanese original.

\bibitem{kon:hom}
M.~Kontsevich.
\newblock Homological algebra of mirror symmetry.
\newblock In {\em Proceedings of the {I}nternational {C}ongress of
  {M}athematicians, {V}ol.\ 1, 2 ({Z}\"urich, 1994)}, pages 120--139, Basel,
  1995. Birkh\"auser.

\bibitem{km} M.~Kontsevich and Y.~Manin.  Gromov-Witten classes,
  quantum cohomology, and enumerative geometry.  Comm. Math. Phys.
  164, 1994, 525--562.

\bibitem{ks:ainfty}
M.~Kontsevich and Y.~Soibelman.
\newblock Notes on {$A_\infty$}-algebras, {$A_\infty$}-categories and
  non-commutative geometry.
\newblock In {\em Homological mirror symmetry}, volume 757 of {\em Lecture
  Notes in Phys.}, pages 153--219. Springer, Berlin, 2009.



\bibitem{lee:flop}
Y.~Iwao, Y.-P. Lee, H.-W. Lin, and C.-L. Wang.
\newblock Invariance of {G}romov-{W}itten theory under a simple flop.
\newblock {\em J. Reine Angew. Math.}, 663:67--90, 2012.


\bibitem{lee:fmi}
Y.~-P.~Lee, H.~-W.~Lin, and C.~-L.~Wang.
\newblock Flops, motives, and invariance of quantum rings.
\newblock {\em Ann. of Math. (2)}, 172(1):243--290, 2010.

\bibitem{le:ai}
K.~Lef\`evre-Hasegawa.
\newblock {\em Sur les $A_\infty$-cat\'egories}.
\newblock PhD thesis, Universit\'e Paris 7, 2003.

\bibitem{le:sy2}
E.~Lerman.
\newblock Symplectic cuts.
\newblock {\em Math. Res. Letters}, 2:247--258, 1995.

\bibitem{lt}
E.~Lerman and S.~Tolman.
\newblock Hamiltonian torus actions on symplectic orbifolds and toric
  varieties.
\newblock {\em Trans. Amer. Math. Soc.}, 349(10):4201--4230, 1997.


\bibitem{li:disj}
Y.~Li.
\newblock Disjoinable Lagrangian tori and semisimple symplectic cohomology.
\newblock \href{http://www.arxiv.org/abs/1605.04700}{arXiv:1605.04700}.


\bibitem{li:degen}
J.~Li. 
\newblock 
A Degeneration formula of GW-invariants
\newblock {\em J. Differential Geom.} 60--199–293, 2002.

\bibitem{loc:ell}
R.~B. Lockhart and R.~C. McOwen. 
\newblock Elliptic differential operators on noncompact manifolds. 
\newblock {\em Ann. Scuola Norm. Sup. Pisa Cl. Sci. (4)}, 12(3):409--447, 1985. 


\bibitem{ma:so}
C.-M. Marle.
\newblock Sous-vari\'et\'es de rang constant d'une vari\'et\'e symplectique.
\newblock {\em Ast{\'e}risque}, 107--108:69--86, 1983.

\bibitem{me:lo}
E.~Meinrenken and C.~Woodward.
\newblock {H}amiltonian loop group actions and {V}erlinde factorization.
\newblock {\em Journal of Differential Geometry}, 50:417--470, 1999.

\bibitem{me:can}
E.~Meinrenken and C.~Woodward.
\newblock Canonical bundles for {H}amiltonian loop group manifolds.
\newblock {\em Pacific J. Math.}, 198(2):477--487, 2001.

\bibitem{milnor:hcobord} J.~Milnor.  \newblock {\em Lectures on the
  {$h$}-cobordism theorem}.  \newblock Notes by L. Siebenmann and
  J. Sondow. Princeton University Press, Princeton, N.J., 1965.

\bibitem{liruan:surg}
A.-M.~Li and Y.~Ruan.
\newblock Symplectic surgery and {G}romov-{W}itten invariants of {C}alabi-{Y}au
  3-folds.
\newblock {\em Invent. Math.}, 145(1):151--218, 2001.

\bibitem{lotay:coupled}
J.~D. {Lotay} and T.~{Pacini}.
\newblock {Coupled flows, convexity and calibrations: Lagrangian and totally
  real geometry}.
\newblock 1404.4227.

\bibitem{opshtein}
 Opshtein, Emmanuel. 
\newblock Singular polarizations and ellipsoid packings.
\newblock  Int. Math. Res. Not. IMRN  2013,  no. 11, 2568--2600.
		

\bibitem{pacini}
T.~Pacini. 
\newblock Maslov, Chern-Weil and Mean Curvature. 
\href{http://www.arxiv.org/abs/1711.07928}{arxiv:1711.07928}. 


\bibitem{ainfty}
S.~Mau, K.~Wehrheim, and C.T. Woodward.
\newblock ${A}_\infty$-functors for {L}agrangian correspondences.
\newblock in preparation.

\bibitem{mau:mult}
S.~Ma'u and C.~Woodward.
\newblock Geometric realizations of the multiplihedra.
\newblock {\em Compos. Math.}, 146(4):1002--1028, 2010.

\bibitem{mc:ex} D.~McDuff.  \newblock Examples of simply-connected
  symplectic non-Kählerian manifolds.  \newblock {\em J. Differential
    Geom.} 20:267--277, 1984.

\bibitem{mc:di} D.~McDuff.  \newblock Displacing {L}agrangian toric
  fibers via probes.  \newblock {\em Low-dimensional and symplectic
    topology}, 131–-160, Proc. Sympos. Pure Math., 82,
  Amer. Math. Soc., Providence, RI, 2011.

\bibitem{ms:jh}
D.~McDuff and D.~Salamon. 
\newblock {\em {$J$}-holomorphic curves and symplectic topology}, volume~52 of 
  {\em American Mathematical Society Colloquium Publications}. 
\newblock American Mathematical Society, Providence, RI, 2004.

\bibitem{ms:pb}
V.~B. Mehta and C.~S. Seshadri.
\newblock Moduli of vector bundles on curves with parabolic structure.
\newblock {\em Math. Ann.}, 248:205--239, 1980.

\bibitem{me:sym}
E.~Meinrenken.
\newblock Symplectic surgery and the {S}pin$^{\rm c}$-{D}irac operator.
\newblock {\em Adv. in Math.}, 134:240--277, 1998.

\bibitem{km:poly}
M.~Kapovich and J.~J. Millson.
\newblock The symplectic geometry of polygons in {E}uclidean space.
\newblock {\em J. Differential Geom.}, 44(3):479--513, 1996.

\bibitem{moon:birat}
H.-B. {Moon} and S.-B. {Yoo}.
\newblock {Birational geometry of the moduli space of rank 2 parabolic vector
  bundles on a rational curve}.
\newblock  {\em  Int. Math. Res. Not. IMRN 2016}, no. 3, 827–859. 
\newblock \href{http://www.arxiv.org/abs/1409.6263}{arXiv:1409.6263}

\bibitem{mu:ge}
D.~Mumford, J.~Fogarty and F.~Kirwan.
\newblock Geometric Invariant Theory, Third Edition. 
\newblock 1994, Springer-Verlag.

\bibitem{nish:degen}
T.~Nishinou, Y.~Nohara, and K.~Ueda.
\newblock Toric degenerations of {G}elfand-{C}etlin systems and potential
  functions.
\newblock {\em Adv. Math.}, 224(2):648--706, 2010.

\bibitem{no:degen}
Y.~Nohara and K.~Ueda. 
\newblock  Toric degenerations of integrable systems on Grassmannians and 
polygon spaces. 
\newblock  {\em Nagoya Math. J.} 214: 125--168, 2014.

\bibitem{pasq:mmp} B.~Pasquier.  \newblock An approach of the Minimal
  Model Program for horospherical varieties via moment polytopes.
  \newblock {\em J. Reine Angew. Math. 708 (2015)}, 173--212. \newblock
  \href{http://www.arxiv.org/abs/1211.6229v1}{arXiv:1211.6229v1}

\bibitem{oh:fc}
Y.-G. Oh.
\newblock {F}loer cohomology, spectral sequences, and the {M}aslov class of
  {L}agrangian embeddings.
\newblock {\em Internat. Math. Res. Notices}, (7):305--346, 1996.

\bibitem{oh:rh} Y.-G.~Oh. Riemann-Hilbert problem and application to
  the perturbation theory of analytic discs. {\em Kyungpook Math. J.},
  35(1):39–75, 1995.

\bibitem{oh:fl1}
Y.-G. Oh.
\newblock Floer cohomology of {L}agrangian intersections and pseudo-holomorphic
  disks. {I}.
\newblock {\em Comm. Pure Appl. Math.}, 46(7):949--993, 1993.

\bibitem{pw} 
J.~Palmer and C.~Woodward. 
\newblock Immersed Floer cohomology and mean curvature flow. 
\newblock
\href{http://www.arxiv.org/abs/1804.06799}{arXiv:1804.06799}. 

\bibitem{parker:compact}  B.~Parker. \newblock Holomorphic curves in 
  exploded manifolds: compactness.  {\em Adv. Math.}  283, 377--457,
  2015.



\bibitem{po:cl}
M.~Po{\'z}niak. 
\newblock Floer homology, {N}ovikov rings and clean intersections. 
\newblock In {\em Northern {C}alifornia {S}ymplectic {G}eometry {S}eminar},
  volume 196 of {\em Amer. Math. Soc. Transl. Ser. 2}, pages 119--181. Amer. 
  Math. Soc., Providence, RI, 1999. 




\bibitem{reid:decomp}
M.~Reid.
\newblock Decomposition of toric morphisms.
\newblock In {\em Arithmetic and geometry, {V}ol. {II}}, volume~36 of {\em
  Progr. Math.}, pages 395--418. Birkh\"auser Boston, Boston, MA, 1983.

\bibitem{reid:flip}
M.~Reid.
\newblock What is a flip?
\newblock Notes from a Utah seminar 1982, available at
\href{http://homepages.warwick.ac.uk/~masda/3folds/.}
{http://homepages.warwick.ac.uk/$\sim$masda/3folds/.}

\bibitem{ruan:surgery}
Yongbin Ruan.
\newblock Surgery, quantum cohomology and birational geometry.
\newblock In {\em Northern {C}alifornia {S}ymplectic {G}eometry {S}eminar},
  volume 196 of {\em Amer. Math. Soc. Transl. Ser. 2}, pages 183--198. Amer.
  Math. Soc., Providence, RI, 1999.

\bibitem{schmaschke} F.~Schm{\"a}schke.  Floer homology of Lagrangians
  in clean intersection.  \newblock
  \href{http://www.arxiv.org/abs/1606.05327}{arXiv:1606.05327}

\bibitem{sch:coh} M.~Schwarz.  \newblock {\em Cohomology Operations
from $S^1$-Cobordisms in {F}loer Homology}.  \newblock \newblock
\href{http://www.math.uni-leipzig.de/~schwarz/}{PhD thesis}, ETH
Zurich, 1995 

\bibitem{sch:morse}
M.~Schwarz.
\newblock {\em Morse homology}, vol. 111 of {\em Progress in Math}.
\newblock Birkh\"auser Verlag, Basel, 1993.

\bibitem{se:gr}
P.~Seidel.
\newblock Graded {L}agrangian submanifolds.
\newblock {\em Bull. Soc. Math. France}, 128(1):103--149, 2000.

\bibitem{seidel:genustwo}
P.~Seidel. 
\newblock Homological mirror symmetry for the genus two curve. 
\newblock {\em J. Algebraic Geom.} 20:727--769, 2011.


\bibitem{se:ho} P.~Seidel.  \newblock Homological mirror symmetry for 
  the quartic surface.  \newblock Memoirs of the American Mathematical 
  Society, 2015, Volume 236. 
 \newblock \href{http://www.arxiv.org/abs/0310414}{arXiv:0310414}

\bibitem{seidel:sub}
P.~Seidel.
\newblock {$A_\infty$}-subalgebras and natural transformations.
\newblock {\em Homology, Homotopy Appl.}, 10(2):83--114, 2008.


\bibitem{se:bo}
P.~Seidel.
\newblock {\em Fukaya categories and {P}icard-{L}efschetz theory}.
\newblock Zurich Lectures in Advanced Mathematics. European Mathematical
  Society (EMS), Z\"urich, 2008.

\bibitem{seidel:susp}
P.~Seidel.
\newblock Suspending {L}efschetz fibrations, with an application to mirror
  symmetry.
\newblock {\em Comm. Math. Phys.}  297:515--528, 2010.

\bibitem{sh:hmsfano}
N.~Sheridan.
\newblock 
On the Fukaya category of a Fano hypersurface in projective space.
\newblock {\em  Pub. math. de l'IH\'{E}S}, 124(1)165--317, 2016.

\bibitem{sikorav} J.-C.~Sikorav.  Some properties of holomorphic 
  curves in almost complex manifolds. Holomorphic curves in symplectic 
  geometry, 165–189, Progr. Math., 117, Birkhäuser, Basel, 1994. 


\bibitem{smoczyk} K.~Smoczyk.  Lagrangian mean curvature flow. 
Habilitation Thesis, Leipzig, 2001. 
\href{http://service.ifam.uni-hannover.de/~smoczyk/publications/preprint07.pdf}{http://service.ifam.uni-hannover.de/~smoczyk/publications/preprint07.pdf}. 


\bibitem{song:krflow}
J.~Song and G.~Tian. 
\newblock The {K}\"ahler-{R}icci flow on surfaces of positive {K}odaira 
  dimension. 
\newblock {\em Invent. Math.}, 170(3):609--653, 2007.

\bibitem{smith:pq}
I.~Smith.
\newblock Floer cohomology and pencils of quadrics.
\newblock {\em Invent. Math.}, 189(1):149--250, 2012.

\bibitem{so:st}
J.-M. Souriau.
\newblock {\em Structure des syst\`emes dynamiques}.
\newblock Dunod, Paris, 1970.

\bibitem{st:ho}
J.~Stasheff.
\newblock {\em {$H$}-spaces from a homotopy point of view}.
\newblock Lecture Notes in Mathematics, Vol. 161. 
Springer-Verlag, Berlin,
  1970.

\bibitem{tz}
M.~Tehrani and A.~Zinger. 
\newblock On Symplectic Sum Formulas in Gromov-Witten Theory.
\newblock \href{http://www.arxiv.org/abs/1404.1898}{arXiv:1404.1898}

\bibitem{tr:sph}
T.~Treloar.
\newblock The symplectic geometry of polygons in the three-sphere.
\newblock {\em Canad. J. Math.} (54):30--54, 2002.

\bibitem{th:fl}
M.~Thaddeus.
\newblock Geometric invariant theory and flips.
\newblock {\em J. Amer. Math. Soc.}, 9(3):691--723, 1996.

\bibitem{vwx} S. Venugopalan, C. Woodward, and G. Xu. Fukaya 
  categories of blowups. 71 pages. 
  \href{http://arxiv.org/abs/2006.12264.}
  {http://arxiv.org/abs/2006.12264.}

\bibitem{orient}
K.~Wehrheim and C.T. Woodward. 
\newblock Orientations for pseudoholomorphic quilts. 
\newblock \href{http://www.arxiv.org/abs/1503.07803}{arXiv:1503.07803}.


\bibitem{wo:gdisk}
C.T.~Woodward.
\newblock Gauged {F}loer theory of toric moment fibers.
\newblock {\em Geom. and Func. Anal.}, 21:680--749, 2011.

\bibitem{wx:partly} 
G.~Xu and C.~T.~Woodward.
\newblock
Partly-local domain-dependent almost complex structures.
\newblock
\href{http://www.arxiv.org/abs/1903.05557}{arXiv:1903.05557}.
\end{thebibliography}
\def\cprime{$'$} \def\cprime{$'$} \def\cprime{$'$} \def\cprime{$'$}
\def\cprime{$'$} \def\cprime{$'$}
\def\polhk#1{\setbox0=\hbox{#1}{\ooalign{\hidewidth 
      \lower1.5ex\hbox{`}\hidewidth\crcr\unhbox0}}} \def\cprime{$'$}
\def\cprime{$'$} \def\cprime{$'$} \def\cprime{$'$}

\printindex
\end{document}